\documentclass[11pt,twoside]{memo-l}

\usepackage{amssymb, enumerate, afterpage, float, url, mparhack}

\usepackage{graphics}

\theoremstyle{plain}
\newtheorem{thm}{Theorem}[chapter]
\newtheorem*{thm1.1}{Theorem~\ref{thm2}}
\newtheorem{cor}[thm]{Corollary}
\newtheorem{lem}[thm]{Lemma}
\newtheorem{prop}[thm]{Proposition}

\theoremstyle{definition}
\newtheorem{clm}[thm]{Claim}
\newtheorem{subclm}{Claim}[thm]
\newtheorem{fct}{Fact}
\newtheorem*{case1}{Case 1}
\newtheorem*{subcase1.1}{Subcase 1.1}
\newtheorem*{subcase1.2}{Subcase 1.2}
\newtheorem*{subsubcase1.2.1}{Sub-subcase 1.2.1}
\newtheorem*{subsubcase1.2.2}{Sub-subcase 1.2.2}
\newtheorem*{case2}{Case 2}
\newtheorem*{subcase2.1}{Subcase 2.1}
\newtheorem*{subcase2.2}{Subcase 2.2}
\newtheorem*{subcase2.3}{Subcase 2.3}
\newtheorem*{subcase2.4}{Subcase 2.4}
\newtheorem*{subcase2.5}{Subcase 2.5}
\newtheorem*{subcase2.6}{Subcase 2.6}
\newtheorem*{subcase2.7}{Subcase 2.7}
\newtheorem*{subcase2.8}{Subcase 2.8}
\newtheorem*{subcase2.9}{Subcase 2.9}
\newtheorem*{case3}{Case 3}
\newtheorem*{subcase3.1}{Subcase 3.1}
\newtheorem*{subcase3.2}{Subcase 3.2}
\newtheorem*{subcase3.3}{Subcase 3.3}
\newtheorem*{subcase3.4}{Subcase 3.4}
\newtheorem*{subsubcase3.4.1}{Sub-subcase 3.4.1}
\newtheorem*{subsubcase3.4.2}{Sub-subcase 3.4.2}
\newtheorem*{subcase3.5}{Subcase 3.5}
\newtheorem*{subcase3.6}{Subcase 3.6}
\newtheorem*{subcase3.7}{Subcase 3.7}
\newtheorem*{subcase3.8}{Subcase 3.8}
\newtheorem*{subcase3.9}{Subcase 3.9}

\theoremstyle{plain}
\newtheorem{sub}[subclm]{}

\newcommand{\cl}{\operatorname{cl}}
\newcommand{\co}{\operatorname{co}}
\newcommand{\si}{\operatorname{si}}
\newcommand{\inter}{\operatorname{int}}

\newcommand{\mob}{M\"{o}bius}
\newcommand{\iso}{\cong}
\newcommand{\del}{\backslash}
\newcommand{\dy}{\ensuremath{\Delta\textrm{-}Y}}
\newcommand{\yd}{\ensuremath{Y\textrm{-}\Delta}}
\newcommand{\mkt}{\ensuremath{M(K_{3,3})}}
\newcommand{\mkf}{\ensuremath{M(K_{5})}}
\newcommand{\ifc}{internally $4$\nobreakdash-\hspace{0pt}connected}
\newcommand{\vfc}{vertically $4$\nobreakdash-\hspace{0pt}connected}
\newcommand{\vtc}{vertically $3$\nobreakdash-\hspace{0pt}connected}
\newcommand{\vts}{vertical $3$\nobreakdash-\hspace{0pt}separation}
\newcommand{\vks}{vertical $k$\nobreakdash-\hspace{0pt}separation}
\newcommand{\dash}{\nobreakdash-\hspace{0pt}}
\newcommand{\mdash}{\nobreakdash--\hspace{0pt}}

\newcommand{\mcal}[1]{\ensuremath{\mathcal{#1}}}

\newcommand{\tup}[1]{\textup{#1}}
\newcommand{\ex}[1]{\ensuremath{\mathcal{EX}(#1)}}
\newcommand{\pg}[1]{\ensuremath{\mathrm{PG}(#1)}}
\newcommand{\gf}[1]{\ensuremath{\mathrm{GF}(#1)}}
\newcommand{\bc}[1]{\ensuremath{\mathrm{bc}(#1)}}
\newcommand{\ov}[1]{\ensuremath{\overline{#1}}}
\newcommand{\cml}[1]{\ensuremath{\mathit{CM}_{\hspace{-1.5pt}#1}}}
\newcommand{\qml}[1]{\ensuremath{\mathit{QM}_{\hspace{-1.5pt}#1}}}

\restylefloat{figure}
\restylefloat{table}

\numberwithin{section}{chapter}
\numberwithin{equation}{chapter}

\newcounter{cross}

\newcommand{\mpar}[1]
{\marginpar[\footnotesize \raggedleft{#1}]{\footnotesize \raggedright{#1}}}

\newcommand{\cross}{\addtocounter{cross}{1}\mpar{\maltese\thecross}}

\begin{document}

\frontmatter

\title[Binary Matroids With No \protect\mkt\protect\dash Minor.]
{The Internally $4$\dash Connected Binary Matroids With
No \protect\mkt\protect\dash Minor.}

\author{Dillon Mayhew}
\address{School of Mathematics, Statistics, and Operations Research\\
Victoria University of Wellington\\
P.O. BOX 600\\
Wellington\\
New Zealand.}
\email{dillon.mayhew@msor.vuw.ac.nz}
\thanks{The first author was supported by a
NZ Science \& Technology Post-doctoral fellowship.}

\author{Gordon Royle}
\address{School of Mathematics and Statistics\\
The University of Western Australia\\
35 Stirling Highway\\
Crawley 6009\\
Western Australia.}
\email{gordon@maths.uwa.edu.au}

\author{Geoff Whittle}
\address{School of Mathematics, Statistics, and Operations Research\\
Victoria University of Wellington\\
P.O. BOX 600\\
Wellington\\
New Zealand.}
\email{geoff.whittle@msor.vuw.ac.nz}
\thanks{The third author was supported by the Marsden Fund of New Zealand.}

\date{}



\begin{abstract}
We give a characterization of the \ifc\ binary matroids
that have no minor isomorphic to \mkt.
Any such matroid is either cographic, or is isomorphic to a
particular single-element extension of the bond matroid
of a cubic or quartic \mob\ ladder, or is isomorphic to
one of eighteen sporadic matroids.
\end{abstract}

\maketitle

\tableofcontents

\mainmatter

\chapter{Introduction}
\label{chp1}

The Fano plane, the cycle matroids of the Kuratowski graphs,
and their duals, are of fundamental importance in the study of
binary matroids.
The famous excluded-minor characterizations of Tutte~\cite{Tut58}
show that the classes of regular, graphic, and cographic
matroids are all obtained by taking the binary matroids with
no minors in some subset of the family
\begin{displaymath}
\{F_{7},\, F_{7}^{*},\, \mkt,\, \mkf,\, M^{*}(K_{3,3}),\, M^{*}(K_{5})\}.
\end{displaymath}
It is natural to consider other classes of binary matroids
produced by excluding some subset of this family.
A number of authors have investigated such classes of binary matroids.
We examine the binary matroids with no
minor isomorphic to \mkt, and completely characterize
the \ifc\ members of this class.

\begin{thm}
\label{thm2}
An \ifc\ binary matroid has no \mkt\dash minor if and only
if it is either:
\begin{enumerate}[(i)]
\item cographic;
\item isomorphic to a triangular or triadic \mob\ matroid; or,
\item isomorphic to one of~$18$ sporadic matroids of rank at
most~$11$ listed in Appendix~{\rm\ref{chp9}}.
\end{enumerate}
\end{thm}

The triangular and triadic \mob\ matroids form two infinite
families of binary matroids.
Each \mob\ matroid is a single-element extension of a cographic
matroid; in particular, a single-element extension of the bond
matroid of a cubic or quartic \mob\ ladder.
For every integer $r \geq 3$ there is a unique
triangular \mob\ matroid of rank~$r$, denoted by $\Delta_{r}$, and for
every even integer $r \geq 4$ there is a unique triadic \mob\
matroid of rank~$r$, denoted by $\Upsilon_{r}$.
We describe these matroids in detail in Chapter~\ref{chp2}.

Seymour's decomposition theorem for regular matroids is one of
the most fundamental results in matroid theory, and
was the first structural decomposition theorem
proved for a class of matroids.
The following characterization of the \ifc\ regular matroids
is an immediate consequence of the decomposition theorem.

\begin{thm}[Seymour~\cite{Sey80}]
\label{thm1}
An \ifc\ regular matroid is either graphic, cographic, or
isomorphic to a particular sporadic matroid ($R_{10}$).
\end{thm}

Our theorem, and its proof, bears some similarity to
Seymour's theorem.
Just as Seymour's theorem leads to a polynomial-time
algorithm for deciding whether a binary matroid (represented
by a matrix over \gf{2}) has a minor isomorphic to $F_{7}$ or $F_{7}^{*}$,
our characterization leads to a polynomial-time algorithm
(to be discussed in a subsequent article) for deciding whether
a represented binary matroid has an \mkt\dash minor.

Other authors who have studied families of binary matroids
with no minors in some subset of
$\{F_{7},\, F_{7}^{*},\, \mkt,\, \mkf,\, M^{*}(K_{3,3}),\, M^{*}(K_{5})\}$
include Walton and Welsh~\cite{WW80}, who examine the characteristic
polynomials of matroids in such classes, and Kung~\cite{Kun86}, who
has considered the maximum size obtained by a simple rank\dash $r$
matroid in one of these classes.
Qin and Zhou~\cite{QZ04} have characterized the \ifc\ binary matroids
with no minor in
$\{\mkt,\, \mkf,\, M^{*}(K_{3,3}),\, M^{*}(K_{5})\}$.
Moreover Zhou~\cite{Zho08} has also characterized the \ifc\ binary
matroids that have no \mkt\dash minor, but which do have an
\mkf\dash minor.
Finally we note that the classic result of Hall~\cite{Hal43}
on graphs with no $K_{3,3}$\dash minor leads to a characterization
of the \ifc\ binary matroids with no minor in
$\{F_{7},\, F_{7}^{*},\, M^{*}(K_{3,3}),\, M^{*}(K_{5}),\, M(K_{3,3})\}$, and
Wagner's~\cite{Wag37} theorem on the graphs with no
$K_{5}$\dash minor leads to a characterization of the \ifc\ binary
matroids with no minor in
$\{F_{7},\, F_{7}^{*},\, M^{*}(K_{3,3}),\, M^{*}(K_{5}),\, M(K_{5})\}$.

The proof of Theorem~\ref{thm2} is unusual amongst results in
matroid theory, in that we rely upon a computer to check certain facts.
All these checks have been carried out using the software
package \textsc{Macek}, developed by Petr Hlin{\v{e}}n{\'y}.
In addition we have written software, which does not depend
upon \textsc{Macek}, to provide us with an independent check
of the same facts.

The \textsc{Macek} package is available to download, along with
supporting documentation.
The current website is \url{http://www.fi.muni.cz/~hlineny/MACEK}.
Thus the interested reader is able to download \textsc{Macek}
and confirm that it verifies our assertions.
Points in the proof where a computer check is required
are marked by a marginal symbol {\footnotesize\maltese} and a number.
The numbers provide a reference for the website of the
second author
(current url: \url{http://www.maths.uwa.edu.au/~gordon}), which
contains a more detailed description of the steps taken to verify
each assertion, and the intermediate data produced during that
computer check.

We emphasize that none of the computer tests relies upon an
exhaustive search of all the matroids of a particular size
or rank.
The only tasks a program need perform to verify our checks
are: determine whether two binary matroids are isomorphic; check
whether a binary matroid has a particular minor; and, generate
all the single-element extensions and coextensions of a binary
matroid.
Whenever the proof requires a computer check, the text includes
a complete description (independent of any particular piece of
software) of the tasks that we ask the computer to perform.
Hence a reader who is able to construct software with the
capabilities listed above could provide another verification of
our assertions.

The article is organized as follows:
In Chapter~\ref{chp3} we develop the basic definitions and results
we will need to prove Theorem~\ref{thm2}.
In Chapter~\ref{chp2} we introduce the \mob\ matroids and
consider their properties in detail.

Chapter~\ref{chp4} is concerned with showing that a minimal
counterexample to Theorem~\ref{thm2} can be assumed to be \vfc.
This chapter depends heavily upon the \dy\ operation and its
dual; we make extensive use of results proved by Oxley, Semple, and
Vertigan~\cite{OSV00}.
The central idea of Chapter~\ref{chp4} is that if a binary matroid
$M$ with no \mkt\dash minor is non-cographic and \ifc\ but not \vfc,
then by repeatedly performing \yd\ operations, we can produce a
\vfc\ non-cographic binary matroid $M'$ with no \mkt\dash minor.
In Lemmas~\ref{lem9} and~\ref{lem5} we show that if $M'$ obeys
Theorem~\ref{thm2}, then $M$ also satisfies the theorem.
From this it follows that a minimal counterexample to
Theorem~\ref{thm2} can be assumed to be \vfc.

The regular matroid $R_{12}$ was introduced by Seymour in the
proof of his decomposition theorem.
He shows that $R_{12}$ contains a $3$\dash separation, and that
this $3$\dash separation persists in any regular matroid that
contains $R_{12}$ as a minor.
In Chapter~\ref{chp5} we introduce a binary matroid $\Delta_{4}^{+}$
that plays a similar role in our proof.
The matroid $\Delta_{4}^{+}$ is a single-element coextension of
$\Delta_{4}$, the rank\dash $4$ triangular \mob\ matroid.
It contains a four-element circuit-cocircuit, which
necessarily induces a $3$\dash separation.
We show that this $3$\dash separation persists in any binary
matroid without an \mkt\dash minor that has $\Delta_{4}^{+}$ as a minor.
Hence no \ifc\ binary matroid without an \mkt\dash minor can
have $\Delta_{4}^{+}$ as a minor.
In Corollary~\ref{cor8} we show that if $M$ is a
$3$\dash connected binary matroid without an \mkt\dash minor
such that $M$ has both a $\Delta_{4}$\dash minor and a four-element
circuit-cocircuit, then $M$ has $\Delta_{4}^{+}$ as a minor.

Suppose that $M$ is a minimal counterexample to Theorem~\ref{thm2}.
It follows easily from a result of Zhou~\cite{Zho04} that $M$ must
have a minor isomorphic to $\Delta_{4}$.
Hence we deduce that if $M'$ is a $3$\dash connected minor of $M$,
and $M'$ has a $\Delta_{4}$\dash minor, then $M'$ has no four-element
circuit-cocircuit.
This is one of the conditions required to apply the connectivity
lemma that we prove in Chapter~\ref{chp6}.

The hypotheses of the connectivity result in Chapter~\ref{chp6}
are that $M$ and $N$ are simple \vfc\ binary matroids such that
$|E(N)| \geq 10$ and $M$ has a proper $N$\dash minor.
Moreover, whenever $M'$ is a $3$\dash connected minor of $M$ with an
$N$\dash minor, then $M'$ has no four-element circuit-cocircuit.
Under these conditions, Theorem~\ref{thm4} asserts that
$M$ has a \ifc\ proper minor $M_{0}$ such that $M_{0}$
has an $N$\dash minor and $|E(M)| - |E(M_{0})| \leq 4$.

The case-checking required to complete our proof would be
impossible if we had no more information than that provided
by Theorem~\ref{thm4}.
Lemma~\ref{lem3} provides a much more fine-grained analysis.
It shows that there are nine very specific ways in which $M_{0}$
can be obtained from $M$.

In Chapter~\ref{chp7} we complete the proof of Theorem~\ref{thm2}.
Our strategy is to assume that $M$ is a \vfc\ minimal
counterexample to the theorem.
We then apply Lemma~\ref{lem3}, and deduce the presence of
a non-cographic proper minor $M_{0}$ of $M$ that must necessarily
obey Theorem~\ref{thm2}.
The rest of the proof consists of a case-check to show that
the counterexample $M$ cannot be produced by reversing the nine
procedures detailed by Lemma~\ref{lem3} and applying them to
the \mob\ matroids and the $18$ sporadic matroids.
Thus a counterexample to Theorem~\ref{thm2} cannot exist.

The first of the three appendices contains a description of
the case-checking required to complete the proof that a
$3$\dash connected binary matroid with both a
$\Delta_{4}$\dash minor and a four-element circuit-cocircuit
has a $\Delta_{4}^{+}$\dash minor.
The second appendix describes the sporadic matroids,
and the third contains information on the three-element
circuits of the sporadic matroids.

In a subsequent article~\cite{MRWb} we consider various
applications of Theorem~\ref{thm2}.
In particular, we consider the classes of binary matroids
produced by excluding any subset of
$\{\mkt,\, \mkf,\, M^{*}(K_{3,3}),\, M^{*}(K_{5})\}$ that
contains either \mkt\ or $M^{*}(K_{3,3})$.
We characterize the \ifc\ members of these classes, and
show that each such class has a polynomial-time recognition algorithm
(where the input consists of a matrix with entries from \gf{2}).
We also consider extremal results and the critical
exponent for these classes.

\chapter{Preliminaries}
\label{chp3}

In this chapter we define basic ideas and develop the fundamental
tools we will need to prove our main result.
Terminology and notation will generally follow that of
Oxley~\cite{Oxl92}.
A \emph{triangle}\index{triangle} is a three-element circuit and a
\emph{triad}\index{triad} is a three-element cocircuit.
We denote the simple matroid canonically associated with
a matroid $M$ by $\si(M)$, and we similarly denote the
canonically associated cosimple matroid by $\co(M)$.
Suppose that $\{M_{1},\ldots, M_{t}\}$ is a collection of binary
matroids.
We denote the class of binary matroids with no minor isomorphic to
one of the matroids in $\{M_{1},\ldots, M_{t}\}$ by
\ex{M_{1},\ldots, M_{t}}.

\section{Connectivity}
\label{chp3.1}

Suppose $M$ is a matroid on the ground set $E$.
The function $\lambda_{M}$, known as the
\emph{connectivity function}\index{connectivity function} of $M$,
takes subsets of $E$ to non-negative integers.
If $X$ is a subset of $E$, then $\lambda_{M}(X)$ (or
$\lambda(X)$ where there is no ambiguity) is defined to
be $r_{M}(X) + r_{M}(E - X) - r(M)$.
Note that $\lambda$ is a symmetric function: that is,
$\lambda(X) = \lambda(E - X)$.
It is well known (and easy to confirm) that
the function $\lambda_{M}$ is submodular, which is
to say that $\lambda_{M}(X) + \lambda_{M}(Y) \geq
\lambda_{M}(X \cup Y) + \lambda_{M}(X \cap Y)$ for
all subsets $X$ and $Y$ of $E(M)$.
Moreover, if $N$ is a minor of $M$ using the subset $X$, then
$\lambda_{N}(X) \leq \lambda_{M}(X)$.

A \emph{$k$\dash separation}\index{separation!$k$\dash} is a
partition $(X,\, Y)$ of $E$ such that $\min\{|X|,\, |Y|\} \geq k$ and
$\lambda(X) = \lambda(Y) < k$.
The subset $X \subseteq E$ is a
\emph{$k$\dash separator}\index{separator!$k$\dash}
if $(X,\, E - X)$ is a $k$\dash separation.
(Note that this definition of $k$\dash separators
differs from that used in~\cite{OSW04}.)
A $k$\dash separation $(X,\, Y)$ is
\emph{exact}\index{separation!exact $k$\dash} if
$\lambda(X) = \lambda(Y) = k - 1$.
A \emph{vertical $k$\dash separation}\index{separation!vertical $k$\dash}
is a $k$\dash separation $(X,\, Y)$ such that
$\min\{r(X),\, r(Y)\} \geq k$, and a subset $X \subseteq E$ is a
\emph{vertical $k$\dash separator}\index{separator!vertical $k$\dash}
if $(X,\, E - X)$ is a \vks.

A matroid is \emph{$n$\dash connected}\index{connectivity!$n$\dash}
if it has no $k$\dash separations where $k < n$.
It is
\emph{vertically $n$\dash connected}\index{connectivity!vertical $n$\dash}
if it has no vertical $k$\dash separations where $k < n$.
It is
\emph{$(n,\, k)$\dash connected}\index{connectivity!$(n,\, k)$\dash}
if it is $(n - 1)$\dash connected and whenever $(X,\, Y)$ is an
$(n - 1)$\dash separation, then either $|X| \leq k$ or $|Y| \leq k$.

In addition, we shall say that a matroid $M$ is
\emph{almost \vfc}\index{connectivity!almost vertical $4$\dash}
if it is vertically $3$\dash connected, and whenever $(X,\, Y)$ is
a \vts\ of $M$, then there
exists a triad $T$ such that either $T \subseteq X$
and $X \subseteq \cl_{M}(T)$, or $T \subseteq Y$ and
$Y \subseteq \cl_{M}(T)$.

We shall use the notion of
$(n,\, k)$\dash connectivity only in the case
$n = 4$.
A $(4,\, 3)$\dash connected matroid is
\emph{\ifc}\index{connectivity!internal $4$\dash}.

The next results collect some easily-confirmed properties
of the different types of connectivity.

\begin{prop}
\label{prop23}
\begin{enumerate}[(i)]
\item A simple binary \vfc\ matroid is \ifc.
\item A \vfc\ matroid with rank at least four has no triads.
\item Both \vfc\ and \ifc\ matroids are also almost \vfc.
\item An almost \vfc\ matroid with no triads is \vfc.
\end{enumerate}
\end{prop}

\begin{prop}
\label{prop7}
Suppose that $(X_{1},\, X_{2})$ is a $k$\dash separation
of the matroid $M$, and that $N$ is a minor of $M$.
If $|E(N) \cap X_{i}| \geq k$ for $i = 1,\, 2$, then
$(E(N) \cap X_{1},\, E(N) \cap X_{2})$ is a
$k$\dash separation of $N$.
\end{prop}

\begin{prop}
\label{prop40}
Suppose that $(X,\, Y)$ is a \vks\ of the matroid $M$ and
that $e \in Y$.
If $e \in \cl_{M}(X)$, then $(X \cup e,\, Y - e)$ is a
vertical $k'$\dash separation of $M$, where
$k' \in \{k,\, k-1\}$.
\end{prop}

\begin{prop}
\label{prop34}
Let $e$ be a non-coloop element of the matroid $M$.
Suppose that $(X,\, Y)$ is a $k$\dash separation of $M \del e$.
Then $(X \cup e,\, Y)$ is a $k$\dash separation of $M$ if and
only if  $e \in \cl_{M}(X)$ and $(X,\, Y \cup e)$ is a
$k$\dash separation of $M$ if and only if $e \in \cl_{M}(Y)$.
\end{prop}

The following result is due to Bixby~\cite{Bix82}.

\begin{prop}
\label{prop16}
Let $e$ be an element of the $3$\dash connected matroid
$M$.
Then either $\si(M / e)$ or $\co(M \del e)$ is $3$\dash connected.
\end{prop}

Suppose that $(e_{1},\ldots, e_{t})$ is an ordered sequence of
at least three elements from the matroid $M$.
Then $(e_{1},\ldots, e_{t})$ is a \emph{fan}\index{fan} if
$\{e_{i},\, e_{i+1},\, e_{i+2}\}$ is a triangle of $M$
whenever $i \in \{1,\ldots, t-2\}$ is odd and a triad
whenever $i$ is even.
Dually, $(e_{1},\ldots, e_{t})$ is a \emph{cofan}\index{cofan}
if $\{e_{i},\, e_{i+1},\, e_{i+2}\}$ is a triad of $M$
whenever $i \in \{1,\ldots, t-2\}$ is odd and a triangle
whenever $i$ is even.
Note that if $(e_{1},\ldots, e_{t})$ is a fan and $t$ is even,
then $(e_{t},\ldots, e_{1})$ is a cofan.
We shall say that the unordered set $X$
is a fan if there is some ordering $(e_{1},\ldots, e_{t})$
of the elements of $X$ such that $(e_{1},\ldots, e_{t})$
is either a fan or a cofan.
The \emph{length}\index{fan!length} of a fan $X$ is the cardinality of $X$.
It is straightforward to check that if $\{e_{1},\ldots, e_{t}\}$
is a fan of $M$, then $\lambda_{M}(\{e_{1},\ldots, e_{t}\}) \leq 2$.
The next result is easy to confirm.

\begin{prop}
\label{prop33}
Suppose that $(X,\, Y)$ is a $3$\dash separation of the
$3$\dash connected binary matroid $M$.
If $|X| \leq 5$ and $r_{M}(X) \geq 3$ then one of the
following holds:
\begin{enumerate}[(i)]
\item $X$ is a triad;
\item $X$ is a fan with length four or five; or,
\item There is a four-element circuit-cocircuit
$C^{*} \subseteq X$ such that $X \subseteq \cl_{M}(C^{*})$
or $X \subseteq \cl_{M}^{*}(C^{*})$.
\end{enumerate}
\end{prop}

The next result follows directly from a theorem of
Oxley~\cite[Theorem~3.6]{Oxl87a}.

\begin{lem}
\label{lem8}
Let $T$ be a triangle of a $3$\dash connected binary
matroid $M$.
If the rank and corank of $M$ are at least three then
$M$ has an $M(K_{4})$\dash minor in which $T$ is a triangle. 
\end{lem}

\section{Fundamental graphs}
\label{chp2.1}

Fundamental graphs provide a convenient way to
visualize a representation of a binary matroid
with respect to a particular basis.
Suppose that $B$ is a basis of the binary matroid $M$.
The
\emph{fundamental graph}\index{fundamental graph} of $M$ with
respect to $B$, denoted by $G_{B}(M)$, has $E(M)$ as its vertex set.
Every edge of $G_{B}(M)$ joins an element in $B$ to
an element in $E(M) - B$.
If $x \in B$ and $y \in E(M) - B$, then $x$ and $y$ are
joined by an edge if and only if $x$ is in the unique
circuit contained in $B \cup y$.
Equivalently, $G_{B}(M)$ is the bipartite graph represented
by the bipartite adjacency matrix $A$, where $M$ is represented over
\gf{2} by the matrix $[I_{r}|A]$ (assuming that the
columns of $I_{r}$ are labeled with the elements of $B$
while the columns of $A$ are labeled by the elements of
$E(M) - B$).
Thus a labeled fundamental graph of a binary matroid
completely determines that matroid.
By convention, the elements of $B$ are represented by solid
vertices in drawings of $G_{B}(M)$, while the elements of
$E(M) - B$ are represented by hollow vertices.

Suppose that $M$ is a binary matroid on the ground set $E$ and
$B$ is a basis of $M$.
Suppose also that $X \subseteq B$ and $Y \subseteq E - B$.
It is easy to see that the fundamental graph
$G_{B - X}(M / X \del Y)$ is equal to the subgraph of
$G_{B}(M)$ induced by $E - (X \cup Y)$.

If $x \in B$ and $y \in E - B$ and $x$ and $y$ are adjacent
in $G_{B}(M)$, then the fundamental graph 
$G_{(B - x) \cup y}(M)$ is obtained by
\emph{pivoting}\index{pivoting on an edge} on the edge $xy$.
In particular, if
\begin{displaymath}
N_{G}(u) = \{v \in V(G) - u \mid v\ \mbox{is adjacent to}\ u\}
\end{displaymath}
is the set of neighbors of $u$ in the graph $G$, then
$G_{(B-x) \cup y}(M)$ is obtained from $G_{B}(M)$ by replacing every
edge between $N_{G_{B}(M)}(x)$ and $N_{G_{B}(M)}(y)$ with a non-edge,
every non-edge between $N_{G_{B}(M)}(x)$ and $N_{G_{B}(M)}(y)$ with an
edge, and then switching the labels on $x$ and $y$.

\section{Blocking sequences}
\label{chp3.2}

Suppose that $B$ is a basis of the matroid $M$ and let $X$ be a
subset of $E(M)$. Define $M[X,\, B]$ to be the minor of $M$
on the set $X$ produced by contracting $B - X$ and deleting
$E(M) - (B \cup X)$. Let $X$ and $Y$ be disjoint subsets of
$E(M)$. It is easy to verify that
$\lambda_{M[X \cup Y,\, B]}(X)$ is equal to the parameter
$\lambda_{B}(X,\, Y)$, defined by Geelen, Gerards, and
Kapoor~\cite{GGK00}.

Suppose $(X,\, Y)$ is a $k$\dash separation of
$M[X \cup Y,\, B]$.
A \emph{blocking sequence}\index{blocking sequence} of $(X,\, Y)$
is a sequence $e_{1},\ldots, e_{t}$ of elements from
$E(M) - (X \cup Y)$ such that:
\begin{enumerate}[(i)]
\item $(X,\, Y \cup e_{1})$ is not a $k$\dash separation
of $M[X \cup Y \cup e_{1},\, B]$;
\item $(X \cup e_{i},\, Y \cup e_{i+1})$ is not a
$k$\dash separation of $M[X \cup Y \cup \{e_{i},\, e_{i+1}\},\, B]$
for $1 \leq i \leq t - 1$;
\item $(X \cup e_{t},\, Y)$ is not a $k$\dash separation
of $M[X \cup Y \cup e_{t},\, B]$; and,
\item no proper subsequence of $e_{1},\ldots,e_{t}$
satisfies conditions (i)\mdash (iii).
\end{enumerate}

Blocking sequences were developed by Truemper~\cite{Tru86},
and by Bouchet, Cunningham, and Geelen~\cite{BCG98}.
They were later employed by Geelen, Gerards, and Kapoor in their
characterization of the excluded-minors for
\gf{4}\dash representability~\cite{GGK00}.

Let $(X,\, Y)$ be a $k$\dash separation
of $M[X \cup Y,\, B]$.
We say that $(X,\, Y)$
\emph{induces}\index{separation!induced $k$\dash}
a $k$\dash separation of $M$ if there is a $k$\dash separation
$(X',\, Y')$ of $M$ such that $X \subseteq X'$ and
$Y \subseteq Y'$.

\begin{lem}
\label{lem6}
\tup{\cite[Theorem~4.14]{GGK00}}
Suppose that $B$ is a basis of the matroid $M$ and that
$(X,\, Y)$ is an exact $k$\dash separation of
$M[X \cup Y,\, B]$. Then $(X,\, Y)$ fails to induce a
$k$\dash separation of $M$ if and only if there is a
blocking sequence of $(X,\, Y)$.
\end{lem}

Halfan and Zhou use Proposition~4.15~(i) and~(iii)
of~\cite{GGK00} to prove the following
result (\cite[Proposition~4.5~(3)]{Hal02}
and~\cite[Lemma~3.5~(3)]{Zho04} respectively).

\begin{prop}
\label{prop11}
Let $B$ be a basis of a matroid $M$. Let $(X,\, Y)$
be an exact $k$\dash separation of $M[X \cup Y,\, B]$
and let $e_{1},\ldots, e_{t}$ be a blocking sequence
of $(X,\, Y)$.
Suppose that $|X| > k$ and that $e_{1}$ is either in parallel or
in series to $x \in X$ in $M[X \cup Y \cup e_{1},\, B]$ where
$x \notin \cl_{M[X \cup Y,\, B]}(Y)$ and
$x \notin \cl^{*}_{M[X \cup Y,\, B]}(Y)$.
If $\{e_{1},\, x\} \subseteq B$ or if
$\{e_{1},\, x\} \cap B = \varnothing$ then let $B' = B$, and otherwise
let $B'$ be the symmetric difference of $B$ and $\{x,\, e_{1}\}$.
Then $M[(X - x) \cup Y \cup e_{1},\, B'] \iso
M[X \cup Y,\, B]$ and $e_{2},\ldots,e_{t}$ is a
blocking sequence of the $k$\dash separation
$((X - x) \cup e_{1},\, Y)$ in
$M[(X - x) \cup Y \cup e_{1},\, B']$.
\end{prop}

\begin{prop}
\label{prop12}
Suppose that $N$ is a $3$\dash connected matroid such
that $|E(N)| \geq 8$ and $N$ contains a four-element
circuit-cocircuit $X$.
If $M$ is an \ifc\ matroid with an $N$\dash minor, then
there exists a $3$\dash connected single-element extension
or coextension $N'$ of $N$, such that $X$ is not a
circuit-cocircuit of $N'$ and $N'$ is a minor of $M$.
\end{prop}

\begin{proof}
Let $B$ be a basis of $M$ and let $X$ and $Y$ be disjoint
subsets of $E(M)$ such that $M[X \cup Y,\, B] \iso N$, where
$X$ is a four-element circuit-cocircuit of
$M[X \cup Y,\, B]$.
Thus $(X,\, Y)$ is an exact $3$\dash separation of
$M[X \cup Y,\, B]$.
Since $|X|,\, |Y| \geq 4$ and $M$ is \ifc\ it cannot be the
case that $(X,\, Y)$ induces a $3$\dash separation of $M$.
Lemma~\ref{lem6} implies that there is a
blocking sequence $e_{1},\ldots, e_{t}$ of $(X,\, Y)$.
Let us suppose that $B$, $X$, and $Y$ have been chosen so
that $t$ is as small as possible.

Since $(X,\, Y \cup e_{1})$ is not a $3$\dash separation
of $M[X \cup Y \cup e_{1},\, B]$ it follows that $X$ is
not a circuit-cocircuit in $M[X \cup Y \cup e_{1},\, B]$.
Thus if $M[X \cup Y \cup e_{1},\, B]$ is
$3$\dash connected there is nothing left to prove.
Therefore we will assume that $M[X \cup Y \cup e_{1},\, B]$
is not $3$\dash connected.
Hence $e_{1}$ is in series or parallel to some element in
$M[X \cup Y \cup e_{1},\, B]$, and in fact $e_{1}$ is
in parallel or series to an element $x \in X$, since
$X$ is not a circuit-cocircuit of $M[X \cup Y \cup e_{1},\, B]$.
But Proposition~\ref{prop11} now implies that our assumption
on the minimality of $t$ is contradicted.
This completes the proof.
\end{proof}

\section{Splitters}
\label{chp3.4}

Suppose that \mcal{M} is a minor-closed class of matroids.
A \emph{splitter}\index{splitter} of \mcal{M} is a matroid
$M \in \mcal{M}$ such that any $3$\dash connected member of
\mcal{M} having an $M$\dash minor is isomorphic to $M$.
We present here two different forms of Seymour's Splitter Theorem.

\begin{thm}
\label{thm8}
\tup{\cite[(7.3)]{Sey80}}
Suppose \mcal{M} is a minor-closed class of matroids.
Let $M$ be a $3$\dash connected member of \mcal{M} such that
$|E(M)| \geq 4$ and $M$ is neither a wheel nor a whirl. If
no $3$\dash connected single-element extension or
coextension of $M$ belongs to \mcal{M}, then $M$ is a
splitter for \mcal{M}.
\end{thm}

\begin{thm}
\label{thm3}
\tup{\cite[Corollary~11.2.1]{Oxl92}}
Let $N$ be a $3$\dash connected matroid with $|E(N)| \geq 4$.
If $N$ is not a wheel or a whirl, and $M$ is a $3$\dash connected
matroid with a proper $N$\dash minor, then $M$ has a $3$\dash connected
single-element deletion or contraction with an $N$\dash minor.
\end{thm}

Recall that for a binary matroid $M$ we use \ex{M} to denote
the set of binary matroids with no $M$\dash minor.

\begin{prop}
\label{prop4}
The only $3$\dash connected matroid in \ex{\mkt} that is regular
but non-cographic is \mkf.
\end{prop}

\begin{proof}
Walton and Welsh~\cite{WW80} note (and it is easy to confirm)
that \mkf\ is a splitter for \ex{F_{7},\, F_{7}^{*},\, \mkt}.
This implies the result, since the set of cographic matroids is
exactly \ex{F_{7},\, F_{7}^{*},\, \mkt,\, \mkf} and the set of
regular matroids is \ex{F_{7},\, F_{7}^{*}}.
\end{proof}

Kingan defined the matroid $T_{12}$\index{T@$T_{12}$}
in~\cite{Kin97}.
All single-element deletions of $T_{12}$ are isomorphic, and
so are all single-element contractions, so the matroids
$T_{12} \del e$ and $T_{12} / e$ are well-defined.
The matroids $N_{10}$ and $\widetilde{K}_{5}$ are defined by
Zhou, who proved the following result.

\begin{prop}
\label{prop3}
\tup{\cite[Corollary~1.2]{Zho04}}
If $M$ is an \ifc\ binary non-regular
matroid that is not isomorphic to $F_{7}$ or $F_{7}^{*}$, then $M$
contains one of the following as a minor: $N_{10}$,
$\widetilde{K}_{5}$, $\widetilde{K}_{5}^{*}$,
$T_{12} \del e$, or $T_{12} / e$.
\end{prop}

The triangular \mob\ matroids are defined in
Chapter~\ref{chp2}.
The rank\dash $4$ triangular \mob\ matroid is
denoted by $\Delta_{4}$, and is isomorphic to
$\widetilde{K}_{5}$.

\begin{cor}
\label{cor3}
If $M \in \ex{\mkt}$ is an \ifc\ non-cographic matroid with no
$\Delta_{4}$\dash minor, then $M$ is isomorphic to one of the
following: $F_{7}$, $F_{7}^{*}$, \mkf, $T_{12} \del e$,
$T_{12} / e$, or $T_{12}$.
\end{cor}

\begin{proof}
Suppose that the result is not true.
Let $M$ be an \ifc\ non-cographic member of \ex{\mkt, \Delta_{4}}
that is isomorphic to none of the six matroids listed in the
statement.
Proposition~\ref{prop4} tells us that $M$ is non-regular, so we
can apply Proposition~\ref{prop3}.
We have noted that $\widetilde{K}_{5}$ is isomorphic to
$\Delta_{4}$, and both $N_{10}$ and $\widetilde{K}_{5}^{*}$ have
\mkt\dash minors~\cite{Zho04}.
Thus $M$ has $T_{12} \del e$ or $T_{12} / e$ as a proper minor.
There is only one $3$\dash connected coextension or extension of
$T_{12} \del e$ in \ex{\mkt,\, \Delta_{4}}, and that is $T_{12}$.
Similarly $T_{12}$ is the only $3$\dash connected coextension or
extension of $T_{12} / e$ in \ex{\mkt,\, \Delta_{4}}.\cross\
It follows from Theorem~\ref{thm3} that $M$ has a $T_{12}$\dash minor.
But $T_{12}$ is a splitter for \ex{\mkt,\, \Delta_{4}},
so $M$ is isomorphic to $T_{12}$, a contradiction.\cross
\end{proof}

Geelen and Zhou have proved a splitter-type theorem for
\ifc\ binary matroids.

\begin{thm}
\label{thm7}
\tup{\cite[Theorem~5.1]{GZ06}}
Suppose that $M$ and $N$ are \ifc\
binary matroids where $|E(N)| \geq 7$ and $M$ has a proper
minor isomorphic to $N$.
Then there exists an element $e \in E(M)$ such that one of
$M \del e$ or $M / e$ is $(4,\, 5)$\dash connected with a minor
isomorphic to $N$.
\end{thm}

\section{Generalized parallel connections}
\label{chp3.3}

In this section we discuss the generalized parallel
connection of Brylawski~\cite{Bry75b}.

A flat $F$ of the matroid $M$ is a
\emph{modular flat}\index{modular flat}
if $r(F) + r(F') = r(F \cap F') + r(F \cup F')$ for every flat
$F'$ of $M$.
Suppose that $M_{1}$ and $M_{2}$ are two matroids and
$E(M_{1}) \cap E(M_{2}) = T$, where $M_{1}|T = M_{2}|T$.
If $\cl_{M_{1}}(T)$ is a modular flat of $M_{1}$ and every
non-loop element in $\cl_{M_{1}}(T) - T$ is parallel to an
element in $T$, then we can define the
\emph{generalized parallel connection}\index{parallel connection}
of $M_{1}$ and $M_{2}$, denoted by $P_{T}(M_{1},\, M_{2})$.
The ground set of $P_{T}(M_{1},\, M_{2})$ is $E(M_{1}) \cup E(M_{2})$
and the flats of $P_{T}(M_{1},\, M_{2})$ are those sets $F$ such that
$F \cap E(M_{i})$ is a flat of $M_{i}$ for $i = 1,\, 2$.

\begin{prop}
\label{prop5}
\tup{\cite[Proposition~12.4.14]{Oxl92}}
Suppose the generalized parallel connection, $P_{T}(M_{1},\, M_{2})$,
of $M_{1}$ and $M_{2}$ is defined.
\begin{enumerate}[(i)]
\item If $e \in E(M_{1}) - T$ then
$P_{T}(M_{1},\, M_{2}) \del e = P_{T}(M_{1} \del e,\, M_{2})$.
\item If $e \in E(M_{2}) - T$ then
$P_{T}(M_{1},\, M_{2}) \del e = P_{T}(M_{1},\, M_{2} \del e)$.
\item If $e \in E(M_{1}) - \cl_{M_{1}}(T)$ then
$P_{T}(M_{1},\, M_{2}) / e = P_{T}(M_{1} / e,\, M_{2})$.
\item If $e \in E(M_{2}) - \cl_{M_{2}}(T)$ then
$P_{T}(M_{1},\, M_{2}) / e = P_{T}(M_{1},\, M_{2} / e)$
\end{enumerate}
\end{prop}

Suppose that $E(M_{1}) \cap E(M_{2})$ contains a single
element $p$.
Then $\cl_{M_{1}}(\{p\})$ is a modular
flat of $M_{1}$, so $P_{T}(M_{1},\, M_{2})$ is
defined, where $T = \{p\}$.
If $p$ is neither a coloop nor a loop in $M_{1}$ or $M_{2}$
then the \emph{$2$\dash sum}\index{$2$\dash sum} of $M_{1}$ and
$M_{2}$ along the \emph{basepoint}\index{basepoint (of a $2$\dash sum)}
$p$, denoted by $M_{1} \oplus_{2} M_{2}$, is defined to be
$P_{T}(M_{1},\, M_{2}) \del T$.
The circuits of $M_{1} \oplus_{2} M_{2}$ are exactly those circuits
of $M_{1}$ or $M_{2}$ that do not contain $p$, and sets of the form
$(C_{1} - p) \cup (C_{2} - p)$, where $C_{i}$ is a circuit of
$M_{i}$ such that $p \in C_{i}$ for $i = 1,\, 2$.

The $2$\dash sum operation has the following properties.

\begin{prop}
\label{prop6}
\tup{\cite[Proposition~7.1.20]{Oxl92}}
Suppose that $M_{1}$ and $M_{2}$ are matroids such that
$E(M_{1}) \cap E(M_{2}) = \{p\}$ and $p$ is not a loop or
a coloop of $M_{1}$ or $M_{2}$.
Then $(M_{1} \oplus_{2} M_{2})^{*} = M_{1}^{*} \oplus_{2} M_{2}^{*}$.
\end{prop}

\begin{prop}
\label{prop17}
\tup{\cite[(2.6)]{Sey80}}
If $(X,\, Y)$ is an exact $2$\dash separation of a matroid
$M$ then there exist matroids $M_{1}$ and $M_{2}$ on the
ground sets $X \cup p$ and $Y \cup p$ respectively, where
$p$ is in neither $X$ nor $Y$, such that $M$ is equal to
$M_{1} \oplus_{2} M_{2}$.
Conversely, if $M$ is the $2$\dash sum of $M_{1}$ and $M_{2}$
along the basepoint $p$, where
$|E(M_{1})|,\, |E(M_{2})| \geq 3$, then
$(E(M_{1}) - p,\, E(M_{2}) - p)$ is an exact $2$\dash separation
of $M$, and $M_{1}$ and $M_{2}$ are isomorphic to minors of $M$.
\end{prop}

The next result is well known, but seems not to appear in the
literature.

\begin{prop}
\label{prop9}
Let $N$ be a $3$\dash connected matroid.
Suppose that $M = M_{1} \oplus_{2} M_{2}$.
If $M$ has an $N$\dash minor then either $M_{1}$ or $M_{2}$
has an $N$\dash minor.
\end{prop}

\begin{proof}
Assume that $M$ is the $2$\dash sum of $M_{1}$ and $M_{2}$ along
the basepoint $p$, and that $M$ has an $N$\dash minor, but neither
$M_{1}$ nor $M_{2}$ has an $N$\dash minor.
Moreover, assume that the proposition holds for all matroids smaller
than $M$.

It is easy to see that the result holds if $M$ is equal to $N$, so
assume that there is an element $e \in E(M)$ such that either
$M \del e$ or $M / e$ has an $N$\dash minor.
By relabeling if necessary we will assume that $e \in E(M_{1}) - p$.

First suppose that $M \del e$ has an $N$\dash minor.
If $p$ is not a coloop in $M_{1} \del e$ then
Proposition~\ref{prop5} implies that
$M \del e = (M_{1} \del e) \oplus_{2} M_{2}$.
The minimality of $M$ now implies that either $M_{1} \del e$
or $M_{2}$ has an $N$\dash minor.
In either case we are done, so we assume that $p$ is a coloop
in $M_{1} \del e$.
This means that $\{e,\, p\}$ is a series pair in $M_{1}$.
From the description of circuits of $M_{1} \oplus_{2} M_{2}$ it
follows that $(E(M_{1}) - \{e,\, p\},\, E(M_{2}) - p)$ is a
$1$\dash separation of $M \del e$.
Proposition~\ref{prop7} implies that $E(N)$ is disjoint with
either $E(M_{1}) - \{e,\, p\}$ or $E(M_{2}) - p$.
The result follows easily.

Next we assume that $M / e$ has an $N$\dash minor.
Proposition~\ref{prop6} implies that $M^{*}$ and
$N^{*}$ provide a minimal counterexample to the proposition.
Furthermore $M^{*} \del e$ has an $N^{*}$\dash minor.
We apply the arguments of the last paragraph to show that
the proposition holds.
\end{proof}

\begin{lem}
\label{lem24}
Let $M$ be a $3$\dash connected binary matroid on the
ground set $E$.
Suppose that $X$ is a subset of $E$ such that
$\lambda_{M}(X) \leq 2$ and $r_{M}(X) \geq 3$.
Let $Y = E - X$.
Assume that $e$ is an element in $X \cap \cl_{M}(Y)$.
Then either $|E| \leq 7$, or there exists an independent set
$I \subseteq X$ and an element $f \in X$, such that $\{e,\, f\}$ is a
parallel pair in $M / I$ and $r_{M}(I \cup Y) = r_{M}(I) + r_{M}(Y)$.
\end{lem}

\begin{proof}
Let $M$ be a counterexample, and assume that $M$ has been
chosen so that $|E|$ is as small as possible.
Let $Y_{0} = \cl_{M}(Y)$, and let $X_{0} = X - Y_{0}$.

Since $M$ is a counterexample, $|E| > 7$, so
every circuit and cocircuit of $M$ contains at least three
elements.
As $e \in \cl_{M}(Y)$ it follows that
there is a circuit contained in $Y \cup e$ which
contains $e$.
Therefore $|Y| \geq 2$.

Suppose that $r_{M}(X_{0}) < r_{M}(X)$.
Then $\lambda_{M}(X_{0}) < 2$.
However $r_{M}(X) \geq 3$ implies that
$r_{M}(Y_{0}) \leq r(M) - 1$, so
$X_{0}$ contains a cocircuit of $M$.
Thus $|X_{0}| \geq 2$.
Furthermore $Y \subseteq Y_{0}$
so $|Y_{0}| \geq 2$.
Hence $(X_{0},\, Y_{0})$ is a $2$\dash separation of
$M$, a contradiction.
Therefore $r_{M}(X_{0}) = r_{M}(X)$.

We start by assuming that $r_{M}(X_{0}) = 3$.
As $|X_{0}|,\, |Y_{0}| \geq 2$ it follows that
$\lambda_{M}(X_{0}) = 2$.
Thus $X_{0}$ is a cocircuit of $M$.
Since $r_{M}(X_{0}) = r_{M}(X)$ it follows that
$e \in \cl_{M}(X_{0})$, so there is a circuit
$C \subseteq X_{0} \cup e$ that contains $e$.
Since $M$ is binary and $X_{0}$ is a rank\dash $3$ cocircuit
it follows that $|X_{0}| \leq 4$.
Thus $|C| \leq 5$.
But it cannot be the case that $|C| = 5$, for then
$|X_{0}| = 4$ and $X_{0}$ is a circuit properly contained
in $C$.
Thus $|C| \leq 4$.
Since $C$ is a circuit and $X_{0}$ is a cocircuit it
follows that $|C \cap X_{0}|$ is even, so we deduce that
$C$ is a triangle.

Suppose that $C = \{e,\, f,\, g\}$.
If we let $I = \{g\}$ then we see that $e$ and
$f$ are parallel in $M / I$ and
$r_{M}(I \cup Y) = r(M) = r_{M}(I) + r_{M}(Y)$.
Thus the result holds for $M$, a contradiction.
Therefore $r_{M}(X_{0}) > 3$.

Choose an element $x \in X_{0}$.
Suppose that $\si(M / x)$ is $3$\dash connected.
If $(A,\, B)$ is a $2$\dash separation of
$M / x$ then either $A$ or $B$ is a parallel pair.
Let $Y' = \cl_{M / x}(Y)$, and let $X' = X - (Y' \cup x)$.
Note that $r_{M}(X) \geq 4$ implies
$r_{M}(Y) \leq r(M) - 2$.
Thus $r_{M / x}(Y) \leq r(M / x) - 1$, so
$Y$ does not span $M / x$, and hence $X'$ contains a
cocircuit of $M / x$.
As every cocircuit of $M / x$ is also a cocircuit of $M$
it follows that $|X'| \geq 3$.

Note that $\lambda_{M / x}(X - x) \leq 2$, so
$\lambda_{M / x}(X') \leq 2$.
Assume that $r_{M / x}(X') < r_{M / x}(X - x)$.
Then $\lambda_{M / x}(X') < 2$.
This means that $X'$ is a parallel pair in $M / x$, a
contradiction as $|X'| \geq 3$.
This shows that $X'$ spans $X - x$ in $M / x$.

Let $P$ be a set containing exactly one element
from each parallel pair of $M / x$, so that
$M / x \del P \iso \si(M / x)$.
Since $x \notin \cl_{M}(Y)$ it follows that any triangle
of $M$ that contains $x$ must contain at least one element
of $X - x$.
Therefore we will choose $P$ so that $P \subseteq X - x$.
If a triangle of $M$ contains both $e$ and $x$, then the
third element of this triangle must be in $X$, for
$e \in \cl_{M}(Y)$, and $x \notin \cl_{M}(Y)$.
Therefore we can also assume that $e \notin P$.

Let $M_{0} = M / x \del P$.
Since $P \subseteq X - x$ it follows that
$e \in \cl_{M_{0}}(Y)$.
Furthermore $\lambda_{M_{0}}(X - (P \cup x)) \leq 2$.
Since $r_{M}(X) \geq 4$ it follows that
$r_{M / x}(X - x) \geq 3$.
Recall that $Y' = \cl_{M / x}(Y)$ and that
$X' = X - (Y' \cup x)$.
If $r_{M_{0}}(X - (P \cup x)) < r_{M / x}(X - x)$ then there
must be a parallel pair in $M / x$ that is not in
$\cl_{M / x}(X')$, containing one element from $X$ and one
element from $Y$.
But $X'$ spans $X - x$ in $M / x$, so this cannot happen.
Thus $r_{M_{0}}(X - (P \cup x)) \geq 3$.

By our assumption on $|E|$ there is an independent
set $I_{0} \subseteq X - (P \cup x)$
and an element $f \in X - (P \cup x)$ such
that $\{e,\, f\}$ is
a parallel pair in $M_{0} / I_{0}$ and
$r_{M_{0}}(I_{0} \cup Y) = r_{M_{0}}(I_{0}) + r_{M_{0}}(Y)$.

Let $I = I_{0} \cup x$.
Then $I$ is an independent set in $M$, and $\{e,\, f\}$
is a parallel pair in $M / I$.
Moreover $r_{M}(I) = r_{M_{0}}(I_{0}) + 1$ and
$r_{M}(I \cup Y) = r_{M_{0}}(I_{0} \cup Y) + 1$.
Also $r_{M}(Y) = r_{M_{0}}(Y)$ since $x \notin \cl_{M}(Y)$.
Thus the result holds for $M$, a contradiction.

Now we must assume that $\si(M / x)$ is not
$3$\dash connected.
Proposition~\ref{prop16} implies that
$\co(M \del x)$ is $3$\dash connected.
It follows that if $(A,\, B)$ is a $2$\dash separation
of $M \del x$ then either $A$ or $B$ is a series pair
in $M \del x$.

Let $T$ be a triad of $M$ which contains $x$.
We will show that $T \subseteq X_{0}$.
Assume otherwise, so $T$ has a non-empty intersection with
$Y_{0}$.
We first suppose that $T$ meets $Y_{0}$ in two elements.
Then $x \in \cl_{M}^{*}(Y_{0})$, and hence
$(X_{0} -  x,\, Y_{0})$ is a $2$\dash separation of $M \del x$.
Thus $|X_{0} - x| = 2$.
But this is a contradiction as $r_{M}(X_{0}) \geq 4$.

Suppose that $T$ contains $e$.
Then $x \in \cl_{M}^{*}(Y \cup e)$.
Thus $(X - \{e,\, x\},\, Y \cup e)$ is a
$2$\dash separation of $M \del x$, so
$|X - \{e,\, x\}| = 2$.
This implies that $|X| = 4$.
As $e \in Y_{0}$ this means that $|X_{0}| \leq 3$,
a contradiction.

Next we suppose that $T$ meets $Y_{0}$ in exactly one
element $y$.
The previous paragraph shows that $y \ne e$.
Now $y \in \cl_{M}^{*}(X)$, so
$y \notin \cl_{M}(Y - y)$.
Therefore $\lambda_{M}(X \cup y) \leq 2$.
Moreover $e \in \cl_{M}(Y - y)$, for
otherwise there is a circuit $C \subseteq Y \cup e$
which contains $e$ and which meets the
triad $T$ in precisely one element, $y$.
It cannot be the case that $y$ is in $\cl_{M}(X)$, for then
$(X \cup y,\, Y - y)$ would be a $2$\dash separation of
$M$.
(Note that $|Y - y| \geq 2$ since $(Y - y) \cup e$ contains
a circuit of $M$.)

Let us suppose that $\si(M / y)$ is $3$\dash connected.
Assume that there is a triangle $T'$ of $M$ containing $y$.
It cannot be the case that $T' \subseteq Y_{0}$, for then
$T'$ and $T$ meet in a single element.
Nor can it be the case that $T' - y \subseteq X_{0}$, for
that would imply that $y \in \cl_{M}(X_{0})$, and
we have already concluded that $y \notin \cl_{M}(X)$.
Thus $T'$ contains exactly two elements from $Y_{0}$, and
one element from $X_{0}$.
This means that the single element in $T' \cap X_{0}$ is
in $\cl_{M}(Y)$, a contradiction.
We have shown that $y$ is contained in no triangles of
$M$, so $M / y$ is $3$\dash connected.

Our assumption on $|E|$ means that there is an
independent set $I \subseteq X$ of $M / y$ and
an element $f \in X$ such that $e$ and $f$ are
parallel in $M / y / I$, and
$r_{M / y}(I \cup (Y - y)) = r_{M / y}(I) + r_{M / y}(Y - y)$.

There is a circuit $C \subseteq I \cup \{e,\, f,\, y\}$
that contains both $e$ and $f$.
It cannot be the case that $y \in C$, for
$y \notin \cl_{M}(X)$.
Thus $e$ and $f$ are parallel in $M / I$.
Moreover $r_{M}(I \cup Y) = r_{M / y}(I \cup (Y - y)) + 1$
and $r_{M}(Y) = r_{M / y}(Y - y) + 1$.
Also $r_{M}(I) = r_{M / y}(I)$, for
$y \notin \cl_{M}(X)$.
Therefore the lemma holds for $M$, a contradiction.

Therefore it cannot be the case that
$\si(M / y)$ is $3$\dash connected, and hence
$\co(M \del y)$ is $3$\dash connected.
However $(X,\, Y - y)$ is a
$2$\dash separation of $M \del y$, so
$|Y - y| \leq 2$.
Since $(Y - y) \cup e$ contains a circuit
it must be the case that $|Y - y| = 2$.
Therefore $|Y| = 3$.
If $Y$ is not independent then it is a
triangle, and in this case $Y$ meets the
cocircuit $T$ in a single element, a
contradiction.
Thus $Y$ is independent, and
$r_{M}(X) = r(M) - 1$.
Thus $Y$ is a triad.

Since $e \in \cl_{M}(Y - y)$ it follows that $(Y - y) \cup e$
is a triangle of $M \del y$.
But to obtain $\co(M \del y)$ from $M \del y$
we must contract a single element from the series pair
$Y - y$, so $\co(M \del y)$ contains a parallel pair.
Since $\co(M \del y)$ is $3$\dash connected this means that
$\co(M \del y)$ is isomorphic to some restriction of $U_{1,3}$.
It is easy to see that this implies $|E| \leq 7$,
contrary to our earlier conclusion.
Thus we have proved that any triad of $M$ that contains
$x$ must be contained in $X_{0}$.

Let $S$ be a set containing a single element from each series
pair of $M \del x$, so that $M \del x / S \iso \co(M \del x)$.
By the previous arguments it follows that
$S \subseteq X_{0} - x$.
We can assume that $e \notin S$.
Let $M_{0} = M \del x / S$.

Note that $\lambda_{M_{0}}(X - (S \cup x)) \leq 2$ and
$e \in \cl_{M_{0}}(Y)$.
Suppose that $r_{M_{0}}(X - (S \cup x)) \geq 3$.
Then by our assumption on the cardinality of
$E$ it follows that there is an independent set
$I_{0} \subseteq X - (S \cup x)$
of $M_{0}$ and an element
$f \in X - (S \cup x)$ such that
$e$ and $f$ are parallel in $M_{0} / I$ and
$r_{M_{0}}(I_{0} \cup Y) = r_{M_{0}}(I_{0}) + r_{M_{0}}(Y)$.
We will assume that $I_{0}$ is minimal with
respect to these properties, so
$I_{0} \cup \{e,\, f\}$ is a circuit of $M_{0}$.

Let \mcal{T} be the set of series pairs of $M \del x$
which meet $I_{0} \cup \{e,\, f\}$.
Each of the series pairs in \mcal{T} contains
exactly one element in $S$.
Let $S_{0} \subseteq S$ be the set containing
these elements.
Let $I = I_{0} \cup S_{0}$.
Since $I_{0} \cup \{e,\, f\}$ is a circuit of
$M_{0}$ it is easy to see that $I \cup \{e,\, f\}$
must be a circuit in $M \del x$, and hence in $M$.
Thus $e$ and $f$ are parallel in $M / I$.

Suppose that it is not the case that
$r_{M}(I \cup Y) = r_{M}(I) + r_{M}(Y)$.
Then there is a circuit $C$ of $M \del x$ contained in
$I \cup Y$ which meets both $I$ and $Y$.
If $C'$ is any circuit and $x'$ is any element contained
in a series pair, then $C' - x'$ is a circuit in the
matroid produced by contracting $x'$.
It follows that $C - S$ is a circuit of $M_{0}$.
Our assumption on $I_{0}$ and $Y$ means that
$C - S$ cannot meet both $I_{0}$ and $Y$.
Moreover $I_{0}$ is independent in $M_{0}$.
Therefore $C - S \subseteq Y$.
Suppose that $T = \{s,\, t\}$ is a series pair
in \mcal{T} and that $s \in S \cap C$.
Since the circuit $C$ has a non-empty intersection with
$T$ it must contain $t$.
Therefore $t \in C - S$, implying that $t \in Y$.
But the definition of \mcal{T} means that
$t \in I_{0} \cup \{e,\, f\}$, and this set
has an empty intersection with $Y$.
Therefore $r_{M}(I \cup Y) = r_{M}(I) + r_{M}(Y)$
and the lemma holds for $M$, a contradiction.

Now we have to assume that $r_{M_{0}}(X - (S \cup x)) < 3$.
Suppose that $r_{M}(X_{0} - x) < r_{M}(X_{0})$.
Thus $(X_{0} - x,\, Y_{0})$ is a
$2$\dash separation of $M \del x$, meaning that
$X_{0} - x$ is a series pair.
This is a contradiction as $r_{M}(X_{0}) \geq 4$.
Thus $r_{M}(X_{0} - x) = r_{M}(X_{0})$.
We observe that $e \in \cl_{M}(X_{0})$ and
$\cl_{M}(X_{0} - x) = \cl_{M}(X_{0})$.
Since $e \notin X_{0} - x$ it follows that
$(X_{0} - x) \cup e$ contains a circuit of $M \del x$.
Therefore $X - (S \cup x)$ contains a circuit in
$M_{0}$.

If $M_{0}$ contains a circuit of size at most two, then
$M_{0}$ is a restriction of $U_{1,3}$, and $|E| \leq 7$,
so we are done.
Therefore every circuit of $M_{0}$ contains at least
three elements.
Since $r_{M_{0}}(X - (S \cup x)) < 3$ it follows that
$X - (S \cup x)$ is a triangle in $M_{0}$.

Suppose that $e$ is contained in a series pair in
$M \del x$.
Then $e$ is contained in a triad of $M$ which also contains
$x$.
But this is a contradiction as we have already shown that
any such triad must be contained in $X_{0}$ and
$e \in \cl_{M}(Y)$.
Therefore $e$ is contained in no series pairs in $M \del x$.

As $r_{M}(X_{0} - x) \geq 4$ and
$r_{M_{0}}(X - (S \cup x)) = 2$ there are at least two
series pairs in $M \del x$.
As $X - (S \cup x)$ contains exactly three elements,
and each series pair of $M \del x$ contributes one
element to $X - (S \cup x)$ it follows that there are
no more than three series pairs in $M \del x$.
However $e$ is in $X - (S \cup x)$, and $e$ is in no
series pair in $M \del x$.
Therefore $M \del x$ contains precisely two series pairs.
Let these series pairs be $\{s_{1},\, t_{1}\}$ and
$\{s_{2},\, t_{2}\}$.
Assume that $S = \{s_{1},\, s_{2}\}$.

Now $\{e,\, t_{1},\, t_{2}\}$ is a circuit in
$M_{0}$.
Therefore $\{e,\, s_{1},\, s_{2},\, t_{1},\, t_{2}\}$ is
a circuit in $M \del x$.
We have already shown that
$r_{M}(X_{0} - x) = r_{M}(X_{0})$.
Therefore there is a circuit $C \subseteq X_{0} \cup x$
which contains $x$.
Both $\{s_{1},\, t_{1},\, x\}$ and
$\{s_{2},\, t_{2},\, x\}$ are triads in $M$, so
$C$ must meet these sets in exactly two elements
each.
By taking the symmetric difference of
$C$ and $\{e,\, s_{1},\, s_{2},\, t_{1},\, t_{2}\}$
we see that there is a circuit $C'$ of $M$ such that
$|C'| = 4$ and $e,\, x \in C'$.

Let $z_{1}$ and $z_{2}$ be the elements in
$C' \cap \{s_{1},\, t_{1}\}$ and
$C' \cap \{s_{2},\, t_{2}\}$ respectively.
Let $I = \{z_{1},\, z_{2}\}$.
Then $e$ and $x$ are parallel in $M / I$.
Moreover, since $r_{M}(X) = 4$ it follows that
$r_{M}(Y) = r(M) - 2$.
Clearly $r_{M}(I) = 2$.
Suppose that $r_{M}(I \cup y) \ne r(M)$.
Because $z_{1},\, z_{2} \notin \cl_{M}(Y)$ this
means that $r(I \cup Y) = r(M) - 1$, and there is
some circuit contained in $I \cup Y$ which
contains both $z_{1}$ and $z_{2}$.
But such a circuit meets the triad
$\{x,\, s_{1},\, t_{1}\}$ in precisely one element,
a contradiction.
Thus the lemma holds for $M$.
This completes the proof.
\end{proof}

\begin{prop}
\label{prop18}
Suppose that $M$ and $N$ are $3$\dash connected binary matroids
such that $|E(M)| > 7$.
Let $e$ be an element of $E(M)$ such that $M \del e$ has
a $2$\dash separation $(X_{1},\, X_{2})$ where
$r_{M}(X_{1}),\, r_{M}(X_{2}) \geq 3$.
If $M \del e$ has an $N$\dash minor then so does $M / e$.
\end{prop}

\begin{proof}
Since $M$ is $3$\dash connected it follows that
$(X_{1},\, X_{2})$ is an exact $2$\dash separation of $M \del e$.
By Proposition~\ref{prop17} there are matroids $M_{1}$ and
$M_{2}$ with the property that $E(M_{i}) = X_{i} \cup p$
for $i = 1,\, 2$, where $p$ is in neither $X_{1}$ nor
$X_{2}$, and $M \del e = M_{1} \oplus_{2} M_{2}$.
By Proposition~\ref{prop9} we can assume that $M_{2}$ has
an $N$\dash minor.

Let $T = \{p\}$ and let $P$ be a set of points in
\pg{r - 1,\, 2}, where $r = r(M \del e)$, such that the
restriction of \pg{r - 1,\, 2} to $P$ is isomorphic to
$P_{T}(M_{1},\, M_{2})$.
We will identify points of $P_{T}(M_{1},\, M_{2})$ with
the corresponding points in $P$ and we will blur the
distinction between $P$, and the restriction of
\pg{r - 1,\, 2} to $P$.
Thus $P \del p = M \del e$.
It is well known, and easy to verify, that
$P|(X_{i} \cup p) \iso M_{i}$ for $i = 1,\, 2$.
We identify $e$ with the unique point in 
\pg{r - 1,\, 2} such that the restriction of
\pg{r - 1,\, 2} to $(P - p) \cup e$ is isomorphic to $M$.
Let $M'$ be the restriction of
\pg{r - 1,\, 2} to $P \cup e$.

The only possible $2$\dash separation of $M'$ is a parallel
pair containing $p$, so either $M'$ or $M' \del p$ is
$3$\dash connected. 
Furthermore $(X_{1} \cup p,\, X_{2} \cup e)$ is a
$3$\dash separation of $M'$ and
$r_{M}(X_{1} \cup p) \geq 3$.
It cannot be the case that $p$ is a coloop in $M_{2}$,
so $p \in \cl_{M'}(X_{2} \cup e)$.
Thus we can apply Lemma~\ref{lem24} and conclude that there is an
independent set $I \subseteq X_{1} \cup p$ and a point
$p' \in X_{1} \cup p$ such that $p$ and $p'$ are parallel in
$M' / I$ and
$r_{M'}(I \cup X_{2} \cup e) = r_{M'}(I) + r_{M'}(X_{2} \cup e)$.
Let us assume that $I$ is minimal with respect to these
properties.

Let $D = E(M_{1}) - (I \cup p \cup p')$.
Then $M_{1} / I \del D$ consists of the parallel
pair $\{p,\, p'\}$.
Proposition~\ref{prop5} implies that
$P / I \del D \del p$ is isomorphic to $M_{2}$.
Thus $P / I \del D \del p = M' / I \del D \del p \del e$
has an $N$\dash minor.

Suppose that $e$ is not a coloop in
$M' \del D \del p$.
Then there is a circuit $C \subseteq I \cup X_{2} \cup \{e,\, p'\}$
that contains $e$.
It cannot be the case that $C \subseteq X_{2} \cup e$, for
that would imply that $e \in \cl_{M}(X_{2})$ and that
$(X_{1},\, X_{2} \cup e)$ is a $2$\dash separation of $M$.
The same argument shows that $C$ is not contained in
$X_{1} \cup e$.
Moreover, $p'$ is contained in $C$, for otherwise the
circuit $C$ contradicts the fact that $(I,\, X_{2} \cup e)$
is a $2$\dash separation of $M'|(X_{2} \cup I \cup e)$.

Our assumption on the minimality of $I$ means that
$I \cup \{p,\, p'\}$ is a circuit of $M'$.
The symmetric difference of $C$ and $I \cup \{p,\, p'\}$
is a union of circuits in $M'$.
Let $C'$ be a circuit in this symmetric difference that
contains $e$.
Since $p' \in C$ it follows that $p' \notin C'$.
Thus $C'$ either demonstrates that
$e \in \cl_{M}(X_{1})$, or $e \in \cl_{M}(X_{2})$, or
$C'$ meets two different connected components of
$M'|(I \cup X_{2} \cup e)$.
In each of these cases we have a contradiction.

Thus $e$ is a coloop in $M' \del D \del p$, and therefore
a coloop in $M' / I \del D \del p$.
Thus $M' / I \del D \del p / e = M' / I \del D \del p \del e$,
so $M' / I \del D \del p / e$, and hence $M / e$, has
an $N$\dash minor.
\end{proof}

\section{The \protect\dy\ operation}
\label{chp3.6}

Suppose that $M_{1}$ and $M_{2}$ are matroids such
that $E(M_{1}) \cap E(M_{2}) = T$.
If $M_{1} \iso M(K_{4})$ and $T$ is a triangle of both
$M_{1}$ and $M_{2}$ then $P_{T}(M_{1},\, M_{2})$ is defined.
In this case $P_{T}(M_{1},\, M_{2}) \del T$ is said to be
the matroid produced from $M_{2}$ by doing a
\emph{\dy\ operation}\index{DeltaB@$\Delta\textrm{-}Y$ operation} on $T$.
The \dy\ operation for matroids has been studied by
Akkari and Oxley~\cite{AO93} and generalized by
Oxley, Semple, and Vertigan~\cite{OSV00}.
This chapter is concerned with results from this last
paper and their consequences.

Let $T$ be a triangle of a matroid $M$.
The authors of~\cite{OSV00} require that $T$
is coindependent for the \dy\ operation to be
defined, and we follow this convention.
The matroid produced from $M$ by a \dy\ operation on $T$
shall be denoted by $\Delta_{T}(M)$.
Suppose that $T = \{a_{1},\, a_{2},\ a_{3}\}$.
Let $T' = \{a_{1}',\, a_{2}',\, a_{3}'\}$.
We shall assume that $\Delta_{T}(M) = P_{T}(M(K_{4}),\, M) \del T$
where the ground set of $M(K_{4})$ is $T \cup T'$ and that
$\{a_{i},\, a_{i}'\}$ is contained in no triangle of
$M(K_{4})$ for $i = 1,\, 2,\, 3$.
It will be convenient to relabel $a_{i}'$ with $a_{i}$ in
$\Delta_{T}(M)$, so that $M$ and $\Delta_{T}(M)$ have the
same ground set.

Suppose that $T$ is a coindependent triangle in the binary
matroid $M$.
Let $B$ be a basis of $M$ that does not contain $T$, and let
$e \in T$ be an element not in $B$.
Then $B \cup e$ is a basis of $\Delta_{T}(M)$.
The fundamental graph $G_{B \cup e}(\Delta_{T}(M))$ is easily
derived from the graph $G_{B}(M)$, as we now show.
The following claims are easy consequences of the definition
of the \dy\ operation and~\cite[Theorem~6.12]{Bry75b}.
If $T - e$ is contained in $B$, then we obtain
$G_{B \cup e}(\Delta_{T}(M))$ from $G_{B}(M)$ by deleting
$e$, adding a new vertex adjacent to those vertices
that are adjacent to exactly one element in $T - e$, and
labeling this new vertex $e$.

Suppose that $T - e = \{f,\, g\}$ and that $f \in B$,
while $g \notin B$.
Then the sets of neighbors of $e$ and $g$ are exactly
the same, with the exception that exactly one of $e$ and
$g$ is adjacent to $f$ in $G_{B}(M)$.
Let us suppose that $e$ is not adjacent to $f$.
Then $G_{B \cup e}(\Delta_{T}(M))$ is obtained from
$G_{B}(M)$ by deleting $e$ and adding a new vertex
adjacent to all the neighbors of $f$ apart from $g$.
If $g$ is not adjacent to $f$ then we obtain
$G_{B \cup e}(\Delta_{T}(M))$ by deleting $e$ and
adding a new vertex adjacent to all the neighbors
of $f$ as well as $g$.
In either case we label the new vertex~$e$.

Finally we note that if $T \cap B = \varnothing$ then
$G_{B \cup e}(\Delta_{T}(M))$ is obtained from
$G_{B}(M)$ by simply deleting $e$, adding a new vertex
adjacent to the two vertices in $T - e$ and labeling this
new vertex $e$.

If $T$ is an independent triad of a matroid $M$, then $T$ is a
coindependent triangle in $M^{*}$, so $\Delta_{T}(M^{*})$ is
defined.
Let $(\Delta_{T}(M^{*}))^{*}$ be denoted by $\nabla_{T}(M)$.
Then $\nabla_{T}(M)$ is said to be obtained by performing a
\yd\ operation on $M$.
The \dy\ and \yd\ operations preserve the property
of being representable over a
field~\cite[Lemma~3.5]{OSV00}.

If $T$ is a coindependent triangle in $M$, then $T$ is an
independent triad of $\Delta_{T}(M)$, so $\nabla_{T}(\Delta_{T}(M))$
is defined.
Similarly, if $T$ is a coindependent triad of $M$, then $T$
is an independent triangle of $\nabla_{T}(M)$, so
$\Delta_{T}(\nabla_{T}(M))$ is defined.

\begin{prop}
\label{prop19}
\tup{\cite[Lemma~2.11 and Corollary~2.12]{OSV00}}
If $T$ is a coindependent triangle of $M$
then $\nabla_{T}(\Delta_{T}(M)) = M$. If $T$ is an
independent triad then
$\Delta_{T}(\nabla_{T}(M)) = M$.
\end{prop}

\begin{prop}
\label{prop31}
\tup{\cite[Lemma~2.6]{OSV00}}
If $T$ is a coindependent triangle of $M$ then
$r(\Delta_{T}(M)) = r(M) + 1$.
If $T$ is an independent triad of $M$ then
$r(\nabla_{T}(M)) = r(M) - 1$.
\end{prop}

\begin{prop}
\label{prop29}
\tup{\cite[Lemma~2.18 and Corollary~2.19]{OSV00}}
If $T$ and $T'$ are disjoint coindependent triangles
of a matroid $M$ then
$\Delta_{T}(\Delta_{T'}(M)) = \Delta_{T'}(\Delta_{T}(M))$.
If $T$ and $T'$ are disjoint independent triads then
$\nabla_{T}(\nabla_{T'}(M)) = \nabla_{T'}(\nabla_{T}(M))$.
\end{prop}

The next result follows easily from Proposition~\ref{prop5}.

\begin{prop}
\label{prop30}
Suppose that $T$ is a coindependent triangle of the matroid
$M$, and that $N$ is a minor of $M$ such that $T$ is a
coindependent triangle of $N$.
Then $\Delta_{T}(N)$ is a minor of $\Delta_{T}(M)$.
\end{prop}

\begin{prop}
\label{prop28}
\tup{\cite[Lemmas~2.8 and~2.13 and Corollary~2.14]{OSV00}}
If $T$ is a coindependent triangle of the matroid $M$ then
$\Delta_{T}(M) \del T = M \del T$ and
$\Delta_{T}(M) / T = M / T$.
Moreover, if $e \in T$ then $\Delta_{T}(M) / e \iso M \del e$.
Similarly, if $T$ is an independent triad then
$\nabla_{T}(M) \del T = M \del T$ and $\nabla_{T}(M) / T = M / T$.
If $e \in T$ then $\nabla_{T}(M) \del e \iso M / e$.
\end{prop}

The function that switches the two members of $T - e$ and
fixes every member of $E(M) - T$ is an isomorphism between
$\Delta_{T}(M) / e$ and $M \del e$, and between
$\nabla_{T}(M) \del e$ and $M / e$.

\begin{cor}
\label{cor5}
Suppose that $T$ is a coindependent triangle in the binary
matroid $M$.
Let $M'$ be a single-element coextension of $M$ by the element
$e$, such that $e$ is in a triad of $M'$ with two elements
of $T$.
Then $\Delta_{T}(M)$ is isomorphic to a minor of $M'$.
\end{cor}

\begin{proof}
Suppose that $T'$ is the triad of $M'$ that contains $e$ and
two elements of $T$.
Let $e'$ be the single element in $T - T'$.
Since $T$ is coindependent in $M$ there is a basis $B$ of
$M$ that avoids $T$.
Then $B \cup e'$ is a basis of $\Delta_{T}(M)$ and
$B \cup e$ is a basis of $M'$.
As noted earlier, we obtain the fundamental graph
$G_{B \cup e'}(\Delta_{T}(M))$ from $G_{B}(M)$ by deleting
$e'$, adding a new vertex adjacent to the two elements
of $T - e'$, and labeling this new vertex $e'$.

In contrast, the fundamental graph $G_{B \cup e}(M')$ is
obtained from $G_{B}(M)$ by adding a vertex labeled with $e$
so that it is adjacent to the two members of $T - e'$.
Thus it is obvious that $\Delta_{T}(M)$ is isomorphic to
$M' \del e'$, under the isomorphism that labels $e$ with $e'$
and fixes the label on every other element. 
\end{proof}

\begin{prop}
\label{prop22}
Suppose that $T_{1}$ and $T_{2}$ are disjoint coindependent triangles
of the matroid $M$.
If $T_{1} \cup T_{2}$ contains a cocircuit $C^{*}$ of size four
then $\Delta_{T_{2}}(\Delta_{T_{1}}(M))$ contains a series pair. 
\end{prop}

\begin{proof}
For $i = 1,\, 2$ let $e_{i}$ be an element in $C^{*} \cap T_{i}$.
Now $C^{*} - \{e_{1},\, e_{2}\}$ is a series pair in
$M \del e_{1} \del e_{2}$.
It is not difficult to see, using Propositions~\ref{prop5}
and~\ref{prop28}, that
$M \del e_{1} \del e_{2} \iso
\Delta_{T_{2}}(\Delta_{T_{1}}(M)) / e_{1} / e_{2}$.
The result follows.
\end{proof}

\begin{lem}
\label{lem23}
Suppose that $T_{1} = \{a_{1},\, b_{1},\, c_{1}\}$ and
$T_{2} = \{a_{2},\, b_{2},\, c_{2}\}$ are disjoint coindependent
triangles in the binary matroid $M$ and that $a_{1}$ and $a_{2}$ are
not parallel in $M$.
Suppose that the binary matroid $M_{1}$ is produced from $M$ by
extending with the elements $x$, $y$, and $z$ so that
$T_{3} = \{x,\, y,\, z\}$ is a triangle of $M_{1}$ and
$\{a_{1},\, y\}$ and $\{a_{2},\, x\}$ are parallel pairs.
Suppose also that the binary matroid $M_{2}$ is obtained from
$\Delta_{T_{2}}(\Delta_{T_{1}}(M))$ by extending with the elements
$x$ and $y$ so that $\{b_{1},\, c_{1},\, x\}$ and
$\{b_{2},\, c_{2},\, y\}$ are triangles and then coextending with the element
$z$ so that $\{x,\, y,\, z\}$ is a triad.
Then
\begin{displaymath}
M_{2} = \Delta_{T_{3}}(\Delta_{T_{2}}(\Delta_{T_{1}}(M_{1}))).
\end{displaymath}
\end{lem}

\begin{proof}
Note that since $a_{1}$ and $a_{2}$ are not parallel in $M$
it is possible to construct $M_{1}$ is the manner described.

Since $a_{1}$ and $y$ are parallel in $M_{1}$ it follows that
they are also parallel in $M_{1} \del b_{1}$.
But Proposition~\ref{prop28} says that $M_{1} \del b_{1}$
is isomorphic to $\Delta_{T_{1}}(M_{1}) / b_{1}$ under the
isomorphism that swaps $a_{1}$ and $c_{1}$.
Thus there is a circuit of $\Delta_{T_{1}}(M_{1})$ contained
in $\{b_{1},\, c_{1},\, y\}$ that contains $c_{1}$ and
$y$.
As $T_{1}$ is a triad in $\Delta_{T_{1}}(M_{1})$ it follows that
$\{b_{1},\, c_{1},\, y\}$ must be a triangle in
$\Delta_{T_{1}}(M_{1})$.

Similarly, $a_{2}$ and $x$ are parallel in
$\Delta_{T_{1}}(M_{1})$, and hence in
$\Delta_{T_{1}}(M_{1}) \del b_{2}$.
But $\Delta_{T_{1}}(M_{1}) \del b_{2}$ is isomorphic to
$\Delta_{T_{2}}(\Delta_{T_{1}}(M_{1})) / b_{2}$ under
the isomorphism that swaps $a_{2}$ and $c_{2}$.
Using the same argument we see that
$\{b_{2},\, c_{2},\, x\}$ is a triangle of
$\Delta_{T_{2}}(\Delta_{T_{1}}(M_{1}))$.

Thus both $\{b_{1},\, c_{1},\, y\}$ and
$\{b_{2},\, c_{2},\, x\}$ are triangles in
$\Delta_{T_{2}}(\Delta_{T_{1}}(M_{1}))$, and hence in
$\Delta_{T_{2}}(\Delta_{T_{1}}(M_{1})) \del z$.
As
\begin{displaymath}
\Delta_{T_{2}}(\Delta_{T_{1}}(M_{1})) \del x \del y \del z
= \Delta_{T_{2}}(\Delta_{T_{1}}(M))
\end{displaymath}
it now follows that
$\Delta_{T_{2}}(\Delta_{T_{1}}(M_{1})) \del z$ is isomorphic
to $M_{2} / z$ under the isomorphism that swaps $x$ and $y$.
However
$\Delta_{T_{2}}(\Delta_{T_{1}}(M_{1})) \del z$ is also
isomorphic to $\Delta_{T_{3}}(\Delta_{T_{2}}(\Delta_{T_{1}}(M_{1}))) / z$
under the isomorphism that swaps $x$ and $y$.
As $\{x,\, y,\, z\}$ is a triad in both
$\Delta_{T_{3}}(\Delta_{T_{2}}(\Delta_{T_{1}}(M_{1}))$ and
$M_{2}$ we are done.
\end{proof}

Recall that a matroid is almost \vfc\ if it is
\vtc\ and whenever $(X,\, Y)$ is a \vts\ then
either $X$ contains a triad that spans $X$, or $Y$
contains a triad that spans $Y$.
We shall say that a
\emph{spanning triad}\index{spanning triad} of $X$ is a
triad contained in $X$ that spans $X$.
A spanning triad of $Y$ is defined analogously.

\begin{lem}
\label{prop20}
Suppose that $T$ is an independent triad of an almost
\vfc\ matroid $M$ where $r(M) \geq 3$.
Then $\nabla_{T}(M)$ is almost \vfc.
\end{lem}

\begin{proof}
We will say that a separation $(X,\, Y)$ is \emph{bad}
if it is a vertical $1$\dash\ or $2$\dash separation,
or if it is a \vts\ such that neither $X$ nor $Y$ contains
a spanning triad.
Let us assume that the result is false, so that $\nabla_{T}(M)$
has a bad separation.

\begin{subclm}
\label{clm3}
There is a bad separation $(X,\, Y)$ of
$\nabla_{T}(M)$ such that $T \subseteq X$.
\end{subclm}

\begin{proof}
Assume that the claim is false.
Let $(X,\, Y)$ be a bad separation.
Assume that $(X,\, Y)$ is a vertical $k$\dash separation
for some $k \leq 3$.
We can assume by relabeling if necessary that $X$
contains two elements of $T$.
Let $e$ be the element in $T \cap Y$.
Since $T$ is an independent triad of $M$ it follows
that $T$ is a triangle of $\nabla_{T}(M)$, so
$e \in \cl_{\nabla_{T}(M)}(X)$.
Thus Proposition~\ref{prop40} implies that
$(X \cup e,\, Y - e)$ is a vertical $k'$\dash separation
of $\nabla_{T}(M)$ for some $k' \leq k$.

By assumption $(X \cup e,\, Y - e)$ is not a bad separation
so $k = k' = 3$ and either $X \cup e$
or $Y - e$ contains a spanning triad.
Let us suppose that $X \cup e$ contains a spanning
triad $T'$.
Since $X$ does not contain a spanning triad, $T'$
contains $e$.
It cannot be the case that
$e \in \cl_{\nabla_{T}(M)}(Y - e)$, for that would
imply the existence of a circuit that meets $T'$ in
one element.
But now $r_{\nabla_{T}(M)}(Y - e) < r_{\nabla_{T}(M)}(Y)$,
and thus $k' < k$, contrary to hypothesis.

We conclude that $Y - e$ contains a spanning triad.
Because $Y$ does not contain a spanning triad
it follows that
$r_{\nabla_{T}(M)}(Y - e) < r_{\nabla_{T}(M)}(Y)$.
But this again leads to a contradiction, so the claim
holds.
\end{proof}

Let $(X,\, Y)$ be a bad separation of $\nabla_{T}(M)$
so that $(X,\, Y)$ is a vertical $k$\dash separation for
some $k \leq 3$.
By Claim~\ref{clm3} we will assume that $T \subseteq X$.
From Proposition~\ref{prop28} we see that
$\nabla_{T}(M)|Y$ is equal to $M|Y$ so
$r_{\nabla_{T}(M)}(Y) = r_{M}(Y)$.
It follows easily from Proposition~\ref{prop31}
that $(X,\, Y)$ is also a vertical $k$\dash separation
of $M$.
Thus $k = 3$ and either $X$ or $Y$ contains a spanning
triad in $M$.

Suppose that $Y$ contains a spanning triad in $M$.
Let $e$ be an element of $T$.
Then $Y$ contains a spanning triad in $M / e$.
But Proposition~\ref{prop28} tells us that
$M / e$ is isomorphic to $\nabla_{T}(M) \del e$.
As $T \subseteq X$ and $T$ is a triangle in
$\nabla_{T}(M)$ it follows that $Y$ contains a spanning
triad in $\nabla_{T}(M)$.
This is a contradiction as $(X,\, Y)$ is a bad separation
of $\nabla_{T}(M)$.
Thus $X$ contains a spanning triad in $M$ and therefore $X$
has rank three in $M$.
It now follows easily from Proposition~\ref{prop31}
that $X$ has rank two in $\nabla_{T}(M)$, contradicting the
fact that $(X,\, Y)$ is a \vts\ of $\nabla_{T}(M)$.
\end{proof}

We fix some notation before proving the next result.
Let $G = (V,\, E)$ be a graph.
If $V' \subseteq V$ then $G[V']$ is the subgraph induced
by $V'$, and $E(V')$ is the set of edges in $G[V']$.
If $E' \subseteq E$ then $V(E')$ is the set of vertices in
the subgraph induced by $E'$.
The \emph{cyclomatic number}\index{cyclomatic number} of $E'$,
denoted by $\xi(E')$, is the cardinality of the complement of a
spanning forest in the subgraph induced by $E'$.
That is, $\xi(E') = |E'| - r_{M(G)}(E')$.

Suppose that $G$ is connected.
It is well known, and easy to show, that if $E' \subseteq E$ and the
subgraphs induced by both $E'$ and $E - E'$ are connected
then
\begin{equation}
\label{eqn3}
\lambda_{M^{*}(G)}(E') = \lambda_{M(G)}(E') =
|V(E') \cap V(E - E')| - 1\ \textrm{and}
\end{equation}
\begin{equation}
\label{eqn1}
r_{M^{*}(G)}(E') = \xi(E') + |V(E') \cap V(E - E')| - 1.
\end{equation}

An \emph{edge cut-set}\index{edge cut-set} or a
\emph{cut-vertex}\index{cut-vertex} of a connected
graph is a set of edges or a vertex respectively whose deletion
disconnects the graph.
A \emph{block}\index{block} is a maximal connected subgraph
that has no cut-vertices.
Suppose that $G$ is a connected graph.
The \emph{block cut-vertex graph}\index{block cut-vertex graph} of $G$,
denoted by $\bc{G}$, has the blocks and the cut-vertices of $G$ as its vertex
set.
Every edge of \bc{G} joins a  block to a cut-vertex,
and a block $B_{0}$ is adjacent to a cut-vertex $c$ in
\bc{G} if and only if $c$ is a vertex of $B_{0}$.
It is well known that \bc{G} is a tree.
A connected graph is
\emph{$n$\dash connected}\index{connectivity!graph} if it
contains no set of fewer than $n$ vertices such that
deleting that set disconnects the graph.
A \emph{cycle minor}\index{cycle minor} of a graph is a minor isomorphic
to a cycle.

\begin{prop}
\label{prop42}
If $v_{1}$, $v_{2}$, and $v_{3}$ are distinct vertices
in $G$, a $2$\dash connected graph, then there is a cycle
minor of $G$ in which $v_{1}$, $v_{2}$, and $v_{3}$ are
distinct vertices.
\end{prop}

\begin{proof}
By Menger's Theorem there is a cycle $C$ that contains
$v_{1}$ and $v_{2}$.
If $C$ also contains $v_{3}$ then we are done, so we assume
otherwise.
Again by Menger's Theorem there are minimal length paths
$P_{1}$ and $P_{2}$ that join $v_{3}$ to $C$ such that
$P_{1}$ and $P_{2}$ meet only in the vertex $v_{3}$.
If either $P_{1}$ or $P_{2}$ meets $C$ in a vertex other than
$v_{1}$ or $v_{2}$, then we can contract that
path to obtain the desired cycle.
Thus we assume that, after relabeling, $P_{1}$ joins $v_{3}$
to $v_{1}$, and $P_{2}$ joins $v_{3}$ to $v_{2}$.
The result now follows easily.
\end{proof}

\begin{lem}
\label{prop24}
Suppose that $M$ is an almost \vfc\ non-cographic matroid
in \ex{\mkt} and that $T$ is an independent triad of $M$.
Then $\nabla_{T}(M)$ is non-cographic.
\end{lem}

\begin{proof}
Let us assume that the lemma is false, and that
there is a graph $G$ such that $\nabla_{T}(M)$ is
isomorphic to $M^{*}(G)$.
We can assume that $G$ has no isolated vertices, and
since  $M^{*}(G)$ is almost \vfc\ by Lemma~\ref{prop20} it
follows that $G$ is connected.
We can also assume that $M$, and therefore $\nabla_{T}(M)$,
has no loops, so $G$ has no vertices of degree one.
Because $M^{*}(G)$ is almost \vfc\ we deduce that $G$ has
no parallel edges.

Since $T$ is a triangle in $\nabla_{T}(M)$ it is a
minimal edge cut-set in $G$, so there exists a partition
$(A,\, B)$ of the vertex set of $G$ such that both $G[A]$
and $G[B]$ are connected, and the set of edges joining vertices
in $A$ to vertices in $B$ is exactly $T$.

First let us assume that only one vertex in $A$ is incident with
edges in $T$.
Then $A$ contains only one vertex, for otherwise $M^{*}(G)$ has a
$1$\dash separation. 
Since $G$ has no parallel edges there are three distinct vertices
in $B$, say $b_{1}$, $b_{2}$, and $b_{3}$, that are incident
with edges in $T$.
Let $G'$ be the graph produced by deleting $T$, and adding edges
between each pair of vertices in $\{b_{1},\, b_{2},\, b_{3}\}$.
It is well known that $M^{*}(G') \iso \Delta_{T}(\nabla_{T}(M)) = M$,
so $M$ is cographic, contrary to hypothesis.

Next we assume that only two vertices in $A$, let us call them
$a_{1}$ and $a_{2}$, are incident with edges in $T$.
It follows from Equation~\eqref{eqn3} that
$\lambda_{M^{*}(G)}(E(A)) = 1$.

Let $T = \{e_{1},\, e_{2},\, e_{3}\}$.
By relabeling we can assume that $a_{1}$ is incident with
$e_{1}$ and $e_{2}$, and that $a_{2}$ is incident with $e_{3}$.
Suppose that $e_{1}$ and $e_{2}$ are incident
with the vertices $b_{1}$ and $b_{2}$ in $B$. Since
$G[B]$ is connected there is a path joining $b_{1}$
to $b_{2}$ in $G[B]$, and hence there is a cycle in
$E(B) \cup T$.
It follows that the cyclomatic number of $E(B) \cup T$ is at
least one, so Equation~\eqref{eqn1} implies that the rank
of $E(B) \cup T$ in $M^{*}(G)$ is at least two.

Since $M^{*}(G)$ is almost \vfc, we deduce
that $r_{M^{*}(G)}(E(A)) < 2$.
By Equation~\eqref{eqn1} we see that $\xi(E(A)) < 1$.
Thus $G[A]$ is a tree, and since $G$ has no
degree-one vertices $G[A]$ is a path
joining $a_{1}$ to $a_{2}$.
Let $e_{3}'$ be the edge in $E(A)$ that is incident with
$a_{1}$ and let $T' = (T - e_{3}) \cup e_{3}'$.
The edges $e_{3}$ and $e_{3}'$ are in series in $G$, and
hence are in parallel in $M^{*}(G)$.
Therefore $T'$ is a triangle of $M^{*}(G) \iso \nabla_{T}(M)$
and $\Delta_{T'}(\nabla_{T}(M))$ is isomorphic
to $\Delta_{T}(\nabla_{T}(M)) = M$.
But because $e_{1}$, $e_{2}$, and $e_{3}'$ are incident
with a common vertex in $G$, we can show as before that
$\Delta_{T'}(\nabla_{T}(M))$ is cographic, a contradiction.

Now we assume that the three edges in $T$ are incident
with three distinct vertices in $A$, and by symmetry,
with three distinct vertices in $B$.
Let the three vertices in $A$ incident with edges in $T$ be
$a_{1}$, $a_{2}$, and $a_{3}$.

Suppose that $G[A]$ is a tree.
Since $G$ has no degree-one vertices there are at most three
degree-one vertices in $G[A]$.
Suppose that there are two, so that $G[A]$ is a path.
By relabeling assume that $a_{1}$ has degree two in $G[A]$.
Assume also that $a_{1}$ is incident with $e_{1}$, and the
two edges $e_{2}'$ and $e_{3}'$ in $E(A)$.
Then $T' = \{e_{1},\, e_{2}',\, e_{3}'\}$ is a triangle in
$M^{*}(G)$, and $e_{2}'$ and $e_{3}'$ are parallel to either
$e_{2}$ or $e_{3}$.
It follows that $\Delta_{T'}(\nabla_{T}(M)) \iso M$, and we
again have a contradiction, since the edges $e_{1}$, $e_{2}'$,
and $e_{3}'$ are incident with a common vertex.

Since $G[A]$ is not a path it contains a vertex $v$
of degree three.
Suppose that $v$ is incident with the three edges $e_{1}'$,
$e_{2}'$ and $e_{3}'$.
By letting $T' = \{e_{1}',\, e_{2}',\, e_{3}'\}$ we obtain a
similar contradiction.

We have shown that $G[A]$ is not a tree, so it has
cyclomatic number at least one, and hence $E(A)$ has
rank at least three in $M^{*}(G)$.
By using the same arguments we can also show that
$r_{M^{*}(G)}(E(B)) \geq 3$.

\begin{subclm}
\label{clm7}
There is a cycle minor of  $G[A]$ in which $a_{1}$,
$a_{2}$, and $a_{3}$ are distinct vertices.
\end{subclm}

\begin{proof}
If $G[A]$ is $2$\dash connected then the result follows
from Proposition~\ref{prop42}, so we assume that $G[A]$ contains
at least two blocks.
We have shown that $G[A]$ is not a tree, so let $B_{0}$ be a
block of $G[A]$ that contains a cycle.

If $B$ is a block of $G[A]$ with degree one in \bc{G[A]}
then $B$ contains one of the vertices $a_{1}$, $a_{2}$,
or $a_{3}$, and this vertex cannot be a cut-vertex in
$G[A]$, for otherwise $G$ contains a cut-vertex, and
this would imply that $M^{*}(G)$ is not connected.
Thus there are at most three degree-one vertices in
\bc{G[A]}.
Suppose that there are exactly three.

First assume that $B_{0}$ has degree less than three in
\bc{G[A]}.
Let $X$ be the edge set of $B_{0}$.
If $B_{0}$ has degree one in \bc{G[A]} then
it contains exactly one of the vertices $a_{1}$,
$a_{2}$, and $a_{3}$.
If $B_{0}$ has degree two in \bc{G[A]} then it
contains none of these vertices.
In either case it follows that $V(X)$ and $V(E(G) - X)$
meet in exactly two vertices.
Moreover, the cyclomatic number of $X$ is at least one
(since $B_{0}$ contains a cycle).
We have already concluded that $E(B)$ has rank at least three
in $M^{*}(G)$, so now Equations~\eqref{eqn3} and~\eqref{eqn1}
imply that $M^{*}(G)$ has a vertical $2$\dash separation,
a contradiction.

Therefore $B_{0}$ has degree at least three in
\bc{G[A]}, and since \bc{G[A]} has exactly three
degree-one vertices it follows that $B_{0}$ has degree
exactly three.
Let $c_{1}$, $c_{2}$, and $c_{3}$ be the cut-vertices of
$G[A]$ that are contained in $B_{0}$.
There are three vertex-disjoint paths in $G[A]$ that
join the vertices in $\{c_{1},\, c_{2},\, c_{3}\}$ to the
vertices in $\{a_{1},\, a_{2},\, a_{3}\}$.
Now $B_{0}$ has a cycle minor in which $c_{1}$, $c_{2}$,
and $c_{3}$ are disjoint vertices by Proposition~\ref{prop42},
so we can find a cycle minor of $G[A]$ in which $a_{1}$,
$a_{2}$, and $a_{3}$ are disjoint vertices.

Next we will assume that \bc{G[A]} has exactly two
degree-one vertices, $B_{1}$ and $B_{2}$, so that
\bc{G[A]} is a path.
By relabeling we will assume that $a_{i}$ is a vertex
of $B_{i}$ for $i = 1,\, 2$.

Let $X$ be the edge set of $B_{0}$.
Suppose that either $a_{3}$ is not contained in the block
$B_{0}$, or that $a_{3}$ is a cut-vertex of $G[A]$.
Then $V(X)$ and $V(E(G) - X)$ have exactly two common vertices.
Moreover, the cyclomatic number of $X$ is at least one,
so we can obtain a contradiction as before.
Therefore $a_{3}$ is contained in $B_{0}$, and is not a
cut-vertex of $G[A]$.
In fact, none of the vertices $a_{1}$, $a_{2}$, or $a_{3}$
can be a cut-vertex of $G[A]$.

Suppose that $B_{0}$ is equal to $B_{1}$.
Let $c$ be the unique cut-vertex of $G[A]$ that is
contained in $B_{0}$.
Then $a_{1}$, $a_{3}$, and $c$ are distinct vertices in
$B_{0}$, so $B_{0}$ has a cycle minor in which
$a_{1}$, $a_{3}$ and $c$ are distinct vertices by
Proposition~\ref{prop42}.
Furthermore, there is a path in $G[A]$ joining $c$
to $a_{2}$, so the result follows.
Therefore let us assume that $B_{0}$ is not equal to
$B_{1}$.
(By symmetry we assume that it is not equal to $B_{2}$.)
There are two cut-vertices $c_{1}$ and $c_{2}$ of
$G[A]$ that are in $B_{0}$.
Then $B_{0}$ has a cycle minor in which $a_{3}$,
$c_{1}$, and $c_{2}$ are distinct vertices.
Furthermore, there are two vertex-disjoint paths
of $G[A]$ joining the vertices in
$\{c_{1},\, c_{2}\}$ to the vertices in
$\{a_{1},\, a_{2}\}$.
Again we can find the desired cycle minor of $G[A]$, so
the result follows.
\end{proof}

We can also apply the arguments of Claim~\ref{clm7} to
show $G[B]$ has a cycle minor in which the three vertices
in $B$ that meet the edges of $T$ are distinct vertices.
Thus $G$ has a minor $G'$ isomorphic to the graph shown in
Figure~\ref{fig13}, where $T$ is the set of edges joining
the two triangles.

Proposition~\ref{prop30} tells us that
$\Delta_{T}(M^{*}(G'))$ is a minor of
$\Delta_{T}(M^{*}(G)) \iso \Delta_{T}(\nabla_{T}(M)) = M$.
But $\Delta_{T}(M^{*}(G'))$ is isomorphic to
\mkt, so we have a contradiction that completes the
proof of the lemma.
\end{proof}

\begin{figure}[htb]

\centering

\includegraphics{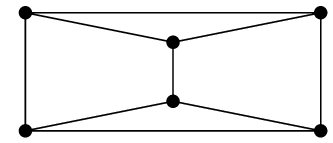}

\caption{$G'$, a minor of $G$.}

\label{fig13}

\end{figure}

\begin{prop}
\label{prop27}
Suppose that $T$ is an independent triad in the matroid $M$.
If $M$ has no \mkt\dash minor then neither does $\nabla_{T}(M)$.
\end{prop}

\begin{proof}
Suppose that $\nabla_{T}(M)$ does have an \mkt\dash minor.
Since $T$ is a triangle in $\nabla_{T}(M)$ and \mkt\ has no
triangles there is some element $e \in T$ such that
$\nabla_{T}(M) \del e$ has an \mkt\dash minor.
But Proposition~\ref{prop28} tells us that
$\nabla_{T}(M) \del e \iso M / e$.
Thus $M$ has an \mkt\dash minor, a contradiction.
\end{proof}

\begin{prop}
\label{prop21}
Suppose that $M$ is an almost \vfc\ non-cographic
matroid in \ex{\mkt} and that $T$ is a triad of $M$.
Every triad of $\nabla_{T}(M)$ is a triad of $M$.
\end{prop}

\begin{proof}
Lemma~\ref{prop24} implies that $\nabla_{T}(M)$ is
non-cographic. 
Thus if $r(\nabla_{T}(M)) \leq 3$ it follows
that $\si(\nabla_{T}(M)) \iso F_{7}$.
In this case $\nabla_{T}(M)$ has no triads and the result
follows.
Therefore we will assume that $r(\nabla_{T}(M)) > 3$.

Suppose that $T'$ is a triad of $\nabla_{T}(M)$ but not of
$M$.
It cannot be the case that $T' \cap T = \varnothing$, for then
$T'$ would be a triad of $\nabla_{T}(M) / T$, which is equal
to $M / T$ by Proposition~\ref{prop28}.
This would imply that $T'$ is a triad of $M$.
Thus $T'$ has a non-empty intersection with $T$.
Since $M$, and hence $\nabla_{T}(M)$, is binary, and $T$ is
a triangle of $\nabla_{T}(M)$, it follows that $T'$ meets
$T$ in exactly two elements.
Let $X$ equal $T \cup T'$ and let $Y = E(\nabla_{T}(M)) - X$.
Now $(X,\, Y)$ is a \vts\ of $\nabla_{T}(M)$ and it is
easy to see that it is also a \vts\ of
$\Delta_{T}(\nabla_{T}(M)) = M$.
It follows easily from Proposition~\ref{prop31} that $X$ has rank
four in $M$.
Since $M$ is almost \vfc\ we conclude that $Y$ contains a
spanning triad in $M$.
But now it is easy to check that $\nabla_{T}(M)$ is
cographic (in fact, up to the addition of parallel points,
$\nabla_{T}(M)$ is a minor of the cographic matroid
corresponding to the graph shown in Figure~\ref{fig13}),
and this is contrary to our earlier conclusion.
\end{proof}

Let $M$ be a binary matroid, and let \mcal{T} be a
multiset of coindependent triangles of $M$.
For each element $e \in E(M)$ let $t_{e}$ denote the
number of triangles in \mcal{T} that contain $e$.
Suppose that $M_{0}$ is obtained from $M$ by the following
procedure:
For each element $e \in E(M)$, if $t_{e} > 1$ then add
$t_{e} - 1$ parallel elements to $e$.
Then $\si(M_{0}) \iso M$, so each triangle of
$M_{0}$ naturally corresponds to a triangle of $M$.
We can find a set $\mcal{T}_{0}$ of pairwise disjoint
coindependent triangles of $M_{0}$ such that
$|\mcal{T}_{0}| = |\mcal{T}|$ and each triangle in $\mcal{T}_{0}$
corresponds to a triangle in \mcal{T}.
Now we let $\Delta(M;\, \mcal{T})$\index{DeltaC@$\Delta(M;\, \mcal{T})$,
$\nabla(M^{*};\, \mcal{T})$}
be the matroid obtained from $M_{0}$ by performing \dy\ operations
in turn on each of the triangles in $\mcal{T}_{0}$.
It is clear that $\Delta(M;\, \mcal{T})$ is well-defined
up to isomorphism.

We shall denote $\Delta(M;\, \mcal{T})^{*}$ by
$\nabla(M^{*};\, \mcal{T})$.
It is easy to see that $\nabla(M^{*};\, \mcal{T})$ can
be obtained from $M_{0}^{*}$ by performing \yd\ operations
on the triads in $\mcal{T}_{0}$, where $M_{0}^{*}$
is obtained from $M^{*}$ by adding $t_{e} - 1$ series elements
to each element $e$ of $M^{*}$ when $t_{e} > 1$,
and $\mcal{T}_{0}$ is a set of pairwise disjoint
independent triads of $M_{0}^{*}$ corresponding to the
set \mcal{T}.

\begin{lem}
\label{lem12}
Suppose that $M$ is a binary matroid and that
\mcal{T} is a multiset of coindependent triangles of $M$.
Let $T$ be a triangle in \mcal{T}.
If there exists a triangle $T' \ne T$ in \mcal{T}
such that $T \cap T' \ne \varnothing$, then
$\Delta(M;\, \mcal{T})$ has a minor isomorphic to
$\Delta(M;\, \mcal{T} - T)$.
\end{lem}

\begin{proof}
Let $e$ be an element contained in both $T$ and
$T'$.
Then there is an element $e' \in E(M_{0}) - E(M)$
such that $e'$ is parallel to $e$ in $M_{0}$.
Let $T_{0}$ and $T_{0}'$ be disjoint triangles of
$M_{0}$ that correspond to $T$ and $T'$ respectively.
We can assume that $e' \in T_{0}$.
It follows from Proposition~\ref{prop19} that
$\Delta(M;\, \mcal{T} - T)$ can be obtained from
$\nabla_{T_{0}}(\Delta(M;\, \mcal{T}))$ by deleting those
elements of $T_{0}$ that are parallel in $M_{0}$ to some
element of $M$ that is contained in more than one
triangle of \mcal{T}.
In particular, $\Delta(M;\, \mcal{T} - T)$ is a minor
of $\nabla_{T_{0}}(\Delta(M;\, \mcal{T})) \del e'$.
But Proposition~\ref{prop28} tells us that
$\nabla_{T_{0}}(\Delta(M;\, \mcal{T})) \del e'$ is isomorphic
to $\Delta(M;\, \mcal{T}) / e'$.
The result follows.
\end{proof}

Several of the sporadic matroids described in Appendix~\ref{chp9}
are best understood as matroids of the form
$\nabla(M;\, \mcal{T})$.
In particular, the matroids $M_{7,15}$, $M_{9,18}$, and
$M_{11,21}$ are isomorphic to $\nabla(F_{7}^{*};\, \mcal{T})$,
where \mcal{T} is a set of five, six, or seven triangles
respectively in the Fano plane $F_{7}$.
Moreover, if $\mcal{T}_{\,\,4}^{a}$ is a set of four
triangles in $F_{7}$ such that three triangles contain
a common point, then
$\nabla(F_{7}^{*};\, \mcal{T}_{\,\,4}^{a}) \iso M_{5,12}^{a}$.
If $\mcal{T}_{\,\,4}^{b}$ is a set of four triangles in
$F_{7}$ such that no three triangles contain a common
point then
$\nabla(F_{7}^{*};\, \mcal{T}_{\,\,4}^{b}) \iso M_{6,13}$.

Let $M = F_{7}$, and suppose that \mcal{T} is a non-empty set of
triangles in $M$.
Let $M_{0}$ be the matroid obtained from $M$ by adding parallel
points, as described earlier, and suppose that
$\mcal{T}_{0} = \{T_{1},\ldots, T_{n}\}$ is a set of disjoint
triangles in $M_{0}$ that corresponds to \mcal{T}.
For $1 \leq i \leq n$ let $M_{i}$ be the matroid obtained
by performing \dy\ operations on the triangles
$T_{1},\ldots, T_{i}$.
Clearly $M_{0}$ has no minor isomorphic to $M^{*}(K_{3,3})$.
Therefore, if $M_{i}$ does have an $M^{*}(K_{3,3})$\dash minor
for some $i \in \{1,\ldots, n\}$, there exists an integer
$i$ such that $M_{i}$ has an $M^{*}(K_{3,3})$\dash minor but
$M_{i-1}$ does not.

However, $M^{*}(K_{3,3})$ contains no triads, and $T_{i}$ is
a triad of $M_{i}$.
Therefore there is an element $e \in T_{i}$ such that
$M_{i} / e$ has an $M^{*}(K_{3,3})$\dash minor.
Now Proposition~\ref{prop28} implies that $M_{i-1} \del e$
has an $M^{*}(K_{3,3})$\dash minor, contrary to our
assumption.
Therefore $M_{i}$ does not have an $M^{*}(K_{3,3})$\dash minor
for any integer $i \in \{1,\ldots, n\}$.

In particular, $M_{n-1}$ does not have an
$M^{*}(K_{3,3})$\dash minor.
Proposition~\ref{prop19} tells us that $M_{n-1}$ is isomorphic
to $\nabla_{T_{n}}(\Delta(M;\, \mcal{T}))$.
Since the triangles in $\mcal{T}_{0}$ are disjoint
Proposition~\ref{prop29} asserts that the order in which
we apply the \dy\ operation to them is immaterial.
The above argument now shows that
$\nabla_{T}(\Delta(M;\, \mcal{T}))$ has no
$M^{*}(K_{3,3})$\dash minor for any triangle
$T \in \mcal{T}_{0}$.

Dualising, we see that neither
$\nabla(M^{*};\, \mcal{T})$ nor
$\Delta_{T}(\nabla(M^{*};\, \mcal{T}))$
has an \mkt\dash minor, for any triangle
$T \in \mcal{T}_{0}$.

In general, if $T$ is a triangle of a matroid
$M \in \ex{\mkt}$ and $\Delta_{T}(M)$ has no
\mkt\dash minor, then we shall say that $T$ is an
\emph{allowable triangle}\index{allowable triangle}.
The previous arguments assert the following fact.

\begin{prop}
\label{prop37}
Suppose that $\mcal{T}$ is a non-empty set of triangles
in $F_{7}$.
Then each triangle in \mcal{T} corresponds to an allowable
triangle in $\nabla(F_{7}^{*};\, \mcal{T})$.
\end{prop}

\chapter{M\"{o}bius matroids}
\label{chp2}

In this chapter we define in detail the two infinite classes featured
in Theorem~\ref{thm2}.
In Section~\ref{chp3.5} we consider their non-cographic minors.

The \emph{cubic \mob\ ladder}\index{M\"{o}bius ladder!cubic}
\cml{2n} is obtained from the even cycle on the vertices
$v_{0},\ldots, v_{2n-1}$ by joining each vertex $v_{i}$ to
the antipodal vertex $v_{i+n}$ (indices are read modulo~$2n$).
The \emph{quartic \mob\ ladder}\index{M\"{o}bius ladder!quartic}
\qml{2n+1} is obtained from the odd cycle on the vertices
$v_{0},\ldots, v_{2n}$ by joining each vertex $v_{i}$ to the two
antipodal vertices $v_{i+n}$ and $v_{i+n+1}$ (in this case indices are
read modulo~$2n+1$).
In both cases, the edges of the cycle are known as
\emph{rim edges}\index{rim edges}, and the edges between antipodal
vertices are \emph{spoke edges}\index{spoke edges}.
We note that $\cml{4} \iso K_{4}$,
$\cml{6} \iso K_{3,3}$, and
$\qml{5} \iso K_{5}$.
In addition, \qml{3} is isomorphic to the graph obtained
from $K_{3}$ by replacing each edge with a parallel pair.

\section{Triangular M\"{o}bius matroids}
\label{chp2.2}

Let $r$ be an integer exceeding two and let
$\{e_{1},\ldots, e_{r}\}$ be the standard basis in the
vector space of dimension $r$ over \gf{2}.
For $1 \leq i \leq r-1$ let $a_{i}$ be the sum of
$e_{i}$ and $e_{r}$, and for $1 \leq i \leq r-2$
let $b_{i}$ be the sum of $e_{i}$ and $e_{i+1}$.
Let $b_{r-1}$ be the sum of $e_{1}$, $e_{r-1}$,
and $e_{r}$.
The \emph{rank\dash $r$ triangular
\mob\ matroid}\index{M\"{o}bius matroid!triangular},
denoted by $\Delta_{r}$, is represented over \gf{2} by the set
$\{e_{1},\ldots, e_{r},\, a_{1},\ldots, a_{r-1},\,
b_{1},\ldots, b_{r-1}\}$.
Thus $\Delta_{r}$ has rank~$r$ and $|E(\Delta_{r})| = 3r - 2$. 
We also take $\{e_{1},\ldots, e_{r},\, a_{1},\ldots, a_{r-1},\,
b_{1},\ldots, b_{r-1}\}$ to be the ground set of
$\Delta_{r}$.

Figure~\ref{fig1} shows a matrix $A$ such that
$\Delta_{4}$ is represented over \gf{2} by
$[I_{4}|A]$, and also the fundamental graph
$G_{B}(\Delta_{4})$, where $B$ is the basis
$\{e_{1},\ldots, e_{4}\}$.

\begin{figure}[htb]

\centering

\includegraphics{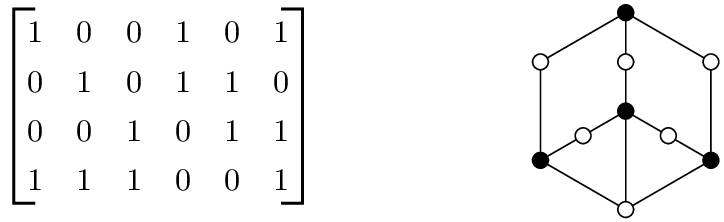}

\caption{Matrix and fundamental graph representations
of $\Delta_{4}$.}

\label{fig1}

\end{figure}

Figure~\ref{fig2} shows geometric representations of
$\Delta_{4}$ and $\Delta_{5}$.
Since $\Delta_{5}$ has rank~$5$ we cannot draw an orthodox
representation, but Figure~\ref{fig2} gives an idea of
its structure by displaying its triangles.

\begin{figure}[htb]

\centering

\includegraphics{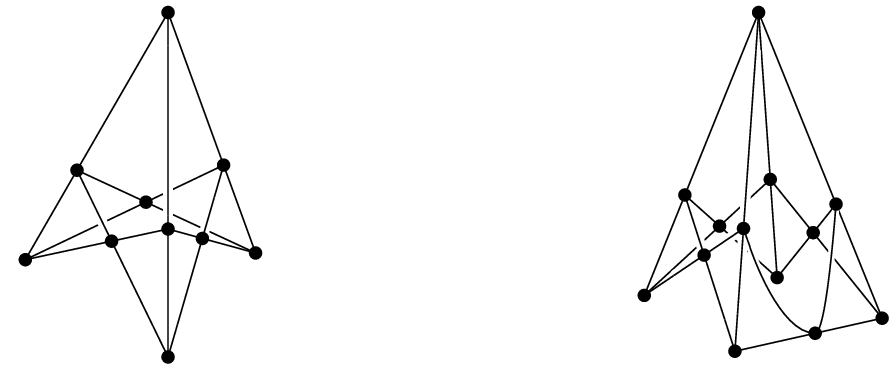}

\caption{Geometric representations of $\Delta_{4}$ and
$\Delta_{5}$.}

\label{fig2}

\end{figure}

Let $2n \geq 2$ be an even integer.
For $i \in {1,\ldots, n}$ let the edge of \cml{2n}
that joins $v_{i-1}$ to $v_{i}$ be labeled $e_{i}$ and let the edge that joins
$v_{i+n-1}$ to $v_{i+n}$ be labeled $a_{i}$.
Let the edge that joins $v_{i}$ to $v_{i+n}$ be labeled $b_{i}$.
Figure~\ref{fig3} shows two drawings of \cml{6}
equipped with this labeling.
Under this labeling
$\Delta_{r} \del e_{r} = M^{*}(\cml{2r-2})$ for $r \geq 3$.
For this reason we will refer to the elements
$e_{1},\ldots, e_{r-1},\, a_{1},\ldots, a_{r-1}$ as
the \emph{rim elements}\index{rim elements} of $\Delta_{r}$, while
referring to the elements $b_{1},\ldots, b_{r-1}$ as
\emph{spoke elements}\index{spoke elements}.
We will also call $e_{r}$ the \emph{tip}\index{tip} of $\Delta_{r}$.

\begin{figure}[htb]

\centering

\includegraphics{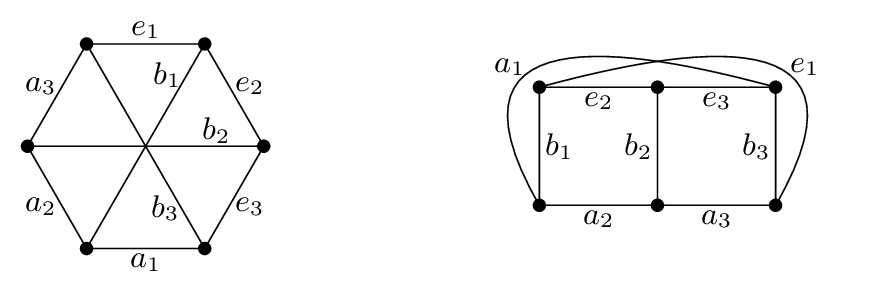}

\caption{Two drawings of the cubic \mob\ ladder on six
vertices.}

\label{fig3}

\end{figure}

The triangular \mob\ matroids do not contain any triads.
When $r \geq 3$ the triangles of $\Delta_{r}$ are the
sets $\{a_{i},\, e_{i},\, e_{r}\}$ for $1 \leq i \leq r - 1$,
the sets $\{a_{i},\, a_{i+1},\, b_{i}\}$ and
$\{e_{i},\, e_{i+1},\, b_{i}\}$ for $1 \leq i \leq r - 2$,
and the sets $\{a_{1},\, e_{r-1},\, b_{r-1}\}$ and
$\{a_{r-1},\, e_{1},\, b_{r-1}\}$.

It is easy to see that $\Delta_{3}$ is isomorphic to $F_{7}$,
the Fano plane.
The rank\dash $4$ triangular \mob\ matroid is known to
Zhou as $\widetilde{K}_{5}$~\cite{Zho04} and to Kung
as $C_{10}$~\cite{Kun86}.
Moreover, $\Delta_{r}$ is known to Kingan and Lemos
as $S_{3r-2}$~\cite{KL02}.

\section{Triadic M\"{o}bius matroids}
\label{chp2.3}

Let $r \geq 4$ be an even integer, and again let
$\{e_{1},\ldots, e_{r}\}$ be the standard basis of
the vector space over \gf{2} of dimension $r$.
Let $c_{i}$ be the sum of $e_{i}$,\, $e_{i+1}$, and
$e_{r}$ for $1 \leq i \leq r - 2$.
Let $c_{r-1}$ be the sum of $e_{1}$,\, $e_{r-1}$, and
$e_{r}$.
The \emph{rank\dash $r$ triadic
\mob\ matroid}\index{M\"{o}bius matroid!triadic}, denoted
by $\Upsilon_{r}$, is represented over \gf{2} by the set
$\{e_{1},\ldots, e_{r},\, c_{1},\ldots, c_{r-1}\}$.
Thus $\Upsilon_{r}$ has rank~$r$ and
$|E(\Upsilon_{r})| = 2r - 1$.
Again we take $\{e_{1},\ldots, e_{r},\, c_{1},\ldots, c_{r-1}\}$
to be the ground set of $\Upsilon_{r}$.
Figure~\ref{fig4} shows a matrix $A$ such that
$[I_{4}|A]$ represents $\Upsilon_{4}$ over \gf{2}
and the fundamental graph $G_{B}(\Upsilon_{4})$, where
$B = \{e_{1},\ldots, e_{4}\}$.
Figure~\ref{fig5} shows a geometric representation of
$\Upsilon_{4}$.

\begin{figure}[htb]

\centering

\includegraphics{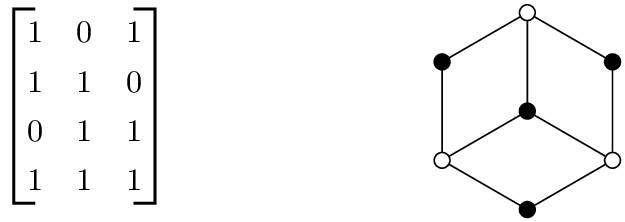}

\caption{Matrix and fundamental graph representations
of $\Upsilon_{4}$.}

\label{fig4}

\end{figure}

\begin{figure}[htb]

\centering

\includegraphics{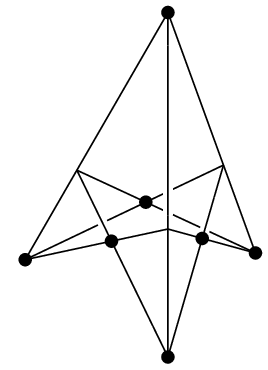}

\caption{A geometric representation of $\Upsilon_{4}$.}

\label{fig5}

\end{figure}

Let $2n + 1 \geq 5$ be an odd integer.
For $i \in \{1,\ldots, 2n + 1\}$ let $e_{i}$ be the edge of
\qml{2n+1} that joins $v_{ni}$ to $v_{ni+1}$.
Let $c_{i}$ be the edge that joins $v_{n(i-1)}$ to $v_{ni}$.
Figure~\ref{fig6} shows two drawings of \qml{7}
labeled in this way.
In addition we label the edges of \qml{3} so that
the parallel pairs are $\{c_{1},\, e_{3}\}$, $\{c_{2},\, e_{1}\}$,
and $\{c_{3},\, e_{2}\}$.
Now $\Upsilon_{r} \del e_{r} = M^{*}(\qml{r-1})$
for any even integer $r \geq 4$.
Thus we will refer to the elements $e_{1},\ldots, e_{r-1}$ as the
\emph{rim elements}\index{rim elements} of $\Upsilon_{r}$ and the elements
$c_{1},\ldots, c_{r-1}$ as \emph{spoke elements}\index{spoke elements}.
We will call $e_{r}$ the \emph{tip}\index{tip} of $\Upsilon_{r}$.

\begin{figure}[htb]

\centering

\includegraphics{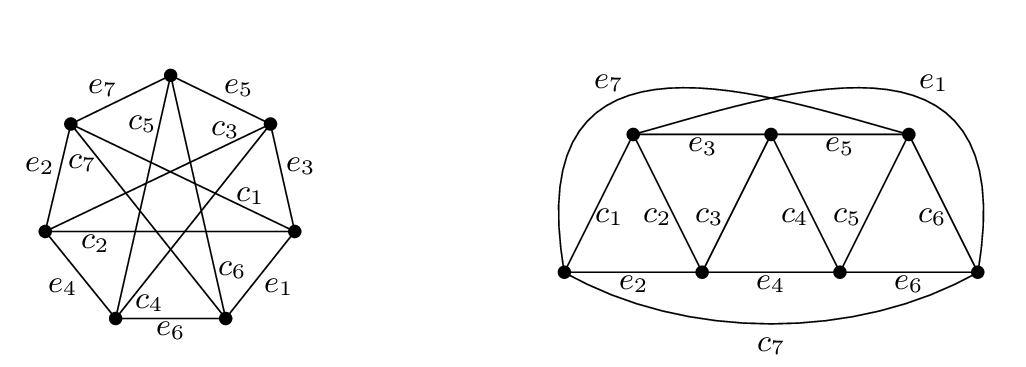}

\caption{Two drawings of the quartic \mob\ ladder on
seven vertices.}

\label{fig6}

\end{figure}

The triadic \mob\ matroids contain no triangles.
For an even integer $r \geq 4$ the triads of $\Upsilon_{4}$ are
the sets $\{c_{i},\, c_{i+1},\, e_{i+1}\}$ for
$1 \leq i \leq r - 2$ and $\{c_{1},\, c_{r-1},\, e_{1}\}$.
When $r = 4$ the sets $\{e_{1},\, e_{2},\, e_{3}\}$,
$\{c_{1},\, e_{3},\, e_{4}\}$, $\{c_{2},\, e_{1},\, e_{4}\}$,
and $\{c_{3},\, e_{2},\, e_{4}\}$ are also triads.
The rank\dash $4$ triadic \mob\ matroid is isomorphic to
the Fano dual, $F_{7}^{*}$.
The rank\dash $6$ triadic \mob\ matroid is isomorphic to any single
element deletion of the matroid $T_{12}$, introduced by
Kingan~\cite{Kin97}.
In addition $\Upsilon_{r}^{*}$ is known to Kingan and Lemos
as $F_{2r-1}$~\cite{KL02}.

The following proposition is not difficult to confirm.

\begin{prop}
\label{prop1}
Every \mob\ matroid is \ifc.
\end{prop}

\section{Minors of M\"{o}bius matroids}
\label{chp3.5}

If $r \geq 4$ then we can obtain $\Delta_{r-1}$ from
$\Delta_{r}$ by contracting a spoke element and deleting
an element from each of the two parallel pairs that
result.
This is made more formal in the next (easily
proved) result.

\begin{prop}
\label{prop2}
Let $r \geq 4$. We can obtain a minor of $\Delta_{r}$
isomorphic to $\Delta_{r-1}$ by:
\begin{enumerate}[(i)]
\item contracting $b_{i}$, deleting an element
from $\{e_{i},\, e_{i+1}\}$ and an element from
$\{a_{i},\, a_{i+1}\}$ where $1 \leq i \leq r - 2$; or,
\item contracting $b_{r-1}$, deleting an element
from $\{a_{1},\, e_{r-1}\}$ and an element from
$\{a_{r-1},\, e_{1}\}$.
\end{enumerate}
\end{prop}

There are two ways to obtain an $\Upsilon_{r-2}$\dash minor
from $\Upsilon_{r}$, where $r \geq 6$ is an even integer.
In the first we contract two consecutive spoke elements. This
produces a triangle that contains $e_{r}$ and whose closure
contains a parallel pair. We delete the element of this
triangle that is not $e_{r}$ and is not in the parallel
pair, and then we delete an element from the parallel pair.

\begin{prop}
\label{prop10}
Let $r \geq 6$ be an even integer. We can obtain a minor
of $\Upsilon_{r}$ isomorphic to $\Upsilon_{r-2}$ by:
\begin{enumerate}[(i)]
\item contracting $c_{i}$ and $c_{i+1}$, deleting
$e_{i+1}$ and an element from $\{e_{i},\, e_{i+2}\}$ where
$1 \leq i \leq r - 3$;
\item contracting $c_{1}$ and $c_{r-1}$, deleting
$e_{1}$ and an element from $\{e_{2},\, e_{r-1}\}$; or,
\item contracting $c_{r-2}$ and $c_{r-1}$, deleting
$e_{r-1}$ and an element from $\{e_{1},\, e_{r-2}\}$.
\end{enumerate}
\end{prop}

The other method involves contracting two spoke elements
that are separated by one other spoke element, and then
deleting the rim elements that lie ``between'' them.

\begin{prop}
\label{prop15}
Let $r \geq 6$ be an even integer. We can obtain a minor
of $\Upsilon_{r}$ isomorphic to $\Upsilon_{r-2}$ by:
\begin{enumerate}[(i)]
\item contracting $c_{i}$ and $c_{i+2}$ and deleting
$e_{i+1}$ and $e_{i+2}$, where $1 \leq i \leq r - 3$;
\item contracting $c_{1}$ and $c_{r-2}$ and deleting
$e_{1}$ and $e_{r-1}$; or,
\item contracting $c_{2}$ and $c_{r-1}$ and deleting
$e_{1}$ and $e_{2}$.
\end{enumerate}
\end{prop}

\begin{prop}
\label{prop26}
Suppose that $(e_{1},\ldots, e_{4})$ is a fan of the matroid $M$.
If $M \del e_{1}$ is cographic then $M$ is cographic.
\end{prop}

\begin{proof}
Let $X = \{e_{1},\, e_{2},\, e_{3},\, e_{4}\}$.
First suppose that $X - e_{1}$ is not a triad in $M \del e_{1}$.
Then $e_{1} \in \cl_{M}^{*}(X - e_{1})$, so there
is a cocircuit $C^{*} \subseteq X$ such that $e_{1} \in C^{*}$.
It follows from cocircuit exchange that $e_{1}$ is contained in
a series pair in $M$.
The result follows easily.

Suppose that $X - e_{1}$ is a triad in $M \del e_{1}$.
Let $G$ be a graph such that $M \del e_{1} = M^{*}(G)$.
Then $X - e_{1}$ is the edge set of a triangle in $G$.
Let $u$ be the vertex of $G$ incident with both $e_{2}$ and $e_{3}$.
Let $G'$ be the graph formed from $G$ by deleting $u$ and replacing
it with $v$ and $w$, where $v$ is incident with $e_{1}$, $e_{2}$, and
$e_{3}$, and $w$ is incident with $e_{1}$ and all edges incident with
$u$ other than $e_{2}$ and $e_{3}$.
Then $M = M^{*}(G')$.
\end{proof}

\begin{prop}
\label{prop13}
Suppose $r \geq 3$. Any minor produced from $\Delta_{r}$
by one of the following operations is cographic.
\begin{enumerate}[(i)]
\item Deleting or contracting $e_{r}$;
\item Contracting $e_{i}$ or $a_{i}$ for
$1 \leq i \leq r - 1$; or,
\item Deleting $b_{i}$ for $1 \leq i \leq r - 1$.
\end{enumerate}
\end{prop}

\begin{proof}
We have already noted (and it is easy to confirm) that
$\Delta_{r} \del e_{r} \iso M^{*}(\cml{2r - 2})$.
It is also easy to confirm that $\Delta_{r} / e_{r}$ is
isomorphic to a matroid obtained from the
rank\dash $(r - 1)$ wheel by adding $r - 1$ parallel elements.
Thus $\Delta_{r} / e_{r}$ is cographic.
This completes the proof of (i).

It follows from (i) that $\Delta_{r} \del e_{r} / e_{i}$
is cographic, where $1 \leq i \leq r - 1$.
But $\Delta_{r} / e_{i}$ is obtained from
$\Delta_{r} \del e_{r} / e_{i}$ by adding $e_{r}$ in
parallel to $a_{i}$.
Thus $\Delta_{r} / e_{i}$ is cographic, and the same argument shows
that $\Delta_{r} / a_{i}$ is also cographic.

If $1 \leq i \leq r - 2$ then $(e_{r},\, a_{i+1},\, e_{i+1},\, b_{i+1})$
is a fan of $\Delta_{r} \del b_{i}$.
By (i) we know that $\Delta_{r} \del b_{i} \del e_{r}$ is cographic.
Proposition~\ref{prop26} now implies that
$\Delta_{r} \del b_{i}$ is cographic.
The case when $i = r - 1$ is similar.
\end{proof}

\begin{prop}
\label{prop14}
Suppose $r \geq 4$ is an even integer.
Any minor produced from $\Upsilon_{r}$ by one of the following
operations is cographic.
\begin{enumerate}[(i)]
\item Deleting or contracting $e_{r}$;
\item Contracting $e_{i}$ for $1 \leq i \leq r - 1$; or,
\item Deleting $c_{i}$ for $1 \leq i \leq r - 1$.
\end{enumerate}
\end{prop}

\begin{proof}
We have noted that
$\Upsilon_{r} \del e_{r} \iso M^{*}(\qml{r-1})$.
It is easy to see that $\Upsilon_{r} / e_{r}$ is isomorphic
to the rank\dash $(r - 1)$ wheel.

If $1 \leq i \leq r - 2$ then $(e_{r},\, c_{i},\, e_{i+1},\, c_{i+1})$
is a fan of $\Upsilon_{r} / e_{i}$.
Since $\Upsilon_{r} / e_{i} \del e_{r}$ is
cographic it follows that $\Upsilon_{r} / e_{i}$
is cographic by Proposition~\ref{prop26}.
A similar argument holds when $i = r - 1$.

Finally, for $1 \leq i \leq r - 1$, the minor
$\Upsilon_{r} \del e_{r} \del c_{i}$ is
isomorphic to $M^{*}(G)$ where $G$ is shown in
Figure~\ref{fig16}.
It is now easy to confirm that
$\Upsilon_{r} \del c_{i} \iso M^{*}(G')$.
\end{proof}

\begin{figure}[htb]

\centering

\includegraphics{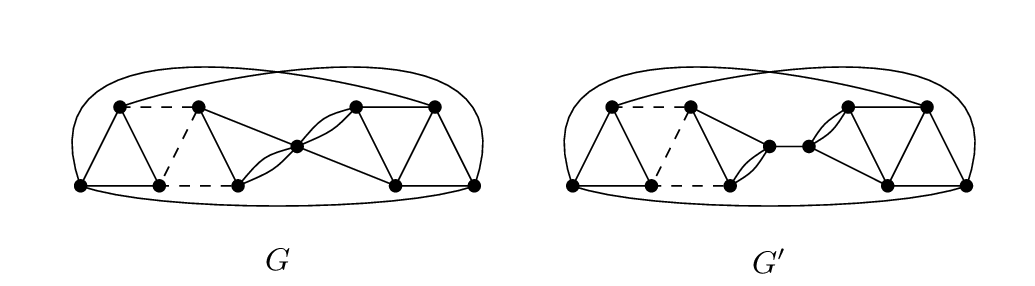}

\caption{$\Upsilon_{r} \del e_{r} \del c_{i} \iso M^{*}(G)$ and
$\Upsilon_{r} \del c_{i} \iso M^{*}(G')$.}

\label{fig16}

\end{figure}

\begin{lem}
\label{lem4}
Suppose that $M$ is a \mob\ matroid and $N$ is an
\ifc\ non-cographic minor of $M$.
Then $N$ is also a \mob\ matroid.
\end{lem}

\begin{proof}
Suppose that the lemma fails, so that $M$ is a \mob\
matroid, and $N$ is an \ifc\ non-cographic minor of $M$
such that $N$ is not a \mob\ matroid.
Let $C$ and $D$ be independent and coindependent subsets of
$E(M)$ respectively such that $N = M / C \del D$.
We will assume that the lemma does not fail for any
\mob\ matroid smaller than $M$. 

We first suppose that $M$ is the triangular \mob\ matroid $\Delta_{r}$.
The result is true if $r = 3$, so we assume that $r \geq 4$.
It follows from Proposition~\ref{prop13} that $C$ can contain
only spoke elements.
However, if we contract a spoke element of $M$ then we obtain
two parallel pairs, and at least one element from each must
be deleted to obtain $N$.
Thus it follows from Proposition~\ref{prop2} that $N$ is a minor
of $\Delta_{r-1}$, and this contradicts our assumption on the
minimality of $M$.
Thus we assume that $C$ is empty.

Proposition~\ref{prop13} shows that $D$ contains only rim elements
of $M$.
The set $\{b_{i-1},\, b_{i}\}$ is a series pair in
$M \del e_{i} \del a_{i}$ for $1 \leq i \leq r - 1$ (henceforth
subscripts are to be read modulo $r - 1$; we identify the subscript
$r - 1$ with zero), so if $D$ contains both $e_{i}$ and $a_{i}$, then at
least one of $b_{i-1}$ or $b_{i}$ is contained in $C$, contrary
to our earlier conclusion.
Thus $D$ contains at most one element from each set
$\{e_{i},\, a_{i}\}$.

By relabeling we can assume that $e_{1} \in D$.
But $(a_{2},\, b_{1},\, a_{1},\, b_{r-1})$ is a fan of
$M \del e_{1}$, so
\begin{displaymath}
\lambda_{M \del e_{1}}(\{a_{1},\, a_{2},\, b_{1},\, b_{r-1}\}) \leq 2.
\end{displaymath}
Suppose that $C \cup D$ contains none of the elements
$\{a_{1},\, a_{2},\, b_{1},\, b_{r-1}\}$.
Since $N$ is \ifc\ it follows from Proposition~\ref{prop7}
that the complement of $\{a_{1},\, a_{2},\, b_{1},\, b_{r-1}\}$
in $N$ contains at most three elements.
Thus $|E(N)| \leq 7$, and since $N$ is a non-cographic matroid
this means that $N$ is isomorphic to either $F_{7}$ or $F_{7}^{*}$.
Both of these are \mob\ matroids, so this is a contradiction.
Thus we must either delete or contract at least one element
from $\{a_{1},\, a_{2},\, b_{1},\, b_{r-1}\}$ to obtain $N$.

By the above discussion and Proposition~\ref{prop13} we must
delete $a_{2}$.
But $(e_{3},\, b_{2},\, e_{2},\, b_{1})$ is a fan of $M \del a_{2}$, so
using the same argument as before we see that we are forced to
delete or contract an element in $\{b_{1},\, b_{2},\, e_{2},\, e_{3}\}$
to obtain $N$.
Since $C$ is empty and $D$ does not contain $\{a_{2},\, e_{2}\}$
we must delete $e_{3}$.
Continuing in this way we see that the elements
$e_{1},\, a_{2},\, e_{3},\, a_{4},\ldots$ are all contained in $D$.

If $r$ is odd then $a_{r-1} \in D$.
However $(e_{r-1},\, b_{r-1},\, a_{1},\, b_{1})$ is a fan of
$M \del e_{1}$, so by using the earlier arguments we can show
that $e_{r-1} \in D$.
Thus $D$ contains both $e_{r-1}$ and $a_{r-1}$, contrary
to our earlier conclusion.

Hence $r$ is even.
It is easy to check that the matroid obtained from
$\Delta_{r}$ by deleting
$\{e_{1},\, a_{2},\ldots, a_{r-2},\, e_{r-1}\}$
is isomorphic to $\Upsilon_{r}$.
Thus $N$ is a minor of $\Upsilon_{r}$, and this
contradicts our assumption on the minimality of $M$.

Now we suppose that $M$ is the triadic \mob\ matroid
$\Upsilon_{r}$.
The result holds if $r = 4$, so we assume that $r \geq 6$.
Assume that $C$ is empty.
By Proposition~\ref{prop14} we have to delete an element $e_{i}$.
But $\{c_{i-1},\, c_{i}\}$ is a series pair of $M \del e_{i}$,
so one of these elements is contracted.
Thus $C$ is not empty.

\begin{subclm}
\label{clm1}
If $c_{i},\, c_{i+1} \in C$, but $c_{i+2} \notin C$, then
$c_{i+3} \in C$.
\end{subclm}

\begin{proof}
Assume that the claim is false, so $c_{i},\, c_{i+1} \in C$, but
$c_{i+2},\, c_{i+3} \notin C$.
Now $(e_{i+1},\, c_{i+2},\, e_{i+3},\, c_{i+3})$ is a fan of
$M / c_{i+1}$, so one of the elements in
$\{e_{i+1},\, e_{i+3},\, c_{i+2},\, c_{i+3}\}$
is deleted or contracted to obtain $N$.
By our assumption either $e_{i+1}$ or $e_{i+3}$ is in $D$.
But $\{c_{i+2},\, c_{i+3}\}$ is a series pair in
$M \del e_{i+3}$, so if $e_{i+3} \in D$ then either $c_{i+2}$
or $c_{i+3}$ is in $C$, a contradiction.
Thus $e_{i+1} \in D$.
But $\{e_{i},\, e_{i+2}\}$ is a parallel pair in
$M / c_{i} / c_{i+1}$, so one of these elements belongs
to $D$.
Now it follows from Proposition~\ref{prop10}
that $N$ is a minor of $\Upsilon_{r-2}$, and this
contradicts our assumption on the minimality of $M$.
\end{proof}

\begin{subclm}
\label{clm2}
If $c_{i} \notin C$ then $c_{i-1}$ and $c_{i+1}$ are
both in $C$.
\end{subclm}

\begin{proof}
Let us assume that the claim fails.
Since $C \ne \varnothing$, we can assume by symmetry,
relabeling if necessary, that $c_{i},\, c_{i+1} \notin C$, but
that $c_{i+2} \in C$.
It cannot be the case that $c_{i+3} \in C$, for then,
by  Claim~\ref{clm1} and symmetry it follows that $c_{i} \in C$.
Since $\{c_{i},\, c_{i+1}\}$ is a series pair of
$M  \del e_{i+1}$ it follows that $e_{i+1} \notin D$.
Note that $(e_{i+3},\, c_{i+1},\, e_{i+1},\, c_{i})$ is a fan
of $M / c_{i+2}$.
Since $e_{i+1} \notin D$ it follows that $e_{i+3} \in D$.

The set $\{c_{i+1},\, c_{i+3}\}$ is a series pair
of $M \del e_{i+3} \del e_{i+2}$.
Since $e_{i+3} \in D$, but neither $c_{i+1}$ nor $c_{i+3}$
is in $C$ this means that $e_{i+2} \notin D$.

Now $(e_{i+4},\, c_{i+3},\, e_{i+2},\, c_{i+1})$ is a fan of
$M / c_{i+2} \del e_{i+3}$, and since $e_{i+2} \notin D$ it
follows that $e_{i+4} \in D$.
But $\{c_{i+3},\, c_{i+4}\}$ is a series pair of
$M \del e_{i+4}$, and since $c_{i+3} \notin C$ this
implies that $c_{i+4} \in C$.

We have shown that $c_{i+2},\, c_{i+4} \in C$ and
that $e_{i+3},\, e_{i+4} \in D$. It now follows from
Proposition~\ref{prop15} that $N$ is a minor of
$\Upsilon_{r-2}$.
This contradiction proves the claim.
\end{proof}

Claim~\ref{clm2} implies that of any pair of consecutive
spoke elements at least one belongs to $C$.
Now it is an easy matter to check that $M / C$ is isomorphic
to a matroid obtained from a triangular \mob\ matroid by
(possibly) adding parallel elements to rim elements.
Thus $N$ is a minor of a triangular \mob\ matroid that
is in turn a proper minor of $M$.
This contradiction completes the proof of Lemma~\ref{lem4}.
\end{proof}

\begin{cor}
\label{cor2}
No \mob\ matroid has an \mkt\dash minor.
\end{cor}

\begin{proof}
In the light of Lemma~\ref{lem4}, to prove the corollary we
need only check that \mkt\ is not a \mob\ matroid.
This is trivial.
\end{proof}

\chapter{From internal to vertical connectivity}
\label{chp4}

The purpose of this chapter is to develop the machinery
we will need to show that a minimal counterexample to
Theorem~\ref{thm2} can be assumed to be \vfc.
Much of the material we use has already been introduced in
Section~\ref{chp3.6} of Chapter~\ref{chp3}.

Recall that a matroid is almost \vfc\ if it is \vtc,
and whenever $(X,\, Y)$ is a \vts, then either $X$
contains a triad that spans $X$, or $Y$ contains a
triad that spans $Y$.
Throughout this chapter we will suppose that $M$ is an
\ifc\ non-cographic member of \ex{\mkt}.
Then $M$ is almost \vfc\ by definition.
If $M$ has no triads then $M$ is \vfc, so we assume that
$T$ is a triad of $M$.
Since $M$ is non-cographic it has rank at least three,
and thus, as $M$ is almost \vfc, it follows that $T$ is
independent.
Lemmas~\ref{prop20} and~\ref{prop24} and
Proposition~\ref{prop27} imply that $\nabla_{T}(M)$ is an
almost \vfc\ non-cographic member of \ex{\mkt}.
Proposition~\ref{prop21} implies that $\nabla_{T}(M)$
has strictly fewer triads than $M$.
Thus we can repeat this process until we obtain a matroid
$M_{0}$ that has no triads, and which is therefore \vfc.
By then deleting all but one element from every parallel
class of $M_{0}$ we obtain a simple \vfc\ matroid.

Suppose that while reducing $M$ to a \vfc\ matroid
we perform the \yd\ operation on the triads
$T_{1},\ldots, T_{n}$ in that order.
Let $\mcal{T}_{0} = \{T_{1},\ldots, T_{n}\}$.

\begin{clm}
\label{clm5}
The triads in $\mcal{T}_{0}$ are pairwise
disjoint.
\end{clm}

\begin{proof}
Let $N_{0} = M$, and for $1 \leq i \leq n$ let $N_{i}$ be
the matroid obtained by performing \yd\ operations on the
triads $T_{1},\ldots, T_{i}$.
Suppose that $T_{i}$ and $T_{j}$ have a non-empty intersection
for $1 \leq i < j \leq n$.
Repeated application of Proposition~\ref{prop21} shows that
$T_{j}$ is a triad of $N_{i}$.
Let $e$ be an element in $T_{i} \cap T_{j}$.
Then $N_{i} \del e$ contains a series pair.
But $N_{i} = \nabla_{T_{i}}(N_{i-1})$, and
$\nabla_{T_{i}}(N_{i-1}) \del e$ is isomorphic
to $N_{i-1} / e$ by Proposition~\ref{prop28}.
Hence $N_{i-1}$ contains a series pair.
But this is a contradiction as the matroids
$N_{0},\ldots, N_{n}$ are all almost \vfc.
\end{proof}

Each member of $\mcal{T}_{0}$ is a triangle of $M_{0}$.
We can recover $M$ from $M_{0}$ by performing a \dy\
operation on each of these triangles in turn.
Proposition~\ref{prop29} and Claim~\ref{clm5} tell us
that the order in which we perform the \dy\ operations is
immaterial.
Obviously each triangle in $\mcal{T}_{0}$ can be identified
with a triangle of $\si(M_{0})$.
Let \mcal{T} be the set of triangles in $\si(M_{0})$
that corresponds to the set of triangles $\mcal{T}_{0}$ in
$M_{0}$.

We may not be able to recover $M$ from $\si(M_{0})$ by
performing \dy\ operations on the triangles of
\mcal{T}, because these triangles may not be
disjoint.
Thus recovering $M$ from $\si(M_{0})$ involves one more step,
namely adding parallel elements to reconstruct $M_{0}$.
The next two results show how we can perform this step
using a knowledge of \mcal{T}.

\begin{clm}
\label{clm6}
Let $T$ and $T'$ be triangles of $M_{0}$
such that $T,\, T' \in \mcal{T}_{0}$.
Then $r_{M_{0}}(T \cup T') > 2$.
\end{clm}

\begin{proof}
Let us assume otherwise.
Then $r_{M_{0}}(T \cup T') = 2$.
Because the order in which we applied the \yd\ operation
to the members of $\mcal{T}_{0}$ while obtaining $M_{0}$ is
immaterial, we can assume that $T = T_{n-1}$ and $T' = T_{n}$.
Since $M_{0}$ is non-cographic by repeated application of
Lemma~\ref{prop24}, it follows that the rank and corank
of $M_{0}$ are at least three.
Now it follows easily from Lemma~\ref{lem8} that
$M_{0}$ has a minor $N$ such that $N$ is isomorphic to the
matroid obtained from $M(K_{4})$ by adding a point in
parallel to each element of a triangle, and both $T$ and
$T'$ are triangles of $N$.
Proposition~\ref{prop30} tells us that
$\Delta_{T}(\Delta_{T'}(N))$ is a minor of
$\Delta_{T}(\Delta_{T'}(M_{0}))$.
But $\Delta_{T}(\Delta_{T'}(M_{0}))$ is obtained from $M$ by
performing \yd\ operations on the triads $T_{1},\ldots, T_{n-2}$
in turn.
Proposition~\ref{prop27} implies that this matroid has
no \mkt\dash minor.
However $\Delta_{T}(\Delta_{T'}(N))$ is isomorphic to
\mkt, so we have a contradiction.
\end{proof}

Claim~\ref{clm6} shows that distinct triangles in
$\mcal{T}_{0}$ correspond to distinct triangles in \mcal{T}.
Therefore the number of triangles in \mcal{T} is exactly
equal to the number of triangles in $\mcal{T}_{0}$.

\begin{clm}
\label{clm10}
Let $F$ be a rank-one flat of $M_{0}$ such that $|F| > 1$.
The size of $F$ is exactly equal to the number of triangles
in $\mcal{T}_{0}$ that have a non-empty intersection with $F$.
\end{clm}

\begin{proof}
Suppose that no member of $\mcal{T}_{0}$ has a non-empty
intersection with $F$.
Since $M$ is recovered from $M_{0}$ by performing
\dy\ operations on the triangles in $\mcal{T}_{0}$ it follows from
Proposition~\ref{prop28} that $M$ contains a parallel pair.
This is a contradiction as $M$ is \ifc.
Therefore there is an element $e \in F$ and a member $T$
of $\mcal{T}_{0}$ such that $e \in T$.

Suppose that $e' \in F$ and $e'$ is contained in no member
of $\mcal{T}_{0}$.
Clearly $T$ is a triad of $\Delta_{T}(M_{0})$ and it is easy to
see that $(T - e) \cup e'$ is a triangle of $\Delta_{T}(M_{0})$.
Moreover, it follows from Proposition~\ref{prop28} that
$T$ is a triad and $(T - e) \cup e'$ is a triangle in
$M$.
Thus $\lambda_{M}(T \cup e') \leq 2$.
Since $M$ is \ifc\ the complement of $T \cup e'$ in $M$
contains at most three elements, so $|E(M)| \leq 7$.
Since $M$ is non-cographic this means that $M$ is isomorphic
to either $F_{7}$ or $F_{7}^{*}$.
However $F_{7}$ has no triads and $F_{7}^{*}$ has no triangles,
which is a contradiction as $M$ has both a triangle and a triad.
Therefore each element of $F$ is contained in at least one member
of $\mcal{T}_{0}$, and by Claim~\ref{clm5} each element of
$F$ is contained in exactly one member of $\mcal{T}_{0}$.
The result follows.
\end{proof}

Using Claims~\ref{clm6} and~\ref{clm10}
we can recover $M_{0}$ from $\si(M_{0})$ as
follows:
For each element $e \in E(\si(M_{0}))$ let $t_{e}$
be the number of triangles in \mcal{T} that contain $e$.
If $t_{e} > 1$ then add $t_{e} - 1$ parallel elements to $e$.
The resulting matroid is isomorphic to $M_{0}$.
Now we can find a set of pairwise disjoint triangles that
correspond to the triangles in \mcal{T} and perform
\dy\ operations on each of them.
In other words, $M$ is isomorphic to the matroid
$\Delta(\si(M_{0});\, \mcal{T})$ described in
Section~\ref{chp3.6} of Chapter~\ref{chp3}.

Next we examine what restrictions apply to the set \mcal{T}.
Let $T$ be a triangle of $\si(M_{0})$ such that $T \in \mcal{T}$.
We can assume that $T$ is also a triangle of $M_{0}$.
Thus $T \in \mcal{T}_{0}$ and we will assume that $T = T_{n}$.
The matroid $\Delta_{T}(M_{0})$ is equal to that obtained
from $M$ by performing \yd\ operations on the triads
$T_{1},\ldots, T_{n-1}$, and therefore has no
\mkt\dash minor by Proposition~\ref{prop27}.
But $\si(M_{0})$ is a minor of $M_{0}$, so
Proposition~\ref{prop30} implies the following fact.

\begin{clm}
\label{clm4}
$\Delta_{T}(\si(M_{0}))$ has no \mkt\dash minor.
\end{clm}

Recall that $T$ is an allowable triangle of a matroid
$M \in \ex{\mkt}$ if $\Delta_{T}(M)$ has no
\mkt\dash minor.
Claim~\ref{clm4} asserts that \mcal{T} contains
only allowable triangles of $\si(M_{0})$.

Suppose that $T,\, T' \in \mcal{T}_{0}$ are triangles
of $M_{0}$, and that $T \cup T'$ contains a cocircuit of
size four.
Proposition~\ref{prop22} tells us that $M$ contains a series
pair, which is a contradiction.
Therefore, if $T,\, T' \in \mcal{T}_{0}$ are triangles
of $M_{0}$ then $T \cup T'$ does not contain
a cocircuit of size four.
Suppose that $T,\, T' \in \mcal{T}$ are triangles of
$\si(M_{0})$ such that $T \cup T'$ does contain a
cocircuit $C^{*}$ of size four in $\si(M_{0})$.
It is easy to see that $T \cap C^{*}$ and
$T' \cap C^{*}$ are disjoint sets of size two.
It cannot be the case that $C^{*}$ is a cocircuit in $M_{0}$.
Thus some element in $C^{*}$ is in a parallel pair in $M_{0}$.
Now it follows easily from Claim~\ref{clm10} that there
is some triangle $T'' \in \mcal{T}$ such that
$T''$ meets both $T \cap C^{*}$ and $T' \cap C^{*}$.

The restrictions that apply to the set \mcal{T}
are summarized here.
\begin{enumerate}[(i)]
\item Every triangle in \mcal{T} is an allowable
triangle of $\si(M_{0})$.
\item If $T$ and $T'$ are in \mcal{T} and $C^{*}$ is a
four-element cocircuit of $\si(M_{0})$ contained in
$T \cup T'$, then \mcal{T} contains a triangle $T''$ that
meets both $T \cap C^{*}$ and $T' \cap C^{*}$.
\end{enumerate}
Any set of triangles that obeys these two conditions
will be called a \emph{legitimate set}\index{legitimate set}.

We first consider the case that $\si(M_{0})$ is a sporadic
matroid.

\begin{lem}
\label{lem9}
Suppose that $M$ is an \ifc\ non-cographic matroid
in \ex{\mkt} and that $M$ has at least one triad.
Let $M_{0}$ be the matroid obtained from $M$ by
repeatedly performing \yd\ operations
until the resulting matroid has no triads.
If $\si(M_{0})$ is isomorphic to any of the sporadic
matroids listed in Appendix~{\rm\ref{chp9}} then $\si(M_{0})$ is
isomorphic to $M_{4,11}$, and $M$ is isomorphic to
$M_{5,11}$.
\end{lem}

\begin{proof}
Appendix~\ref{chp9} lists $18$ sporadic matroids.
Each is simple and \vfc, except for $M_{5,11}$, which
is not \vfc.
For each of the remaining $17$ matroids $M$ we
will find all possible non-empty legitimate sets of
triangles, and for each such set \mcal{T}, we will
construct $\Delta(M;\, \mcal{T})$.
By the preceding discussion, if the lemma fails then this
procedure will uncover at least one \ifc\ non-cographic
member of \ex{\mkt} other than $M_{5,11}$.
By the results of Appendix~\ref{chp10}, there are only six
sporadic matroids that contain allowable triangles, so
we need only consider these six.

Performing a \dy\ operation on any one of the
allowable triangles in $M_{4,11}$ produces a matroid
isomorphic to $M_{5,11}$.\cross\
If \mcal{T} is a legitimate set of triangles in
$M_{4,11}$ that contains more than one triangle
then $\Delta(M_{4,11};\, \mcal{T})$ has an
\mkt\dash minor.\cross

Since the allowable triangles in $M_{5,12}^{a}$,
$M_{6,13}$, $M_{7,15}$, $M_{9,18}$, and $M_{11,21}$
are pairwise disjoint and any pair contains a cocircuit
of size four it follows that a non-empty legitimate set
contains exactly one allowable triangle.
The following argument demonstrates that performing a
\dy\ operation on a single allowable triangle in any
of these matroids produces a matroid that is not \ifc.

Suppose that $M_{1}$ is one of the five matroids listed in the
previous paragraph.
Then $M_{1} \iso \nabla(F_{7}^{*};\, \mcal{T})$, where \mcal{T}
is a set of at least four triangles in $F_{7}$.
Let $T$ be an allowable triangle in $M_{1}$.
We will show that $(\Delta_{T}(M_{1}))^{*}$ is not \ifc.
Since internal $4$\dash connectivity is preserved by
duality this will suffice.

Let $M = F_{7}$, so that $M_{1}^{*} \iso \Delta(M;\, \mcal{T})$.
Let $M'$ be the matroid obtained
by adding parallel elements to $M$ in the appropriate way
so that we can find a set $\mcal{T}_{0} = \{T_{1},\ldots, T_{n}\}$
of disjoint triangles in $M'$ such that each of these
corresponds to a triangle in \mcal{T}.
Thus $M_{1}^{*}$ is isomorphic to the matroid obtained
from $M'$ by performing \dy\ operations on each of the
triangles in $\mcal{T}_{0}$ in turn.
Since $T$ is an allowable triangle of $M_{1}$ it is
a member of $\mcal{T}_{0}$ by Proposition~\ref{prop37}
and the results in Appendix~\ref{chp10}.
Since the triangles in $\mcal{T}_{0}$ are disjoint we can
assume that $T = T_{n}$.
Note that $(\Delta_{T}(M_{1}))^{*} = \nabla_{T}(M_{1}^{*})$ is
isomorphic to $\nabla_{T_{n}}(\Delta(M;\, \mcal{T}))$ and this
matroid is equal to the matroid obtained from $M'$ by
performing \dy\ operations on the triangles $T_{1},\ldots, T_{n-1}$.
For any integer $i \in \{1,\ldots, n - 1\}$ there is
a parallel class $F$ of $M'$ such that
$T_{i}$ contains an element $e \in F$ and $T_{n}$ contains
an element $e' \in F$.
Now $T_{i}$ is a triad in $\nabla_{T}(\Delta(M;\, \mcal{T}))$,
and it is easy to see that $(T_{i} - e) \cup e'$ is a
triangle.
Therefore $\lambda(T_{i} \cup e') = 2$, and it follows
easily that $\nabla_{T}(\Delta(M;\, \mcal{T}))$ is not \ifc.
Thus $(\Delta_{T}(M_{1}))^{*}$ is not \ifc.
This completes the proof.
\end{proof}

Next we consider the case that $\si(M_{0})$ is isomorphic to
a triangular \mob\ matroid.
Recall that the triangles of $\Delta_{r}$ are exactly the sets
$\{a_{i},\, e_{i},\, e_{r}\}$ where $1 \leq i \leq r - 1$,
the sets $\{a_{i},\, a_{i+1},\, b_{i}\}$ and
$\{e_{i},\, e_{i+1},\, b_{i}\}$ where $1 \leq i \leq r - 2$,
along with the sets $\{a_{1},\, e_{r-1},\, b_{r-1}\}$
and $\{a_{r-1},\, e_{1},\, b_{r-1}\}$.

\begin{clm}
\label{clm11}
Let $r \geq 4$ be an integer.
The triangle $\{a_{i},\, e_{i},\, e_{r}\}$ is not
allowable in $\Delta_{r}$ for $1 \leq i \leq r - 1$.
\end{clm}

\begin{proof}
Let  $T = \{a_{i},\, e_{i},\, e_{r}\}$. There is an
automorphism of $\Delta_{r}$ taking $T$ to
$\{a_{1},\, e_{1},\, e_{r}\}$, so we will assume that
$i = 1$.
If $r \geq 5$ then Proposition~\ref{prop2} implies
that by contracting the elements $b_{3},\ldots, b_{r-2}$
from $\Delta_{r}$ we obtain a minor isomorphic to
$\Delta_{4}$ (up to the addition of parallel elements) in
which $T$ is a triangle.
Proposition~\ref{prop30} shows that $\Delta_{T}(\Delta_{4})$
is a minor of $\Delta_{T}(\Delta_{r})$.
Figure~\ref{fig8} gives matrix and fundamental graph
representations of $\Delta_{T}(\Delta_{4})$.
This matroid has an \mkt\dash minor so we are
done.\cross
\end{proof}

\begin{figure}[htb]

\centering

\includegraphics{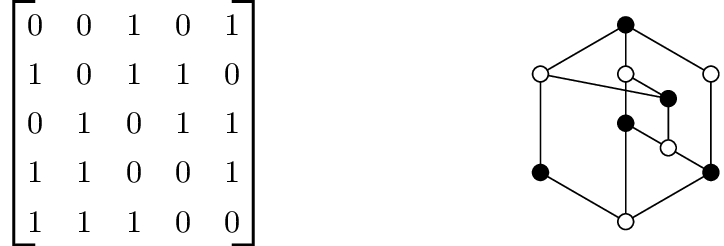}

\caption{Matrix and fundamental graph representations
of $\Delta_{T}(\Delta_{4})$.}

\label{fig8}

\end{figure}

\begin{clm}
\label{clm12}
Let $r \geq 4$ be an integer.
Suppose that $N$ is obtained from $\Delta_{r}$ by adding
an element $b_{i}'$ in parallel to $b_{i}$, where
$i \in \{1,\ldots, r-1\}$.
Let $T$ and $T'$ be the disjoint triangles of
$N$ that contain $b_{i}$ and $b_{i}'$ respectively.
Then $\Delta_{T}(\Delta_{T'}(N))$ has an
\mkt\dash minor.
\end{clm}

\begin{proof}
As in the previous proof we can assume
that $T = \{b_{1},\, e_{1},\, e_{2}\}$ and that
$T' = \{a_{1},\, a_{2},\, b_{1}'\}$.
If $r \geq 5$ then we again obtain $N'$ by contracting the
elements $b_{3},\ldots, b_{r-2}$, so that $N'$ is isomorphic
to $\Delta_{4}$ up to the addition of parallel elements.
It remains to consider
$\Delta_{T}(\Delta_{T'}(N'))$.
Figure~\ref{fig7} shows matrix and fundamental
graph representations of this matroid.
It has an \mkt\dash minor, so by Proposition~\ref{prop30}
we are done.\cross
\end{proof}

\begin{figure}[htb]

\centering

\includegraphics{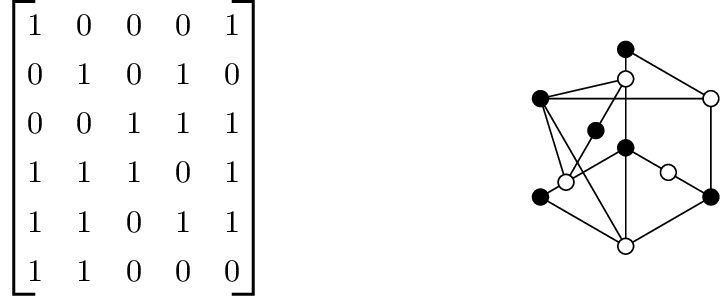}

\caption{Matrix and fundamental graph representations
of $\Delta_{T}(\Delta_{T'}(N'))$.}

\label{fig7}

\end{figure}

For any positive integer $r \geq 3$ let $\Delta_{r}^{\circ}$ be
the matroid obtained by adding a parallel point to each rim
element of $\Delta_{r}$.
Suppose that the ground set of $\Delta_{r}^{\circ}$ is
$E(\Delta_{r}) \cup \{a_{1}',\ldots, a_{r-1}'\} \cup
\{e_{1}',\ldots, e_{r-1}'\}$, where $a_{i}'$ is parallel
to $a_{i}$ and $e_{i}'$ is parallel to $e_{i}$ for
$1 \leq i \leq r - 1$.

\begin{lem}
\label{lem10}
Let $r \geq 3$ be an integer and let $T$ be a triangle of
$\Delta_{r}^{\circ}$ that contains a spoke element $b$.
Then $\Delta_{T}(\Delta_{r}^{\circ})$ is isomorphic
to a restriction of $\Delta_{r+1}^{\circ}$, and this
isomorphism takes spoke elements of
$\Delta_{r}^{\circ}$ other than $b$ to spoke elements of
$\Delta_{r+1}^{\circ}$.
\end{lem}

\begin{proof}
It follows from Proposition~\ref{prop2} (and is easy to
confirm) that we can obtain an isomorphic copy of
$\Delta_{r}^{\circ}$ from $\Delta_{r+1}^{\circ}$ by deleting
two elements from $\{a_{r},\, a_{r}',\, e_{1},\, e_{1}'\}$,
two elements from $\{a_{1},\, a_{1}',\, e_{r},\, e_{r}'\}$,
and then contracting $b_{r}$.
In particular, if we delete
$\{a_{r}',\, e_{1}',\, e_{r},\, e_{r}'\}$, contract $b_{r}$, and
then perform the following relabeling of elements
we obtain a matroid that is actually equal to $\Delta_{r}^{\circ}$.
\begin{displaymath}
e_{r+1} \rightarrow e_{r} \quad a_{r} \rightarrow e_{1}'
\end{displaymath}

Let $N = \Delta_{r+1}^{\circ} \del a_{r}' \del e_{1}'
\del e_{r} \del e_{r}'$ and let
$T' = \{a_{r},\, b_{r-1},\, b_{r}\}$.
Now $T'$ is a triad of $N$.
It follows from the discussion following
Proposition~\ref{prop28} that
$N / b_{r}$ is isomorphic to $\nabla_{T'}(N) \del b_{r}$,
where the isomorphism involves switching the labels
on $a_{r}$ and $b_{r-1}$.
From this we derive another way to obtain $\Delta_{r}^{\circ}$
from $\Delta_{r+1}^{\circ}$:
Delete $\{a_{r}',\, e_{1}',\, e_{r},\, e_{r}'\}$, then
perform a \yd\ operation on $\{a_{r},\, b_{r-1},\, b_{r}\}$.
Then delete $b_{r}$ and perform the following relabeling.
\begin{displaymath}
e_{r+1} \rightarrow e_{r} \quad a_{r} \rightarrow b_{r-1} \quad
b_{r-1} \rightarrow e_{1}'
\end{displaymath}

Note that $\{a_{r-1}',\, a_{r},\, b_{r-1}\}$ is a triangle of
$N$.
It is easy to see that it is also a triangle in $\nabla_{T'}(N)$.
However, $T' = \{a_{r},\, b_{r-1},\, b_{r}\}$ is a triangle
in $\nabla_{T'}(N)$, so it follows that $a_{r-1}'$ and $b_{r}$
are parallel in $\nabla_{T'}(N)$.
Thus we can delete $a_{r-1}'$ from $\nabla_{T'}(N)$ instead
of $b_{r}$.
But $\nabla_{T'}(N) \del a_{r-1}' = \nabla_{T'}(N \del a_{r-1}')$
since $a_{r-1}' \notin T'$.
This provides us with yet another way of deriving
$\Delta_{r}^{\circ}$ from $\Delta_{r+1}^{\circ}$:
Delete $\{a_{r-1}',\, a_{r}',\, e_{1}',\, e_{r},\, e_{r}'\}$.
Perform a \yd\ operation on $\{a_{r},\, b_{r-1},\, b_{r}\}$
and then perform the following relabeling.
\begin{displaymath}
e_{r+1} \rightarrow e_{r} \quad a_{r} \rightarrow b_{r-1} \quad
b_{r-1} \rightarrow e_{1}' \quad b_{r} \rightarrow a_{r-1}'
\end{displaymath}

Suppose that $T$ is a triangle of $\Delta_{r}^{\circ}$ that
contains a spoke element.
There is an automorphism of $\Delta_{r}^{\circ}$ that takes
$T$ to $\{a_{r-1}',\, b_{r-1},\, e_{1}'\}$, and this
automorphism takes spoke elements to spoke elements.
Therefore we will assume that $T$ is equal to
$\{a_{r-1}',\, b_{r-1},\, e_{1}'\}$.

By reversing the procedure discussed above, we see that
if we relabel $e_{r}$ with $e_{r+1}$, $b_{r-1}$ with
$a_{r}$, $e_{1}'$ with $b_{r-1}$, and $a_{r-1}'$ with
$b_{r}$, then perform a \dy\ operation on $T$,
then we obtain the matroid
\begin{displaymath}
\Delta_{r+1}^{\circ} \del a_{r-1}' \del a_{r}' \del e_{1}'
\del e_{r} \del e_{r}'.
\end{displaymath}
Thus $\Delta_{T}(\Delta_{r}^{\circ})$ is isomorphic to
$\Delta_{r+1}^{\circ}$ restricted to the set
$E(\Delta_{r+1}^{\circ}) -
\{a_{r-1}',\, a_{r}',\, e_{1}',\, e_{r},\, e_{r}'\}$,
and the isomorphism is determined naturally by the
relabeling.
Any spoke element other than the one in
$T$ is taken to another spoke element in this procedure,
so we are done.
\end{proof}

\begin{lem}
\label{lem5}
Suppose that $M$ is an \ifc\ non-cographic matroid
in \ex{\mkt} and that $M$ has at least one triad.
Let $M_{0}$ be the matroid obtained from $M$ by repeatedly
performing \yd\ operations until the resulting matroid has
no triads.
If $\si(M_{0})$ is isomorphic to a triangular \mob\ matroid
then $M$ is also a \mob\ matroid.
\end{lem}

\begin{proof}
Suppose that $M_{0}$ is obtained from $M$ by performing
\yd\ operations on the triads $T_{1},\ldots, T_{n}$ in turn.
Let $\mcal{T}_{0} = \{T_{1},\ldots, T_{n}\}$, and
let \mcal{T} be the set of triangles in $\si(M_{0})$ that
correspond to the triangles in $\mcal{T}_{0}$.
Thus $M \iso \Delta(\si(M_{0});\, \mcal{T})$ and
$\si(M_{0}) \iso \Delta_{r}$ for some $r \geq 3$.

We start by assuming that $r = 3$.
This means that $\si(M_{0}) \iso \Delta_{3} \iso F_{7}$.
Up to relabeling there is only one legitimate set of
triangles \mcal{T} in $\si(M_{0})$ such that
$|\mcal{T}| = 1$.
Moreover, if $|\mcal{T}| = 1$ then
$\Delta(\si(M_{0});\, \mcal{T}) \iso F_{7}^{*}$.
The Fano dual is also a \mob\ matroid, so in the
case that $n = 1$ we are done.

Since any pair of triangles in $F_{7}$ contains a cocircuit
of size four we see that there is no legitimate set of
triangles in $\si(M_{0})$ containing exactly two triangles.
Therefore we assume that $n \geq 3$.

It is easy to see that since \mcal{T} is a legitimate set
it contains a set of three triangles having no common point
of intersection.
Let $\mcal{T}_{3}$ be such a set of three triangles in
$\si(M_{0})$.
It follows from Lemma~\ref{lem12} that
$\Delta(\si(M_{0});\, \mcal{T})$ has a minor isomorphic
to $\Delta(F_{7};\, \mcal{T}_{3})$.
However this last matroid is isomorphic to the dual of
$\Delta_{4}$, and it has an \mkt\dash minor, so we have a
contradiction.\cross

Therefore we assume that $r \geq 4$.
For $1 \leq i \leq n$ let $M_{i}$ be the matroid obtained
from $M_{0}$ by performing \dy\ operations on the triangles
$T_{1},\ldots, T_{i}$.
Thus $M_{n}$ is equal to $M$.

\begin{subclm}
\label{clm14}
For $0 \leq i \leq n$ there is an isomorphism $\psi_{i}$
between $M_{i}$ and a restriction of $\Delta_{r+i}^{\circ}$.
Moreover, if $i < j \leq n$, then $T_{j}$ is a triangle of
$M_{i}$ containing an element $b$ such that $\psi_{i}(b)$ is
a spoke element of $\Delta_{r+i}^{\circ}$.
\end{subclm}

\begin{proof}
It follows from Claim~\ref{clm11} that the only allowable
triangles in $\Delta_{r}$ are triangles that contain
spoke elements.
Moreover, it is not difficult to use Claim~\ref{clm12}
to show that any spoke element is contained in at most
one triangle of \mcal{T}.
Therefore if a pair of triangles in \mcal{T} have a
non-empty intersection, they meet in a rim element
of $\si(M_{0}) \iso \Delta_{r}$.
It follows from these facts that $M_{0}$ is
obtained from $\Delta_{r}$ by replacing
certain rim elements with parallel pairs.
Thus the claim holds when $i = 0$.

Suppose that $i > 0$, and the claim holds for $M_{i - 1}$.
Clearly $T_{i}$ is a triangle in $M_{i-1}$.
Let $T = \psi_{i-1}(T_{i})$.
Then $T$ contains a spoke element of
$\Delta_{r+i-1}^{\circ}$ by the inductive hypothesis.
It follows easily from Proposition~\ref{prop5} that
$\psi_{i-1}(\Delta_{T_{i}}(M_{i-1}))$ is a restriction of
$\Delta_{T}(\Delta_{r+i-1}^{\circ})$.
But Lemma~\ref{lem10} tells us that
$\Delta_{T}(\Delta_{r+i-1}^{\circ})$ is isomorphic
to a restriction of $\Delta_{r+i}^{\circ}$ and this
isomorphism takes spoke elements of
$\Delta_{r+i-1}^{\circ}$ (other than the spoke
element contained in $T$) to spoke elements of
$\Delta_{r+i}^{\circ}$.
We let $\psi_{i}'$ be the restriction of this isomorphism
to the ground set of $\psi_{i-1}(\Delta_{T_{i}}(M_{i-1}))$.
Since $M_{i} = \Delta_{T_{i}}(M_{i-1})$, if we act upon
$\Delta_{T_{i}}(M_{i-1})$ first with $\psi_{i-1}$ and then
with $\psi_{i}'$, we obtain an isomorphism between
$M_{i}$ and a restriction of $\Delta_{r+i}^{\circ}$ that
takes spoke elements other than the spoke element in $T_{i}$
to spoke elements.
We let this isomorphism be $\psi_{i}$.
Now it is easy to see that if $i < j \leq n$ then $T_{j}$
contains an element $b$ such that $\psi_{i}(b)$ is a spoke
element of $\Delta_{r+i}^{\circ}$.
\end{proof}

Since $M = M_{n}$ it follows from Claim~\ref{clm14}
that $M$ is isomorphic to a restriction of
$\Delta_{r + n}^{\circ}$.
In fact $M$ is a restriction of $\Delta_{r+n}$,
since $M$ is \ifc\ and therefore simple.
However Lemma~\ref{lem4} tells us that any
\ifc\ non-cographic minor of $\Delta_{r+n}$ is in
fact a \mob\ matroid.
Thus $M$ is a \mob\ matroid, so
Lemma~\ref{lem5} holds.
\end{proof}

\chapter{An $R_{12}$-type matroid}
\label{chp5}

The matroid $R_{12}$ plays an important role in Seymour's
decomposition of regular matroids.
He shows that any regular matroid with an $R_{12}$\dash minor
cannot be \ifc.
In this chapter we introduce a matroid that plays a similar role
in our proof.

Consider the single-element coextension
$\Delta_{4}^{+}$\index{DeltaA@$\Delta_{4}^{+}$}
of $\Delta_{4}$ by the element $e_{5}$ represented
over \gf{2} by $[I_{5}|A]$ (the matrix $A$ is displayed
in Figure~\ref{fig9}, along with the fundamental graph
$G_{B}(\Delta_{4}^{+})$ where $B = \{e_{1},\ldots, e_{5}\}$).
We can check that $\Delta_{4}^{+}$ has no \mkt\dash minor.\cross\
The set $\{a_{1},\, a_{2},\, b_{1},\, e_{5}\}$ is a four-element
circuit-cocircuit of $\Delta_{4}^{+}$.

\begin{figure}[htb]

\centering

\includegraphics{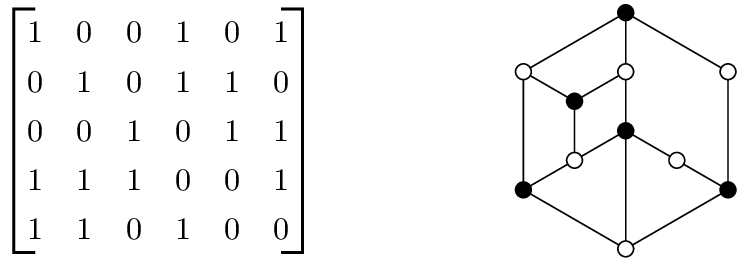}

\caption{Matrix and fundamental graph representations
of $\Delta_{4}^{+}$.}

\label{fig9}

\end{figure}

\begin{lem}
\label{lem1}
If $M \in \ex{\mkt}$ and $M$ has a $\Delta_{4}^{+}$\dash minor,
then $M$ is not \ifc.
\end{lem}

\begin{proof}
Suppose that the lemma fails, and that $M$ is an
\ifc\ member of \ex{\mkt} with a $\Delta_{4}^{+}$\dash minor.
By Proposition~\ref{prop12} there
is a $3$\dash connected single-element extension
or coextension $N'$ of $\Delta_{4}^{+}$ such that
$\{a_{1},\, a_{2},\, b_{1},\, e_{5}\}$ is not a
circuit-cocircuit of $N'$ and $N'$ is a minor of $M$.
But every $3$\dash connected binary single-element extension
or coextension of $\Delta_{4}^{+}$ in which
$\{a_{1},\, a_{2},\, b_{1},\, e_{5}\}$ is not a
circuit-cocircuit has an \mkt\dash minor.\cross\
This contradiction completes the proof.
\end{proof}

\begin{lem}
\label{lem2}
Suppose that $M$ and $N$ are $3$\dash connected binary matroids
such that $|E(N)| > 7$.
Assume that $M$ contains a four-element circuit-cocircuit
$C^{*}$ and an $N$\dash minor.
Then $M$ has a $3$\dash connected minor $M'$ such that
$C^{*}$ is a four-element circuit-cocircuit of $M'$,
furthermore $M'$ has an $N$\dash minor, but if
$e \in E(M') - C^{*}$ then neither $M' \del e$ nor
$M' / e$ has an $N$\dash minor.
\end{lem}

\begin{proof}
Assume that $M$ is a minimal counterexample.
This means that $M \ne N$, so let $e$ be an element of
$M$ such that either $M \del e$ or $M / e$ has an $N$\dash minor.
If the only such elements are contained in
$C^{*}$ then the result holds for $M$ so we are done.
Therefore we assume that $e \notin C^{*}$.

Let us assume that $M \del e$ has an $N$\dash minor.
Then $\co(M \del e)$ has an $N$\dash minor.
Suppose that $e \notin \cl_{M}^{*}(C^{*}) \cup \cl_{M}(C^{*})$.
Then it is easy to see that $C^{*}$ is a four-element
circuit-cocircuit of $\co(M \del e)$.
Thus if $\co(M \del e)$ is $3$\dash connected we
have a contradiction to the minimality of $M$.
Therefore we assume that $\co(M \del e)$ is not
$3$\dash connected, so Proposition~\ref{prop16}
implies that $\si(M / e)$ is $3$\dash connected.

Since $e \notin \cl_{M}(C^{*})$ it follows that
$C^{*}$ is a four-element circuit-cocircuit in
$\si(M / e)$.
If $M / e$ has an $N$\dash minor then
so does $\si(M / e)$, and in this case we again
have a contradiction to the minimality of $M$.
Therefore we must assume that $M / e$ does not have
an $N$\dash minor.

Since $\co(M \del e)$ is not $3$\dash connected there
is a $2$\dash separation $(X_{1},\, X_{2})$ of $M \del e$ such that
$|X_{1}|,\, |X_{2}| \geq 3$.
If both $r(X_{1}),\, r(X_{2}) \geq 3$ then $M / e$ has
an $N$\dash minor by Proposition~\ref{prop18}, a
contradiction.
Thus we will assume that $r(X_{1}) \leq 2$.
Since $M$ has no parallel pairs it follows that $X_{1}$ is
a triangle in $M$.

As $(X_{1},\, X_{2})$ is a $2$\dash separation of
$M \del e$ we conclude that $r(X_{2}) = r(M) - 1$.
Thus $X_{1}$ contains a cocircuit of $M \del e$, and
in fact it must contain a series pair.
Let $X_{1} = \{x,\, y,\, z\}$, and suppose that
$x$ and $y$ are in series in $M \del e$.
Since $N$ has no series pairs $M \del e / x$ has an
$N$\dash minor.
But $y$ and $z$ are in parallel in $M \del e / x$, so
$M \del e / x \del y$ has an $N$\dash minor.
However since $M$ has no series pairs
$\{e,\, x,\, y\}$ must be a triad of $M$.
Thus $x$ is a coloop in $M \del e \del y$, so
$M \del e / x \del y = M \del e \del x \del y$.
But $e$ is a coloop in $M \del x \del y$, so
$M \del e \del x \del y = M / e \del x \del y$.
Thus $M / e \del x \del y$, and hence $M / e$, has an
$N$\dash minor, a contradiction.

Now we must assume that
$e \in \cl_{M}^{*}(C^{*}) \cup \cl_{M}(C^{*})$.
Suppose that $e \in \cl_{M}(C^{*})$.
Then $e \notin \cl_{M}^{*}(C^{*})$, for otherwise
$(C^{*} \cup e,\, E - (C^{*} \cup e))$ is a
$2$\dash separation of $M$.
Thus $C^{*}$ is a four-element cocircuit in
$\co(M \del e)$.
Note that in $M / e$ there are two parallel pairs
in $C^{*}$, and deleting a single element from each of these
produces a series pair.
Thus $\si(M / e)$ is not $3$\dash connected, so
$\co(M \del e)$ is $3$\dash connected.
Since $\co(M \del e)$ has an $N$\dash minor
we have a contradiction to the minimality of $M$.

Therefore we suppose that $e \notin \cl_{M}(C^{*})$,
so that $e \in \cl_{M}^{*}(C^{*})$.
Now $C^{*}$ contains two series pairs in $M \del e$
and contracting an element from each produces a parallel
pair.
Thus $\co(M \del e)$ is not $3$\dash connected, and
hence $\si(M / e)$ is $3$\dash connected.
Moreover $(C^{*},\, E - (C^{*} \cup e))$ is a
$2$\dash separation of $M \del e$ and
$r(C^{*}) \geq 3$.
As $E - (C^{*} \cup e)$ contains at least four
elements of $E(N)$ it also follows that
$r(E - (C^{*} \cup e)) \geq 3$.
Since $M \del e$ has an $N$\dash minor
Proposition~\ref{prop18} tells us that $M / e$,
and hence $\si(M / e)$, has an $N$\dash minor.
But $C^{*}$ is a four-element circuit-cocircuit in
$\si(M / e)$, so we again have a contradiction to the
minimality of $M$.

We have shown that $M \del e$ cannot have an
$N$\dash minor.
Therefore $M / e$ has an $N$\dash minor.
But $M^{*}$ and $N^{*}$ also provide a
minimal counterexample to the problem, and
$C^{*}$ is a four-element circuit-cocircuit of
$M^{*}$.
Since $M^{*} \del e$ has an $N^{*}$\dash minor
we can use exactly the same arguments as before
to find a contradiction.
Thus the result holds.
\end{proof}

\begin{cor}
\label{cor8}
Suppose $M \in \ex{\mkt}$ is a $3$\dash connected matroid
that has a $\Delta_{4}$\dash minor and a four-element
circuit-cocircuit.
Then $M$ has a $\Delta_{4}^{+}$\dash minor.
\end{cor}

\begin{proof}
Lemma~\ref{lem2} implies that $M$ has a
$3$\dash connected minor $M'$ such that $M'$ contains a
four-element circuit-cocircuit $C^{*}$, and $M'$ has
a $\Delta_{4}$\dash minor, but if $e \in E(M') - C^{*}$
then neither $M' \del e$ nor $M' / e$ has a
$\Delta_{4}$\dash minor.
The rest of the proof is a straightforward case-check,
an outline of which is given in Proposition~\ref{prop25}.
\end{proof}

The following corollary of Lemma~\ref{lem1}
and Corollary~\ref{cor8} is the main result of this chapter.

\begin{cor}
\label{cor1}
Suppose that $M \in \ex{\mkt}$ is \ifc\ and has a
$\Delta_{4}$\dash minor.
If $M'$ is a $3$\dash connected minor of $M$ and $M'$ has a
$\Delta_{4}$\dash minor then $M'$ has no four-element
circuit-cocircuit.
\end{cor}

\chapter{A connectivity lemma}
\label{chp6}

Connectivity results such as Seymour's Splitter Theorem are
essential tools for inductive proofs in structural matroid
theory.
Theorem~\ref{thm3} implies that if $N$ is a
$3$\dash connected matroid such that $|E(N)| \geq 4$ and
$N$ is not a wheel or whirl, and $M$ has a proper
$N$\dash minor, then $M$ has a proper $3$\dash connected minor
$M_{0}$ such that $M_{0}$ has an $N$\dash minor and
$|E(M)| - |E(M_{0})| = 1$.

Because we are considering matroids that exhibit higher
connectivity than $3$\dash connectivity we need a new set
of inductive tools.

\begin{thm}
\label{thm4}
Suppose that $M$ and $N$ are simple \vfc\ binary matroids
such that $N$ is a proper minor of $M$ and $|E(N)| \geq 10$.
Suppose also that whenever $M'$ is a $3$\dash connected
minor of $M$ and $M'$ has a minor isomorphic to $N$ then
$M'$ has no four-element circuit-cocircuit.
Then $M$ has a proper \ifc\ minor $M_{0}$ such that $M_{0}$
has an $N$\dash minor and $|E(M)| - |E(M_{0})| \leq 4$.
\end{thm}

The inductive step of our proof would be infeasible if we had
to search through all extensions and coextensions on up to four
elements.
We need a refined version of Theorem~\ref{thm4} that
tells us more about the way in which $M_{0}$ is derived from
$M$.
The next result is a step in this direction.

\begin{thm}
\label{thm5}
Under the hypotheses of Theorem~{\rm\ref{thm4}} one of the
following cases holds.
\begin{enumerate}[(i)]
\item There exists an element $x \in E(M)$ such that
$M \del x$ is \ifc\ with an $N$\dash minor;
\item There exists an element $x \in E(M)$ such that
$\si(M / x)$ is \ifc, has an $N$\dash minor, and
$|E(M)| - |E(\si(M / x))| \leq 3$;
\item There exist elements $x,\, y \in E(M)$ such that
$M / x / y$ is \vfc, has an $N$\dash minor, and
$|E(M)| - |E(\si(M / x / y))| \leq 3$; or,
\item There exist elements $x,\, y,\, z \in E(M)$ such
that $M / x / y / z$ is \vfc, has an $N$\dash minor, and
$|E(M)| - |E(\si(M / x / y / z))| \leq 4$.
\end{enumerate}
\end{thm}

Clearly Theorem~\ref{thm5} implies Theorem~\ref{thm4}.

Theorem~\ref{thm5} is also too coarse a
tool for our inductive proof: we need yet another refinement.
Theorem~\ref{thm5} follows from Lemma~\ref{lem3}, which
is the main result of this chapter.

Before proving Lemma~\ref{lem3} we discuss some preliminary
ideas.
Suppose that $M$ is a \vtc\ matroid
and that $N$ is an \ifc\ minor of
$M$ with $|E(N)| \geq 7$.
Suppose that $(X_{1},\, X_{2})$ is a $3$\dash separation of $M$.
It follows from Proposition~\ref{prop7} that either
$|E(N) \cap X_{1}| \leq 3$ or $|E(N) \cap X_{2}| \leq 3$.
If $|E(N) \cap X_{i}| \leq 3$ then we say that
$X_{i}$ is a \emph{small $3$\dash separator}\index{separator!small $3$\dash}.
Since $|E(N)| \geq 7$ exactly one of $X_{1}$ and $X_{2}$ is a small
$3$\dash separator.
If $X$ is a small $3$\dash separator and $X$ is not
properly contained in any other small $3$\dash separator
then we shall say that $X$ is a
\emph{maximal small
$3$\dash separator}\index{separator!maximal small $3$\dash}.
A \emph{small vertical
$3$\dash separator}\index{separator!small vertical $3$\dash} is a small
$3$\dash separator that is also vertical, and a
\emph{maximal small vertical
$3$\dash separator}\index{separator!maximal small vertical $3$\dash} is a
maximal small $3$\dash separator that is also vertical.

\begin{prop}
\label{prop36}
Suppose that $M$ is a \vtc\ matroid
on the ground set $E$ and that $N$ is an \ifc\ minor of $M$ with
$|E(N)| \geq 10$.
If $X_{1}$ and $X_{2}$ are maximal small
$3$\dash separators of $M$ such that
$X_{1} \ne X_{2}$ and $r_{M}(E - X_{1}),\, r_{M}(E - X_{2}) \geq 2$
then
\begin{displaymath}
r_{M}(X_{1} \cap X_{2}) \leq 1.
\end{displaymath}
Suppose that $X_{1} \cap X_{2} = F$ where $r_{M}(F) = 1$.
If, in addition, $r_{M}(X_{1}),\, r_{M}(X_{2}) \geq 2$
and both $X_{1}$ and $X_{2}$ contain at least three rank-one flats
then either
\begin{enumerate}[(i)]
\item $F \subseteq \cl_{M}(E - X_{1}) \cap \cl_{M}(E - X_{2})$; or,
\item $F \cap \cl_{M}(X_{1} - F) = \varnothing$ and
$F \cap \cl_{M}(X_{2} - F) = \varnothing$.
\end{enumerate}
\end{prop}

\begin{proof}
Suppose that $r_{M}(X_{1} \cap X_{2}) \geq 2$.
Since $\lambda_{M}(X_{i}) \leq 2$ for all $i \in \{1,\, 2\}$
it follows from the submodularity of the connectivity
function that
\begin{displaymath}
\lambda_{M}(X_{1} \cup X_{2}) + \lambda_{M}(X_{1} \cap X_{2})
\leq 4.
\end{displaymath}
Since $r_{M}(X_{1} \cap X_{2}) \geq 2$ and
$r_{M}(E - (X_{1} \cap X_{2})) \geq 2$ it cannot be the case  that
$\lambda_{M}(X_{1} \cap X_{2}) \leq 1$, for then $M$ would
have a vertical $2$\dash separation.
Thus $\lambda_{M}(X_{1} \cup X_{2}) \leq 2$.
Suppose that $X_{1} \cup X_{2}$ is not a $3$\dash separator.
This implies that $|E(M) - (X_{1} \cup X_{2})| \leq 2$.
Since $E(M) - X_{1}$ contains at least seven elements
of $E(N)$ it follows that $X_{2} - X_{1}$ contains
at least five elements of $E(N)$, which contradicts the
fact that $X_{2}$ is a small $3$\dash separator.
Therefore $X_{1} \cup X_{2}$ is a $3$\dash separator.

As $X_{1}$ and $X_{2}$ are distinct maximal $3$\dash separators
they are each properly contained in $X_{1} \cup X_{2}$, so
$X_{1} \cup X_{2}$ is not a small $3$\dash separator.
Therefore $E - (X_{1} \cup X_{2})$ is a small $3$\dash separator,
so
\begin{displaymath}
|E(N) \cap (E - (X_{1} \cup X_{2}))| \leq 3.
\end{displaymath}
However $|E(N) \cap X_{i}| \leq 3$ for $i = 1,\, 2$
so $|E(N) \cap (X_{1} \cup X_{2})| \leq 6$.
Since $|E(N)| \geq 10$ this leads to a contradiction.
We have shown that $r_{M}(X_{1} \cap X_{2}) \leq 1$.

Now we suppose that $r_{M}(X_{1}),\, r_{M}(X_{2}) \geq 2$, and
both $X_{1}$ and $X_{2}$ contain at least three rank-one flats.
Let $F = X_{1} \cap X_{2}$, where $r_{M}(F) = 1$.
Assume that $F \subseteq \cl_{M}(E - X_{1})$.
If $F$ were not contained in $\cl_{M}(X_{1} - F)$ then
$(X_{1} - F,\, (E - X_{1}) \cup F)$ would be a vertical
$2$\dash separation of $M$.
Hence $F \subseteq \cl_{M}(X_{1} - F)$, and this implies
that $F \subseteq \cl_{M}(E - X_{2})$.

Next assume that $F \nsubseteq \cl_{M}(E - X_{1})$.
This implies that $F \nsubseteq \cl_{M}(X_{2} - F)$, so
in fact $F \cap \cl_{M}(X_{2} - F) = \varnothing$.
If $F$ were contained in $\cl_{M}(E - X_{2})$ then
$(X_{2} - F,\, (E - X_{2}) \cup F)$ would be a vertical
$2$\dash separation of $M$, so
$F \nsubseteq \cl_{M}(E - X_{2})$, and this implies
that $F \nsubseteq \cl_{M}(X_{1} - F)$.
Thus $F \cap \cl_{M}(X_{1} - F) = \varnothing$ and this
completes the proof.
\end{proof}

If $M$ is a matroid and $X$ is a subset of $E(M)$ then let
$G_{M}(X)$ denote the set $X \cap \cl_{M}(E(M) - X)$.
We use $G_{M}^{*}(X)$ to denote $X \cap \cl_{M}^{*}(E(M) - X)$.
We will make use of the fact that if $X$ is a subset of
$E(M)$ and $e \in X$ then $e \in \cl_{M}^{*}(E(M) - X)$
if and only if $e \notin \cl_{M}(X - e)$.

If $X_{1}$ and $X_{2}$ are small $3$\dash separators in
a $3$\dash connected matroid then they automatically
satisfy the hypotheses of Proposition~\ref{prop36},
so the next result follows as a corollary.

\begin{cor}
\label{cor4}
Suppose that $M$ is a $3$\dash connected matroid and that $N$ is an
\ifc\ minor of $M$ such that $|E(N)| \geq 10$.
If $X_{1}$ and $X_{2}$ are distinct maximal small $3$\dash separators of
$M$ then $|X_{1} \cap X_{2}| \leq 1$, and if $e \in X_{1} \cap X_{2}$
then either $e \in G_{M}(X_{1}) \cap G_{M}(X_{2})$ or
$e \in G_{M}^{*}(X_{1}) \cap G_{M}^{*}(X_{2})$.
\end{cor}

Suppose that $M$ is a matroid and $X$ is a subset of $E(M)$.
Let $\inter_{M}(X)$ denote the set
$X - G_{M}(X) = X - \cl_{M}(E(M) - X)$.

\begin{prop}
\label{prop39}
Suppose that $M$ is a \vtc\ matroid
and that $X$ is a vertical $3$\dash separator of $M$.
Then
\begin{enumerate}[(i)]
\item $r_{M}(\inter_{M}(X)) = r_{M}(X)$;
\item $X \subseteq \cl_{M}(\inter_{M}(X))$; and,
\item $\inter_{M}(X)$ is a vertical
$3$\dash separator of $M$.
\end{enumerate}
\end{prop}

\begin{proof}
Let $Y = E(M) - X$ and let $Y' = \cl_{M}(Y)$.
Also, let $X' = \inter_{M}(X) = E(M) - Y'$.
It cannot be the case that $X'$ is empty, for that would
imply that $r_{M}(Y) = r(M)$ and that $r_{M}(X) = 2$.
If $r_{M}(X') < r_{M}(X)$ then $(X',\, Y')$ is a
\vks\ for some $k < 3$, so
$r_{M}(\inter_{M}(X)) = r_{M}(X)$.
Statements~(ii) and~(iii) follow easily.
\end{proof}

Suppose that $A = (e_{1},\ldots, e_{4})$ is a cofan
of a binary matroid $M$.
(Note that in this case $A$ is also a fan, with the
elements taken in reverse order.) 
We define $e_{1}$ to be a
\emph{good element}\index{good element (of a small separator)}
of $A$.
Similarly, if $A = (e_{1},\ldots, e_{5})$ is a fan then
$e_{2}$ and $e_{4}$ are good elements of $A$, and if
$A = (e_{1},\ldots, e_{5})$ is a cofan then $e_{1}$ and
$e_{5}$ are good elements.

\begin{prop}
\label{prop41}
Suppose that $N$ is a proper minor of the binary matroid $M$
such that $N$ has no triads, and is simple and cosimple.
Suppose that $A$ is a fan or cofan with length four or five in $M$.
If $x$ is a good element of $A$ then $M / x$ has an $N$\dash minor.
\end{prop}

\begin{proof}
Suppose that $A = (e_{1},\ldots, e_{4})$ is a cofan.
Then $x = e_{1}$.
Suppose that $M / e_{1}$ does not have an $N$\dash minor.
Since $\{e_{1},\, e_{2},\, e_{3}\}$ is a triad of
$M$ and $N$ is cosimple with no triads it follows that
either $M / e_{2}$ or $M / e_{3}$ has an $N$\dash minor.
But it is easy to check that both $M / e_{2}$ and $M / e_{3}$
can be obtained from $M / e_{1}$ by deleting an
element and adding a point in parallel to an existing
element.
Since $N$ is simple it follows that $M / e_{1}$ also
has an $N$\dash minor.

The argument is similar when $|A| = 5$.
\end{proof}

We are now ready to tackle Lemma~\ref{lem3}.

\begin{lem}
\label{lem3}
Suppose that $M$ and $N$ are simple \vfc\ binary matroids
such that $N$ is a proper minor of $M$ and $|E(N)| \geq 10$.
Suppose also that whenever $M'$ is a $3$\dash connected
minor of $M$ and $M'$ has a minor isomorphic to $N$ then
$M'$ has no four-element circuit-cocircuit.
Then one of the following cases holds:
\begin{enumerate}[(i)]
\item\label{state1}
There is an element $x \in E(M)$ such that
$M \del x$ is \ifc\ with an $N$\dash minor;
\item\label{state7}
There is an element $x \in E(M)$ such that
$M / x$ is simple and \vfc\ with an $N$\dash minor;
\item\label{state3}
There is an element $x \in E(M)$ such that
$\si(M / x)$ is \ifc\ with an $N$\dash minor.
Furthermore $\si(M / x)$ contains at least one triangle and
at least one triad.
Moreover $M / x$ has no loops and exactly one parallel pair;
\item\label{state2}
There is an element $x \in E(M)$ such that
$M / x$ is \vfc\, and has an $N$\dash minor.
Furthermore, there is a triangle $T$ of $M / x$ such that
$x$ is in a four-element cocircuit $C^{*}$ of $M$ with the
property that $|C^{*} \cap T| = 2$.
Moreover $x$ is in at most two triangles in $M$, and if
$x$ is in two triangles of $M$, then exactly one of these
triangles contains an element of $T$;
\item\label{state6}
There is an element $x \in E(M)$ such that
$M / x$ is \vfc\, and has an $N$\dash minor.
Furthermore, there exist triangles $T_{1}$ and $T_{2}$ in
$M / x$ such that $|T_{1} \cap T_{2}| = 1$, there is
a four-element cocircuit $C^{*}$ of $M$ such that
$x \in C^{*}$, $|C^{*} \cap T_{i}| = 2$ for $i = 1,\, 2$,
and $T_{1} \cap T_{2} \subseteq C^{*}$.
Moreover, there are no loops and at most two parallel
pairs in $M / x$, and $x$ is not contained in a triangle
of $M$ with the element in $T_{1} \cap T_{2}$.
\item\label{state9}
There is a triangle $\{x,\, y,\, z\}$ of
$M$ such that $M / x / y$ is \vfc\ with an
$N$\dash minor.
Furthermore, there exist triangles $T_{1}$ and $T_{2}$ of
$M / x / y$ such that $r_{M / x / y}(T_{1} \cup T_{2}) = 4$,
and $x$ is in a triad of $M \del z$ with two elements of
$T_{1}$, and $y$ is in a triad of $M \del z$ with two elements
of $T_{2}$.
Moreover, $M / x / y$ has exactly one loop and no parallel
elements;
\item\label{state4}
There is a triangle $\{x,\, y,\, z\}$ of
$M$ such that $M / x / y$ is \vfc\ with an
$N$\dash minor.
Furthermore, there exist triangles $T_{1}$ and $T_{2}$ of
$M / x / y$ such that $|T_{1} \cap T_{2}| = 1$, and
both $(T_{1} - T_{2}) \cup x$ and $(T_{2} - T_{1}) \cup y$
are triads of $M \del z$.
Moreover, $M / x / y$ has exactly one loop and no parallel elements;
\item\label{state8}
There is a triangle $\{x,\, y,\, z\}$ of
$M$ such that $M / x / y$ is \vfc\ with an
$N$\dash minor.
Furthermore, there exist triangles $T_{1}$ and $T_{2}$ of
$M / x / y$ such that $|T_{1} \cap T_{2}| = 1$, and
in $M \del z$ the element $x$ is in a triad with the single
element from $T_{1} \cap T_{2}$ and a single element from
$T_{1} - T_{2}$, and $(T_{2} - T_{1}) \cup y$ is a triad.
Moreover, $M / x / y$ has exactly one loop and no parallel
elements; or,
\item\label{state5}
There is an element $x \in E(M)$ such that $M \del x$
contains three cofans, $(x_{1},\ldots, x_{5})$,
$(y_{1},\ldots, y_{5})$, and $(z_{1},\ldots, z_{5})$, where
$x_{5} = y_{1}$, $y_{5} = z_{1}$, and $z_{5} = x_{1}$, and
$M / x_{1} / y_{1} / z_{1}$ is \vfc\ with an
$N$\dash minor.
Furthermore, $\{x,\, x_{1},\, y_{1},\ z_{1}\}$ is a circuit
of $M$, and $M / x_{1} / y_{1} / z_{1}$ has exactly one loop and
no parallel elements.
\end{enumerate}
\end{lem}

\begin{proof}
Suppose that the lemma does not hold for the pair of
matroids $M$ and $N$.
Since both $M$ and $N$ have ground sets of size at least
ten, and both $M$ and $N$ are simple we deduce
that $r(M),\, r(N) \geq 4$.
It follows from Proposition~\ref{prop23} that neither $M$
nor $N$ has any triads.

\begin{sub}
\label{sub1}
There exists an element $e \in E(M)$ such that $M \del e$
is $(4,\, 5)$\dash connected and has an $N$\dash minor.
\end{sub}

\begin{proof}
If the claim is false, then by Theorem~\ref{thm7}
there is an element $e \in E(M)$ such that $M / e$
is $(4,\, 5)$\dash connected with an $N$\dash minor.
Since $M$ has no triads $M / e$ has no triads.
Every fan or cofan with length four or five contains a triad
and as $M / e$ has no four-element circuit-cocircuit
it follows from Proposition~\ref{prop33} that if $(X,\, Y)$
is a $3$\dash separation of $M / e$ then either
$r_{M / e}(X) \leq 2$ or $r_{M / e}(Y) \leq 2$.
Thus $M / e$ is \vfc, and since $M / e$ is
$(4,\, 5)$\dash connected it is simple, so
statement~\eqref{state7} of the lemma holds.
This contradiction implies that the sublemma is true.
\end{proof}

Henceforth we will suppose that $e \in E(M)$ has been
chosen so that $M \del e$ is $(4,\, 5)$\dash connected
and has an $N$\dash minor.
Let us fix a particular $N$\dash minor of $M \del e$, so
that we can define small $3$\dash separators of $M \del e$,
just as in the introduction to this chapter.

Since statement~\eqref{state1} does not hold we deduce that $M \del e$
has a $3$\dash separation $(X,\, Y)$ such that
$|X|,\, |Y| \geq 4$.
Since $M \del e$ is simple $r(X),\, r(Y) \geq 3$, so
$(X,\, Y)$ is a \vts.
Therefore $M \del e$ has at least one small vertical
$3$\dash separator.

Since $M$ has no vertical $3$\dash separations the
next result follows from Proposition~\ref{prop34}.

\begin{sub}
\label{sub9}
If $(X,\, Y)$ is a \vts\ of $M \del e$
then $e \notin \cl_{M}(X)$ and $e \notin \cl_{M}(Y)$.
\end{sub}

Suppose that $A$ is a small vertical $3$\dash separator
of $M \del e$.
Then $E(M \del e) - A$ contains at least seven elements of $E(N)$,
and as $M \del e$ is $(4,\, 5)$\dash connected it
follows that $|A| \leq 5$.
Since $M \del e$ has no four-element circuit-cocircuit we
conclude from Proposition~\ref{prop33} that $A$ is either
a triad (which is to say a cofan of length
three), or a fan or cofan with length four or five.

Figure~\ref{fig18} shows the four possible small vertical
$3$\dash separators of $M \del e$.
They are, in order, a triad, a cofan of length four, a fan of
length five, and a cofan of length five.
In each case a hollow square represents a point in the
underlying projective space that may not be in
$E(M \del e)$.
In the non-triad cases, a hollow circle indicates a good element.

\begin{figure}[htb]

\centering

\includegraphics{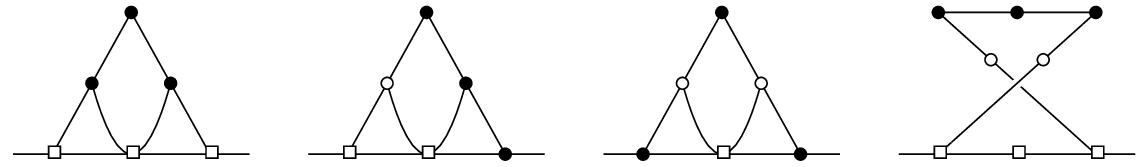}

\caption{Small vertical $3$\protect\dash separators of $M \del e$.}

\label{fig18}

\end{figure}

\begin{sub}
\label{sub3}
Suppose that $A$ is a small vertical $3$\dash separator of
$M \del e$ and that $x$ is a good element of $A$.
Then $M \del e / x$ is \vtc. 
\end{sub}

\begin{proof}
Suppose that the claim is false.
Let $(X,\, Y)$ be a \vks\ of
$M \del e / x$ where $k < 3$.
Since $M \del e$ is (vertically) $3$\dash connected
it is easy to see that both $(X \cup x,\, Y)$ and
$(X,\, Y \cup x)$ are vertical $3$\dash separations of
$M \del e$ and that
$x \in \cl_{M \del e}(X) \cap \cl_{M \del e}(Y)$.
Suppose that $X \cup x$ is not a small
$3$\dash separator of $M \del e$.
Then $|Y| \leq 5$, and hence $Y \cup x$ cannot
contain seven elements of $E(N)$, so $Y \cup x$ is a small
$3$\dash separator.
Therefore, by relabeling if necessary, we will assume
that $X \cup x$ is a small $3$\dash separator of
$M \del e$.
Let \ov{X} be a maximal small $3$\dash separator of
$M \del e$ that contains $X \cup x$ and let
\ov{A} be a maximal small $3$\dash separator that
contains $A$.

First let us suppose that \ov{X} and \ov{A} are not
equal.
Then, since $M \del e$ is $3$\dash connected,
Corollary~\ref{cor4} tells us that
$\ov{X} \cap \ov{A} = \{x\}$.
But then $x \in \cl_{M \del e}(X)$ implies that
$x \in G_{M \del e}(\ov{A})$, and therefore
$x \in G_{M \del e}(A)$.
However, it is easy to see that $x \in \inter_{M \del e}(A)$,
as $x$ is a good element, so we have a contradiction.
Hence we assume that $\ov{X} = \ov{A}$.
This means that $X$ is a vertical $3$\dash separator of
$M \del e$ contained in \ov{A} such that
$x \in \cl_{M \del e}(X)$, and $x \in \cl_{M \del e}(E(M \del e) - X)$,
while $x \notin G_{M \del e}(A)$.
An easy case-check confirms that $A$ must be a
cofan $(e_{1},\ldots, e_{5})$, and that $x$ must be either
$e_{2}$ or $e_{4}$, contradicting the fact that $x$ is a good
element of $A$.
\end{proof}

If $A$ is a cofan or fan of $M \del e$ with length four or five
then $A$ contains a good element $x$ by definition, and
clearly $x \in \inter_{M \del e}(A)$.
The next result summarizes these observations and the
consequences of Proposition~\ref{prop41} and~\ref{sub3}.

\begin{sub}
\label{sub2}
Let $A$ be a small vertical $3$\dash separator of $M \del e$
such that $|A| \geq 4$.
Then $A$ contains a good element $x$.
Moreover, $x \in \inter_{M \del e}(A)$, and $M \del e / x$ is
\vtc\ with an $N$\dash minor.
\end{sub}

If $A$ is a small vertical $3$\dash separator of $M \del e$
such that $|A| \geq 4$ and $x \in A$ is a good element then,
since $M \del e / x$ has an $N$\dash minor, we can
define small $3$\dash separators of $M \del e / x$ in the
same way that we defined them for $M \del e$.

\begin{sub}
\label{sub5}
Suppose that $A$ is a maximal small vertical $3$\dash separator
of $M \del e$ and $|A| \geq 4$.
Let $x \in A$ be a good element and let $X$ be a small vertical
$3$\dash separator of $M \del e / x$.
Then either
\begin{enumerate}[(i)]
\item $X \subseteq A - x$ and $X$ is a small
vertical $3$\dash separator of $M \del e$; or,
\item there exists a small vertical $3$\dash separator
$X_{0}$ of $M \del e$ such that
$\inter_{M \del e / x}(X) = \inter_{M \del e}(X_{0})$.
\end{enumerate}
\end{sub}

\begin{proof}
Let $X' = \inter_{M \del e /x}(X)$. 
Proposition~\ref{prop39} tells us that $X'$ is a
vertical $3$\dash separator of $M \del e / x$, and clearly it
is a small vertical $3$\dash separator of $M \del e / x$.
Let us suppose that $x \notin \cl_{M \del e}(X')$.
Then $r_{M \del e / x}(X') = r_{M \del e}(X')$ and $X'$ is a
$3$\dash separator in $M \del e$.
If $X'$ is not a small $3$\dash separator of $M \del e$, then
the complement of $X'$ in $M \del e$ is a small
$3$\dash separator, so this complement has size at most five.
But $X'$ is a small $3$\dash separator of $M \del e / x$,
so its complement in $M \del e / x$ contains at least
seven elements of $E(N)$.
Therefore $X'$ is a small vertical $3$\dash separator of
$M \del e$, and it is easy to see that $\inter_{M \del e}(X') = X'$,
so we can take $X_{0}$ to be $X'$.
Therefore we will assume that $x \in \cl_{M \del e}(X')$.
This in turn implies that $X' \cap (A - x) \ne \varnothing$
for otherwise $x \in G_{M \del e}(A)$, and this
contradicts~\ref{sub2}.

It is a trivial exercise to verify that $A - x$ is a
$3$\dash separator of $M \del e / x$.
Since $|A| \leq 5$ it is clear that $A - x$ is a
small $3$\dash separator of $M \del e / x$.
Suppose that $A - x$ is properly contained in $A'$, a
small $3$\dash separator of $M \del e / x$.
Then $x \notin \cl_{M \del e}(E(M \del e) - A')$, so
$A' \cup x$ is a $3$\dash separator of $M \del e$.
Since $A' \cup x$ properly contains $A$ it cannot be a
small $3$\dash separator of $M \del e$.
This easily leads to a contradiction, so $A - x$ is a maximal
small $3$\dash separator of $M \del e / x$.

Let \ov{X} be a maximal small $3$\dash separator of
$M \del e / x$ that contains $X$.
Suppose $\ov{X} = A - x$.
Then $X \subseteq A - x$.
Since $r_{M \del e / x}(X) \geq 3$ it follows that
$A$ has rank at least four in $M \del e$, so
$A$ is a cofan $(e_{1},\ldots, e_{5})$.
By symmetry we can assume that $x = e_{1}$ and
that $X \subseteq \{e_{2},\ldots, e_{5}\}$.
Now it is easy to check that $X$ is a small vertical
$3$\dash separator of $M \del e$, so the first
statement holds and we are done.

Thus we assume that \ov{X} and $A - x$ are
distinct maximal small $3$\dash separators of $M \del e / x$.
Since $X' \cap (A - x) \ne \varnothing$ it follows
that \ov{X} and $A - x$ are not disjoint.

Since the complements of \ov{X} and $A - x$ in $M \del e / x$
each contain at least seven elements of $E(N)$, it is easy to
see that they both have rank at least two in $M \del e / x$.
Proposition~\ref{prop36} now tells us that
\ov{X} and $A - x$ meet in $F$, a rank-one flat of
$M \del e / x$.
Note that, as $X' \cap (A - x) \ne \varnothing$, it follows
that $X'$ contains at least one element of $F$, and in fact
$X'$ contains $F$, for $E(M \del e / x) - X'$ is a flat
of $M \del e / x$ by definition.

\begin{subclm}
\label{clm8}
$F \cap G_{M \del e / x}(A - x) = \varnothing$.
\end{subclm}

\begin{proof}
By inspection $A - x$ has rank at least two, and contains at
least three rank-one flats in $M \del e / x$.
Clearly the same statement applies to \ov{X}, as it is
a vertical $3$\dash separator of $M \del e / x$.
Since $F$ is contained in $X'$ and the complement
of $X'$ is a flat in $M \del e / x$ we deduce that
$F$ is not contained in
$\cl_{M \del e / x}(E(M \del e / x) - \ov{X})$, as
$X' \subseteq \ov{X}$.
Therefore statement~(i) in Proposition~\ref{prop36} cannot hold.
Hence statement~(ii) holds, so
$F \cap \cl_{M \del e / x}(A - (F \cup x)) = \varnothing$.
Now, if $F$ had a non-empty intersection with
$\cl_{M \del e / x}(E(M \del e / x) - (A - x))$ then
\begin{displaymath}
(A - (F \cup x),\, (E(M \del e / x)  - (A - x)) \cup F)
\end{displaymath}
would be a vertical $2$\dash separation of $M \del e / x$,
contradicting~\ref{sub3}.
Therefore the claim holds.
\end{proof}

We have assumed that $x \in \cl_{M \del e}(X')$.
Let $C$ be a circuit of $M \del e$ that is contained in
$X' \cup x$ and that contains $x$.
Since $x \notin G_{M \del e}(A)$ it follows that $C$
contains an element $f$ of $X' \cap (A - x) = F$.
Since $x \in C$ and $F \cup x$ has rank two in $M \del e$
it follows that $f$ is the only element of $F$ in $C$.
Now $C - x$ is a circuit of $M \del e / x$ which contains
$f$, and $(C - x) - f$ is contained in
$E(M \del e / x) - (A - x)$, so $f \in G_{M \del e / x}(A - x)$,
contradicting Claim~\ref{clm8}.
This completes the proof of~\ref{sub5}.
\end{proof}

In the case that the second statement of~\ref{sub5} holds,
$\inter_{M \del e}(X_{0})$ is a vertical $3$\dash separator
of $M \del e$ by Proposition~\ref{prop39}~(iii), and is clearly
a small vertical $3$\dash separator.
Therefore the next fact follows as a corollary of~\ref{sub5}.

\begin{sub}
\label{sub10}
Suppose that $A$ is a maximal small vertical $3$\dash separator
of $M \del e$ and $|A| \geq 4$.
Let $x \in A$ be a good element.
If $X$ is a small vertical $3$\dash separator of $M \del e / x$
then $X$ contains a small vertical $3$\dash separator
of $M \del e$.
\end{sub}

We have defined small $3$\dash separators of $M \del e$, and
in the case that $A$ is a small vertical
$3$\dash separator $M \del e$ with $|A| \geq 4$ and $x \in A$
is a good element we have defined small $3$\dash separators
of $M \del e / x$.
The fact that $M \del e / x$ is \vtc\
means that $M / x$ is also \vtc.
We next define small $3$\dash separators of $M / x$.
We do so in such a way that the definition is compatible
with the definition for $M \del e / x$.
Suppose that $(X_{1},\, X_{2})$ is a partition of
$E(M) - \{e,\, x\}$ such that either
$(X_{1} \cup e,\, X_{2})$ or $(X_{1},\, X_{2} \cup e)$ is
a \vts\ of $M / x$.
Since $M \del e / x$ is \vtc\
by~\ref{sub3} it follows that $(X_{1},\, X_{2})$ is a
\vts\ of $M \del e / x$.
Assume that $X_{i}$ is a small $3$\dash separator of
$M \del e / x$, where $\{i,\, j\} = \{1,\, 2\}$.
Then either $(X_{i} \cup e,\, X_{j})$ or
$(X_{i},\, X_{j} \cup e)$ is a \vts\
of $M / x$.
In the first case we say that $X_{i} \cup e$ is a small
$3$\dash separator of $M / x$ and in the second we say that
$X_{i}$ is.
Thus if $(X,\, Y)$ is a \vts\ of
$M / x$ either $X$ or $Y$ is a small $3$\dash separator.
Note that we have not defined small $3$\dash separators
of $M / x$ in full generality: we have only defined small
vertical $3$\dash separators.

\begin{sub}
\label{sub7}
Suppose that $A$ is a maximal small vertical $3$\dash separator
of $M \del e$ and $|A| \geq 4$.
Let $x \in A$ be a good element.
If $X$ is a small vertical $3$\dash separator of $M / x$ then
$e \in X$ and $e \in \cl_{M}((X \cup x) - e)$.
\end{sub}

\begin{proof}
First suppose that $e \notin X$.
Then $e \in \cl_{M / x}(Y - e)$,
for otherwise $(X,\, Y - e)$ is a vertical
$2$\dash separation of $M \del e / x$.
By definition $X$ is a small $3$\dash separator of
$M \del e / x$, so it follows from~\ref{sub10} that $X$
contains a small vertical $3$\dash separator $X'$ of $M \del e$.
Let $Y' = E(M \del e) - X'$, so that $(X',\, Y')$ is
a \vts\ of $M \del e$.
Now $x \notin X'$ and $Y - e \subseteq Y'$, so
$(Y - e) \cup x \subseteq Y'$.
The fact that $e \in \cl_{M / x}(Y - e)$ implies
that $e \in \cl_{M}(Y')$, which contradicts~\ref{sub9}.

Therefore $e \in X$.
Then $e \in \cl_{M / x}(X - e)$, for
otherwise $(X - e,\, Y)$ is a vertical $2$\dash separation
of $M \del e / x$.
But $e \in \cl_{M / x}(X - e)$ implies that
$e \in \cl_{M}((X \cup x) - e)$, so we are done.
\end{proof}

We have assumed that $M \del e$ has at least one small
vertical $3$\dash separator.
We next suppose that, in particular, $M \del e$ contains a 
cofan of length five.

\begin{sub}
\label{sub8}
Suppose that $A = (e_{1},\ldots, e_{5})$ is a cofan of
$M \del e$.
Let $X$ be a maximal small vertical $3$\dash separator of
$M / e_{1}$.
Then $e \in X$, and $X - e = \{f_{1},\ldots, f_{5}\}$, where
$(f_{1},\ldots, f_{5})$ is a cofan of $M \del e$ such that
\begin{enumerate}[(i)]
\item $f_{1} = e_{5}$;
\item $(X - e) \cap A = \{f_{1}\}$; and,
\item $\{e,\, e_{1},\, f_{1},\, f_{5}\}$ is
a circuit of $M$.
\end{enumerate}
\end{sub}

\begin{proof}
The fact that $e \in X$ follows from~\ref{sub7}, as does
the fact that $e \in \cl_{M / e_{1}}(X - e)$.
Now $X - e$ is a vertical small $3$\dash separator of
$M \del e / e_{1}$ by definition, and since
$e \in \cl_{M / e_{1}}(X - e)$ the next claim is easy to
check.

\begin{fct}
\label{fct1}
$X - e$ is a maximal small $3$\dash separator of
$M \del e / e_{1}$.
\end{fct}

Clearly $A - e_{1}$ is a fan of length four in
$M \del e / e_{1}$, and hence a $3$\dash separator.
Obviously $A - e_{1}$ is a small $3$\dash separator,
and in fact the next result is easy to prove.

\begin{fct}
\label{fct2}
$A - e_{1}$ is a maximal small $3$\dash separator of
$M \del e / e_{1}$.
\end{fct}

We know by~\ref{sub3} that $M \del e / e_{1}$ is
\vtc.
Suppose that it is not $3$\dash connected.
It follows $M \del e / e_{1}$ is not simple, and that
therefore $e_{1}$ is contained in a triangle $T$ of $M \del e$.
Since $M \del e$ is binary $T$ meets the triad
$\{e_{1},\, e_{2},\, e_{3}\}$ in exactly two element.
Let $f$ be the element in
$T - \{e_{1},\, e_{2},\, e_{3}\}$, so that
$f \in \cl_{M \del e}(A)$.
By inspection $f \notin A$, so $A \cup f$ is a
$3$\dash separator of $M \del e$ that properly contains $A$.
This easily leads to a contradiction, so the next claim
follows.

\begin{fct}
\label{fct3}
$M \del e / e_{1}$ is $3$\dash connected.
\end{fct}

It cannot be the case that $A - e_{1} = X - e$,
for, as $e \in \cl_{M}((X \cup e_{1}) - e)$, this
would imply that $e \in \cl_{M}(A)$, in contradiction
to~\ref{sub9}.
Thus $A - e_{1}$ and $X - e$ are distinct maximal small
$3$\dash separators of $M \del e / e_{1}$.
Fact~\ref{fct3} and Corollary~\ref{cor4} imply that
$A - e_{1}$ and $X - e$ meet in at most one element.
Note that $T = \{e_{3},\, e_{4},\, e_{5}\}$ is a triad of
$M \del e$.
If $X - e$ is disjoint from $T$ then~\ref{sub7}
implies that $e \in \cl_{M}(E(M) - (T \cup e))$ and this implies
that $T$ is a triad of $M$, a contradiction.
Thus $A - e_{1}$ and $X - e$ meet in exactly one
element of $\{e_{3},\, e_{4},\, e_{5}\}$.
Since none of these three elements is in
$G_{M \del e / e_{1}}(A - e_{1})$ we see that
Corollary~\ref{cor4} also implies that the
single element in $(A - e_{1}) \cap (X - e)$ is in
$G_{M \del e / e_{1}}^{*}(A - e_{1})$.
This set contains only one element: $e_{5}$.
Thus we have proved that $(X - e) \cap (A - e_{1})= \{e_{5}\}$,
and since $e_{1} \notin X - e$ we have established
the following fact.

\begin{fct}
\label{fct4}
$(X - e) \cap A = \{e_{5}\}$.
\end{fct}

We next assume that $e_{1} \in \cl_{M \del e}(X - e)$.
Let $C$ be a circuit of $M \del e$ contained in
$(X - e) \cup e_{1}$ such that $e_{1} \in C$.
Fact~\ref{fct4} implies that $C$ meets the
triad $\{e_{1},\, e_{2},\, e_{3}\}$ in exactly one
element, $e_{1}$, which is impossible.
Thus $e_{1} \notin \cl_{M \del e}(X - e)$.
This implies that $X - e$ is a $3$\dash separator of
$M \del e$, and in fact the next statement is easy
to confirm.

\begin{fct}
\label{fct5}
$X - e$ is a small vertical $3$\dash separator of $M \del e$.
\end{fct}

Since $e_{5} \notin G_{M \del e}(A)$ we deduce that
$e_{5} \notin \cl_{M \del e}(X - \{e,\, e_{5}\})$.
By examining the possible small vertical $3$\dash separators
of $M \del e$ we see that one of the following
cases holds:
\begin{enumerate}[(i)]
\item $X - e$ is a triad of $M \del e$;
\item $X - e$ is a cofan $(f_{1},\ldots, f_{4})$
of $M \del e$, and $f_{1} = e_{5}$; or,
\item $X - e$ is a cofan $(f_{1},\ldots, f_{5})$
of $M \del e$, and $f_{1} = e_{5}$.
\end{enumerate}

Let us assume that either~(i) or~(ii) applies.
In either of these cases $X - e$ contains a triad $T$
such that $e_{5} \in T$.
Moreover, $X - e \subseteq \cl_{M \del e}(T)$,
so~\ref{sub7} implies that $e \in \cl_{M}(T \cup e_{1})$.
Let $C \subseteq T \cup \{e,\, e_{1}\}$ be a circuit of
$M$ that contains $e$.

If $e_{5} \notin C$ then
$e \in \cl_{M}(E(M) - \{e,\, e_{3},\, e_{4},\, e_{5}\})$,
and this implies that $\{e_{3},\, e_{4},\, e_{5}\}$
is a triad (and hence a vertical $3$\dash separator)
of $M$.
Similarly, if $e_{1} \notin C$ then $T \cup e$ is a
vertical $3$\dash separator of $M$.
Thus $\{e_{1},\, e_{5}\} \subseteq C$.
Now, since $M$ has no triads, $T \cup e$ is a
cocircuit of $M$, and therefore $C$ meets $T \cup e$
in an even number of elements.
We know that $C$ contains at least two elements,
$e$ and $e_{5}$, of $T \cup e$, but it cannot be the case
that $C = \{e,\, e_{1},\, e_{5}\}$, for that would imply
that $e \in \cl_{M}(A)$.
Thus $C = T \cup \{e,\, e_{1}\}$.
Fact~\ref{fct3} tells us that $M \del e / e_{1}$ is
$3$\dash connected.
If $M / e_{1}$ were not $3$\dash connected then
$e$ would have to be in a triangle of $M$ with $e_{1}$.
Such a triangle would necessarily meet the cocircuit
$T \cup e$ in exactly one element, a contradiction.
Thus $M / e_{1}$ is $3$\dash connected.
However, since $T \cup e$ is a cocircuit of $M$, and
$T \cup \{e,\, e_{1}\}$ is a circuit, it follows
that $T \cup e$ is a four-element circuit-cocircuit in
$M / e_{1}$.
Since $M / e_{1}$ has an $N$\dash minor this contradicts
our hypotheses on $M$.
Thus we have established the following fact.

\begin{fct}
\label{fct6}
$X - e$ is a cofan $(f_{1},\ldots, f_{5})$ of $M \del e$,
and $f_{1} = e_{5}$.
\end{fct}

Let $C \subseteq X \cup e_{1}$ be a circuit of $M$ that contains
$e$.
As before, we can argue that $\{e_{1},\, e_{5}\} \subseteq C$,
but that $C \ne \{e,\, e_{1},\, e_{5}\}$.
Let $T_{1} = \{f_{1},\, f_{2},\, f_{3}\}$, and let
$T_{2} = \{f_{3},\, f_{4},\, f_{5}\}$.
Since both $T_{1} \cup e$ and $T_{2} \cup e$ are cocircuits
of $M$ it follows that $C$ meets both of these sets in
an even number of elements.
We know that $C$ meets $T_{1} \cup e$ in the two elements
$e$ and $f_{1}$.
Suppose that $C$ contains $T_{1} \cup e$.
Then $C$ does not contain $f_{4}$, for then it would properly
contain the triangle $\{f_{2},\, f_{3},\, f_{4}\}$.
Nor can $C$ contain $f_{5}$, for then $C$ would meet
$T_{2} \cup e$ in an odd number of elements.
Thus $C = T_{1} \cup \{e,\, e_{1}\}$.
But then we can again show that $M / e_{1}$ is
$3$\dash connected, and since $T_{1} \cup e$ is a four-element
circuit-cocircuit of $M / e_{1}$ we again have a
contradiction to our hypotheses on $M$.

Therefore we conclude that $C \cap T_{1} = \{f_{1}\}$.
Since $C \ne \{e,\, e_{1},\, e_{5}\}$  we see that
$C$ contains $f_{4}$ or $f_{5}$, and since
$C$ meets $T_{2} \cup e$ in an even number of elements,
$C$ contains exactly one of $f_{4}$ or $f_{5}$.
Suppose that $f_{4} \in C$.
Then $C = \{e,\, e_{1},\, f_{1},\, f_{4}\}$.
By properties of circuits in binary matroids,
the symmetric difference of $C$ and
$\{f_{2},\, f_{3},\, f_{4}\}$, which is
$T_{1} \cup \{e,\, e_{1}\}$, is a disjoint
union of circuits.
Since $M$ has no circuits of size less than three this means
that $T_{1} \cup \{e,\, e_{1}\}$ is a circuit.
But we have already shown that this leads to a contradiction
to the hypotheses of the lemma.
Thus we have proved the following fact, and hence have
completed the proof of~\ref{sub8}.

\begin{fct}
\label{fct7}
$\{e,\, e_{1},\, f_{1},\, f_{5}\}$ is a circuit of $M$. \qedhere
\end{fct}
\end{proof}

We have assumed that $M \del e$ contains a cofan
$A_{1} = (e_{1},\ldots, e_{5})$.
We use this to prove that statement~\eqref{state5} of
Lemma~\ref{lem3} holds, and hence derive a contradiction.
Since $M \del e / e_{1}$ is \vtc\
by~\ref{sub3}, it follows that $M / e_{1}$ is also \vtc.
Suppose that $M / e_{1}$ has no vertical $3$\dash separations.
In this case we will show that statement~\eqref{state2} of
Lemma~\ref{lem3} holds with $x = e_{1}$.
By assumption $M / e_{1}$ is \vfc.
Also $M / e_{1}$ has an $N$\dash minor by~\ref{sub2}.
The triangle $T$ of statement~\eqref{state2} is equal to
$\{e_{2},\, e_{3},\, e_{4}\}$.
Since $\{e_{1},\, e_{2},\, e_{3}\}$ is a triad of $M \del e$
it follows that $\{e,\, e_{1},\, e_{2},\ e_{3}\}$ is a cocircuit
in $M$.
Moreover, it is easy to see that $e_{1}$ is contained in no
triangles in $M \del e$.
Therefore any triangle of $M$ that contains $e_{1}$ contains
$e$, and hence $e_{1}$ is in at most one such triangle.
Thus statement~\eqref{state2} holds.

Therefore we assume that $M / e_{1}$ has at least one
\vts.
Now~\ref{sub8} guarantees the existence of a cofan
$A_{2} = (f_{1},\ldots, f_{5})$ of $M \del e$ such that
$f_{1} = e_{5}$.
Using an identical argument we see that $M / f_{1}$
has a \vts, and by again
applying~\ref{sub8} we see that there exists
a cofan $A_{3} = (g_{1},\ldots, g_{5})$ of $M \del e$
such that $g_{1} = f_{5}$.

We know from~\ref{sub8}~(ii) that $f_{5} \notin A_{1}$, and
therefore $g_{1} \notin A_{1}$.
Similarly $e_{5} \notin A_{3}$.
Now $G_{M \del e}(A_{1}) = G_{M \del e}(A_{3}) = \varnothing$,
but $G_{M \del e}^{*}(A_{1}) = \{e_{1},\, e_{5}\}$ and
$G_{M \del e}^{*}(A_{3}) = \{g_{1},\, g_{5}\}$.
Since $A_{1}$ and $A_{3}$ are distinct maximal
small $3$\dash separators of $M \del e$, we can deduce
from Corollary~\ref{cor4} that if $A_{1}$ and
$A_{3}$ have a non-empty intersection, then they meet
in the element $e_{1} = g_{5}$.

Suppose that $A_{1} \cap A_{3} = \varnothing$.
Now $T = \{e_{1},\, e_{2},\, e_{3}\}$ is a triad
of $M \del e$.
However~\ref{sub8}~(iii) implies that $e$ is
in the closure of $\{f_{1},\, g_{1},\, g_{5}\}$ in
$M$, and this set is contained in $E(M) - (T \cup e)$.
Thus $T$ is a triad of $M$. 
This contradiction shows that $A_{1}$ and $A_{3}$ have
a non-empty intersection, and hence $e_{1} = g_{5}$.

Let $M' = M / e_{1} / f_{1} / g_{1}$.
We first show that $M'$ has an $N$\dash minor.
Since $A_{1}$ is a cofan of length five in $M \del e$
we see that $M \del e / e_{1}$ has an $N$\dash minor
by Proposition~\ref{prop41}.
Now $A_{1} - e_{1}$ is a cofan of length four in
$M \del e / e_{1}$, and $e_{5} = f_{1}$ is a good
element of $A_{1} - e_{1}$.
Applying Proposition~\ref{prop41} again we conclude
that $M \del e / e_{1} / f_{1}$ has an
$N$\dash minor.
Similarly, $A_{2} - f_{1}$ is a cofan of length four
in $M \del e / e_{1} / f_{1}$, and $f_{5} = g_{1}$
is a good element, so
$M \del e / e_{1} / f_{1} / g_{1}$, and
hence $M'$, has an $N$\dash minor.

Next we wish to show that $M'$ is vertically
$4$\dash connected.
Assume otherwise and let $(X,\, Y)$ be a \vks\ of
$M'$, where $k < 4$.
If $T$ is a triangle of $M'$ and $|T \cap X| \geq 2$
then we shall say that $T$ is \emph{almost contained} in
$X$.

It is clear that $T_{1} = \{e_{2},\, e_{3},\, e_{4}\}$,
$T_{2} = \{f_{2},\, f_{3},\, f_{4}\}$, and
$T_{3} = \{g_{1},\, g_{2},\, g_{3}\}$ are triangles
of $M'$.
By relabeling if necessary we can assume that both
$T_{1}$ and $T_{2}$ are almost contained in $X$.
Now it follows from Proposition~\ref{prop40} that
$(X \cup T_{1} \cup T_{2},\, Y - (T_{1} \cup T_{2}))$
is a vertical $k'$\dash separation of $M'$, where
$k' \leq k$.
Therefore we may as well assume that
$X$ contains both $T_{1}$ and $T_{2}$.

Note that $\{e,\, e_{1},\, f_{1},\, g_{1}\}$ is a circuit
of $M$ by~\ref{sub8}~(iii), so $e$ is a loop in $M'$.
Therefore we can also assume that $e \in X$.
Now $\{e,\, e_{1},\, e_{2},\, e_{3}\}$,
$\{e,\, f_{1},\, f_{2},\, f_{3}\}$, and
$\{e,\, g_{1},\, f_{3},\, f_{4}\}$ are all cocircuits of $M$.
If $C$ is a circuit of $M$ contained in
$Y \cup \{e_{1},\, f_{1},\, g_{1}\}$ that has a non-empty
intersection with $\{e_{1},\, f_{1},\, g_{1}\}$ then
$C$ meets at least one of these cocircuits in a single
element, which is impossible.
Therefore no such circuit can exist, and it follows that
\begin{displaymath}
r_{M}(Y \cup \{e_{1},\, f_{1},\, g_{1}\}) =
r_{M}(Y) + 3.
\end{displaymath}
Hence $r_{M}(Y) = r_{M'}(Y)$, and it now follows that
$(X \cup \{e_{1},\, f_{1},\, g_{1}\},\, Y)$ is a
\vks\ of $M$ for some $k < 4$.
This is a contradiction because $M$ is \vfc.
Thus $M'$ is \vfc.

Next we show that $M'$ contains exactly one loop, $e$,
and no parallel pairs.
We have already noted that $e$ is a loop in $M'$.
Suppose that $e' \ne e$ is also a loop.
Then there is a circuit
$C \subseteq \{e',\, e_{1},\, f_{1},\, g_{1}\}$ of $M$ such
that $e' \in C$.
There are no loops in $M$, so let us assume that $C$ contains
$e_{1}$ (the other cases are identical).
Since neither $e_{2}$ nor $e_{3}$ is a loop in $M'$ it follows
that $C$ is a circuit in $M \del e$ that meets the triad
$\{e_{1},\, e_{2},\, e_{3}\}$ in exactly one element, a
contradiction.

Now assume that $\{e',\, f'\}$ is a parallel pair in
$M'$.
There is a circuit $C$ of $M$ such that
$C \subseteq \{e',\, f',\, e_{1},\, f_{1},\, g_{1}\}$ and
$e',\, f' \in C$.
As before, we will assume that $e_{1} \in C$.
As $\{e_{1},\, e_{2},\, e_{3}\}$ and
$\{e_{1},\, g_{3},\, g_{4}\}$ are triads in $M \del e$ and
$C$ is a circuit of $M \del e$ it follows that
(relabeling if necessary) $e' \in \{e_{2},\, e_{3}\}$ and
$f' \in \{g_{3},\, g_{4}\}$.
Suppose that $C = \{e',\, f',\, e_{1}\}$.
Now, since $\{e_{3},\, e_{4},\, e_{5}\}$ and
$\{g_{1},\, g_{2},\, g_{3}\}$ are also triads of $M \del e$,
it follows that $e' = e_{2}$ and $f' = g_{4}$.
This implies that $(g_{2},\, g_{3},\, g_{4},\, e_{1},\, e_{2},\, e_{3})$
is a fan of $M \del e$.
As $M' \del e$ has an $N$\dash minor and $|E(N)| \geq 10$ it follows that
the complement of $\{g_{2},\, g_{3},\, g_{4},\, e_{1},\, e_{2},\, e_{3}\}$
in $M \del e$ contains at least six elements.
Moreover $\{g_{2},\, g_{3},\, g_{4},\, e_{1},\, e_{2},\, e_{3}\}$ is a
$3$\dash separator of $M \del e$ and it provides a contradiction to the
fact that $M \del e$ is $(4,\, 5)$\dash connected.
Therefore we assume that $|C| > 3$, so that $f_{1} \in C$
(the case when $g_{1} \in C$ is identical).
Now $C$ meets the triad $\{f_{1},\, f_{2},\, f_{3}\}$ in exactly
one element.
This contradiction shows that $M'$ has no parallel pairs.

We have shown that $M'$ is \vfc\ with an $N$\dash minor, and
that $M'$ has exactly one loop and no parallel pairs.
Therefore statement~\eqref{state5} holds with $x = e$,
$x_{i} = e_{i}$, $y_{i} = f_{i}$, and $z_{i} = g_{i}$ for
$1 \leq i \leq 5$.
Thus our assumption that $M \del e$ contains
a cofan of length five has lead to a contradiction.
Hence we have established the following claim.

\begin{sub}
\label{sub4}
There are no cofans of length five in $M \del e$.
\end{sub}

However, $M \del e$ does have at least one small
vertical $3$\dash separator with at least four elements.
It is a consequence of~\ref{sub4} that any small
vertical $3$\dash separator of $M \del e$ has rank three.

\begin{sub}
\label{sub11}
Let $A$ be a small vertical $3$\dash separator of $M \del e$
such that $|A| \geq 4$, and let $x \in A$ be a good element.
Then $M / x$ is not \vfc.
\end{sub}

\begin{proof}
Suppose that $M / x$ is \vfc.
We will show that statement~\eqref{state2} of Lemma~\ref{lem3}
holds.
We know that $M / x$ has an $N$\dash minor by~\ref{sub2}.
Since $A$ is a cofan of length four or a fan of length five in
$M \del e$ we see that $A$ contains a triad $T_{A}$ and a
triangle $T$ of $M \del e$ such that $x \in T_{A}$,
$x \notin T$, and $T \cap T_{A} = T_{A} - x$.
Now $T$ is a triangle of $M / x$, and since $M$ has no triads
$T_{A} \cup e$ is a cocircuit of $M$ that meets $T$ in two elements.

Next we show that $x$ can be in at most two triangles of $M$.
Any triangle of $M \del e$ that contains $x$ contains a
member of $T_{A} - x$, so $x$ can be in at most two triangles in
$M \del e$.
If $x$ is in two triangles then $A$ is contained in a
$3$\dash separator \ov{A} of $M \del e$ such that
$|\ov{A}| = 6$.
Since $M \del e$ is $(4,\, 5)$\dash connected this means that
\ov{A} is not a small $3$\dash separator, so it contains at least
seven elements of $E(N)$, an impossibility.
Furthermore, $x$ can be contained in at most one triangle of $M$
that contains $e$, and any such triangle cannot contain an element
of $T$, for otherwise $e \in \cl_{M}(T_{A})$,
contradicting~\ref{sub9}.
Thus $x$ is in at most two triangles in $M$, and
if $x$ is in two triangles, then it is contained in exactly
one triangle that contains $e$, and exactly one triangle
that meets $T_{A} - x$.
Therefore statement~\eqref{state2} holds with
$C^{*} = T_{A} \cup e$.
\end{proof}

\begin{sub}
\label{sub12}
Let $A$ be a maximal small vertical $3$\dash separator of
$M \del e$ and suppose that $|A| \geq 4$.
Let $x \in A$ be a good element.
Suppose that $X$ is a small vertical $3$\dash separator of
$M / x$.
Then $e \in X$, and there is a triad $T$ of $M \del e$
such that the following statements hold.
\begin{enumerate}[(i)]
\item $T \subseteq X - e$.
\item $X - e \subseteq \cl_{M \del e}(T \cup x)$.
\item There is an element $y \in T$ such that $\{e,\, x,\, y\}$ is
a triangle of $M$.
\end{enumerate}
\end{sub}

\begin{proof}
By~\ref{sub3} we see that $M \del e / x$ is vertically
$3$\dash connected.
This implies that $M / x$ is also \vtc.
We know that $e \in X$ by~\ref{sub7}.
Moreover, $X - e$ is a small vertical $3$\dash separator of
$M \del e / x$ by definition.
It cannot be the case that $X - e \subseteq A - x$, for
then $A$ would have rank four in $M \del e$.
Thus, by~\ref{sub5}, there is some vertical
$3$\dash separator $X_{0}$ of $M \del e$ such that
$\inter_{M \del e / x}(X - e) = \inter_{M \del e}(X_{0})$.
Since $X_{0}$ is a small $3$\dash separator of $M \del e$
we deduce from~\ref{sub4} that $r_{M \del e}(X_{0}) = 3$
and $X_{0}$ contains a triad $T$ of $M \del e$ such that
$X_{0} \subseteq \cl_{M \del e}(T)$.
It is elementary to demonstrate that $T$ is contained
in $\inter_{M \del e}(X_{0})$, so $T$ is contained in
$X - e$.

Let $X' = \inter_{M \del e / x}(X - e)$.
Note that $X' \subseteq X_{0}$, so $X' \subseteq \cl_{M \del e}(T)$.
Proposition~\ref{prop39}~(ii) tells us that
$X - e \subseteq \cl_{M \del e / x}(X')$, so
$X - e \subseteq \cl_{M \del e}(X' \cup x)$.
This combined with the fact that
$X' \subseteq \cl_{M \del e}(T)$ shows that 
$X - e \subseteq \cl_{M \del e}(T \cup x)$.

It remains to show that there is an element $y \in T$ such that
$\{e,\, x,\, y\}$ is a triangle of $M$.
From~\ref{sub7} we see that $e \in \cl_{M}((X \cup x) - e)$.
Since $X - e \subseteq \cl_{M \del e}(T \cup x)$ it follows that
$e \in \cl_{M}(T \cup x)$.
Let $C \subseteq T \cup \{e,\, x\}$ be a circuit of $M$ that
contains $e$.
Then $x \in C$, for otherwise $T \cup e$
is a vertical $3$\dash separator of $M$.
Since $T \cup e$ is a cocircuit of $M$ it follows that
$C$ meets $T \cup e$ in either two or four elements.
Suppose that the latter case holds, so that
$C = T \cup \{e,\, x\}$.
This means that $T \cup e$ is a four-element
circuit-cocircuit in $M / x$.
Any triangle of $M$ that contains $x$ cannot contain
an element of $T \cup e$, for if it did it would meet
the cocircuit $T \cup e$ in exactly one element.
Thus any parallel pair in $M / x$ contains no element
of $T \cup e$.
It follows that $T \cup e$ is a four-element
circuit-cocircuit in $\si(M / x)$.
Moreover, $\si(M / x)$ is $3$\dash connected
since $M / x$ is \vtc.
We know that $M \del e / x$ has an $N$\dash minor
by~\ref{sub2}, so $M / x$, and hence $\si(M / x)$, has
an $N$\dash minor.
Thus $\si(M / x)$ provides a contradiction to the
hypotheses of the lemma.

Therefore $C$ does not meet $T \cup e$ in four elements, and
hence there is an element $y \in T$ such that $\{e,\, x,\, y\}$
is a circuit of $M$.
\end{proof}

Assume that $A$ is a maximal small vertical $3$\dash separator
of $M \del e$ and that $|A| \geq 4$.
Let $x \in A$ be a good element.
By inspection we see that $A$ contains a triad $T_{A}$
of $M \del e$ such that $x \in T_{A}$ and
$A \subseteq \cl_{M \del e}(T_{A})$.
We know from~\ref{sub11} that $M / x$ is not \vfc, so
let $X$ be a small vertical $3$\dash separator of $M / x$.
Let $\mcal{T}_{X}$ be the set of triads of $M \del e$
satisfying~\ref{sub12}.
Let \mcal{T} be the union of the collections $\mcal{T}_{X}$,
taken over all small vertical $3$\dash separators of $M / x$.

\begin{sub}
\label{sub6}
If $T_{B} \in \mcal{T}$ then $T_{A} \cap T_{B} = \varnothing$.
\end{sub}

\begin{proof}
Let $T_{B}$ be a member of \mcal{T}, and let $X$ be a small
vertical $3$\dash separator of $M / x$ such that
$T_{B} \in \mcal{T}_{X}$.
Then $T_{B}$ is contained in \ov{T_{B}}, a maximal small vertical
$3$\dash separator of $M \del e$.
Clearly $A$ and \ov{T_{B}} are distinct maximal small
$3$\dash separators of $M \del e$.
Since $A$ is either a cofan of length four or a fan of length
five in $M \del e$ it is easy to see that
$G_{M}^{*}(A)$ is either empty or equal to $\{x\}$.
Suppose that \ov{T_{B}} contains $x$.
Then as $T_{B} \subseteq \ov{T_{B}}$, and both $T_{B}$ and
\ov{T_{B}} have rank three in $M \del e$ it follows that
$r_{M \del e}(T_{B} \cup x) = 3$.
Now $X - e \subseteq \cl_{M \del e}(T_{B} \cup x)$
by~\ref{sub12}~(ii) and $e \in \cl_{M}((X \cup x) - e)$
by~\ref{sub7}, so $X \subseteq \cl_{M}(T_{B} \cup x)$.
Thus $r_{M}(X) \leq r_{M}(T_{B} \cup x) = 3$, so
$X$ has rank two in $M / x$.
This is a contradiction, so \ov{T_{B}} does not contain $x$ and
therefore \ov{T_{B}} cannot meet $G_{M}^{*}(A)$.
Since $T_{A}$ has an empty intersection with
$G_{M}(A)$ it follows from Corollary~\ref{cor4} that
$T_{B}$ cannot contain any element of $T_{A}$.
\end{proof}

The remainder of the proof is divided into two cases, each
of which is then divided into various subcases and sub-subcases.

\begin{case1}
There is a triad $T_{B} \in \mcal{T}$ such that
$T_{A} \cup T_{B}$ contains a circuit $C$ of $M \del e$.
\end{case1}

Let $X$ be the small vertical $3$\dash separator of $M / x$
such that $T_{B} \subseteq X - e$ and
$X - e \subseteq \cl_{M \del e}(T_{B} \cup x)$.

We need one more technical result before we proceed.

\begin{sub}
\label{sub13}
There is at most one triangle of $M$ contained in
$T_{A} \cup T_{B} \cup e$ that is not also contained in $C \cup e$.
\end{sub}

\begin{proof}
First note that any triangle of $M$ contained in $T_{A} \cup T_{B} \cup e$
contains $e$, for otherwise it meets one of the cocircuits
$T_{A} \cup e$ and $T_{B} \cup e$ in exactly one element.
Similarly, any such triangle contains exactly one element from
$T_{A}$ and one element from $T_{B}$, for otherwise it meets
either $T_{A} \cup e$ or $T_{B} \cup e$ in three elements.

Now suppose that $\{e,\, x,\, y\}$ and $\{e,\, x'\, y'\}$ are
two distinct triangles of $M$ such that
$x,\, x' \in T_{A}$ and $y,\, y' \in T_{B}$, but that
neither $\{e,\, x,\, y\}$ nor $\{e,\, x'\, y'\}$ is contained
in $C \cup e$.
Note that $x \ne x'$, for if $x$ were equal to $x'$ then
$y$ and $y'$ are distinct elements (for $\{e,\, x,\, y\}$ and
$\{e,\, x'\, y'\}$ are distinct triangles) and $\{y,\, y'\}$ is
a parallel pair.
But $M$ is simple.
Thus $x \ne x'$, and similarly $y \ne y'$.
Thus $\{x,\, x',\, y,\, y'\}$ is a circuit of $M$.

Since $T_{A}$ and $T_{B}$ are disjoint triads of $M \del e$
by~\ref{sub6} we see that
$|T_{A} \cap C| = |T_{B} \cap C| = 2$.
Therefore we can assume by relabeling that $x,\, y' \in C$.
Assume that $x' \in C$.
Then the symmetric difference of $C$ and
$\{x,\, x',\, y,\, y'\}$ contains two elements, so $M$ contains
a circuit of size at most two, a contradiction.
Thus $x' \notin C$, and by a symmetrical argument
$y \notin C$.

Now it is easy to see that $T_{A} \cup e$ spans
$T_{A} \cup T_{B} \cup e$ in $M$, so
$r_{M}(T_{A} \cup T_{B} \cup e) = 4$.
Also, both $T_{A} \cup e$ and $T_{B} \cup e$ are
cocircuits in $M$, so $T_{A} \cup T_{B}$ is a disjoint
union of cocircuits in $M$.
Since $M$ has no cocircuits of size less than four this
means that $T_{A} \cup T_{B}$ is a cocircuit.
Hence $r_{M}^{*}(T_{A} \cup T_{B} \cup e) = 5$.

For any subset $S \subseteq E(M)$ the equation
$r_{M}^{*}(S) = |S| - r(M) + r_{M}(E(M) - S)$ holds.
From this we deduce that
\begin{equation}
\label{eqn2}
\lambda_{M}(S) = r_{M}(S) + r_{M}(E(M) - S) - r(M) =
r_{M}(S) + r_{M}^{*}(S) - |S|.
\end{equation}
Equation~\eqref{eqn2} implies that
$\lambda_{M}(T_{A} \cup T_{B} \cup e) = 2$.
Since $M$ has no vertical $3$\dash separations this means
that the complement of $T_{A} \cup T_{B} \cup e$ in $M$
has rank at most two.
Therefore the complement of $T_{A} \cup T_{B} \cup e$
contains at most three elements, which means that
$|E(M)| \leq 10$, a contradiction.
\end{proof}

Now we proceed to analyze the subcases.

\begin{subcase1.1}
$x \notin C$.
\end{subcase1.1}

We will show that statement~\eqref{state4} of Lemma~\ref{lem3}
holds.

Let $T_{x}$ be the triangle supplied by~\ref{sub12}~(iii),
so that $e,\, x \in T_{x}$, and $T_{x}$ contains an element
of $T_{B}$.

Since $|T_{A} \cap C| = 2$ it follows that
$C \cap T_{A} = T_{A} - x$.
Since $x$ is a good element of $A$, and $A$ is either a cofan of
length four or a fan of length five in $M \del e$ there is a
triangle $T$ of $M \del e$ such that $T \subseteq A$ and
$x \notin A$.
Moreover, $T \cap T_{A} = C \cap T_{A} = T_{A} - x$.
Let $f$ be the element in $T - T_{A}$.

Since $(T_{B} \cap C) \cup f$ is the symmetric difference
of the circuits $C$ and $T$ it is a triangle.
Therefore $f \in \cl_{M \del e}(T_{B})$, and it
follows that $T_{B} \cup f$ is a small vertical
$3$\dash separator of $M \del e$ with four elements.
Let \ov{T_{B}} be a maximal small $3$\dash separator
of $M \del e$ that contains $T_{B} \cup f$.
The single element in $T_{B} - C$ is a good
element of \ov{T_{B}}.
Let us call this element $y$.

By~\ref{sub11} we see that $M / y$ is not \vfc, so 
let $X'$ be a small vertical $3$\dash separator of $M / y$.
By applying~\ref{sub12} to \ov{T_{B}} and $X'$
we see that there is a triangle $T_{y}$ of $M$ such that
$e,\, y \in T_{y}$ and $T_{y}$ meets $X'$ in the single
element $x'$.

If $x' \notin T_{A}$ then $e$ is in $\cl_{M}(E(M) - (T_{A} \cup e))$,
implying that $T_{A}$ is a triad of $M$.
Thus $x' \in T_{A}$, so $T_{y} \subseteq T_{A} \cup T_{B} \cup e$.
We already know that $T_{x} \subseteq T_{A} \cup T_{B} \cup e$.
Since neither $x$ nor $y$ is contained in $C$, neither
$T_{x}$ nor $T_{y}$ is contained in $C \cup e$.
Now~\ref{sub13} implies that $T_{x}$ and $T_{y}$ are equal,
so that $\{e,\, x,\, y\}$ is a triangle.

We know that $M \del e / x$ has an $N$\dash minor by~\ref{sub2}.
Moreover, $T_{B} \cup f$ is a cofan of length four in
$M \del e / x$, and $y$ is a good element of this cofan so
$M \del e / x / y$ has an $N$\dash minor by Proposition~\ref{prop41}.

Next we show that $M / x / y$ is \vfc.
Suppose otherwise and let $(X',\, Y')$ be a \vks\ of $M / x / y$
where $k < 4$.
By relabeling if necessary we can assume that at least two elements
of the triangle $(T_{B} \cap C) \cup f$ are contained in $X'$.
Then Proposition~\ref{prop40} implies that
\begin{displaymath}
(X' \cup (T_{B} \cap C) \cup f,\, Y' - ((T_{B} \cap C) \cup f))
\end{displaymath}
is a vertical $k'$\dash separation of $M / x / y$ where
$k' \leq k$, so we may as well assume that $(T_{B} \cap C) \cup f$
is contained in $X'$.
Since $e$ is a loop in $M / x / y$ we can assume that
it is in $X'$.
Now $T_{B} \cup e$ is a cocircuit of $M / x$, and
since $(T_{B} \cup e) - y$ is contained in $X'$ this means
that $y \notin \cl_{M / x}(Y')$.
Therefore $(X' \cup y,\, Y')$ is a \vks\ of $M / x$, and $k$ is three.
If $Y'$ is a small vertical $3$\dash separator of $M / x$, then
by~\ref{sub12} there is an element $y' \in Y'$ such that
$\{e,\, x,\, y'\}$ is a triangle of $M$.
Since $y'$ cannot be equal to $y$ this means that
$y'$ and $y$ are parallel in $M$, a contradiction.
Thus $X' \cup y$ is a small vertical $3$\dash separator of
$M / x$.
But it follows from~\ref{sub12}~(ii) and~\ref{sub7} that
$X' \cup x \cup y$ is spanned in $M$ by the union of $x$ with a
triad of $M \del e$.
Therefore $X' \cup x \cup y$ has rank at most four in $M$ and
$X'$ has rank at most two in $M / x / y$, a contradiction.
Therefore $M / x / y$ is \vfc.

Finally we show that $M / x / y$ has exactly one loop and no
parallel pairs.
We have noted that $e$ is a loop in $M / x / y$.
If $M / x / y$ has more than one loop then $M$ has
a circuit of size at most two, a contradiction.

Suppose that $\{e',\, f'\}$ is a parallel pair in $M / x / y$.
Note that neither $e'$ nor $f'$ is equal to $e$.
Then there is a circuit $C' \subseteq \{e',\, f',\, x,\, y\}$
of $M$ such that $e',\, f' \in C'$.
Assume that $|C'| = 4$.
By relabeling if necessary we can assume that $e' \in T_{A} - x$
and $f' \in T_{B} - y$, for otherwise $C'$ meets one of the
cocircuits $T_{A} \cup e$ or $T_{B} \cup e$ in exactly one
element.
Now we see that $T_{A} \cup e$ spans $T_{A} \cup T_{B} \cup e$
in $M$, so $r_{M}(T_{A} \cup T_{B} \cup e) = 4$, and
since $T_{A} \cup T_{B}$ is a cocircuit we deduce that
$r^{*}_{M}(T_{A} \cup T_{B} \cup e) = 5$.
From this we can obtain a contradiction, as in the proof
of~\ref{sub13}.

Therefore we will assume that $|C'| = 3$.
We will assume that $x \in C'$, since the case when $y \in C'$ is
identical.
Since $C'$ meets $T_{A} \cup e$ in two elements we will
assume that $e' \in T_{A}$ while $f' \notin T_{A}$.
Now $T_{A} \cup \{f,\, f'\}$ is a fan of length five in $M \del e$
and since $A$ is a maximal small $3$\dash separator of $M \del e$
this means that $A = T_{A} \cup \{f,\, f'\}$.
Thus $A$ contains a good element distinct from $x$.
Let this element be $x'$.
By~\ref{sub11} we see that $M / x'$ contains a small
vertical $3$\dash separator, and~\ref{sub12} implies the
existence of a triangle $T_{x'}$ that contains $e$ and
$x'$.
Then $T_{x'}$ contains an element $y'$
of $T_{B}$, for otherwise $T_{B}$ is a triad of $M$.
Note that $y'$ is not equal to $y$, for if it were then $x$ and
$x'$ would be parallel in $M$.
Thus $\{x,\, y,\, x',\, y'\}$ is a circuit of $M$.
From this we again conclude that 
$r_{M}(T_{A} \cup T_{B} \cup e) = 4$ while
$r_{M}^{*}(T_{A} \cup T_{B} \cup e) = 5$ and hence derive
a contradiction.

Now if we let $z = e$, $T_{1} = T$, and
$T_{2} = (T_{B} - y) \cup f$ we see that statement~\eqref{state4}
of the lemma holds.
Our assumption that $x \notin C$ has lead us to a contradiction,
so we consider the next subcase.

\begin{subcase1.2}
$x \in C$.
\end{subcase1.2}

We have already noted that $C$ contains exactly two
elements of $T_{A}$ and two elements of $T_{B}$.
Let $c_{0}$ be the single element other than $x$
contained in $T_{A} \cap C$.
Let $c_{1}$ and $c_{2}$ be the two elements in
$T_{B} \cap C$, and let $y$ be the single
element in $T_{B} - \{c_{1},\, c_{2}\}$.

Now $\{c_{0},\, c_{1},\, c_{2}\}$ is a triangle of $M \del e / x$,
and it intersects the triad $T_{B}$ in two elements.
Thus $(c_{0},\, c_{1},\, c_{2},\, y)$ is a
four-element fan in $M \del e / x$ and $y$ is a good element of
this fan.
We know that $M \del e / x$ has an $N$\dash minor by~\ref{sub2},
so it follows from Proposition~\ref{prop41} that
$M \del e / x / y$, and hence $M / y$, has an $N$\dash minor.

By~\ref{sub12} there is a triangle of $M$ that contains
$e$ and $x$ and an element of $T_{B}$.

\begin{subsubcase1.2.1}
$\{e,\, x,\, y\}$ is not a triangle of $M$.
\end{subsubcase1.2.1}

By relabeling if necessary we can assume that
$\{e,\, x,\, c_{1}\}$ is a triangle.
We will show that statement~\eqref{state6} of the lemma
holds.

The set $\{e,\, c_{0},\, c_{2}\}$ is the symmetric difference
of the circuits $C$ and $\{e,\, x,\, c_{1}\}$, and therefore
is a triangle of $M$.

Suppose that $M / y$ is not \vfc, and let $(X',\, Y')$ be a
\vks\ of $M / y$, where $k < 4$.
There is a triangle $T$ contained in $A$ that contains two
elements of $T_{A}$, but does not contain $x$.
In particular, $T$ contains $c_{0}$.
By relabeling if necessary we can assume that at least
two elements of $T$ are contained in $X'$.
Then by Proposition~\ref{prop40} we can assume
that $T$ is contained in $X'$.

Suppose that one of $e$ or $c_{2}$ is contained in $\cl_{M / y}(X')$.
Then the entire triangle $\{e,\, c_{0},\, c_{2}\}$ is
contained in $\cl_{M / y}(X')$, and we can again apply
Proposition~\ref{prop40} and assume that $X'$ contains
$\{e,\, c_{0},\, c_{2}\}$.
In this case $X'$ contains three elements of the cocircuit
$T_{A} \cup e$, so in particular $x \in \cl_{M}^{*}(X')$.
Suppose that $(X' \cup x,\, Y' - x)$ is not a
vertical $k'$\dash separation of $M / y$ for some $k' \leq k$.
As $M / y$ is \vtc\ and $(X',\, Y')$ is a \vks\ of $M / y$
for some $k < 4$ it follows that $r_{M / y}(Y') = 3$ and
$r_{M / y}(Y' - x) = 2$.
But it is easy to see that this implies that $Y'$ contains
a triad of $M / y$, and hence of $M$.
This is a contradiction, so we can now assume that
$(X',\, Y')$ is a \vks\ of $M / y$ such that
$X'$ contains $T \cup T_{A} \cup \{c_{2},\, e\}$.
Then $X'$ contains two elements of the triangle
$\{e,\, x,\, c_{1}\}$, so we assume that $X'$ contains the
entire triangle.
This means that $X'$ contains $(T_{B} - y) \cup e$.
Since $T_{B} \cup e$ is a cocircuit of $M$ we conclude
that $y \notin \cl_{M}(Y')$.
This means that $(X' \cup y,\, Y')$ is a \vks\ of $M$,
a contradiction.

Therefore we assume that neither $e$ nor $c_{2}$ is contained
in $\cl_{M / y}(X')$, and in particular $e,\, c_{2} \in Y'$.
If $c_{1}$ were in $Y'$, then $Y'$ would contain
$(T_{B} - y) \cup e$, which would imply that
$y \notin \cl_{M}(X')$ and that $(X',\, Y' \cup y)$
is a \vks\ of $M$.
Thus $c_{1} \in X'$.
This implies that $x \in Y'$, for $\{e,\, x,\, c_{1}\}$
is a triangle in $M / y$, and we have concluded that
$e \notin \cl_{M / y}(X')$.
But now, since both $x$ and $e$ are in $Y'$, we see that
$c_{1}$ is in $\cl_{M / y}(Y')$, and therefore
$(X' - c_{1},\, Y' \cup c_{1})$ is a vertical
$k'$\dash separation for some $k' \leq k$.
Since $Y' \cup c_{1}$ contains $(T_{B} - y) \cup e$ we
easily derive a contradiction.

Therefore $M / y$ is \vfc\ with an $N$\dash minor.
The set $T_{B} \cup e$ is a cocircuit of $M$, and any
triangle of $M$ that contains $y$ contains exactly
two members of this cocircuit.
Suppose that there is a triangle $T$ of $M$ that contains
both $y$ and $e$.
Since $T_{A} \cup e$ is also a cocircuit of $M$,
it follows that the single element in $T - \{e,\, y\}$
is contained in $T_{A}$.
Now we see that $T_{A} \cup e$ spans $T_{A} \cup T_{B} \cup e$
in $M$, so as before we deduce that
$r_{M}(T_{A} \cup T_{B} \cup e) = 4$ and
$r_{M}^{*}(T_{A} \cup T_{B} \cup e) = 5$.
This again leads to a contradiction.

It follows from these observations that $y$ is in at most two
triangles in $M$ and hence $M / y$ has at most two parallel pairs.
Clearly $M / y$ has no loops, since $M$ has no parallel pairs.
Now, by relabeling $y$ with $x$ we see that
statement~\eqref{state6} holds with $C^{*} = T_{B} \cup e$,
$T_{1} = \{e,\, c_{0},\, c_{2}\}$, and
$T_{2} = \{e,\, x,\, c_{1}\}$.
The assumption that $\{e,\, x,\, y\}$ is not a triangle of
$M$ has lead to a contradiction, so we are forced into
the next subcase.

\begin{subsubcase1.2.2}
$\{e,\, x,\, y\}$ is a triangle of $M$.
\end{subsubcase1.2.2}

We will show in this sub-subcase that statement~\eqref{state8}
of the lemma holds.

We noted at the beginning of Subcase~1.2 that $M \del e / x / y$
has an $N$\dash minor.
Suppose that $M / x / y$ is not \vfc\ and that $(X',\, Y')$
is a \vks\ of $M / x / y$, where $k < 4$.
The set $\{c_{0},\, c_{1},\, c_{2}\}$ is a
triangle of $M / x / y$.
By relabeling if necessary we assume that
it is contained in $X'$.
Since $e$ is a loop in $M / x / y$ we assume that it
too is contained in $X'$.
Then $X'$ contains three elements of the cocircuit
$T_{B} \cup e$ of $M / x$, and it follows that
$y \notin \cl_{M / x}(Y')$.
Thus $(X' \cup y,\, Y')$ is a \vks\ of $M / x$ and
$k = 3$ since $M / x$ is \vtc\ by~\ref{sub2}.
If $Y'$ is a small vertical $3$\dash separator of
$M / x$ then~\ref{sub12} implies that there is a triangle
of $M$ that contains $e$, $x$, and an element of $Y'$.
Since $y \notin Y'$ this implies that $y$ is parallel to
an element of $Y'$ in $M$, a contradiction.
Therefore $X' \cup y$ is a small vertical $3$\dash separator
of $M / x$.
But then~\ref{sub12} tells us that $X' \cup x \cup y$ has
rank at most four in $M$ and hence $X'$ has rank at most
two in $M / x / y$, contradicting the fact that
$(X',\, Y')$ is a \vts\ of $M / x / y$.
Thus $M / x / y$ is \vfc.

Next we show that $M / x / y$ has exactly one loop and no
parallel pairs.
Clearly $e$ is a loop of $M / x / y$, and the presence of any
other loop implies a circuit of a size at most two in $M$.
Suppose that $\{e',\, f'\}$ is a parallel pair in
$M / x / y$.
Let $C' \subseteq \{e',\, f',\, x,\, y\}$ be a circuit of
$M$ that contains $e'$ and $f'$.
Suppose that $|C'| = 4$.
We can assume that $e' \in T_{A}$ and that $f' \in T_{B}$,
for otherwise $C'$ intersects $T_{A} \cup e$ or $T_{B} \cup e$
in a single element.
Now we see that $T_{A} \cup e$ spans $T_{A} \cup T_{B} \cup e$
in $M$.
As before we can obtain a contradiction.

Therefore we assume that $|C'| = 3$.
Let us suppose that $C' = \{e',\, f',\, x\}$.
We can assume that $e' \in T_{A}$ while $f' \notin T_{A}$.
Then $A \cup f'$ contains a five-element fan of $M \del e$,
so in fact $f'$ is contained in $A$, since $A$ is a
maximal small $3$\dash separator of $M \del e$.
Now $A$ contains a good element $x' \ne x$.
Then~\ref{sub11} implies that $M / x'$ contains a
small vertical $3$\dash separator, and~\ref{sub12} tells
us that there is a triangle $T_{x'}$ of $M$ that contains
$e$ and $x'$.
There must be an element $y'$ of $T_{B}$ in $T_{x'}$, and this
element cannot be equal to $y$.
Thus $\{x,\, y,\, x',\, y'\}$ is a circuit of $M$ and we can
again obtain a contradiction.

If $C' = \{e',\, f',\, y\}$ then we assume that
$e' \in T_{B}$ while $f' \notin T_{B}$.
Then $T_{B} \cup f'$ is a four-element cofan in $M \del e$,
and it contains a good element $y'$ not equal to $y$.
Then~\ref{sub11} and~\ref{sub12} imply the existence
of a triangle $T_{y'}$ containing $e$ and $y'$, and
$T_{y'}$ contains an element $x' \in T_{A}$ not equal
to $x$, so $\{x,\, y,\, x',\, y'\}$ is a circuit in $M$.
If $x' \notin C$ then we can obtain a contradiction as before. 
If $x'$ is in $C$ then the symmetric difference of $C$ and
$\{x,\, y,\, x',\, y'\}$ contains at most two elements,
leading to a contradiction.

Thus statement~\eqref{state8} of the lemma holds with
$z = e$, $T_{1}$ equal to the triangle in $A$ that
meets $T_{A}$ in $T_{A} - x$, and
$T_{2} = \{c_{0},\, c_{1},\, c_{2}\}$.

This contradiction means that we must now consider Case~2.

\begin{case2}
There is no triad $T_{B} \in \mcal{T}$ such that
$T_{A} \cup T_{B}$ contains a circuit of $M \del e$.
\end{case2}

\begin{subcase2.1}
No small vertical $3$\dash separator of $M / x$ contains
more than four elements.
\end{subcase2.1}

Our aim in this subcase is to show that statement~\eqref{state3}
of Lemma~\ref{lem3} holds.
We know that $M / x$ has an $N$\dash minor by~\ref{sub2}.
Next we show that $\si(M/ x)$ is \ifc.
Suppose otherwise and let $(X,\, Y)$ be a
$3$\dash separation of $\si(M / x)$ such that
$|X|,\, |Y| \geq 4$.
Then both $X$ and $Y$ have rank at least three in
$\si(M / x)$.
The separation $(X,\, Y)$ naturally induces a \vts\ of $M / x$.
Let us call this \vts\ $(X',\, Y')$.
By relabeling assume that $X'$ is a small vertical
$3$\dash separator of $M / x$.
Then by~\ref{sub12} there is an element $y \in X'$ such
that $\{e,\, x,\, y\}$ is a triangle of $M$.
Therefore $X'$ contains a parallel pair in $M / x$.
But because $|X'| \leq 4$ by assumption, and $X$
is contained in $X'$ this means that
$|X| \leq 3$, contrary to hypothesis.
Thus $\si(M / x)$ is \ifc.

Next we show that $M / x$ has no loops and exactly one
parallel pair.
Clearly $M / x$ cannot have a loop, since $M$ has no parallel
pairs.
We know from~\ref{sub11} that $M / x$ has a small
vertical $3$\dash separator $X$.
Let $T_{B}$ be the triad of $M \del e$ contained in $X - e$
that~\ref{sub12} supplies.
Then there is an element $y \in T_{B}$ such that
$\{e,\, x,\, y\}$ is a triangle of $M$.
Thus $M / x$ has at least one parallel pair.
Suppose that $M / x$ has another parallel pair.
Obviously $x$ can be contained in only one triangle of
$M$ that also contains $e$.
It follows that $x$ is contained in a triangle of
$M \del e$.
Therefore that $A$ is a five-element fan, so
$A$ contains a good element $x' \ne x$.

Now $\{e,\,x'\}$ is contained in a triangle $T_{x'}$
of $M$ by~\ref{sub11} and~\ref{sub12}.
Moreover, the other element in $T_{x'}$ is an element
$y'$ of $T_{B}$, for otherwise $T_{B}$ is a triad of
$M$.
Also, $y'$ is not equal to $y$, for that would imply that
$x$ and $x'$ are parallel in $M$.
Thus $\{x,\, y,\, x',\, y'\}$ is a circuit of $M \del e$
contained in $T_{A} \cup T_{B}$, contrary to our
assumption.
Thus $M / x$ has exactly one parallel pair.

It remains to show that $\si(M / x)$ has at least one triangle
and at least one triad.
There is a triangle of $A$ that does not contain $x$, and this
is also a triangle of $\si(M / x)$.
We know that $M / x$ has a small vertical $3$\dash separator $X$
and that $X$ contains a triad $T_{B}$ of $M \del e$.
By the above discussion $\si(M / x)$ is isomorphic to
$M / x \del e$, and $T_{B}$ is a triad of $M / x \del e$.

Therefore statement~\eqref{state3} of the lemma holds, so we
are forced to consider the final subcase.

\begin{subcase2.2}
There is a small vertical $3$\dash separator $X$ of
$M / x$ such that $|X| > 4$.
\end{subcase2.2}

We will show that statement~\eqref{state9} holds.
Let $T_{B}$ be the triad of $M \del e$ supplied by~\ref{sub12},
so that $T_{B} \in \mcal{T}_{X}$.
We know from~\ref{sub12} that
$X - e \subseteq \cl_{M \del e}(T_{B} \cup x)$.
Assume that $f \in X - e$ is not contained in
$\cl_{M \del e}(T_{B})$.
Then there is a circuit
$C \subseteq (T_{B} \cup x) \cup f$ that contains $f$ and $x$.
Since $C$ meets $T_{A}$ in two elements we deduce
from~\ref{sub6} that $f \in T_{A}$.
But now $C$ is contained in $T_{A} \cup T_{B}$, contrary
to our assumption.

Therefore $X - e \subseteq \cl_{M \del e}(T_{B})$, and
since $|X - e| \geq 4$, it follows that $X - e$ is a
fan or cofan with length four or five in $M \del e$.
Let $y \in X - e$ be a good element.
Let $y' \in T_{B}$ be the element such that
$\{e,\, x,\, y'\}$ is a triangle of $M$, and
suppose that $y' \ne y$.

By~\ref{sub11} and~\ref{sub12} there is a triangle
$T_{y}$ of $M$ such that $e,\, y \in T_{y}$.
There exists an element $x' \in T_{y} \cap T_{A}$,
for otherwise $T_{A}$ is a triad of $M$.
It cannot be the case that $x' = x$ for that would imply
that $y$ and $y'$ are parallel in $M$.
Thus $\{x,\, y,\, x',\, y'\}$ is a circuit of $M \del e$
contained in $T_{A} \cup T_{B}$, contrary to our assumption.

It follows that $y = y'$, so $\{e,\, x,\, y\}$ is a triangle of
$M$.
As before, we can show that $M / x / y$ has an $N$\dash minor.
Next we show that $M / x / y$ is \vfc.
Suppose that $(X',\, Y')$ is a \vks\ of $M / x / y$ where
$k < 4$.
There is an element $f \in X - e$ such that
$(T_{B} - y) \cup f$ is a triangle of $M \del e$ and of
$M / x / y$.
We assume that $(T_{B} - y) \cup f$ is contained in $X'$.
We can also assume that $e \in X'$ as it is a loop of $M / x / y$.
Since $T_{B} \cup e$ is a cocircuit in $M / x$ and
$(T_{B} - y) \cup e$ is contained in $X'$ it follows that
$y \notin \cl_{M / x}(Y')$ so $(X' \cup y,\, Y')$ is a
\vts\ of $M / x$, and hence $k = 3$ by~\ref{sub2}.
Assuming that $Y'$ is a small vertical $3$\dash separator of
$M / x$ implies that there is a triangle of $M$ containing $e$,
$x$, and an element of $Y'$, and hence $y$ is in parallel with
an element of $Y'$ in $M$.
Therefore $X' \cup y$ is a small vertical $3$\dash separator
of $M / x$.
But then $X' \cup x \cup y$ has rank at most four in $M$
and $X'$ has rank at most two in $M / x / y$, contradicting the
fact that $(X',\, Y')$ is a \vts.
Thus $M / x / y$ is \vfc.

It is easy to see that $M / x / y$ has exactly one loop.
Suppose that $\{e',\, f'\}$ is a parallel pair in $M / x / y$.
There is a circuit $C' \subseteq \{e',\, f',\, x,\, y\}$
that contains both $e'$ and $f'$.
If $|C'| = 4$ then $C' \subseteq T_{A} \cup T_{B}$, contrary
to assumption.
If $C' = \{e',\, f',\ x\}$, then $A$ is a five-element
fan, so there is a good element $x' \ne x$ in $A$.
Then $x'$ is contained in a triangle $T_{x'}$ of $M$ along
with $e$ and an element of $T_{B}$.
This again implies the existence of a circuit contained in
$T_{A} \cup T_{B}$.
We can obtain a similar contradiction if
$C' = \{e',\, f',\, y\}$.
This shows that there are no parallel pairs in
$M / x / y$.

Let $T_{1}$ be the triangle contained in $A$ that avoids $x$,
and $T_{2}$ be the triangle contained in $X$ that avoids
$y$.
Suppose that $T_{1}$ and $T_{2}$ have a non-empty intersection.
Then they meet in an element in neither $T_{A}$ nor
$T_{B}$, for otherwise one of these triangles meets either
$T_{A}$ or $T_{B}$ in one element in $M \del e$.
Now the symmetric difference of $T_{1}$ and $T_{2}$ is a circuit
contained in $T_{A} \cup T_{B}$, a contradiction.
Thus $T_{1}$ and $T_{2}$ are disjoint triangles.
Since there are no parallel pairs in $M / x / y$ this means
that $r_{M / x / y}(T_{1} \cup T_{2}) = 4$.
Hence statement~\eqref{state9} holds with $z = e$,

This exhausts the possible cases, so no counterexample to
Lemma~\ref{lem3} can exist.
\end{proof}

Theorem~\ref{thm5}, and hence Theorem~\ref{thm4},
follows from Lemma~\ref{lem3}.

\chapter{Proof of the main result}
\label{chp7}

In this chapter we prove Theorem~\ref{thm2}.
In Section~\ref{sct7.1} we assemble a collection of lemmas,
and in Section~\ref{sct7.2} we move to the proof of the
theorem.

\section{A cornucopia of lemmas}
\label{sct7.1}

\begin{lem}
\label{lem11}
Let $r \geq 4$ be an even integer.
If $M$ is a simple \vfc\ single-element extension of
$\Upsilon_{r}$ such that $M \in \ex{\mkt}$,
then $r = 6$ and $M$ is isomorphic to $T_{12}$.
\end{lem}

\begin{proof}
Let $M$ be a simple \vfc\ binary single-element extension of
$\Upsilon_{r}$ by the element $e$.
Consider the unique circuit $C$ of $M$ contained in $B \cup e$,
where $B$ is the basis $\{e_{1},\ldots, e_{r}\}$ of $M$.
If $e_{i} \notin C$ for some $i \in \{1,\ldots, r-1\}$
then $e \in \cl_{M}(B - e_{i})$, and therefore the triad
$\{c_{i-1},\, c_{i},\, e_{i}\}$ of $\Upsilon_{r}$ is also
a triad of $M$ (subscripts are to be read modulo $r - 1$).
Since $M$ is \vfc\ this contradicts Proposition~\ref{prop23}.
Therefore $C$ is equal to either $B$ or to
$B - e_{r}$, and $M$ can be represented over \gf{2} by
the set of vectors
$\{e_{1},\ldots, e_{r},\, c_{1},\ldots, c_{r-1},\, e\}$,
where, by an abuse of notation, $e$ stands for either
the column of all ones, or the column which has ones in every
position other than that corresponding to $e_{r}$.

Using a computer we can check the following facts:
If $r = 4$ then neither of the matroids represented by this set
of vectors is \vfc.
If $r = 6$ and $e$ is the column of all ones then the resulting
matroid has an \mkt\dash minor.
If $r = 6$ and $e$ is the column containing all ones except in its
last position then the resulting matroid is isomorphic to
$T_{12}$.\cross

Suppose that $r \geq 8$.
We prove by induction that the matroids represented over \gf{2} by
$\{e_{1},\ldots, e_{r},\, c_{1},\ldots, c_{r-1},\, e\}$
have \mkt\dash minors.
Clearly this will complete the proof of the lemma.
Here $e$ is either the sum of the vectors
$e_{1},\ldots, e_{r-1}$ or the sum of $e_{1},\ldots, e_{r}$.

A computer check shows that this claim is true when
$r = 8$.\cross\
Suppose that $r > 8$.
Let $M$ be the matroid represented over \gf{2} by
$\{e_{1},\ldots, e_{r},\, c_{1},\ldots, c_{r-1},\, e\}$.
It is elementary to check that
$M / c_{r-2} / c_{r-1} \del e_{r-2} \del e_{r-1}$ is isomorphic
to the matroid represented by
$\{e_{1},\ldots, e_{r-2},\, c_{1},\ldots, c_{r-3},\, e\}$,
where $e$ is the sum of either $e_{1},\ldots, e_{r-3}$ or
of $e_{1},\ldots, e_{r-2}$.
By the inductive assumption
$M / c_{r-2} / c_{r-1} \del e_{r-2} \del e_{r-1}$ has
an \mkt\dash minor, so we are done.
\end{proof}

\begin{lem}
\label{lem7}
Let $r \geq 4$ be an integer.
If $M$ is a $3$\dash connected binary single-element
extension of $\Delta_{r}$ such that $M$ has no \mkt\dash minor,
then $r = 4$ and $M$ is isomorphic to either
$C_{11}$ or $M_{4,11}$.
\end{lem}

\begin{proof}
The proof is by induction on $r$.
A computer check shows that up to isomorphism
$\Delta_{4}$ has only two $3$\dash connected single-element
extensions in \ex{\mkt}, and that these are isomorphic to
$C_{11}$ and $M_{4,11}$;
moreover, $\Delta_{5}$ has no $3$\dash connected
single-element extensions in \ex{\mkt}.\cross\
Suppose that $r > 5$ and that $\Delta_{r - 1}$ has
no $3$\dash connected single-element extensions in
\ex{\mkt}.
Assume that the lemma fails for $\Delta_{r}$.
Let $M$ be a $3$\dash connected binary extension of $\Delta_{r}$
by the element $e$ such that $M$ has no \mkt\dash minor.

Proposition~\ref{prop2} tells us that
$M / b_{1} \del a_{1} \del e_{1}$ is isomorphic to
a single-element extension of $\Delta_{r-1}$.
By the inductive hypothesis $M / b_{1} \del a_{1} \del e_{1}$
cannot be $3$\dash connected.
Thus $e$ is in a parallel pair in
$M / b_{1} \del a_{1} \del e_{1}$ and this means that $e$ is
in a triangle of $M$ with $b_{1}$.
Let $x_{1}$ be the element of this triangle that is not equal
to $b_{1}$ or $e$.
Note that $x_{1}$ cannot be equal to $a_{1}$ or $e_{1}$, for
then $e$ would be parallel to $a_{2}$ or $e_{2}$ in $M$.

Since $r > 5$ the sets
$\{a_{3},\, b_{2},\, b_{3},\, e_{3}\}$ and
$\{a_{r-1},\, b_{r-2},\, b_{r-1},\, e_{r-1}\}$
are disjoint.
We assume that $x_{1}$ is not in the first of these two sets.
The argument when $x_{1}$ is not in the second set is
identical.
Now $M / b_{2} \del a_{3} \del e_{3}$ is isomorphic to a
single-element extension of $\Delta_{r-1}$, so by using
the same argument as before we see that $e$ is in a
triangle of $M$ that contains $b_{2}$.
Let $x_{2}$ be the element in this triangle not equal to
$b_{2}$ or $e$.
It cannot be the case that $x_{2}$ is equal to
$a_{3}$ or $e_{3}$.
Nor can it be the case that $x_{1} = x_{2}$, for that
would imply that $b_{1}$ and $b_{2}$ are parallel.

Since $\{b_{1},\, b_{2},\, x_{1},\, x_{2}\}$ is the symmetric
difference of two triangles it is a circuit of
$M \del e = \Delta_{r}$.
But $\{a_{3},\, b_{2},\, b_{3},\, e_{3}\}$ is a cocircuit
of $\Delta_{r}$, and if
$x_{2} \notin \{a_{3},\, b_{2},\, b_{3},\, e_{3}\}$ then
$\{b_{1},\, b_{2},\, x_{1},\, x_{2}\}$ meets this cocircuit
in one element, a contradiction.
It follows that $x_{2}$ is equal to $b_{3}$.
The set $\{a_{1},\, b_{1},\, b_{r-1},\, e_{1}\}$ is also
a cocircuit in $\Delta_{r}$, and by using the same
argument we see that $x_{1}$ is equal to $b_{r-1}$.
Now $\{b_{1},\, b_{2},\, b_{3},\, b_{r-1}\}$ is a circuit,
which is a contradiction, as this set is independent
in $\Delta_{r}$.
\end{proof}

\begin{lem}
\label{lem13}
Let $r \geq 5$ be an integer.
Suppose that $M$ is a $3$\dash connected coextension of
$\Delta_{r}$ by the element $e$ such that $M \in \ex{\mkt}$.
Let $B$ be the basis $\{e_{1},\ldots, e_{r},\, e\}$ of $M$.
Let $C^{*}$ be the unique cocircuit of $M$ contained in
$(E(M) - B) \cup e$.
Then $C^{*} - e$ is equal to one of the following sets.
\begin{enumerate}[(i)]
\item A subset of two or three elements from the set
$\{a_{i},\, a_{i+1},\, b_{i}\}$ where $1 \leq i \leq r - 2$;
\item The set $\{a_{i+1},\, a_{i+2},\, b_{i},\, b_{i+2}\}$
where $1 \leq i \leq r - 3$, or the set
$\{a_{1},\, a_{2},\, b_{2},\, b_{r-1}\}$;
\item The set $\{a_{i+1},\, a_{i+2},\, b_{i},\, b_{i+1},\, b_{i+2}\}$
where $1 \leq i \leq r - 3$, or the set
$\{a_{1},\, a_{2},\, b_{1},\, b_{2},\, b_{r-1}\}$;
\item One of the sets $\{a_{1},\, b_{r-1}\}$ or
$\{a_{r-1},\, b_{r-1}\}$;
\item One of the sets $\{a_{1},\, a_{r-1},\, b_{1}\}$ or
$\{a_{1},\, a_{r-1},\, b_{r-2}\}$; or,
\item One of the sets $\{a_{1},\, a_{r-1},\, b_{1},\, b_{r-1}\}$ or
$\{a_{1},\, a_{r-1},\, b_{r-2},\, b_{r-1}\}$.
\end{enumerate}
\end{lem}

\begin{proof}
The proof is by induction on $r$.
To check the base case we need to generate all the
$3$\dash connected binary single-element coextensions of
$\Delta_{5}$ that have no \mkt\dash minor.
There are $24$ such coextensions (ignoring isomorphisms).
In each case one of the statements of the lemma holds.\cross

Suppose that $r > 5$ and that the lemma holds for
$\Delta_{r-1}$.
Let $M \in \ex{\mkt}$ be a $3$\dash connected coextension
of $\Delta_{r}$ by the element $e$ such that
$M \in \ex{\mkt}$.

For $1 \leq i \leq r - 2$ let $E_{i}$ be the set
$\{a_{i},\, a_{i+1},\, b_{i},\, e_{i},\, e_{i+1}\}$ and let
$E_{r - 1} = \{a_{1},\, a_{r-1},\, b_{r-1},\, e_{1},\, e_{r - 1}\}$.
We consider the fundamental graph $G_{B}(M)$.

\begin{subclm}
\label{clm13}
There is an integer $i \in \{1,\ldots, r - 1\}$ such that
$e$ is adjacent in $G_{B}(M)$ to at most one element
in $E_{i}$ and $e$ is not adjacent to $b_{i}$.
\end{subclm}

\begin{proof}
Recall from Proposition~\ref{prop2} that
$M_{0} = M / b_{r-2} \del a_{r-1} \del e_{r-1}$ is isomorphic to a
single-element coextension of $\Delta_{r-1}$.
Indeed, we will relabel the ground set of $M_{0}$ by giving $e_{r}$
the label $e_{r-1}$ and $b_{r-1}$ the label $b_{r-2}$; under this
relabeling $M_{0} / e = \Delta_{r-1}$.

Let $B_{0}$ be the basis $\{e_{1},\ldots, e_{r-1},\, e\}$ of
$M_{0}$.
Let $G'$ be the graph obtained from $G_{B}(M)$ by pivoting
on the edge $b_{r-2}e_{r-1}$.
Then the fundamental graph $G_{B_{0}}(M_{0})$ can be obtained from
$G'$ by deleting the vertices $a_{r-1}$, $b_{r-2}$, and
$e_{r-1}$, and then relabeling $e_{r}$ with $e_{r-1}$ and $b_{r-1}$
with $b_{r-2}$.

Suppose that $M_{0}$ is not $3$\dash connected.
Then $e$ is either a coloop in $M_{0}$ or is in series with some
element in $M_{0}$.
If $e$ is a coloop of $M_{0}$ then it is an isolated vertex in
$G_{B_{0}}(M_{0})$.
Thus $e$ is adjacent to at most two vertices in $G'$:
$a_{r-1}$ and $e_{r-1}$.
Since $G_{B}(M)$ can be obtained from $G'$ by pivoting on the
edge $b_{r-2}e_{r-1}$ it follows that $e$ is adjacent to at
most three vertices in $G_{B}(M)$: $a_{r-1}$, $b_{r-2}$, and
$b_{r-1}$.
The claim follows easily.

Next we suppose that $e$ is in a series pair in $M_{0}$.
In this case either: (i)~$e$ has degree one in
$G_{B_{0}}(M_{0})$, or (ii)~there is some element
$e' \in B_{0} - e$ such that $e$ and $e'$ are adjacent to
exactly the same vertices in $G_{B_{0}}(M_{0})$.
Suppose that case~(i) holds.
Then $e$ is adjacent to exactly one vertex $x$ in
$G_{B_{0}}(M_{0})$, and is therefore adjacent to at most
three vertices in $G'$: $x$, $a_{r-1}$, and $e_{r-1}$.
Thus $e$ is adjacent in $G_{B}(M)$ to at most four
vertices: $x$, $a_{r-1}$, $b_{r-1}$, and $b_{r-2}$.
The claim follows.

Suppose that case~(ii) hold.
If $e$ has exactly the same set of neighbors in
$G_{B_{0}}(M_{0})$ as the vertex $e_{i}$, where
$1 \leq i \leq r - 2$, then $e$ has degree three in
$G_{B_{0}}(M_{0})$.
Since $e_{i}$ is adjacent to vertices in exactly two sets
$E_{i-1}$ and $E_{i}$, by again pivoting on the edge
$b_{r-2}e_{r-1}$ in $G'$ we can see that the claim holds.

Next we assume that $e$ has exactly the same set of
neighbors in $G_{B_{0}}(M_{0})$ as $e_{r-1}$.
Thus $e$ and $e_{r-1}$ are in series in
$M_{0} = M / b_{r-2} \del a_{r-1} \del e_{r-1}$.
Thus $\{a_{r-1},\, e,\, e_{r-1},\, e_{r}\}$
contains a cocircuit $C^{*}$ in $M$ such that 
$e,\, e_{r} \in C^{*}$.
Since $M$ is $3$\dash connected $|C^{*}| \geq 3$.
Suppose that $|C^{*}| = 3$.
Then $M$ is a coextension of $\Delta_{r}$ by the element
$e$ such that $e$ is in a triad of $M$ with two elements
from the triangle $T = \{a_{r-1},\, e_{r-1},\, e_{r}\}$
of $\Delta_{r}$.
Corollary~\ref{cor5} implies that
$\Delta_{T}(\Delta_{r})$ is a minor of $M$.
But $\Delta_{T}(\Delta_{r})$ has an \mkt\dash minor by
Claim~\ref{clm11}, so we have a contradiction.

Suppose that $C^{*} = \{a_{r-1},\, e,\, e_{r-1},\, e_{r}\}$.
Then $C^{*}$ is also a circuit, for otherwise
$\{a_{r-1},\, e_{r-1},\, e_{r}\}$ is a circuit of
$M$ that meets $C^{*}$ in an odd number of elements.
Then the fundamental graph $G_{B}(M)$ is obtained
from $G_{B - e}(\Delta_{r})$ by adding the vertex $e$ and
making it adjacent to $a_{r-1}$ and every vertex that is
adjacent to exactly one of $e_{r-1}$ and $e_{r}$.
Now it is easy to confirm that $M \del a_{r-1}$ is isomorphic
to $\Delta_{T}(\Delta_{r})$, which has an \mkt\dash minor.
Thus we have a contradiction.

Therefore we must assume that $M_{0}$ is $3$\dash connected.
By the inductive assumption Lemma~\ref{lem13} holds for $M_{0}$.
If $C^{*}$ is the unique cocircuit contained in
$(E(M_{0}) - B_{0}) \cup e$ then the elements of $C^{*} - e$
are exactly the neighbors of $e$ in $G_{B_{0}}(M_{0})$.
Thus $e$ has degree at most five in $G_{B_{0}}(M_{0})$ and
degree at most eight in $G_{B}(M)$.
A straightforward case-check confirms that if the claim is false,
then $r$ is equal to six, and $C^{*} - e$ is the
union of
$\{a_{2},\, a_{3},\, b_{1},\, b_{3},\, b_{4},\, b_{5}\}$
with some subset of $\{a_{5},\, b_{2}\}$.
We can eliminate this case by considering the four
corresponding coextensions of $\Delta_{6}$.
They all have \mkt\dash minors, so this completes the
proof of the claim.\cross
\end{proof}

Let $i$ be the integer supplied by Claim~\ref{clm13}.
We will assume that $1 \leq i \leq r - 2$, for the proof
when $i = r - 1$ is similar.
It follows from Claim~\ref{clm13} that either $e$ is
adjacent in $G_{B}(M)$ to no vertex in
$\{a_{i},\, b_{i},\, e_{i}\}$, or $e$ is adjacent to no vertex
in $\{a_{i+1},\, b_{i},\, e_{i+1}\}$.
We will assume the former; the argument in the latter case
is similar.

Consider $M_{1} = M / b_{i} \del a_{i} \del e_{i}$.
We relabel the ground set of $M_{1}$ so that every
element in $M_{1}$ with an index $j > i$ is relabeled
with the index $j - 1$.
Then $M_{1} / e = \Delta_{r-1}$.
Let $B_{1} = \{e_{1},\ldots, e_{r-1},\, e\}$ and let $G''$
be the fundamental graph derived from $G_{B}(M)$ by
pivoting on the edge $b_{i}e_{i}$.
Then $G_{B_{1}}(M_{1})$ is obtained from $G''$ by deleting
$a_{i}$, $b_{i}$, and $e_{i}$, and then doing the appropriate
relabeling.

Since $e$ is not adjacent to any of the vertices in
$\{a_{i},\, b_{i},\, e_{i}\}$ in $G_{B}(M)$ it is
not adjacent to any of them in $G''$.
From this observation it follows that if $e$ is a coloop
or in a series pair in $M_{1}$ then it is in a coloop
or series pair in $M$, a contradiction.
Thus $M_{1}$ is $3$\dash connected.

The inductive hypothesis implies that the lemma holds for
$M_{1}$.
Therefore in $G_{B_{1}}(M_{1})$ the vertex $e$ is adjacent
to vertices with at most three different indices.
Now the only case it which it is not immediate that the
lemma holds for $M$ is the case in which $e$ is adjacent
to both at least one vertex with the index $i - 1$, and
at least one with the index $i + 1$.
It is not difficult to see that if this is the case then
the hypotheses of Claim~\ref{clm13} apply to either
$E_{i-2}$ or $E_{i+2}$.
Assuming the former, either
$M / b_{i-2} \del a_{i-2} \del e_{i-2}$ or
$M / b_{i-2} \del a_{i-1} \del e_{i-1}$ provides
a contradiction to the inductive hypothesis.
The second case is similar.
\end{proof}

Recall that if $M \in \ex{\mkt}$, then an allowable
triangle of $M$ is a triangle $T$ such that
$\Delta_{T}(M)$ has no \mkt\dash minor.
Claim~\ref{clm11} tells us that the allowable
triangles in $\Delta_{r}$ are exactly those triangles
that contain spoke elements.

\begin{cor}
\label{cor7}
Suppose that $r \geq 4$ is an integer.
If $M$ is a $3$\dash connected coextension of $\Delta_{r}$
by the element $e$ such that $M \in \ex{\mkt}$ then either:
\begin{enumerate}[(i)]
\item $r = 4$ and $M$ is isomorphic to $M_{5,11}$; or,
\item there exists an allowable triangle $T$ of $\Delta_{r}$ such
that either $T \cup e$ is a four-element circuit-cocircuit of
$M$, or $e$ is in a triad of $M$ with two elements of $T$.
\end{enumerate}
\end{cor}

\begin{proof}
To prove the lemma when $r = 4$ we consider all the
$3$\dash connected single-element coextensions of $\Delta_{4}$
that belong to \ex{\mkt}.
There are $21$ such coextensions (ignoring isomorphisms).
The lemma holds for each of these.\cross

When $r \geq 5$ we apply Lemma~\ref{lem13}.
The set $\{a_{i},\, a_{i+1},\, b_{i}\}$ is a triangle
of $\Delta_{r}$ for $1 \leq i \leq r - 2$, so the lemma
holds if statement~(i) of Lemma~\ref{lem13}
applies.
If statement~(ii) holds then $\{e,\, e_{i},\, e_{i+1}\}$
is a triad of $M$ for some value of
$i \in \{1,\ldots, r - 2\}$.
Thus the lemma holds.
Similarly, if statement~(iii) holds, then
$\{e,\, b_{i},\,e_{i},\, e_{i+1}\}$ is a four-element
circuit-cocircuit for some value of $i$.
The sets $\{a_{1},\, b_{r-1},\, e_{r-1}\}$ and
$\{a_{r-1},\, b_{r-1},\, e_{1}\}$ are triangles
of $\Delta_{r}$, so the lemma holds if
statement~(iv) applies.
If statement~(v) holds then either
$\{e,\, a_{1},\, b_{r-1},\, e_{r-1}\}$ or
$\{e,\, a_{r-1},\, b_{r-1},\, e_{1}\}$ is a four-element
circuit-cocircuit.
In case~(vi) either $\{e,\, a_{1},\, e_{r-1}\}$ or
$\{e,\, a_{r-1},\, e_{1}\}$ is a triad.
\end{proof}

\begin{lem}
\label{lem14}
Let $T_{0}$ be a triangle of $\Delta_{4}$ containing the spoke
element $b$.
Suppose that the binary matroid $M_{1}$ is a coextension
of $\Delta_{4}$ by the element $e$ such that $e$ is in a
triad $T_{1}$ of $M_{1}$ with two elements of $T_{0}$.
Suppose that the binary matroid $M_{2}$ is an extension of
$M_{1}$ by the element $f$ so that there is a triangle $T_{2}$
of $M_{2}$ that contains $e$ and $f$.
Then either:
\begin{enumerate}[(i)]
\item $M_{2}$ has an \mkt\dash minor;
\item $M_{2}$ is isomorphic to $M_{5,12}^{a}$ or $M_{5,12}^{b}$;
\item $T_{2} \subseteq T_{0} \cup \{e,\, f\}$; or,
\item $T_{1}$ contains $b$, and the single element in
$T_{2} - \{e,\, f\}$ comes from the unique allowable triangle
$T$ of $\Delta_{4}$ such that $T \cup (T_{1} - e)$ contains a
cocircuit of size four in $\Delta_{4}$.
\end{enumerate}
\end{lem}

\begin{proof}
Because the automorphisms of $\Delta_{4}$ act transitively upon
the allowable triangles we can assume that
$T_{0} = \{a_{1},\, a_{2},\, b_{1}\}$, so that $b = b_{1}$.
First suppose that $T_{1} = \{e,\, a_{1},\, a_{2}\}$.
Let $x$ be the single element in $T_{2} - \{e,\, f\}$.
If $x \in T_{0}$ then statement~(iii) holds, so we assume that
$x \notin T_{0}$.
If $x$ is equal to $b_{2}$, $b_{3}$, $e_{1}$, or $e_{2}$, then
$M_{2}$ has an \mkt\dash minor.
If $x = e_{4}$ then $M_{2}$ is isomorphic to $M_{5,12}^{a}$
and if $x$ is equal to $a_{3}$ or $e_{3}$ then $M_{2}$ is isomorphic
to $M_{5,12}^{b}$.\cross\
Thus the result holds in this case.

Next we consider the case that
$T_{1} = \{e,\, a_{1},\, b_{1}\}$ is a triad of $M_{1}$.
We again let $x$ be the single element in $T_{2} - \{e,\, f\}$.
If $x \in T_{0}$ then we are done, so we assume that
$x \notin T_{0}$.
The unique triangle $T$ of $\Delta_{4}$ such that
$T \cup \{a_{1},\, b_{1}\}$ contains a cocircuit of size four
is $\{a_{3},\, b_{3},\, e_{1}\}$, so if $x$ is one of these elements
statement~(iv) holds.
Thus we need only check the case that
$x \in \{b_{2},\, e_{2},\, e_{3},\, e_{4}\}$.
If $x$ is equal to $b_{2}$, $e_{2}$, or $e_{3}$ then $M_{2}$
has an \mkt\dash minor.
If $x = e_{4}$ then $M_{2}$ is isomorphic to
$M_{5,12}^{a}$.\cross

The final case is that in which $\{e,\, a_{2},\, b_{1}\}$
is a triad of $M_{1}$.
We can use the symmetries of $\Delta_{4}$ to show that the
result also holds in this case.
\end{proof}

\begin{lem}
\label{lem15}
Suppose that $r \geq 5$ is an integer.
Let $T_{0}$ be a triangle of $\Delta_{r}$ containing the
spoke element $b$.
Suppose that the binary matroid $M_{1}$ is a coextension
of $\Delta_{r}$ by the element $e$ such that $e$ is in a
triad $T_{1}$ of $M_{1}$ with two elements of $T_{0}$.
Suppose that the binary matroid $M_{2}$ is an extension of
$M_{1}$ by the element $f$ so that there is a triangle $T_{2}$
of $M_{2}$ that contains $e$ and $f$.
Then either:
\begin{enumerate}[(i)]
\item $M_{2}$ has an \mkt\dash minor;
\item $T_{2} \subseteq T_{0} \cup \{e,\, f\}$; or,
\item $T_{1}$ contains $b$, and the single element in
$T_{2} - \{e,\, f\}$ comes from the unique allowable triangle
$T$ of $\Delta_{r}$ such that $T \cup (T_{1} - e)$ contains a
cocircuit of size four in $\Delta_{r}$.
\end{enumerate}
\end{lem}

\begin{proof}
Assume that the lemma fails, and that $r \geq 5$ is the
smallest value for which there is a counterexample $M_{2}$.
Suppose that $r = 5$.
Since the automorphism group of $\Delta_{5}$ is transitive
on allowable triangles we will assume that
$T_{0} = \{a_{1},\, a_{2},\, b_{1}\}$, so that $b = b_{1}$.
Suppose that $T_{1} = \{e,\, a_{1},\, a_{2}\}$.
Let $x$ be the single element in $T_{2} - \{e,\, f\}$.
It is easily checked that if $x \notin T_{0}$ then
$M_{2}$ has an \mkt\dash minor.\cross

Suppose that $T_{1} = \{e,\, a_{1},\, b_{1}\}$.
Again let $x$ be the single element in $T_{2} - \{e,\, f\}$.
The unique allowable triangle $T$ of $\Delta_{5}$ such
that $T \cup \{a_{1},\, b_{1}\}$ contains a cocircuit of
size four is $\{a_{4},\, b_{4},\, e_{1}\}$.
If $x \notin T_{0}$ and $x \notin T$ then $M_{2}$ has an
\mkt\dash minor.\cross\
We can see by symmetry that the result also holds if
$T_{1} = \{e,\, a_{2},\, b_{1}\}$.
Thus the lemma holds when $r = 5$, so we must assume that
$r > 5$.

Again we assume that $T_{0} = \{a_{1},\, a_{2},\, b_{1}\}$.
Let $x$ be the element in $T_{2} - \{e,\, f\}$.
Suppose that $x \in \{a_{r-2},\, e_{r-2}\}$.
Consider the minor
$N_{1} = M_{2} / b_{r-2} \del a_{r-1} \del e_{r-1}$
and relabel the ground set of $N_{1}$ so that 
$b_{r-1}$ receives the label $b_{r-2}$ and $e_{r}$ receives
the label $e_{r-1}$.
Proposition~\ref{prop2} implies that
$N_{1} / e \del f = \Delta_{r-1}$.
Since $a_{r-1}$ and $e_{r-1}$ are in
$\cl_{M_{2}}(E(M_{2}) - (T_{1} \cup \{a_{r-1},\, e_{r-1}\}))$
it follows that $T_{1}$ is a triad in $N_{1}$.
Moreover $b_{r-2}$ cannot be parallel to any element in $T_{2}$,
so $T_{2}$ is a triangle in $N_{1}$.
Therefore the lemma holds for $N_{1}$.

If $N_{1}$ has an \mkt\dash minor, then so does $M_{2}$,
a contradiction.
Thus statement~(i) does not apply to $N_{1}$.
It is easy to see that if statement~(ii) holds for $N_{1}$
then it also holds for $M_{2}$, a contradiction.
The only way that statement~(iii) can hold is if
$T_{1} - e = \{a_{1},\, b_{1}\}$ and $x = a_{r-2}$, for the 
unique triangle $T$ of $N_{1}$ specified in statement~(iii)
consists of the elements $\{a_{r-2},\, b_{r-2},\, e_{1}\}$.
Let us assume that this is the case.
Then we consider the minor
$N_{2} = M_{2} / b_{r-3} \del a_{r-3} \del e_{r-3}$ and
relabel the ground set of $N_{2}$ so that any element with
the index $j \geq r - 2$ is relabeled with the index $j - 1$.
Then $N_{2} / e \del f = \Delta_{r-1}$ and we again
conclude that the lemma holds for $N_{2}$.
It is not difficult to demonstrate the neither statement~(ii)
nor~(iii) holds, so $N_{2}$ has an \mkt\dash minor.
Thus $M_{2}$ has an \mkt\dash minor, a contradiction.

Therefore we will assume that $x \notin \{a_{r-2},\, e_{r-2}\}$.
Thus $x$ is in at most one of the sets
$\{a_{r-2},\, b_{r-3},\, e_{r-2}\}$ and
$\{a_{r-2},\, b_{r-2},\, e_{r-2}\}$.
Let us assume that $x$ is not contained in the second of these
sets (the argument is similar in either case).
Let $N_{3}$ be the minor
$M_{2} / b_{r-2} \del a_{r-2} \del e_{r-2}$, relabeled so that
any element with the index $j \geq r - 1$ receives the
index $j - 1$.
Then $N_{3} / e \del f = \Delta_{r-1}$ and the lemma holds for
$N_{3}$.

If $N_{3}$ has an \mkt\dash minor then we are done.
Similarly, if $x \in T_{0}$ in $N_{3}$ then the lemma
holds.
Thus we will assume that statement~(iii) of the lemma
holds in $N_{3}$.
Then either $x \in \{b_{2},\, e_{2},\, e_{3}\}$, or
$x \in \{a_{r-2},\, b_{r-1},\, e_{1}\}$.
In either case we can see that statement~(iii) also holds in
$M_{2}$.
This contradiction completes the proof.
\end{proof}

\begin{cor}
\label{cor6}
Let $r \geq 4$ be an integer.
Suppose that $T_{0}$ is a triangle of $\Delta_{r}$ that
contains a spoke element.
Let $M_{1}$ be a binary coextension of $\Delta_{r}$ by the
element $e$ such that $e$ is in a triad $T_{1}$ of $M_{1}$ with
two elements of $T_{0}$.
Suppose that $M_{2}$ is an extension of $M_{1}$ by the element
$f$ such that there is a triangle $T_{2}$ of $M_{2}$ that contains
$e$ and $f$.
If $M_{2}$ is \vfc\ then either it contains a minor isomorphic
to \mkt, or $r = 4$ and $M_{2}$ is isomorphic to $M_{5,12}^{a}$
or $M_{5,12}^{b}$.
\end{cor}

\begin{proof}
Suppose that $M_{2}$ is a counterexample to the corollary.
Thus $M_{2}$ is \vfc.
We apply Lemmas~\ref{lem14} and~\ref{lem15}.
It is easy to demonstrate that if
$T_{2} \subseteq T_{0} \cup \{e,\, f\}$ then $M_{2}$ is
not \vfc.
Therefore there is a triangle $T$ of $\Delta_{r}$ such that
$T \cup (T_{1} - e)$ contains a four-element cocircuit
in $\Delta_{r}$, and the single element in
$T_{2} - \{e,\, f\}$ is contained in $T$.
We will show that $T \cup T_{2}$ is a vertical
$3$\dash separator of $M_{2}$, and this will provide a
contradiction that completes the proof.
Now $T$ is a triangle of $M_{1}$, for otherwise $T \cup e$ is
a circuit that meets the cocircuit $T_{1}$ in one element.
Thus $T$ is a triangle in $M_{2}$, and the rank of
$T \cup T_{2}$ in $M_{2}$ is three.
Moreover, because $T \cup (T_{1} - e)$ contain a cocircuit
of size four in $\Delta_{r}$, and $T_{1}$ is a triad in
$M_{1}$, by the properties of cocircuits in binary matroids
it follows that $T \cup e$ contains a triad in $M_{1}$.
Thus $T \cup \{e,\, f\}$ contains a cocircuit in $M_{2}$,
so the complement of $T \cup T_{2}$ has rank at most
$r(M_{2}) - 1$.
The result follows.
\end{proof}

\begin{lem}
\label{lem16}
Suppose that $T_{0}$ is a triangle of $\Delta_{4}$
containing a spoke element.
Let $M_{1}$ be a binary coextension of $\Delta_{4}$ by the
element $e$ such that $e$ is contained in a triad
$T_{1}$ of $M_{1}$, where $T_{1}$ contains two elements
of $T_{0}$.
Suppose that $M_{2}$ is a binary extension of $M_{1}$ by the
elements $f$ and $g$ such that $M_{2}$ contains triangles
$T_{f}$ and $T_{g}$, where $\{e,\, f\} \subseteq T_{f}$ and
$\{e,\, g\} \subseteq T_{g}$.
If $M_{2}$ is \vfc\ and contains no minor isomorphic to
\mkt, $M_{5,12}^{a}$, or $M_{5,12}^{b}$, then
$M_{2}$ is isomorphic to $\Delta_{5}$ or $M_{5,13}$.
\end{lem}

\begin{proof}
By the symmetries of $\Delta_{4}$ we can assume that
$T_{0} = \{a_{1},\, a_{2},\, b_{1}\}$.
If $T_{1} = \{e,\, a_{1},\, a_{2}\}$ then
Lemma~\ref{lem14} implies that
$T_{f} \subseteq T_{0} \cup \{e,\, f\}$ and
$T_{g} \subseteq T_{0} \cup \{e,\, g\}$.
In this case it is easy to see that $T_{0} \cup \{e,\, f,\, g\}$
is a vertical $3$\dash separator of $M_{2}$.
Therefore, by the symmetries of $\Delta_{4}$, we can assume
that $T_{1} = \{e,\, a_{1},\, b_{1}\}$.
Let $x_{f}$ and $x_{g}$ be the single elements in
$T_{f} - \{e,\, f\}$ and $T_{g} - \{e,\, g\}$
respectively.
By Lemma~\ref{lem14} the elements $x_{f}$ and $x_{g}$ are
contained in $T_{0}$ or in $\{a_{3},\, b_{3},\, e_{1}\}$.
If both are contained in $T_{0}$ then $T_{0} \cup \{e,\, f,\, g\}$
is a vertical $3$\dash separator of $M_{2}$.
Similarly, if both $x_{f}$ and $x_{g}$ are contained
in $\{a_{3},\, b_{3},\, e_{1}\}$, then, as in the proof
of Corollary~\ref{cor6}, the set
$\{e,\, f,\, g,\, a_{3},\, b_{3},\, e_{1}\}$ is a
vertical $3$\dash separator.

Thus we will assume that $x_{f} \in T_{0}$, while
$x_{g} \in \{a_{3},\, b_{3},\, e_{1}\}$.
A computer check verifies that 
if $x_{f} = a_{1}$ while $x_{g} = e_{1}$ then
$M_{2}$ is isomorphic to $\Delta_{5}$.
If $x_{f} = b_{1}$ and $x_{g} = b_{3}$ then $M_{2}$ is
isomorphic to $M_{5,13}$.
In each of the seven remaining cases $M_{2}$ has an
\mkt\dash minor.\cross
\end{proof}

\begin{lem}
\label{lem17}
Let $r \geq 5$ be an integer.
Suppose that $T_{0}$ is a triangle of $\Delta_{r}$
containing a spoke element.
Let $M_{1}$ be a binary coextension of $\Delta_{r}$ by the
element $e$ such that $e$ is contained in a triad
$T_{1}$ of $M_{1}$, where $T_{1}$ contains two elements
of $T_{0}$.
Suppose that $M_{2}$ is a binary extension of $M_{1}$ by the
elements $f$ and $g$ such that $M_{2}$ contains triangles
$T_{f}$ and $T_{g}$, where $\{e,\, f\} \subseteq T_{f}$ and
$\{e,\, g\} \subseteq T_{g}$.
If $M_{2}$ is \vfc\ and has no \mkt\dash minor then
$M_{2}$ is isomorphic to $\Delta_{r+1}$.
\end{lem}

\begin{proof}
We suppose that $M_{2}$ is \vfc\ with no \mkt\dash minor.
We can assume that $T_{0} = \{a_{1},\, a_{2},\, b_{1}\}$.
If $T_{1} = \{e,\, a_{1},\, a_{2}\}$ then Lemma~\ref{lem15}
implies that $T_{f} \subseteq T_{0} \cup \{e,\, f\}$ and
$T_{g} \subseteq T_{0} \cup \{e,\, g\}$.
In this case $T_{0} \cup \{e,\, f,\, g\}$ is a
vertical $3$\dash separator of $M_{2}$, so we will assume that
$T_{1} = \{e,\, a_{1},\, b_{1}\}$.
Let $x_{f}$ and $x_{g}$ be the single elements in
$T_{f} - \{e,\, f\}$ and $T_{g} - \{e,\, g\}$ respectively.
By Lemma~\ref{lem15} the elements $x_{f}$ and $x_{g}$
are contained in either $T_{0}$ or in
$\{a_{r-1},\, b_{r-1},\, e_{1}\}$.
As before, if both are contained in either $T_{0}$ or
$\{a_{r-1},\, b_{r-1},\, e_{1}\}$ then $M_{2}$ is not \vfc.
Thus we assume that $x_{f} \in T_{0}$ and
$x_{g} \in \{a_{r-1},\, b_{r-1},\, e_{1}\}$.

\begin{subclm}
\label{clm15}
$x_{f} = a_{1}$ and $x_{g} = e_{1}$.
\end{subclm}

\begin{proof}
Suppose $r = 5$.
There are eight cases in which the claim is false.
In each of these $M_{2}$ has an \mkt\dash minor.\cross\
This provides the base case for an inductive argument.

Suppose that $r > 5$ and that the claim holds for $r - 1$.
Consider $M_{2}' = M_{2} / b_{r-2} \del a_{r-2} \del e_{r-2}$.
Then $M_{2}' / e \del f \del g$ is isomorphic to $\Delta_{r-1}$
under the relabeling that reduces by one the index of any
element with an index that exceeds $r - 2$.
It is easy to see that $T_{1}$ is a triad of $M_{2}'$
and both $T_{f}$ and $T_{g}$ are triangles.
Thus we can apply the inductive hypothesis, and the claim
follows.
\end{proof}

It remains to show that $M_{2}$ is isomorphic to $\Delta_{r+1}$.
This can be accomplished by considering the basis graph
$G_{B_{2}}(M_{2})$, where $B_{2} = \{e_{1},\ldots, e_{r},\, e\}$.
This graph can be obtained from the basis graph $G_{B}(\Delta_{r})$
(where $B = \{e_{1},\ldots, e_{r}\}$) by adding the vertex
$e$ so that it is adjacent to $a_{1}$ and $b_{1}$, the vertex
$f$ so that it is adjacent to $e_{1}$ and $e_{r}$, and the
vertex $g$ so that it is adjacent to $e$ and $e_{1}$.
Now consider the graph obtained from $G_{B_{2}}(M_{2})$ by
pivoting on the edges $ea_{1}$, $fe_{1}$, and $ga_{1}$.
This produces a graph isomorphic to $G_{B'}(\Delta_{r+1})$
(where $B' = \{e_{1},\ldots, e_{r+1}\}$) under the relabeling
that takes $e$ to $b_{r}$, $f$ to $e_{r}$, $g$ to $e_{1}$,
and $e_{r}$ to $e_{r+1}$.
Thus $M_{2}$ is isomorphic to $\Delta_{r+1}$ and the lemma is
true.
\end{proof}

\begin{lem}
\label{lem20}
Let $T_{1}$ and $T_{2}$ be allowable triangles of $\Delta_{4}$
such that $r_{\Delta_{4}}(T_{1} \cup T_{2}) = 4$.
Suppose that $M_{1}$ is the binary matroid obtained from
$\Delta_{4}$ by coextending with the elements $e$ and $f$
so that in $M_{1}$ the element $e$ is in a triad with two
elements from $T_{1}$, and $f$ is in a triad with two
elements from $T_{2}$.
Let $M_{2}$ be the binary matroid obtained from $M_{1}$
by extending with the element $g$ so that
$\{e,\, f,\, g\}$ is a triangle of $M_{2}$.
If $M_{2}$ is \vfc\ then it is isomorphic to $M_{6,13}$.
\end{lem}

\begin{proof}
By the symmetries of $\Delta_{4}$ we can assume that
$T_{1} = \{b_{1},\, e_{1},\, e_{2}\}$ while
$T_{2} = \{a_{2},\, a_{3},\, b_{2}\}$.
The element $e$ may be in a triad of $M_{1}$ with any two
elements of $T_{1}$, and $f$ may be in a triad with any
two elements of $T_{2}$.
It follows that there are nine cases to check.
If $\{e,\, b_{1},\, e_{2}\}$ and
$\{f,\, a_{2},\, b_{2}\}$ are triads of $M_{1}$ then
$e$ and $f$ are in series in $M_{1}$.
In this case $M_{2}$ is not \vfc.
If $\{e,\, e_{1},\, e_{2}\}$ and
$\{f,\, a_{2},\, a_{3}\}$ are triads of $M_{1}$, then
$M_{2}$ is isomorphic to $M_{6,13}$.
A computer check reveals that in the seven remaining
cases $M_{2}$ is not \vfc.\cross
\end{proof}

\begin{lem}
\label{lem21}
Let $r \geq 5$ be an integer.
Suppose that $T_{1}$ and $T_{2}$ are disjoint allowable
triangles of $\Delta_{r}$ such that either $b_{i} \in T_{1}$
and $b_{i+1} \in T_{2}$ for some $i \in \{1,\ldots, r-2\}$, or
$b_{r-1} \in T_{1}$ and $b_{1} \in T_{2}$.
Suppose that $M_{1}$ is the binary matroid obtained from
$\Delta_{r}$ by coextending with the elements $e$ and $f$
so that in $M_{1}$ the element $e$ is in a triad $T_{e}$ with two
elements from $T_{1}$, and $f$ is in a triad $T_{f}$ with two
elements from $T_{2}$.
Let $M_{2}$ be the binary matroid obtained from $M_{1}$
by extending with the element $g$ so that
$\{e,\, f,\, g\}$ is a triangle of $M_{2}$.
If $M_{2}$ is \vfc\ then it has an \mkt\dash minor.
\end{lem}

\begin{proof}
We will assume that $T_{1} = \{e_{1},\, e_{2},\, b_{1}\}$
and that $T_{2} = \{a_{2},\, a_{3},\, b_{2}\}$.

We first suppose that $T_{e} = \{e,\, b_{1},\, e_{2}\}$.
We claim that $T_{2} \cup \{e,\, f,\, g\}$ is a vertical
$3$\dash separator of $M_{2}$.
Note that $r_{M_{2}}(T_{2} \cup \{e,\, f,\, g\}) = 4$ and
$T_{f} \cup g$ is a cocircuit in $M_{2}$.
If we can show that $T_{2} \cup \{e,\, f,\, g\}$ contains a
cocircuit distinct from $T_{f} \cup g$ then we will
have shown that the complement of $T_{2} \cup \{e,\, f,\, g\}$
has rank at most $r(M_{2}) - 2$, and this will prove the
claim.
Since $C^{*} = \{a_{2},\, b_{1},\, b_{2},\, e_{2}\}$
is a cocircuit of $M_{1}$, and $g$ is in the closure
of the complement of $C^{*}$ in $M_{2}$ it follows that
$C^{*}$ is also a cocircuit of $M_{2}$.
Since $T_{e} \cup g$ is a cocircuit of $M_{2}$ the
symmetric difference of $C^{*}$ with $T_{e} \cup g$
contains a cocircuit in $M_{2}$.
This symmetric difference is $\{e,\, g,\, a_{2},\, b_{2}\}$.
Any cocircuit contained in this set is different from
$T_{f} \cup g$, since it cannot contain $f$.
This proves the claim.

If $T_{f} = \{f,\, a_{2},\, b_{2}\}$ then we can give a
symmetric argument and show that $T_{1} \cup \{e,\, f,\, g\}$
is a vertical $3$\dash separator of $M_{2}$.
Thus there are only four cases left to consider:
$T_{e}$ is equal to either
$\{e,\, b_{1},\, e_{1}\}$ or $\{e,\, e_{1},\, e_{2}\}$,
and $T_{f}$ is equal to either
$\{f,\, a_{2},\, a_{3}\}$ or $\{f,\, a_{3},\, b_{2}\}$.

We prove by induction on $r$ that in any of these four
cases $M_{2}$ has an \mkt\dash minor.
A computer check confirms that this is the case when $r = 5$.\cross\
Assume that claim holds for $r - 1$.
Let $M' = M_{2} / b_{r-1} \del a_{r-1} \del e_{r-1}$.
Then $M' / e / f \del g$ is equal to
$\Delta_{r-1}$ once $e_{r}$ is relabeled as $e_{r-1}$.
Moreover it is easy to confirm that $T_{e}$ and $T_{f}$
are triads in $M' \del g$, and that
$\{e,\, f,\, g\}$ is a triangle of $M'$.
Thus $M$ has an \mkt\dash minor by the inductive
hypothesis.
\end{proof}

\begin{lem}
\label{lem22}
Let $r \geq 5$ be an integer.
Suppose that $T_{1}$ and $T_{2}$ are disjoint allowable
triangles of $\Delta_{r}$.
Suppose that $M_{1}$ is the binary matroid obtained from
$\Delta_{r}$ by coextending with the elements $e$ and $f$
so that in $M_{1}$ the element $e$ is in a triad $T_{e}$ with two
elements from $T_{1}$, and $f$ is in a triad $T_{f}$ with two
elements from $T_{2}$.
Let $M_{2}$ be the binary matroid obtained from $M_{1}$
by extending with the element $g$ so that
$\{e,\, f,\, g\}$ is a triangle of $M_{2}$.
If $M_{2}$ is \vfc\ then it has an \mkt\dash minor.
\end{lem}

\begin{proof}
Since the automorphism group of $\Delta_{r}$ is
transitive on the allowable triangles we assume that
$T_{1} = \{b_{1},\, e_{1},\, e_{2}\}$.
Lemma~\ref{lem21} and symmetry show that the lemma
is true when $b_{2} \in T_{2}$ or $b_{r-1} \in T_{2}$,
so henceforth we will assume that $T_{2}$ contains
neither $b_{2}$ nor $b_{r-1}$.

\begin{subclm}
\label{clm9}
If $M_{2}$ does not have an \mkt\dash minor, then
either:
\begin{enumerate}[(i)]
\item $T_{2} = \{b_{3},\, e_{3},\, e_{4}\}$,
$T_{e} = \{e,\, b_{1},\, e_{2}\}$,
$T_{f} = \{f,\, b_{3},\, e_{3}\}$; or,
\item $T_{2} = \{a_{r-2},\, a_{r-1},\, b_{r-2}\}$,
$T_{e} = \{e,\, b_{1},\, e_{1}\}$,
$T_{f} = \{f,\, a_{r-1},\, b_{r-2}\}$.
\end{enumerate}
\end{subclm}

\begin{proof}
We start by confirming that the claim holds when
$r = 5$.
Since $T_{2}$ does not contain $b_{2}$ or $b_{3}$ it
follows that $T_{2}$ can only be $\{a_{3},\, a_{4},\, b_{3}\}$
or $\{b_{3},\, e_{3},\, e_{4}\}$.
In either case $e$ can be in a triad of $M_{1}$ with any two
elements of $T_{1}$ and $f$ can be in a triad with any two
elements of $T_{2}$.
Thus there are $18$ cases to check.
A computer check reveals that in only two of these cases does
$M_{2}$ not have an \mkt\dash minor:
$T_{2} = \{b_{3},\, e_{3},\, e_{4}\}$ and
$T_{e} = \{e,\, b_{1},\, e_{2}\}$ while
$T_{f} = \{f,\, b_{3},\, e_{3}\}$; or
$T_{2} = \{a_{3},\, a_{4},\, b_{3}\}$ and
$T_{e} = \{e,\, b_{1},\, e_{1}\}$ while
$T_{f} = \{f,\, a_{4},\, b_{3}\}$.\cross

Let us assume the claim is false, and let $r$ be
the smallest value for which it fails.
The discussion above shows that $r > 5$.
Suppose that $M_{2}$ is constructed from $\Delta_{r}$
as described in the hypotheses of the lemma and that
$M_{2}$ provides a counterexample to the claim.

We first suppose that $b_{r-2} \in T_{2}$.
Then either $T_{2} = \{a_{r-2},\, a_{r-1},\, b_{r-2}\}$
or $\{b_{r-2},\ e_{r-2},\, e_{r-1}\}$.
We consider $M' = M_{2} / b_{2} \del a_{3} \del e_{3}$.
Then $M' / e / f \del g$ is equal to $\Delta_{r-1}$
after all the elements with index $j > 3$ are relabeled
with the index $j - 1$.
Furthermore $\{e,\, f,\, g\}$ is a triangle in
$M'$, and $T_{e}$ and $T_{f}$ are triads in $M' \del g$.
Since $M_{2}$ has no \mkt\dash minor, it follows that
$M'$ has no \mkt\dash minor.
By the inductive hypothesis there are two cases to
consider:
In the first case $T_{2} = \{a_{r-2},\, a_{r-1},\, b_{r-2}\}$
and $T_{e} = \{e,\, b_{1},\, e_{1}\}$ while
$T_{f} = \{f,\, a_{r-1},\, b_{r-2}\}$.
But this cannot occur since then $M_{2}$ would not be
a counterexample to the claim.
Thus we consider the other case: $r = 6$, and
$T_{2} = \{b_{4},\, e_{4},\, e_{5}\}$, while
$T_{e} = \{e,\, b_{1},\, e_{2}\}$ and
$T_{f} = \{f,\, b_{4},\, e_{4}\}$.
But a computer check shows that if this is the case then
$M_{2}$ has an \mkt\dash minor.\cross\
Thus we conclude that $T_{2}$ does not contain
$b_{r-2}$.

We consider the minor
$M' = M_{2} / b_{r-1} \del a_{r-1} \del e_{r-1}$.
Then $M' / e / f \del g$ is equal to $\Delta_{r-1}$
after $e_{r}$ is relabeled with $e_{r-1}$.
We can apply the inductive hypothesis to $M'$.
Since $M'$ does not have an \mkt\dash minor we are forced
to consider two possible cases:
In the first $T_{2} = \{b_{3},\, e_{3},\, e_{4}\}$, and
$T_{e} = \{e,\, b_{1},\, e_{2}\}$ while
$T_{f} = \{f,\, b_{3},\, e_{3}\}$.
However in this case $M_{2}$ is not a counterexample
to the claim.
Therefore we assume that the other case holds:
$T_{2} = \{a_{r-3},\, a_{r-2},\, b_{r-3}\}$, and
$T_{e} = \{e,\, b_{1},\, e_{1}\}$ while
$T_{f} = \{f,\, a_{r-2},\, b_{r-3}\}$.
A computer check reveals that if $r = 6$ then $M_{2}$ has
an \mkt\dash minor, contrary to hypothesis, so $r > 6$.\cross\
But then $M_{2} / b_{2} \del a_{3} \del e_{3}$
is a counterexample to the claim, contradicting our
assumption on the minimality of~$r$.
\end{proof}

We complete the proof of Lemma~\ref{lem22} by showing
that in either of the two cases in the statement of
Claim~\ref{clm9} we can find a \vts\ of $M_{2}$.
If the first case holds, we show that
$X = \{e,\, f,\, g,\, a_{2},\, a_{3},\, b_{2}\}$ is a
vertical $3$\dash separator.
Note that $X$ is the union of the two triangles
$\{e,\, f,\, g\}$ and $\{a_{2},\, a_{3},\, b_{2}\}$, so
$r_{M_{2}}(X) = 4$.
Next we observe that $\{a_{2},\, b_{1},\, b_{2},\, e_{2}\}$
is a cocircuit of $\Delta_{r}$, and also of $M_{2}$.
Furthermore $T_{e} \cup g$ is a cocircuit of $M_{2}$.
Thus the symmetric difference of these two sets,
$\{e,\, g,\, a_{2},\, b_{2}\}$, contains a cocircuit of $M_{2}$.
By using a symmetric argument on
$T_{f} \cup g$ and $\{a_{3},\, b_{2},\, b_{3},\, e_{3}\}$ we
can show that $\{f,\, g,\, a_{3},\, b_{2}\}$ contains
a cocircuit.
Since $M_{2}$ contains no series pairs it follows that
$X$ contains two distinct cocircuits.
Thus the complement of $X$ has rank at most $r(M_{2}) - 2$,
and $X$ is a vertical $3$\dash separator of $M_{2}$.
A similar argument shows that if the second case of
Claim~\ref{clm9} holds then
$\{e,\, f,\, g,\, a_{1},\, b_{r-1},\, e_{r-1}\}$ is
a vertical $3$\dash separator.
This completes the proof of the lemma.
\end{proof}

\begin{lem}
\label{lem18}
Suppose that $r \geq 4$ and that $T_{1}$ and $T_{2}$ are
distinct allowable triangles of $\Delta_{r}$ such that
$T_{1} \cap T_{2} \ne \varnothing$.
Suppose that $M_{1}$ is the binary matroid obtained from
$\Delta_{r}$ by coextending with the elements $e$ and $f$
so that $(T_{1} - T_{2}) \cup e$ and $(T_{2} - T_{1}) \cup f$
are triads of $M_{1}$.
Let $M_{2}$ be the binary matroid obtained from $M_{1}$
by extending with the element $g$ so that
$\{e,\, f,\, g\}$ is a triangle of $M_{2}$.
Then $M_{2}$ has an \mkt\dash minor.
\end{lem}

\begin{proof}
First suppose that $T_{1}$ and $T_{2}$ meet in the spoke
element $b$.
Let $T_{1}' = (T_{1} - T_{2}) \cup e$, and let
$T_{2}' = (T_{2} - T_{1}) \cup f$.
It is not difficult to see that $\nabla_{T_{1}'}(\nabla_{T_{2}'}(M_{1}))$
is obtained from $\Delta_{r}$ by adding $e$ and $f$
as parallel elements to $b$.
It follows that $M_{1} \del b$ is isomorphic to the matroid
obtained from $\Delta_{r}$ by adding an element $b'$ in
parallel to $b$, and then performing \dy\ operations on
$T_{1}$ and $(T_{2} - b) \cup b'$.
Thus $M_{1} \del b$ has an \mkt\dash minor, by Claim~\ref{clm12},
so we are done.
Therefore we now assume that $T_{1}$ and $T_{2}$ meet in a rim
element.

The rest of the proof is by induction on $r$.
Suppose that $r = 4$.
By the symmetries of $\Delta_{r}$ we can assume that
$T_{1} = \{b_{1},\, e_{1},\, e_{2}\}$ and
$T_{2} = \{b_{2},\, e_{2},\, e_{3}\}$.
Then $M_{2}$ is obtained by coextending with $e$ and
$f$ so that $\{e,\, b_{1},\, e_{1}\}$ and
$\{f,\, b_{2},\, e_{3}\}$ are triads, and then
extending with $g$ so that $\{e,\, f,\, g\}$ is a
triangle.
The resulting matroid has an \mkt\dash minor.\cross\
This establishes the base case of our inductive
argument.

Suppose that $r > 4$, and that the result holds for
$\Delta_{r-1}$.
By the symmetries of $\Delta_{4}$ we assume that
$T_{1} = \{b_{1},\, e_{1},\, e_{2}\}$ and
$T_{2} = \{b_{2},\, e_{2},\, e_{3}\}$.
Let $M_{2}$ be the matroid obtained from
$\Delta_{r}$ as described in the hypotheses of the lemma.

We consider the minor
$M' = M_{2} / b_{r-1} \del a_{r-1} \del e_{r-1}$.
Note that $\{e,\, g,\, b_{1},\, e_{1}\}$ and
$\{f,\, g,\, b_{2},\, e_{3}\}$ are cocircuits in $M_{2}$,
so it cannot be the case that $g$ and $b_{r-1}$ are in
parallel in $M_{2}$, for then $\{e,\, f,\, b_{r-1}\}$
would be a triangle meeting these cocircuits in one
element.
Thus $\{e,\, f,\, g\}$ is a triangle in $M'$.
Moreover, $\{e,\, b_{1},\, e_{1}\}$ and
$\{f,\, b_{2},\, e_{3}\}$ are triads in $M_{2} \del g$.
Since $\{a_{r-2},\, a_{r-1},\, b_{r-2}\}$ and
$\{e_{r-2},\, e_{r-1},\, b_{r-2}\}$ are triangles in
$M_{2} \del g$, it follows that
$a_{r-1}$ and $e_{r-1}$ are not in
$\cl_{M \del g}^{*}(\{e,\, b_{1},\, e_{1}\})$ or
$\cl_{M \del g}^{*}(\{f,\, b_{2},\, e_{3}\})$, so
$\{e,\, b_{1},\, e_{1}\}$ and $\{f,\, b_{2},\, e_{3}\}$
are triads in $M' \del g$.
Using Proposition~\ref{prop2} we can now see that
$M'$ is constructed from $\Delta_{r-1}$ in the way
described in the statement of the lemma.
By induction $M'$ has an \mkt\dash minor, so we are done.
\end{proof}

\begin{lem}
\label{lem19}
Suppose that $r \geq 4$ and that $T_{1}$ and $T_{2}$ are
distinct allowable triangles of $\Delta_{r}$ such that
$T_{1} \cap T_{2} \ne \varnothing$.
Let $M_{1}$ be the binary matroid obtained from
$\Delta_{r}$ by coextending with the element $e$ so that
$e$ is in triad with the element in $T_{1} \cap T_{2}$ and
a single element from $T_{1} - T_{2}$, and let $M_{2}$ be
obtained from $M_{1}$ by coextending with the element
$f$ so that $(T_{2} - T_{1}) \cup f$ is a triad of $M_{2}$.
Finally, let $M_{3}$ be obtained from $M_{2}$ by extending
with the element $g$ so that $\{e,\, f,\, g\}$ is a
triangle in $M_{3}$.
Then $M_{3}$ has an \mkt\dash minor.
\end{lem}

\begin{proof}
Suppose that $T_{1}$ and $T_{2}$ meet in the spoke
element $b$.
Let $T_{1}'$ and $T_{2}'$ be the triads of $M_{2}$
that contain $e$ and $f$ respectively, and
let $x$ be the element in $T_{1} - T_{1}'$.
Then $\nabla_{T_{1}'}(\nabla_{T_{2}'}(M_{2}))$
is obtained from $\Delta_{r}$ by adding $f$
as a parallel element to $b$, and adding $e$ in parallel
to $x$.
Thus $M_{2} \del x$ is isomorphic to the matroid
obtained from $\Delta_{r}$ by adding an element $b'$ in
parallel to $b$, and then performing \dy\ operations on
$T_{1}$ and $(T_{2} - b) \cup b'$.
Therefore $M_{2} \del x$ has an \mkt\dash minor by Claim~\ref{clm12}.
Henceforth we assume that $T_{1}$ and $T_{2}$ meet in a spoke
element.

We complete the proof by induction on $r$.
Suppose that $r = 4$.
By the symmetries of $\Delta_{4}$ we can assume that
$T_{1} = \{b_{1},\, e_{1},\, e_{2}\}$ and
$T_{2} = \{b_{2},\, e_{2},\, e_{3}\}$.
There are two cases to check.
In the first case $\{e,\, e_{1},\, e_{2}\}$ is a triad
of $M_{1}$, and in the second $\{e,\, b_{1},\, e_{2}\}$
is a triad.
In either case $M_{3}$ has an \mkt\dash minor.\cross

Suppose that $r > 4$, and that the result holds for
$\Delta_{r-1}$.
By the symmetries of $\Delta_{r}$ we assume that
$T_{1} = \{b_{1},\, e_{1},\, e_{2}\}$ and
$T_{2} = \{b_{2},\, e_{2},\, e_{3}\}$.

Let $M'$ be the minor
$M_{3} / b_{r-1} \del a_{r-1} \del e_{r-1}$.
It cannot be the case that $g$ and $b_{r-1}$ are in
parallel in $M_{2}$ so $\{e,\, f,\, g\}$ is a triangle in
$M'$.
Furthermore, as in the proof of the previous lemma,
we can show that $\{f,\, b_{2},\, e_{3}\}$ is a
triad of $M' \del g$, and that either
$\{e,\, e_{1},\, e_{2}\}$ or
$\{e,\, b_{1},\, e_{2}\}$ is a triad of $M' \del g$.
It follows by induction that $M'$ has an \mkt\dash minor.
\end{proof}

\section{The coup de gr\^{a}ce}
\label{sct7.2}

Finally we move to the proof of the main theorem.

\begin{thm1.1}
\label{thm6}
An \ifc\ binary matroid has no \mkt\dash minor if and only
if it is either:
\begin{enumerate}[(i)]
\item cographic;
\item isomorphic to a triangular or triadic \mob\ matroid; or,
\item isomorphic to one of the~$18$ sporadic matroids
listed in Appendix~{\rm\ref{chp9}}.
\end{enumerate}
\end{thm1.1}

\begin{proof}
A simple computer check shows that none of the sporadic matroids
in Appendix~\ref{chp9} has an \mkt\dash minor\cross.
It follows easily from this fact and Corollary~\ref{cor2} that
the matroids listed in the statement of the theorem do not have
\mkt\dash minors.

To prove the converse we start by showing that if the theorem
fails, then there is a \vfc\ counterexample.
Suppose that the theorem is false, and that $M'$ is a minimum-sized
counterexample.
If $M'$ is not \vfc, then we can reduce it to a \vfc\
matroid $M''$ by repeatedly performing \yd\ operations, as
shown in Chapter~\ref{chp4}.
Let $M = \si(M'')$.
Lemma~\ref{prop24} shows that $M''$, and hence $M$, is not
cographic.
Certainly $M$ must contain at least one triangle, so
if $M$ is not a counterexample to the theorem then $M$ is
either a triangular \mob\ matroid or one of the sporadic matroids
in Appendix~\ref{chp9}.
But in this case Lemmas~\ref{lem9} and~\ref{lem5} show
that $M'$ is not a counterexample to Theorem~\ref{thm6}.
Thus $M$ is simple and \vfc\ and a counterexample to
Theorem~\ref{thm6}.
Moreover, since $|E(M)| \leq |E(M')|$ it follows that
$M$ too is a minimum-sized counterexample.

We note that $F_{7}$ is a triangular \mob\ matroid, that $F_{7}^{*}$
and $T_{12} \del e$ are not \vfc, and that \mkf, $T_{12} / e$, and
$T_{12}$ are all listed as sporadic matroids in Appendix~\ref{chp9}.
It follows from Corollary~\ref{cor3} that $M$ has a
$\Delta_{4}$\dash minor.
Consider the \vfc\ proper minors of $M$ that have
$\Delta_{4}$\dash minors, and among such minors let $N$
be as large as possible.

Corollary~\ref{cor1} implies that $M$ has no
$3$\dash connected minor $M'$ having both a
$N$\dash minor and a four-element circuit-cocircuit.
Thus we can apply Lemma~\ref{lem3} to $M$.
Therefore $M$ has a proper minor $M_{0}$ such that $M_{0}$ is \ifc\
with an $N$\dash minor, and $M_{0}$ is obtained from $M$ by
one of the operations described in Lemma~\ref{lem3}.
We assume that amongst such proper minors of $M$ the size of
$M_{0}$ is as large as possible.
Since $M_{0}$ is \ifc\ and $|E(M_{0})| < |E(M)|$ it follows
that $M_{0}$ obeys Theorem~\ref{thm6}.
It cannot be the case that $M_{0}$ is cographic, for $M_{0}$
has an $N$\dash minor, $N$ has a $\Delta_{4}$\dash minor, and
$\Delta_{4}$ is non-cographic.
Therefore $M_{0}$ is either a \mob\ matroid, or is isomorphic
to one of the sporadic matroids.
We divide what follows into various cases and subcases.

\begin{case1}
$M_{0}$ is a triadic \mob\ matroid.
\end{case1}

Since the triadic \mob\ matroids are not \vfc, it follows
that the only statements in Lemma~\ref{lem3} that could
apply are~\eqref{state1} or~\eqref{state3}.
We can quickly eliminate the case that statement~\eqref{state3}
applies, since no \mob\ matroid has both a triangle and a triad.
If statement~\eqref{state1} applies then $M$ is a \vfc\
single-element extension of a triadic \mob\ matroid, so
Lemma~\ref{lem11} implies that $M$ is isomorphic to
$T_{12}$, a contradiction.

This disposes of the case that $M_{0}$ is a triadic \mob\
matroid.

\begin{case2}
$M_{0}$ is a triangular \mob\ matroid.
\end{case2}

\begin{clm}
\label{clm16}
If statement~\eqref{state2} of Lemma~{\rm\ref{lem3}} applies then there
is an allowable triangle $T$ of $M_{0}$ such that one of the
following cases applies.
\begin{enumerate}[(i)]
\item There is an element $x \in E(M)$ such that $M_{0} = M / x$;
\item There are elements $x,\, y \in E(M)$ such that
$M_{0} \iso M / x \del y$ and $x$ is contained in a triad of
$M \del y$ with two elements of $T$.
Moreover, $M \del y$ is $3$\dash connected; or,
\item There are elements $x,\, y,\, z \in E(M)$ and triangles $T_{xy}$
and $T_{xz}$ of $M$ such that $M_{0} \iso M / x \del y \del z$ and
$x,\, y \in T_{xy}$ while $x,\, z \in T_{xz}$.
Moreover, $T_{xy} \cap T = \varnothing$ while $T_{xz} \cap T \ne \varnothing$,
and $x$ is contained in a triad of $M \del y \del z$ with two elements of $T$.
Furthermore $M \del y \del z$ and $M \del z$ are both $3$\dash connected.
\end{enumerate}
\end{clm}

\begin{proof}
Statement~\eqref{state2} asserts that there is a triangle $T$ of
$M / x$ and a four-element cocircuit $C^{*}$ of $M$ such that
$x \in C^{*}$ and $|C^{*} \cap T| = 2$.
If $x$ is in no triangles of $M$ then $M_{0} = M / x$, so~(i) holds.
Suppose that $x$ is in exactly one triangle $T_{xy}$ of $M$.
Clearly $T_{xy}$ meets $C^{*}$ in exactly two elements.
Let $y$ be the element in $(T_{xy} \cap C^{*}) - x$.

It cannot be the case that $T = T_{xy}$, for $T_{xy}$ is not a triangle
in $M / x$.
Now we see that $y \notin T$, for otherwise there would
be a parallel element in $M$.
Therefore $T$ is a triangle of $M / x \del y \iso M_{0}$.
Moreover $C^{*} - y$ is a triad of $M \del y$ which contains $x$ and two
elements of $T$, as desired.
If $T$ is not an allowable triangle of $M_{0}$, then
Corollary~\ref{cor5} implies that $M \del y$ has an
\mkt\dash minor, a contradiction.

It remains to show only that $M \del y$ is $3$\dash connected.
If this is not the case then $x$ is in a cocircuit of size
at most two in $M \del y$, implying that $M$ contains a cocircuit
of size at most three, a contradiction as $M$ is \vfc.

Next we assume that $x$ is contained in exactly two triangles,
$T_{xy}$ and $T_{xz}$, in $M$.
Statement~\eqref{state2} of Lemma~\ref{lem3} asserts that
exactly one of $T_{xy}$ and $T_{xz}$ meets $T$.
Let us assume without loss of generality that $T_{xy} \cap T = \varnothing$
while $T_{xz} \cap T \ne \varnothing$.
Obviously $T_{xy}$ and $T_{xz}$ meet $C^{*}$ in exactly two elements each.
Let $y$ be the element in $(T_{xy} \cap C^{*}) - x$, and let
$z$ be the element in $T_{xz} - C^{*}$.
Then $M _{0} \iso M / x \del y \del z$.
Neither $T_{xy}$ nor $T_{xz}$ is equal to $T$, and it follows that
neither $y$ nor $z$ is contained in $T$, so $T$ is a triangle
of $M / x \del y \del z$.

It cannot be the case that $z$ is contained in $\cl_{M}^{*}(C^{*})$,
for if this were the case then properties of cocircuits in binary
matroids would imply that $M$ contains a cocircuit of size at
most three, a contradiction.
Therefore $C^{*}$ is a cocircuit in $M \del z$, so $C^{*} - y$ is a
triad in $M \del y \del z$ and $C^{*} - y$ contains both $x$ and
two elements of $T$, as desired.
Corollary~\ref{cor5} again implies $T$ is an
allowable triangle of $M_{0}$.

Suppose that $M \del y \del z$ is not $3$\dash connected.
Then $x$ is contained in a cocircuit of size at most two
in $M \del y \del z$.
In fact it must be in a cocircuit of size exactly two, for
otherwise $M$ contains a cocircuit of size at most three.
Let $x'$ be the other element in this cocircuit.
Since $C^{*} - y$ is a triad in $M \del y \del z$ it follows
that $(C^{*} - \{x,\, y\}) \cup x'$ is a triad in
$M \del y \del z$, and hence in $M / x \del y \del z$.
This is a contradiction, as $M / x \del y \del z$ is
\vfc.

Therefore $M \del y \del z$ is $3$\dash connected.
If $M \del z$ is not $3$\dash connected then $y$ is in
a circuit of size at most two in $M \del z$ and hence in $M$,
a contradiction.
This completes the proof of the claim.
\end{proof}

\begin{clm}
\label{clm17}
If statement~\eqref{state6} of Lemma~{\rm\ref{lem3}} holds then:
\begin{enumerate}[(i)]
\item there is an element $x \in E(M)$ such that $M_{0} = M / x$;
\item there are elements $x,\, y \in E(M)$ such that
$M_{0} \iso M / x \del y$ and $M \del y$ are $3$\dash connected; or,
\item there are elements $x,\, y,\, z \in E(M)$ such that
$M_{0} \iso M / x \del y \del z$, and $M \del y \del z$ and
$M \del z$ are both $3$\dash connected.
\end{enumerate}
\end{clm}

\begin{proof}
Statement~\eqref{state6} asserts that there is an element $x$ such that
$M / x$ is \vfc, and a cocircuit $C^{*}$ such that $|C^{*}| = 4$
and $x \in C^{*}$.
Moreover $x$ is in at most two triangles in $M$.

If $x$ is in no triangles then~(i) applies and we are done.
Suppose that $x$ is in exactly one triangle $T$.
Let $y$ be an element of $T - x$.
Then $M_{0} \iso M  / x \del y$.
If $M \del y$ is not $3$\dash connected then $x$ is contained
in a cocircuit of size at most two in $M \del y$, and $M$ contains
a cocircuit of size at most three, a contradiction.

Suppose $x$ is in exactly two triangles, $T_{y}$ and $T_{z}$,
in $M$.
Let $y$ be the single element in $(C^{*} \cap T_{y}) - x$, and
let $z$ be the single element in $T_{z} - C^{*}$.
Then $M_{0} \iso M / x \del y \del z$.
If $M \del y \del z$ is not $3$\dash connected then
$x$ is contained in a series pair with some element, $x'$, in
$M \del y \del z$.
In this case $(C^{*} - \{x,\, y\}) \cup x'$ is a triad of
$M \del y \del z$, and hence of $M / x \del y \del z$, a
contradiction.
Thus $M \del y \del z$ is $3$\dash connected, and it is easy to
see that $M \del z$ is $3$\dash connected.
\end{proof}

\begin{subcase2.1}
Statement~\eqref{state1} of Lemma~{\rm\ref{lem3}} applies.
\end{subcase2.1}

In this case $M$ is a single-element extension of $\Delta_{r}$
for some $r \geq 4$.
Then Lemma~\ref{lem7} implies that $M$ is isomorphic
to either $C_{11}$ or $M_{4,11}$, a contradiction.

\begin{subcase2.2}
Statement~\eqref{state7} of Lemma~{\rm\ref{lem3}} applies.
\end{subcase2.2}

In this case $M$ is a $3$\dash connected coextension of
$\Delta_{r}$ for some $r \geq 4$.
We apply Corollary~\ref{cor7}.
It cannot be the case that $M$ contains a triad or a four-element
circuit-cocircuit, so we deduce that $M$ is isomorphic to
$M_{5,11}$, a contradiction.

\begin{subcase2.3}
Statement~\eqref{state3} of Lemma~{\rm\ref{lem3}} applies.
\end{subcase2.3}

This case immediately leads to a contradiction, as
no \mob\ matroid has both a triangle and a triad.

\begin{subcase2.4}
Statement~\eqref{state2} of Lemma~{\rm\ref{lem3}} applies.
\end{subcase2.4}

We use Claim~\ref{clm16}.
It cannot be the case that $M_{0} = M / x$, for then
statement~\eqref{state7} applies, and we have already
considered that possibility.

Suppose that~(ii) of Claim~\ref{clm16} holds.
Then $M \del y$ is a $3$\dash connected coextension of
$\Delta_{r}$ for some $r \geq 4$ and there exists an allowable
triangle $T$ of $\Delta_{r}$ such that $x$ is contained in
a triad of $M \del y$ with two elements of $T$.
Now Corollary~\ref{cor6} tells us that
$M$ is isomorphic to $M_{5,12}^{a}$ or $M_{5,12}^{b}$.

This contradiction means that we must assume~(iii) of
Claim~\ref{clm16} holds.
Therefore $M \del y \del z$ is a $3$\dash connected
single-element coextension of $\Delta_{r}$ for some $r \geq 4$, and
there is an allowable triangle $T$ of $\Delta_{r}$ such that
$x$ is in a triad of $M \del y \del z$ with two elements of $T$.
Suppose that $r = 4$.
Then $N \iso \Delta_{4}$, and by the maximality of $N$
it follows that $M$ has no minor isomorphic to $M_{5,12}^{a}$ or $M_{5,12}^{b}$.
Thus Lemma~\ref{lem16} tells us that $M$ is isomorphic
to either $\Delta_{5}$ or $M_{5,13}$.
In either case we have a contradiction.
Similarly, if $r \geq 5$ then Lemma~\ref{lem17} asserts
that $M$ is isomorphic to $\Delta_{r+1}$, and again we
have a contradiction.

\begin{subcase2.5}
Statement~\eqref{state6} of Lemma~{\rm\ref{lem3}} applies.
\end{subcase2.5}

Consider Claim~\ref{clm17}.
If~(i) applies then we are in Subcase~2.1.
If~(ii) holds then $M \del y$ is a coextension of $\Delta_{r}$
by the element $x$.
Since $M \del y$ cannot have a four-element circuit-cocircuit
Corollary~\ref{cor7} tells us that there is a allowable
triangle $T$ of $\Delta_{r}$ such that $x$ is in a triad with
two elements of $T$ in $M \del y$.
Since $y$ is in a triangle of $M$ with $x$ Corollary~\ref{cor6}
tells us that either $M$ has an \mkt\dash minor, or $M$ is
isomorphic to $M_{5,12}^{a}$ or $M_{5,12}^{b}$.
Thus we have a contradiction.

Suppose that~(iii) of Claim~\ref{clm17} holds.
We can derive a contradiction by using Lemmas~\ref{lem16}
and~\ref{lem17}.

\begin{subcase2.6}
Statement~\eqref{state9} of Lemma~{\rm\ref{lem3}} applies.
\end{subcase2.6}

In this case we immediately derive a contradiction from
Lemmas~\ref{lem20} and~\ref{lem22}.

\begin{subcase2.7}
Statement~\eqref{state4} of Lemma~{\rm\ref{lem3}} applies.
\end{subcase2.7}

In this case Lemma~\ref{lem18} shows that we have a contradiction.

\begin{subcase2.8}
Statement~\eqref{state8} of Lemma~{\rm\ref{lem3}} applies.
\end{subcase2.8}

In this case Lemma~\ref{lem19} shows that we have a contradiction.

\begin{subcase2.9}
Statement~\eqref{state5} of Lemma~{\rm\ref{lem3}} applies.
\end{subcase2.9}

In this case there is an element $x \in E(M)$ such that $M \del x$
contains three cofans, $\{x_{1},\ldots, x_{5}\}$,
$\{y_{1},\ldots, y_{5}\}$, and $\{z_{1},\ldots, z_{5}\}$, where
$x_{5} = y_{1}$, $y_{5} = z_{1}$, and $z_{5} = x_{1}$, and
$M_{0} = M \del x / x_{1} / y_{1} / z_{1}$.
It is easy to see that $\{x_{2},\, x_{3},\, x_{4}\}$,
$\{y_{2},\, y_{3},\, y_{4}\}$, and
$\{z_{2},\, z_{3},\, z_{4}\}$ are triangles of
$M_{0}$, and that
$\{x_{3},\, x_{4},\, y_{2},\, y_{3}\}$,
$\{y_{3},\, y_{4},\, z_{2},\, z_{3}\}$, and
$\{z_{3},\, z_{4},\, x_{2},\, x_{3}\}$ are cocircuits.

Thus we seek pairwise disjoint triangles $T_{1}$, $T_{2}$,
and $T_{3}$ in $M_{0}$ such that $T_{i} \cup T_{j}$ contains a
cocircuit $C_{ij}^{*}$ of size four for $1 \leq i < j \leq 3$.
We can reconstruct $M$ from $M_{0}$ by coextending
with the element $e$ so that $e$ forms a triad with
$T_{1} \cap C_{12}^{*}$, then coextending by $f$ so that
$(T_{2} \cap C_{23}^{*}) \cup f$ is a triad, coextending
by $g$ so that $(T_{3} \cap C_{13}^{*}) \cup g$ is a triad, and finally
extending by $x$ so that $\{e,\, f,\, g,\, x\}$ is a circuit.

It follows from this discussion and Corollary~\ref{cor5}
that each of $T_{1}$, $T_{2}$, and $T_{3}$ is
an allowable triangle in $M_{0}$.
It is simple to prove by induction on $r$ that when $r \geq 4$,
the only four-element cocircuits in $\Delta_{r}$ are sets of
the form $\{a_{i},\, b_{i},\, b_{i-1},\, e_{i}\}$, for
$2 \leq i \leq r - 1$, and $\{a_{1},\, b_{1},\, b_{r-1},\, e_{1}\}$.
Since the only allowable triangles of $\Delta_{r}$ are
those triangles that contain spoke elements, it is
now straightforward to see that $M_{0}$ is isomorphic to $\Delta_{4}$.

By the symmetries of $\Delta_{4}$ we can assume that
$T_{1} = \{b_{1},\, e_{1},\, e_{2}\}$,
$T_{2} = \{a_{2},\, a_{3},\, b_{2}\}$, and
$T_{3} = \{a_{1},\, b_{3},\, e_{3}\}$.
Now it is a simple matter to construct a representation of
$M$ and check that it has an \mkt\dash minor.\cross

This contradiction means that we have completed the case-checking
when $M_{0}$ is a triangular \mob\ matroid.

\begin{case3}
$M_{0}$ is isomorphic to one of the sporadic matroids listed
in Appendix~{\rm\ref{chp9}}.
\end{case3}

\begin{subcase3.1}
Statement~\eqref{state1} of Lemma~{\rm\ref{lem3}} applies.
\end{subcase3.1}

A computer check reveals that all the $3$\dash connected binary
single-element extensions of the matroids in Appendix~\ref{chp9}
either have \mkt\dash minors, or are isomorphic to
other sporadic matroids.\cross\
Thus we have a contradiction.

\begin{subcase3.2}
Statement~\eqref{state7} of Lemma~{\rm\ref{lem3}} applies.
\end{subcase3.2}

In this case $M$ is a single-element coextension of one of the
matroids listed in Appendix~\ref{chp9}.
For each of the matroids in Appendix~\ref{chp9}, other than $M_{9,18}$ and
$M_{11,21}$, we can use a computer to generate all the
$3$\dash connected single-element coextensions that belong to \ex{\mkt}.
The only \vfc\ matroids we produce in this way are isomorphic to
$T_{12} / e$ and $T_{12}$.\cross

The matroids $M_{9,18}$ and $M_{11,21}$ are large enough that generating
all their single-element coextensions is non-trivial, so we provide
an inductive argument to show that they have no
\vfc\ single-element coextensions in \ex{\mkt}.
Using a computer we can generate all the $3$\dash connected
coextensions of $M_{7,15}$ by the element $e$ such that the resulting
matroid is in \ex{\mkt}.
There are $12$ such coextensions, ignoring isomorphisms.
In each of these coextensions $e$ is contained in a cocircuit
$C^{*}$ such that there is a triangle $T$ of $M_{7,15}$ with
the property that $C^{*} \subseteq T \cup e$.\cross\
In every such coextension it must be the case that $|C^{*}| \geq 3$,
and it is easy to see that $C^{*} \cup T$ is a vertical
$3$\dash separator.
Therefore no such coextension is \vfc.

We prove inductively that if $M'$ is a $3$\dash connected
coextension of $M_{9,18}$ or (respectively) $M_{11,21}$ by the element
$e$, such that $M' \in \ex{\mkt}$, then there is a cocircuit
$C^{*}$ of $M'$ and a triangle $T$ of $M_{9,18}$ or (respectively)
$M_{11,21}$, such that $e \in C^{*}$ and $C^{*} \subseteq T \cup e$.

Recall that if $\mcal{T}_{0}$ is a set of triangles in $F_{7}$, and
$t_{e}$ stands for the number of triangles in $\mcal{T}_{0}$ that
contain $e$ for each element $e \in F_{7}$, then we obtain
$M_{1}$ by adding $t_{e} - 1$ parallel elements to each element $e$ in
$F_{7}$.
We then let \mcal{T} be a set of pairwise disjoint triangles in
$M_{1}$ corresponding in the natural way to triangles in $\mcal{T}_{0}$;
and finally we perform \dy\ operations on each of the triangles in
\mcal{T} to obtain the matroid $\Delta(F_{7};\, \mcal{T}_{0})$.
Then $\nabla(F_{7}^{*};\, \mcal{T}_{0})$ is the dual of
$\Delta(F_{7};\, \mcal{T}_{0})$.
We know that $M_{7,15}$, $M_{9,18}$, and $M_{11,21}$ are isomorphic
respectively to $\nabla(F_{7}^{*};\, \mcal{T}_{5})$,
$\nabla(F_{7}^{*};\, \mcal{T}_{6})$, and $\nabla(F_{7}^{*};\, \mcal{T}_{7})$,
where $\mcal{T}_{5}$, $\mcal{T}_{6}$, and $\mcal{T}_{7}$ are
sets of five, six, and seven pairwise distinct triangles in $F_{7}$.

Suppose that $M_{9,18}^{*}$ is obtained from $M_{1}$ by performing
\dy\ operations on the triangles in \mcal{T}, where the
triangles in \mcal{T} correspond naturally to the triangles in
$\mcal{T}_{6}$.
Let $T'$ be any triangle in \mcal{T}.
Then $T'$ is also a triangle in $M_{9,18} = \nabla(F_{7}^{*};\, \mcal{T}_{6})$.
Since $\mcal{T}_{6}$ contains six triangles it follows that every
element in $T'$ is contained in a non-trivial parallel class in $M_{1}$.
From this fact and Proposition~\ref{prop28} we deduce that 
\begin{equation}
\label{eqn4}
M_{9,18}^{*} \del T' = \Delta(F_{7};\, \mcal{T}_{6}) \del T'
= \Delta(F_{7};\, \mcal{T}_{5}) = M_{7,15}^{*}.
\end{equation}
Thus $M_{7,15} = M_{9,18} / T'$.

Let $M'$ be a coextension of $M_{9,18}$ by the element
$e$ such that $M'$ is a $3$\dash connected member of \ex{\mkt}.
Let $M'' = M' / T'$.
By Equation~\eqref{eqn4} it follows that $M''$ is a coextension
of $M_{7,15}$ by the element $e$.
If $M''$ is not $3$\dash connected then $e$ is contained in
a cocircuit of size at most two in $M''$.
But this implies that $M'$ contains a cocircuit of size at most two,
a contradiction.
Thus $M''$ is a $3$\dash connected single-element coextension of
$M_{7,15}$ that belongs to \ex{\mkt}.
Our earlier argument tells us that in $M''$ the element $e$ belongs
to a cocircuit $C^{*}$ and that there is a triangle $T$ of $M_{7,15}$
such that $C^{*} \subseteq T \cup e$.
Now $M_{7,15}$ has exactly five triangles, and these triangles
are exactly the members of $\mcal{T}_{5}$.
It follows that $T$ is also a triangle of $M_{9,18}$, and $C^{*}$ is
certainly a cocircuit of $M_{9,18}$, so our claim holds for
$M_{9,18}$.
We can use exactly the same argument to prove that the claim holds for
$M_{11,21}$.
Thus no $3$\dash connected single-element extension of $M_{9,18}$
or $M_{11,21}$ can be a \vfc\ member of \ex{\mkt}.
This completes the proof that statement~\eqref{state7} cannot hold.

\begin{subcase3.3}
Statement~\eqref{state3} of Lemma~{\rm\ref{lem3}} applies.
\end{subcase3.3}

There is only one matroid listed in Appendix~\ref{chp9} that contains
both a triad and a triangle, so we assume that
$M_{0}$ is isomorphic to $M_{5,11}$.
Thus $M$ contains elements $x$ and $y$ such that $M / x \del y$
is isomorphic to $M_{5,11}$, and there is a triangle
of $M$ that contains both $x$ and $y$.
Suppose that $M \del y$ is not $3$\dash connected.
Then $x$ is contained in a cocircuit of size at most two in
$M \del y$, and this implies that $M$ contains a cocircuit
of size at most three, a contradiction.
Therefore $M \del y$ is $3$\dash connected.
However, a computer check reveals that $M_{5,11}$ has
no $3$\dash connected single-element coextensions in
\ex{\mkt}.\cross

\begin{subcase3.4}
Statement~\eqref{state2} of Lemma~{\rm\ref{lem3}} applies.
\end{subcase3.4}

If statement~(i) of Claim~\ref{clm16} holds then the case reduces
to Subcase~3.2.
So we suppose that either~(ii) or~(iii) of Claim~\ref{clm16}
holds.

For each of the matroids in Appendix~\ref{chp9} other than $M_{7,15}$,
$M_{9,18}$ and $M_{11,21}$, we perform the following procedure:
we generate all single-element coextensions, and then extend by either
one or two more elements, at each step restricting to those
$3$\dash connected matroids that belong to \ex{\mkt}.
The only \ifc\ matroids uncovered by this procedure are
$M_{5,11}$, $T_{12} / e$, $M_{5,12}^{b}$, $M_{5,13}$,
$T_{12}$, and $M_{7,15}$.\cross

We give an inductive argument to cover the possibility that $M_{0}$
is isomorphic to $M_{7,15}$, $M_{9,18}$ or $M_{11,21}$.
The base case of our argument is centered on $M_{5,12}^{a}$.
Recall that $M_{5,12}^{a}$ is isomorphic to
$\nabla(F_{7}^{*};\, \mcal{T}_{\,\,4}^{a})$, where $\mcal{T}_{\,\,4}^{a}$ is a
set of four triangles in $F_{7}$, three of which contain a common point.
Proposition~\ref{prop37} tells us that each of the triangles
in $\mcal{T}_{\,\,4}^{a}$ corresponds to an allowable triangle in
$\nabla(F_{7}^{*};\, \mcal{T}_{\,\,4}^{a})$, and
Proposition~\ref{prop38} implies that the allowable triangles
on $M_{5,12}^{a}$ are precisely the triangles that arise in this way.

\begin{subsubcase3.4.1}
Statement~(ii) of Claim~{\rm\ref{clm16}} holds.
\end{subsubcase3.4.1}

We verify the next claim by constructing each of the coextensions
of $M_{5,12}^{a}$ by the element $x$ so that $x$ is in a triad
with two elements of an allowable triangle, and then using a
computer to check every $3$\dash connected single-element
extension in \ex{\mkt}, ignoring isomorphisms.\cross

\begin{clm}
\label{clm18}
Let $T$ be an allowable triangle in $M_{5,12}^{a}$.
Suppose that we construct $M'$ by coextending with $x$ so that
it forms a triad $T_{x}$ with two elements of $T$, and then extending
by the element $y$ so that $y$ is in a triangle $T_{xy}$ with $x$.
If $M'$ is a $3$\dash connected member of \ex{\mkt}, then either:
\begin{enumerate}[(i)]
\item $T_{xy}$ meets $T$; or,
\item there exists an allowable triangle $T'$ of $M_{5,12}^{a}$ such
that $T' \ne T$, $|T_{xy} \cap T'| = 1$, and $T_{xy} \cup T'$ contains
a cocircuit.
\end{enumerate}
\end{clm}

If~(i) holds in Claim~\ref{clm18} then $T_{xy} \cup \{x,\, y\}$
is a vertical $3$\dash separator.
Suppose that~(ii) holds.
Then $T'$ is a triangle in $M'$, for otherwise $T' \cup x$
is a circuit in $M' \del y$ which meets the triad $T_{x}$ in one
element, $x$.
Now $T_{xy} \cup T'$ is a vertical $3$\dash separator as
$r_{M'}(T_{xy} \cup T') = 4$ and $T_{xy} \cup T'$ contains
a cocircuit.
Thus Claim~\ref{clm18} shows that any matroid obtained from
$M_{5,12}^{a}$ in this way can not be \vfc.

Claim~\ref{clm18} provides the base case for our inductive argument.
Recall that $M_{7,15}$ is isomorphic to $\nabla(F_{7}^{*};\, \mcal{T}_{5})$,
where $\mcal{T}_{5}$ is a set of five triangles in $F_{7}$.
Proposition~\ref{prop37} and Proposition~\ref{prop44} show that
the allowable triangles of $M_{7,15}$ correspond exactly to the
members of $\mcal{T}_{5}$.
Let the allowable triangles of $M_{7,15}$ be $T_{1},\ldots, T_{5}$.

Suppose that $M'$ is obtained from $M_{7,15}$ by coextending with
the element $x$ so that it is in a triad $T_{x}$ with two elements
of an allowable triangle (which we will assume to be $T_{1}$ by
relabeling if necessary), and then extending by the element $y$
so that it is in a triangle $T_{xy}$ with $x$.
Assume that $M'$ is a $3$\dash connected member of \ex{\mkt}.
We wish to show that either $T_{xy}$ meets $T_{1}$, or that
there is an allowable triangle $T'$ of $M_{7,15}$ such that
$T' \ne T_{1}$, $T_{xy}$ meets $T'$ in exactly one element, and
$T_{xy} \cup T'$ contains a cocircuit.
If $T_{xy}$ meets $T_{1}$ then obviously we are done, so
we assume that $T_{xy} \cap T_{1} = \varnothing$.

For the element $e$ in $F_{7}$ we let $t_{e}$ be the number of
triangles of $\mcal{T}_{5}$ that contain $e$.
Then $M_{1}$ is the matroid obtained by adding $t_{e} - 1$
elements in parallel to $e$, for every element $e$.
If \mcal{T} is a set of pairwise disjoint triangles in
$M_{1}$ which correspond to the triangles in $\mcal{T}_{5}$,
then $M_{7,15}^{*}$ is isomorphic to the matroid obtained by
performing \dy\ operations on the triangles in \mcal{T}.
It is easy to see that $t_{e} > 1$ for every element $e$, and
from this fact and Proposition~\ref{prop28} we deduce that
for $i \in \{2,\ldots, 5\}$ we have
\begin{equation}
\label{eqn5}
\Delta(F_{7};\, \mcal{T}_{5}) \del T_{i} =
\nabla_{T_{i}}(\Delta(F_{7};\, \mcal{T}_{5})) \del T_{i} =
\Delta(F_{7};\, \mcal{T}_{5} - T_{i}).
\end{equation}
Therefore
$\nabla(F_{7}^{*};\, \mcal{T}_{5}) / T_{i}
= \nabla(F_{7}^{*};\, \mcal{T}_{5} - T_{i})$.
It is easy to see that there is at most one triangle $T_{i}$
in $\mcal{T}_{5}$ such that $\mcal{T}_{5} - T_{i}$ is not the same
arrangement of triangles as $\mcal{T}_{\,\,4}^{a}$.
Therefore we will assume that $M_{7,15}/T_{2}$ and $M_{7,15}/T_{3}$
are both isomorphic to $M_{5,12}^{a}$.

Next we argue that $T_{xy}$ is the only triangle of $M'$ that contains $y$.
Suppose that $T_{y}$ is some other triangle of $M'$ such that $y \in T_{y}$.
Then $(T_{y} \cup T_{xy}) - y$ is a circuit in $M \del y$, and it meets
the triad $T_{x}$ in the element $x$.
Therefore it contains exactly one elements of $T_{x} - x$,
and since $T_{x} - x$ is contained in $T_{1}$ this means that
$T_{y}$ meets $T_{1}$.
Now $T_{1}$ and $(T_{y} \cup T_{xy}) - \{x,\, y\}$ are both triangles
in $M' / x \del y = M_{7,15}$.
But this is a contradiction, as all triangles of $M_{7,15}$ are
disjoint.

Since $T_{xy}$ cannot meet both $T_{2}$ and $T_{3}$ we assume
without loss of generality that $T_{xy}$ does not meet $T_{2}$.
Therefore there is no triangle of $M'$ that both contains $y$ and
meets $T_{2}$.
Let $M''$ be $M'/T_{2}$.
Then $M'' / x \del y = M_{5,12}^{a}$.
Now $T_{x}$ is a triad of $M' \del y$, and hence of $M'' \del y$.
Furthermore, if $T_{xy}$ is not a triangle of $M''$ then $T_{xy}$ has a
non-empty intersection with $T_{2}$, contrary to our assumption.
Thus $T_{xy}$ is a triangle of $M''$ which contains both $x$ and $y$.
By applying Claim~\ref{clm18} to $M''$ we deduce that either
$T_{xy}$ meets $T_{1}$, or that there is an allowable triangle $T'$
in $M_{5,12}^{a}$ other than $T_{1}$ such that $T_{xy}$ meets $T'$ in
one element, and $T_{xy} \cup T'$ contains a cocircuit of $M''$.
We have assumed that $T_{xy} \cap T_{1} = \varnothing$, so the latter
holds.

The allowable triangles of $M_{5,12}^{a} = M'' / x \del y$ are
exactly $T_{1},\, T_{3},\, T_{4},\, T_{5}$.
Hence $T'$ is also an allowable triangle in $M_{7,15}$.
Furthermore, the cocircuit which is contained in $T_{xy} \cup T'$
is also a cocircuit of $M'$.
We have shown that Claim~\ref{clm18} also holds when $M_{5,12}^{a}$
is replaced with $M_{7,15}$.
By the same argument as used earlier, no matroid obtained from $M_{7,15}$
in the way described can be \vfc.

But we can use the inductive argument again, and show that
Claim~\ref{clm18} also holds when $M_{5,12}^{a}$ is replaced with
$M_{9,18}$ or $M_{11,21}$.
Therefore we have dealt with the case that~(ii) of Claim~\ref{clm16}
holds.

\begin{subsubcase3.4.2}
Statement~(iii) of Claim~{\rm\ref{clm16}} holds.
\end{subsubcase3.4.2}

We use a computer to check all $3$\dash connected members
of \ex{\mkt} that are formed from $M_{5,12}^{a}$ by coextending
with $x$ so that it is in a triad with two elements of an
allowable triangle $T$, and then extending by $y$ and $z$, at
each stage considering only those matroids which are $3$\dash connected
members of \ex{\mkt} and ignoring isomorphisms.
The check reveals that in any such matroid, if $x$ and $y$ lie in a
triangle $T_{xy}$, and $x$ and $z$ lie in a triangle $T_{xz}$, then
it is not the case that $T_{xy}$ meets $T$ and $T_{xz}$ avoids
$T$.\cross\
We summarize this in the following claim.

\begin{clm}
\label{clm19}
Let $T$ be an allowable triangle in $M_{5,12}^{a}$.
Suppose that we construct $M'$ by coextending with $x$ so that
it forms a triad $T_{x}$ with two elements of $T$, and then extending
by the element $y$ so that $y$ is in a triangle $T_{xy}$ such
that $x \in T_{xy}$ and $T_{xy} \cap T = \varnothing$, and then extending by
the element $z$ so that $z$ is in a triangle $T_{xz}$ such that
$x \in T_{xz}$ and $T_{xz} \cap T \ne \varnothing$.
If $M'$, $M' \del z$, and $M' \del y \del z$ are all
$3$\dash connected then $M'$ has an \mkt\dash minor.
\end{clm}

Let $T_{1},\ldots, T_{5}$ be the allowable triangles of $M_{7,15}$.
Suppose $M'$ is obtained from $M_{7,15}$ by coextending with the
element $x$ so that $x$ is in a triad $T_{x}$ with two elements
of $T_{1}$, and then extending by the elements $y$ and $z$ so that
there are triangles $T_{xy}$ and $T_{xz}$ such that
$x,\, y \in T_{xy}$, $x,\, z \in T_{xz}$, and
$T_{xy} \cap T_{1} = \varnothing$, while $T_{xz} \cap T_{1} \ne \varnothing$.
Suppose also that $M'$, $M' \del z$, and $M' \del y \del z$ are
all $3$\dash connected.

Since there is exactly one triangle $T_{i}$ in $\mcal{T}_{5}$ such that
$\mcal{T}_{5} - T_{i}$ is not the same configuration as
$\mcal{T}_{\,\,4}^{a}$ we can assume that $M_{7,15}/T_{2}$, $M_{7,15}/T_{3}$,
and $M_{7,15}/T_{4}$ are all isomorphic to $M_{5,12}^{a}$.
We can argue, as before, that $T_{xy}$ is the only triangle of
$M' \del z$ that contains $y$.
Therefore any triangle of $M'$ other than $T_{xy}$ that contains
$y$ contains $z$.
Therefore there can be at most two triangles of $M'$ that contain
$y$.
Without loss of generality we can assume that there is no triangle
of $M'$ that contains $y$ and that intersects $T_{2}$.

Next we note that $T_{xz}$ has an empty intersection with $T_{2}$,
since the members of $T_{xz}$ are $x$, $z$, and a single element of $T_{1}$.
Suppose that there is a triangle $T_{z}$ of $M'$ such that
$z \in T_{z}$ and $T_{z} \cap T_{2} \ne \varnothing$.
It cannot be the case that $y \in T_{z}$, since we have assumed that
no triangle of $M'$ containing $y$ meets $T_{2}$.
Therefore $(T_{z} \cup T_{xz}) - z$ is a circuit of
$M' \del y \del z$ that contains $x$.
Since $T_{x}$ is a triad of $M' \del y \del z$ and $x \in T_{x}$
it follows that $(T_{z} \cup T_{xz}) - z$ contains an element
of $T_{x} - x \subseteq T_{1}$.
This means that $(T_{z} \cup T_{xz}) - \{x,\, z\}$ is a triangle of
$M' / x \del y \del z = M_{7,15}$ which meets $T_{1}$.
As the triangles of $M_{7,15}$ are disjoint this is a
contradiction, so no such triangle $T_{z}$ exists.

Now we let $M''$ be $M' / T_{2}$.
Then $M'' / x \del y \del z$ is isomorphic to $M_{5,12}^{a}$.
Moreover $T_{x}$ is a triad of $M'' \del y \del z$ and
both $T_{xy}$ and $T_{xz}$ are triangles of $M''$, since they
do not intersect $T_{2}$.

If $M'' \del y \del z$ is not $3$\dash connected, then $x$ is
in a cocircuit of size at most two in $M'' \del y \del z$,
and hence in $M' \del y \del z$.
This is a contradiction as $M' \del y \del z$ is $3$\dash connected.
If $M'' \del z$ is not $3$\dash connected then $y$ is contained
in a circuit of size at most two in $M'' \del z$.
But this implies that $y$ is contained in a triangle of $M' \del z$
which meets $T_{2}$, contrary to hypothesis.
Therefore we assume that $M'' \del z$ is $3$\dash connected.
Similarly, if $M''$ is not $3$\dash connected, then $z$ is
contained in a circuit of size at most two in $M''$, and
this leads to a contradiction, since $z$ is contained
in no triangle of $M'$ which meets $T_{2}$.

We have shown that $M''$, $M'' \del z$, and $M'' \del y \del z$
are all $3$\dash connected.
Now Claim~\ref{clm19} implies that $M''$, and hence $M'$,
has an \mkt\dash minor.
Thus Claim~\ref{clm19} holds when $M_{5,12}^{a}$ is replaced
with $M_{7,15}$.
By using the same inductive argument we can show that
Claim~\ref{clm19} holds when $M_{5,12}^{a}$ is replaced with
$M_{9,18}$ or $M_{11,21}$.
This finishes the case that statement~\eqref{state2} of Lemma~\ref{lem3}
holds.

\begin{subcase3.5}
Statement~\eqref{state6} of Lemma~{\rm\ref{lem3}} applies.
\end{subcase3.5}

In this case $M_{0}$ contains a pair of intersecting triangles.
Thus $M_{0}$ is isomorphic to one of the $13$ matroids
listed in Proposition~\ref{prop35}.
Since $M$ is not a single-element coextension of $M_{0}$, it
follows from Claim~\ref{clm17} that we can find an isomorphic
copy of $M$ by constructing all the $3$\dash connected single-element
coextensions of $M_{0}$ in \ex{\mkt}, and then extending
by one or two more elements, at each step considering only
the $3$\dash connected members of \ex{\mkt}.
When we perform this procedure for each of the matroids listed in
Proposition~\ref{prop35} we find that the only \ifc\ matroids
we uncover are isomorphic to $M_{5,11}$, $T_{12} / e$,
$M_{5,12}^{b}$, $M_{5,13}$, and $T_{12}$.\cross

\begin{subcase3.6}
Statement~\eqref{state9} of Lemma~{\rm\ref{lem3}} applies.
\end{subcase3.6}

In this case there is a triangle $\{x,\,y,\, z\}$ in $M$ such
that $M_{0} = M / x / y \del z$.
If $M / y \del z$ is not $3$\dash connected then $x$ is
contained in a cocircuit of size at most two in $M / y \del z$.
This means that there is a cocircuit of size at most three in
$M$, a contradiction.
Similarly, $M \del z$ is $3$\dash connected.
It follows from this argument that we can recover $M$ from
$M_{0}$ by coextending by two elements, and then extending
by a single element, at each step considering only
$3$\dash connected binary matroids with no \mkt\dash minor.
We perform this procedure on the matroids in Appendix~\ref{chp9},
other than $M_{7,15}$, $M_{9,18}$, and $M_{11,21}$.
The only \ifc\ matroids produced are isomorphic to $M_{5,11}$,
$T_{12} / e$, $T_{12}$, and $M_{7,15}$.\cross

Suppose that $M_{0} = M_{7,15}$.
Recall that $M_{7,15} = \nabla(F_{7}^{*};\, \mcal{T}_{5})$, where
$\mcal{T}_{5}$ is a set of five distinct triangles in $F_{7}$.
Let $M_{1}$ be the matroid derived from $F_{7}$ by adding
$t_{e} - 1$ parallel elements to each element $e$ in $F_{7}$.
Then \mcal{T} is the set of disjoint triangles in $M_{1}$ which
correspond to the triangles in $\mcal{T}_{5}$.
In this case $M_{7,15}^{*}$ is the matroid obtained from $M_{1}$ by
performing \dy\ operations on each of the triangles in \mcal{T}
in turn.

Suppose that $T_{1} = \{a_{1},\, b_{1},\, c_{1}\}$ and
$T_{2} = \{a_{2},\, b_{2},\, c_{2}\}$ are triangles of $M_{7,15}$.
Proposition~\ref{prop37} and the fact that $M_{7,15}$ contains
exactly five triangles means that $T_{1}$ and $T_{2}$ are members
of \mcal{T}.
Let $M'$ be the binary matroid obtained from $M_{7,15}$ by coextending with
the elements $x$ and $y$ such that $\{b_{1},\, c_{1},\, x\}$
and $\{b_{2},\, c_{2},\, y\}$ are triads, and then extending
with the element $z$ so that $\{x,\, y,\, z\}$ is a triangle.
Assume that $M'$ is \vfc.

Suppose that $a_{1}$ and $a_{2}$ are parallel elements in $M_{1}$.
Then $\{a_{1},\, a_{2}\}$ is a circuit in $M_{1} \del b_{1} \del b_{2}$.
But Proposition~\ref{prop28} implies that $M_{1} \del b_{1} \del b_{2}$
is isomorphic to
$\Delta_{T_{2}}(\Delta_{T_{1}}(M_{1})) / b_{1} / b_{2}$ under the
function that switches $a_{1}$ with $c_{1}$ and $a_{2}$ with $c_{2}$.
This implies that $\{b_{1},\, c_{1},\, b_{2},\, c_{2}\}$ contains
a circuit in $\Delta_{T_{2}}(\Delta_{T_{1}}(M_{1}))$, and hence
in $M_{7,15}^{*}$.
Since $M_{7,15}$ contains no cocircuits of size less than four it
follows that $\{b_{1},\, c_{1},\, b_{2},\, c_{2}\}$ is a cocircuit
of $M_{7,15}$.
Since $\{b_{1},\, c_{1},\, x\}$ and $\{b_{2},\, c_{2},\, y\}$ are 
triads of $M' \del z$ it follows that $x$ and $y$ are in series
in $M' \del z$, and that therefore $M'$ contains a cocircuit of
size at most three, a contradiction as $M'$ is \vfc.

Now we assume that $a_{1}$ and $a_{2}$ are not in parallel in $M_{1}$.
Suppose that we add $y$ in parallel to $a_{1}$ and $x$ in parallel
to $a_{2}$ in $M_{1}$.
Let $z'$ be an element of $M_{1}$ such that $\{x,\, y,\, z'\}$
is a triangle.
If $z'$ is contained in a non-trivial parallel class of $M_{1}$
then we obtain $M_{2}$ by adding $z$ in parallel to $z'$,
otherwise we simply relabel $z'$ with $z$.
Let $T$ be the triangle of $F_{7}$ that corresponds to the triangle
$\{x,\, y,\, z\}$ in $M_{2}$.
It follows easily from the dual of Lemma~\ref{lem23} that
$(M')^{*}$ is equal to the matroid obtained from $M_{2}$
by performing \dy\ operations on all the triangles in
$\mcal{T} \cup \{\{x,\, y,\, z\}\}$.
Thus $M'$ is equal to $\nabla(F_{7}^{*};\, \mcal{T}_{5} \cup \{T\})$.

If $T$ does not belong to $\mcal{T}_{5}$ then
$\mcal{T}_{5} \cup T$ is equal to $\mcal{T}_{6}$, a set of six
distinct triangles in $F_{7}$, and therefore $M'$ is isomorphic
to $M_{9,18}$.
If $T$ is already contained in $\mcal{T}_{5}$ then in $M_{2}$
there is a triangle $T'$ such that $T'$ and $\{x,\, y,\, z\}$ are
disjoint but $r_{M_{2}}(T' \cup \{x,\, y,\, z\}) = 2$.
It follows easily that in $(M')^{*}$ the rank of 
$T' \cup \{x,\, y,\, z\}$ is four.
However $T'$ and $\{x,\, y,\, z\}$ are both triads of $(M')^{*}$, and
therefore $T' \cup \{x,\, y,\, z\}$ is a $3$\dash separator of
$(M')^{*}$.
Hence $M'$ is not \ifc, so if $M_{0} = M_{7,15}$ then $M$ can only
be isomorphic to $M_{9,18}$.

By using a similar argument, we can show that if $M_{0} = M_{9,18}$,
then $M$ is equal to $M_{11,21}$.
However, if $M_{0} = M_{11,21}$, then
$M_{0} = \nabla(F_{7}^{*};\, \mcal{T}_{7})$.
As every triangle of $F_{7}$ is contained in $\mcal{T}_{7}$ it
follows that there is no \ifc\ binary matroid obtained from
$M_{0}$ by coextending with elements $x$ and $y$ so that they
form triads with two elements from triangles and then
extending by $z$ so that $\{x,\, y,\, z\}$ is a triangle.
This completes the case that statement~\eqref{state9} applies.

\begin{subcase3.7}
Statement~\eqref{state4} of Lemma~{\rm\ref{lem3}} applies.
\end{subcase3.7}

In this case $M_{0}$ contains two intersecting
triangles, and is therefore one of the matroids listed in
Proposition~\ref{prop35}.
There is a triangle $\{x,\, y,\ z\}$ in $M$ such that
$M_{0} = M / x / y \del z$.
As in Subcase~3.6, we can argue that $M \del z$ and
$M / y \del z$ are $3$\dash connected, so we can construct
$M$ from $M_{0}$ by coextending twice, and then extending once,
at each step considering only $3$\dash connected binary matroids
with no \mkt\dash minor.
When we perform this procedure on the matroids in
Proposition~\ref{prop35} the only \ifc\ matroids produced
are isomorphic to $M_{5,11}$, $T_{12} / e$, $T_{12}$, and
$M_{7,15}$.\cross

\begin{subcase3.8}
Statement~\eqref{state8} of Lemma~{\rm\ref{lem3}} applies.
\end{subcase3.8}

In this case we $M_{0}$ again contains two intersecting
triangles.
Furthermore, there is a triangle $\{x,\, y,\, z\}$ in
$M$ such that $M_{0} = M / x / y \del z$.
Therefore we can apply exactly the same arguments as in
Subcase~3.7 and obtain a contradiction.

\begin{subcase3.9}
Statement~\eqref{state5} of Lemma~{\rm\ref{lem3}} applies.
\end{subcase3.9}

As in Subcase~2.9 we look for pairwise disjoint allowable triangles
$T_{1}$, $T_{2}$, and $T_{3}$ in $M_{0}$ such that $T_{i} \cup T_{j}$
contains a cocircuit $C_{ij}^{*}$ of size four for $1 \leq i < j \leq 3$.
We call any such triple of allowable triangles a \emph{good triple}.
We then try to reconstruct $M$ from $M_{0}$ by coextending in turn
with the elements $e$, $f$, and $g$ so that
$(T_{1} \cap C_{12}^{*}) \cup e$, $(T_{2} \cap C_{23}^{*}) \cup f$, and
$(T_{3} \cap C_{13}^{*}) \cup g$, are triads, and finally
extending by $x$ so that $\{e,\, f,\, g,\, x\}$ is a circuit.

The proof sketched in Appendix~\ref{chp10} shows that only six of the matroids
listed in Appendix~\ref{chp9} have allowable triangles.
Proposition~\ref{prop47} makes it clear that $M_{4,11}$ contains
no good triple of allowable triangles.
Let $T_{1}$, $T_{2}$, $T_{3}$, and $T_{4}$ be the four allowable triangles
of $M_{5,12}^{a}$ in the order listed in Proposition~\ref{prop38}.
Any pair of these triangles contains a cocircuit of size four.
Omitting $T_{1}$, $T_{2}$, or $T_{3}$ from the set of allowable
triangles leaves us with a good triple, but performing the procedure
described above on this triple results in a matroid with an
\mkt\dash minor.\cross\
The only four-element cocircuit contained in $T_{1} \cup T_{2}$ is
$\{1,\, -1,\, -5,\, -6\}$ and the only four-cocircuit contained
in $T_{1} \cup T_{3}$ is $\{1,\, 5,\, -1,\, -4\}$.
Thus if we try to perform the procedure on the triple
$\{T_{1},\, T_{2},\, T_{3}\}$, we will be forced to coextend twice
by an element so that it forms a triad with $\{1,\, -1\}$.
This means that $M \del x$ will have a series pair, and that therefore
$M$ will have a cocircuit of size at most three, a contradiction.
This finishes the case-check for $M_{5,12}^{a}$.

Proposition~\ref{prop43} lists the four allowable triangles of $M_{6,13}$.
Any triple of these is a good triple, but performing the procedure
described above produces a matroid with an \mkt\dash minor.\cross

Let $T_{1},\ldots, T_{5}$ be the allowable triangles of $M_{7,15}$,
in the order listed in Proposition~\ref{prop44}.
Each pair contains a cocircuit of size four.
The only four-element cocircuit contained in $T_{1} \cup T_{2}$
is $\{1,\, -1,\, -5,\, -6\}$ and the only four-element cocircuit
contained in $T_{1} \cup T_{3}$ is $\{1,\, -1,\, -4,\, 5\}$, and both
these cocircuits meet $T_{1}$ in $\{1,\, -1\}$.
This means that performing the procedure on the triple 
$\{T_{1},\, T_{2},\, T_{3}\}$ results in a cocircuit of size at most three.
Similarly, the only four-element cocircuit contained in
$T_{2} \cup T_{4}$ is $\{3,\, 6,\, 7,\, -6\}$ and the only four-element
cocircuit contained in $T_{2} \cup T_{5}$ is $\{3,\, -2,\, -3,\, -6\}$,
and both these cocircuits meet $T_{2}$ in $\{3,\, -6\}$, so we
can dismiss the triple $\{T_{2},\, T_{4},\, T_{5}\}$.
Performing the procedure on any other triple of triangles in
$\{T_{1},\ldots, T_{5}\}$ produces a matroid with an
\mkt\dash minor.\cross

Suppose that $\{T_{1},\, T_{2},\, T_{3}\}$ is a triple of allowable
triangles in $M_{9,18}$ such that $T_{i} \cup T_{j}$ contains a
four-element cocircuit $C_{ij}^{*}$ for $1 \leq i < j \leq 3$.
Suppose that $M'$ is produced from $M_{9,18}$ by coextending by
$e$, $f$, and $g$, and then extending by $x$ is the way
described above.
Then $M' \del x / g / f / e = M_{9,18}$. 

Let $T$ be a triangle of $M_{9,18}$ that is not in
$\{T_{1},\, T_{2},\, T_{3}\}$.
It is not difficult to see that $T$ is a triangle in $M'$, for if
it is not then there is a circuit that meets one of the cocircuits
$(T_{1} \cap C_{12}^{*}) \cup \{e,\, x\}$,
$(T_{2} \cap C_{23}^{*}) \cup \{f,\, x\}$,
or $(T_{3} \cap C_{13}^{*}) \cup \{g,\, x\}$ in exactly one element.
Let $M'' = M' / T$.
Then Equation~\eqref{eqn4} tells us that
$M'' \del x / g / f / e = M_{7,15}$.

Suppose that $\{e,\, f,\, g,\, x\}$ is not a circuit in $M''$.
Then there is a circuit $C$ of $M'$ which has a non-empty
intersection with both $T$ and a proper subset of $\{e,\, f,\, g,\, x\}$.
If $x \notin C$ then, by relabeling if necessary, we assume that
$e \in C$, and therefore $C$ meets the cocircuit
$(T_{1} \cap C_{12}^{*}) \cup \{e,\, x\}$ of $M'$ in exactly one
element.
On the other hand, if $x \in C$, then we assume without loss
of generality that $e \notin C$, and we reach the same contradiction.
Thus $\{e,\, f,\, g,\, x\}$ is indeed a circuit of $M''$.
Moreover, $(T_{3} \cap C_{13}^{*}) \cup g$ is a triad in $M' \del x$,
and hence in $M'' \del x$.
We can use exactly the same arguments to show that
$(T_{2} \cap C_{23}^{*}) \cup f$ and $(T_{1} \cap C_{12}^{*}) \cup e$
are triads of $M'' \del x / g$ and $M'' \del x / g / f$ respectively.
If $T_{i}$ is not a triangle of $M'' \del x / g / f / e$ for some
$i \in \{1,\, 2,\, 3\}$, then $T_{i} \cup T$ has rank at
most three in $M' \del x / g / f / e = M_{9,18}$.
This would imply that $M_{9,18}$ contains a parallel pair, a
contradiction.
Thus $T_{i}$ is a triangle of $M'' \del x / g / f / e$ for
all $i \in \{1,\, 2,\, 3\}$.
Similarly, $C_{ij}^{*}$ is a cocircuit of
$M'' \del x / g / f / e$ for $1 \leq i < j \leq 3$.

The above arguments show that $M''$ can be obtained from
$M_{7,15}$ in exactly the same way that $M'$ was obtained from
$M_{9,18}$.
By the computer checking described earlier this means that $M''$
either has an \mkt\dash minor, or a cocircuit with at most three
elements.
If $M''$ has an \mkt\dash minor then so does $M'$, and similarly,
if $M''$ contains a cocircuit of size at most three, then so does $M'$.
These arguments mean that $M_{0}$ cannot be isomorphic $M_{9,18}$.
By using the same arguments we can show that $M_{0}$ cannot be equal
to $M_{11,21}$.

This completes the case that statement~\eqref{state5} applies.
With this we have completed the case-checking and we conclude that
the counterexample $M$ does not exist.
Therefore Theorem~\ref{thm6} holds.
\end{proof}

\appendix

\chapter{Case-checking}
\label{chp8}

The next proposition sketches the case-check
needed to complete the proof of Corollary~\ref{cor8}.

\begin{prop}
\label{prop25}
Suppose that $M \in \ex{\mkt}$ is a $3$\dash connected
matroid with a four-element circuit-cocircuit $C^{*}$ such that
$M$ has a $\Delta_{4}$\dash minor, but if $e \in E(M) - C^{*}$,
then neither $M \del e$ nor $M / e$ has a
$\Delta_{4}$\dash minor.
Then $M$ has a $\Delta_{4}^{+}$\dash minor.
\end{prop}

\begin{proof}
Let \mcal{M} be the class of labeled $3$\dash connected
matroids in \ex{\mkt,\, \Delta_{4}^{+}}.
Assume that $M \in \mcal{M}$ is a counterexample to
the proposition.
Thus $M$ has both a four-element circuit-cocircuit $C^{*}$ and
a $\Delta_{4}$\dash minor.
Let us first suppose that there is no element $x \in E(M)$
such that $M / x$ has a $\Delta_{4}$\dash minor.
It cannot be the case that $M \iso \Delta_{4}$, since
$\Delta_{4}$ has no four-element circuit-cocircuit.
Thus $M \del y$ has a $\Delta_{4}$\dash minor for some
$y \in C^{*}$.
But $C^{*} - y$ is a triad of $M \del y$ and $\Delta_{4}$ has
no triads.
Hence we must contract some element $x \in C^{*} - y$ from
$M \del y$ to obtain a $\Delta_{4}$\dash minor.
Thus there is an element $x \in C^{*}$ such that
$M / x$ has a $\Delta_{4}$\dash minor.

Suppose that $M / x$ has a $2$\dash separation
$(X,\, Y)$ such that $|X|,\, |Y| \geq 3$.
Then $x \in \cl_{M}(X) \cap \cl_{M}(Y)$.
Without loss of generality we will assume that
$X$ contains two elements of $C^{*} - x$.
Thus $C^{*} \subseteq \cl_{M}(X)$, and the assumption
$|X|$ and $|Y|$ means that
$(X \cup C^{*},\, Y - C^{*})$ is a
$2$\dash separation of $M$, a contradiction.
Thus if $(X,\, Y)$ is a $2$\dash separation of
$M / x$ then without loss of generality $|X| = 2$.
In this case there is a triangle $T$ of
$M$ which contains $x$.
It must be the case that $T$ contains a single
element $t$ of $E(M) - C^{*}$.
But $t$ is in a parallel pair in $M / x$, and as
$\Delta_{4}$ has no parallel pairs it follows that
$M / x \del t$ has a $\Delta_{4}$\dash minor.
This is a contradiction as $t \notin C^{*}$.
Therefore $M / x$ is $3$\dash connected.

Suppose that $\Delta_{4}$ is represented over \gf{2} by the
matrix $[I_{4}|A]$, where $A$ is the matrix shown in
Figure~\ref{fig1}.
Consider all the binary matrices obtained by adding a column
to $[I_{4}|A]$.
Suppose that this new column is labeled by the element $e$.
Let $\mathrm{EX}$ be the set of labeled matroids
in \mcal{M} represented over \gf{2} by such a matrix.
Ignoring isomorphism, there are five matroids in
$\mathrm{EX}$ (recall that the members of \mcal{M} are
$3$\dash connected).\cross\
Every single-element extension of $\Delta_{4}$ in \mcal{M} is
isomorphic to a matroid in $\mathrm{EX}$.
Similarly, let $\mathrm{CO}$ be the set of matroids
in \mcal{M} that are represented over \gf{2} by a matrix
$[I_{5}|A']$, where $A'$ is obtained by
adding a row to $A$.
Again we will assume that every matroid in $\mathrm{CO}$ is
a coextension of $\Delta_{4}$ by the element $e$.
There are fifteen matroids in $\mathrm{CO}$, ignoring
isomorphisms.\cross

None of the matroids in $\mathrm{CO}$ has a four-element
circuit-cocircuit, so $M / x \ncong \Delta_{4}$.\cross\
By Theorem~\ref{thm3} there is an element $y \in E(M / x)$
such that either $M / x \del y$ or $M / x / y$ is $3$\dash connected
with a $\Delta_{4}$\dash minor.
Then $y$ must be in $C^{*} - x$.
Contracting any element of $C^{*} - x$ in $M / x$ creates a
parallel pair, so $M / x \del y$ is $3$\dash connected with a
$\Delta_{4}$\dash minor.

Assume that $M / x \del y \iso \Delta_{4}$.
Then $M / x$ is isomorphic to a member of
$\mathrm{EX}$, and this isomorphism takes $y$ to $e$.
Thus $e$ appears in a triangle.
In three of the five matroids in $\mathrm{EX}$,
$e$ appears in four triangles, and in two $e$ appears in
three triangles.
Given a binary matroid $N$ and a triangle $T$ of $N$, there is a
unique coextension of $N$ by the element $x$ such that
$T \cup x$ is a four-element circuit-cocircuit of the
coextension.
Thus for each matroid $N \in \mathrm{EX}$
and each triangle of $N$ that contains $e$ we construct
the corresponding unique coextension of $N$.
One of the resulting matroids is isomorphic to $M$.
However, each of these eighteen matroids has either
an \mkt\dash minor or a $\Delta_{4}^{+}$\dash minor.\cross\
We conclude that $M / x \del y \ncong \Delta_{4}$.

We will next suppose that there is some element
$z \in C^{*} - \{x,\, y\}$ such that
$M / x \del y / z \iso \Delta_{4}$.
Then $M / x \del y$ is isomorphic to a member of
$\mathrm{CO}$.
Thus, for each matroid $N \in \mathrm{CO}$ we consider all the
extensions of $N$ by the element $f$ such that the extension
belongs to \mcal{M}.
If $e$ and $f$ appear together in a triangle in the resulting
matroid we construct the unique coextension that creates a four-element
circuit-cocircuit.
One of the resulting matroids is isomorphic to $M$.

The fifteen members of $\mathrm{CO}$
each have either six or seven single-element extensions
that belong to \mcal{M} (ignoring isomorphisms), and in each case,
either five or six of these extensions have triangles containing
both $e$ and $f$.
In total there are $78$ candidate to check, but each of these
has a $\Delta_{4}^{+}$\dash minor.\cross

Next we assume that
$M / x \del y \del z \iso \Delta_{4}$ for some element
$z \in C^{*} - \{x,\, y\}$.
For each matroid $N \in \mathrm{EX}$ we consider the
extensions of $N$ belonging to \mcal{M} by the element~$f$.
If $e$ and $f$ are contained in a triangle
of the resulting matroid we construct the corresponding
unique coextension.
Each of the five matroids in $\mathrm{EX}$ has four
single-element extensions in \mcal{M} (ignoring isomorphisms).
Three of the matroids in $\mathrm{EX}$ have two extensions each,
in which the two new elements are contained in a triangle.
Every extension of the other two matroids in
$\mathrm{EX}$ has a triangle containing the
two new elements.
This leads to a total of $14$ matroids to be checked.
Each of the $14$ has an \mkt\dash minor.\cross

Let $C^{*} - \{x,\, y\} = \{z,\, w\}$.
There are four remaining cases to check.
In the first $M / x \del y / z$ is $3$\dash connected and
$M / x \del y / z / w \iso \Delta_{4}$.
For each matroid $N \in \mathrm{CO}$ we
construct the coextensions by the element $f$ that
belong to \mcal{M}.
We then extend so that the new element makes a triangle with
$e$ and $f$, and then construct the unique coextension that
creates a four-element circuit-cocircuit.
The fifteen matroids in $\mathrm{CO}$ have either zero, six, or
eight coextensions in \mcal{M}.
There are $84$ candidates to check.\cross

In the second of the four cases $M / x \del y / z$
is $3$\dash connected and
$M / x \del y / z \del w \iso \Delta_{4}$.
For each matroid $N \in \mathrm{EX}$ we
construct the coextensions of $N$ by the element $f$ that
belong to \mcal{M}.
In no such coextension are $e$ and $f$ contained in a triangle,
so we proceed as in the previous paragraph.
Three of the five matroids in $\mathrm{EX}$ have nine
single-element coextensions in \mcal{M} each.
The other two have no such coextensions.
Thus there are $27$ candidates to check.\cross

In the next case we assume that $M / x \del y \del z$
is $3$\dash connected and
$M / x \del y \del z / w \iso \Delta_{4}$.
For each matroid $N \in \mathrm{CO}$
we construct the extensions of $N$ belonging to
\mcal{M} by the element $f$.
If $e$ and $f$ are contained in a triangle of the resulting
matroid then we need proceed no further, since $y$ is not
contained in a parallel pair of $M / x$.
Otherwise we proceed as in the previous two cases.
The fifteen matroids in $\mathrm{CO}$ have
six or seven extensions each that belong to
\mcal{M}.
In each case either one or two of these extensions does not have
a triangle containing both $e$ and $f$.
There are $21$ candidates to check.\cross

Finally we assume that $M / x \del y \del z$
is $3$\dash connected and
$M / x \del y \del z \del w \iso \Delta_{4}$.
Each of the five matroids in $\mathrm{EX}$
has four single-element extensions in \mcal{M}.
In three of the five cases exactly two of the
extensions do not have triangles containing both
$e$ and $f$.
In the other two cases all of the extensions have triangles containing
$e$ and $f$.
This leads to a total of $6$ candidates to check.
Each of these has either an \mkt\dash minor or a
$\Delta_{4}^{+}$\dash minor.\cross\
Thus we have exhausted the supply of matroids that could
potentially be isomorphic to $M$.
We conclude that no counterexample exists and that the
result holds. 
\end{proof}

\chapter{Sporadic matroids}
\label{chp9}

There are $27$ \ifc\ non-cographic
matroids in \ex{\mkt} with rank at most $11$ and ground
sets of size at most $21$.
They are categorized in Tables~\ref{tbl1}
and~\ref{tbl2}.
Five of these matroids are triangular \mob\ matroids and
four are triadic \mob\ matroids (recall that $\Delta_{3}$ and
$\Upsilon_{4}$ are isomorphic to the Fano plane and its
dual respectively).
The other $18$ matroids are sporadic
members of the class.
All of the matroids listed in this appendix are \vfc,
with the exception of $M_{5,11}$, which contains a
single triad.\cross

In the following pages we give matrix representations of these
sporadic matroids.
If $A$ is one of the matrices displayed in Figures~\ref{fig10}
to~\ref{fig20} then $[I_{r(M)}|A]$ is a
\gf{2}\dash representation of a sporadic matroid $M$.\newpage

\begin{table}[H]
\begin{center}
\begin{tabular}{|r|ccccc|}
\hline
size&&&&&\\
$16$&&&&$\Delta_{6}$&\\
$15$&&\pg{3,\, 2}&&&\\
$14$&&$M_{4,14}$&&&\\
$13$&&$M_{4,13}$&$\Delta_{5},\,M_{5,13}$&$M_{6,13}$&\\
$12$&&$C_{12},\, D_{12}$&$M_{5,12}^{a},\, M_{5,12}^{b}$&
$T_{12}$&\\
$11$&&$C_{11},\, M_{4,11}$&$M_{5,11},\, T_{12} / e$&
$\Upsilon_{6}$&\\
$10$&&$\Delta_{4},\, \mkf$&&&\\
$9$&&&&&\\
$8$&&&&&\\
$7$&$F_{7}$&$F_{7}^{*}$&&&\\
\hline
&$3$&$4$&$5$&$6$&rank\\
\hline
\end{tabular}
\end{center}
\caption{Internally $4$\protect \dash connected non-cographic matroids
in \ex{\mkt}.}
\label{tbl1}
\end{table}

\begin{table}[H]
\begin{center}
\begin{tabular}{|r|cccccc|}
\hline
size&&&&&&\\
$21$&&&&&$M_{11,21}$&\\
$20$&&&&&&\\
$19$&$\Delta_{7}$&&&$\Upsilon_{10}$&&\\
$18$&&&$M_{9,18}$&&&\\
$17$&&&&&&\\
$16$&&&&&&\\
$15$&$M_{7,15}$&$\Upsilon_{8}$&&&&\\
\hline
&$7$&$8$&$9$&$10$&$11$&rank\\
\hline
\end{tabular}
\end{center}
\caption{Internally $4$\protect \dash connected non-cographic matroids
in \ex{\mkt}.}
\label{tbl2}
\end{table}\newpage

\begin{figure}[H]

\centering

\includegraphics{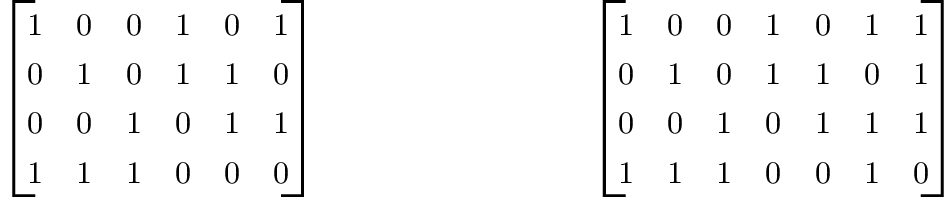}

\caption{Matrix representations of \mkf\ and $C_{11}$.}

\label{fig10}

\end{figure}

\begin{figure}[H]

\centering

\includegraphics{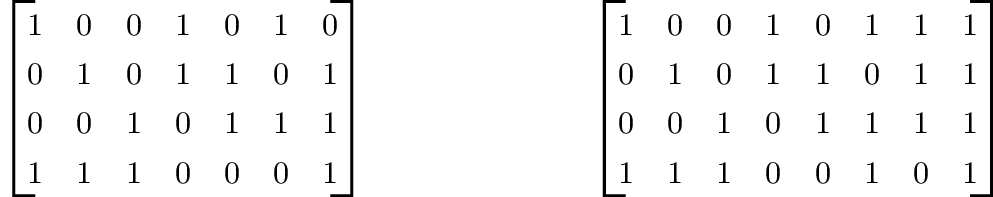}

\caption{Matrix representations of $M_{4,11}$ and $C_{12}$.}

\label{fig12}

\end{figure}

$C_{12}$ is isomorphic to the matroid produced by deleting a set
of three collinear points from \pg{3,\, 2}.

\begin{figure}[H]

\centering

\includegraphics{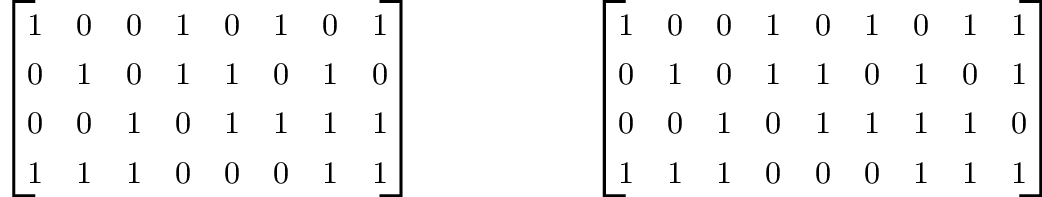}

\caption{Matrix representations of $D_{12}$ and $M_{4,13}$.}

\label{fig14}

\end{figure}

$D_{12}$ is isomorphic to the matroid produced by deleting
a set of three non-collinear points from \pg{3,\, 2}.

\begin{figure}[H]

\centering

\includegraphics{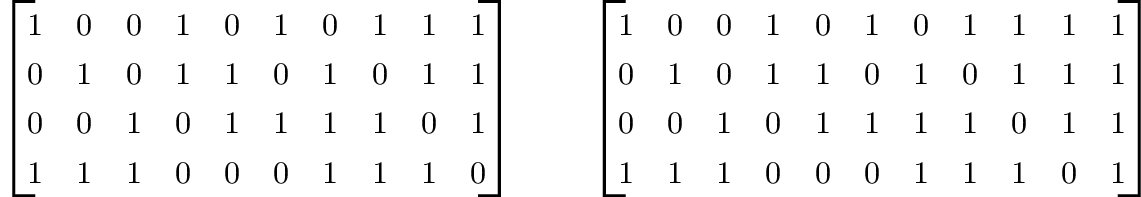}

\caption{Matrix representations of $M_{4,14}$ and \pg{3,\, 2}.}

\label{fig11}

\end{figure}

\begin{figure}[H]

\centering

\includegraphics{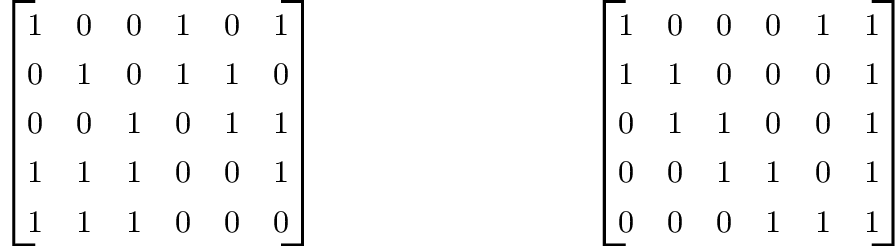}

\caption{Matrix representations of $M_{5,11}$ and $T_{12} / e$.}

\label{fig15}

\end{figure}

$M_{5,11}$ is not \vfc.
It contains exactly one triad: the set $\{4,\, 5,\, -6\}$,
using the convention that if $A$ is the matrix given so that
$M_{5,11}$ is represented by $[I_{5}|A]$, then the columns of
$I$ are labeled by $1,\ldots, 5$, and the columns of $A$ are
labeled by $-1,\ldots, -6$.

\begin{figure}[H]

\centering

\includegraphics{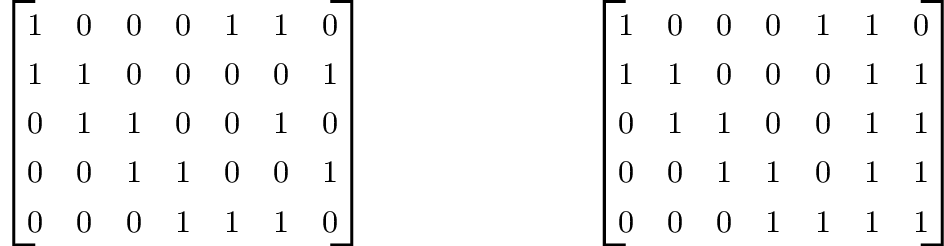}

\caption{Matrix representations of $M_{5,12}^{a}$ and $M_{5,12}^{b}$.}

\label{fig17}

\end{figure}

$M_{5,12}^{a}$ is isomorphic to $\nabla(F_{7}^{*};\, \mcal{T}_{\,\,4}^{a})$,
where $\mcal{T}_{\,\,4}^{a}$ is a set of four triangles in the Fano
plane, three of which contain a common point.

\begin{figure}[H]

\centering

\includegraphics{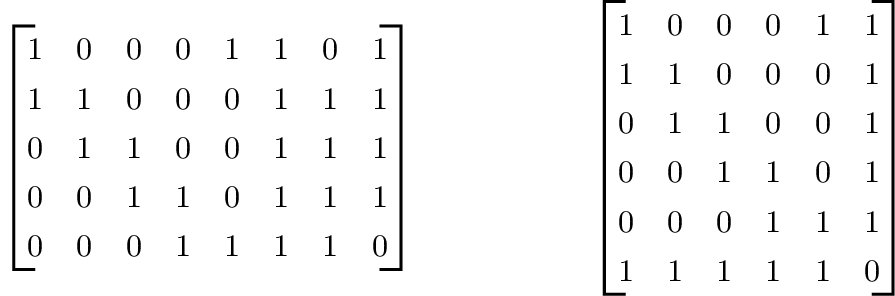}

\caption{Matrix representations of $M_{5,13}$ and $T_{12}$.}

\label{fig19}

\end{figure}

\begin{figure}[H]

\centering

\includegraphics{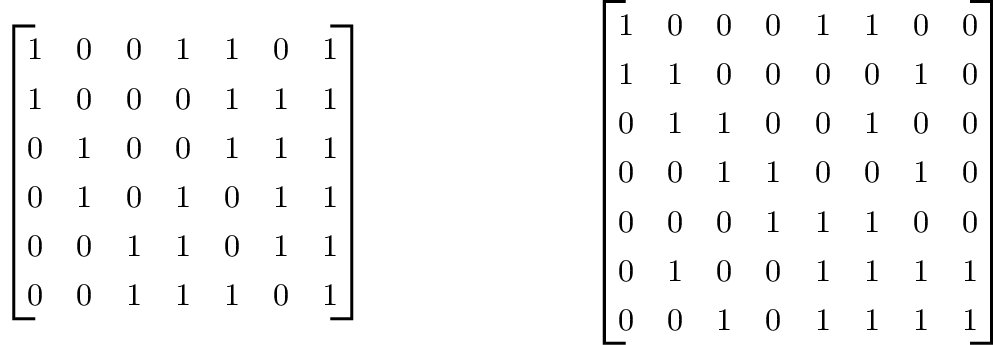}

\caption{Matrix representations of $M_{6,13}$ and $M_{7,15}$.}

\label{fig21}

\end{figure}

$M_{6,13}$ is isomorphic to $\nabla(F_{7}^{*};\, \mcal{T}_{\,\,4}^{b})$,
where $\mcal{T}_{\,\,4}^{b}$ is a set of four triangles in the
Fano plane, no three of which contain a common point, and
$M_{7,15}$ is isomorphic to $\nabla(F_{7}^{*};\, \mcal{T}_{5})$,
where $\mcal{T}_{5}$ is a set of five triangles in the
Fano plane.

\begin{figure}[H]

\centering

\includegraphics{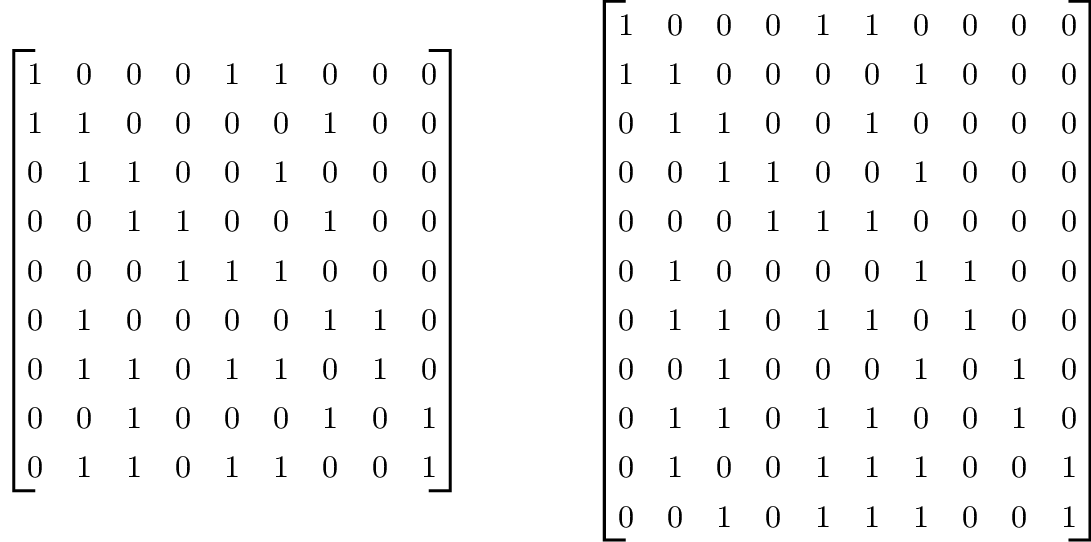}

\caption{Matrix representations of $M_{9,18}$ and
$M_{11,21}$.}

\label{fig20}

\end{figure}

$M_{9,18}$ is isomorphic to $\nabla(F_{7}^{*};\, \mcal{T}_{6})$, where
$\mcal{T}_{6}$ is a set of six triangles in the Fano plane, and
$M_{11,21}$ is isomorphic to $\nabla(F_{7}^{*};\, \mcal{T}_{7})$,
where $\mcal{T}_{7}$ is the set of all seven lines in
the Fano plane.

\chapter{Allowable triangles}
\label{chp10}

Recall that if $T$ is a triangle of the matroid
$M \in \ex{\mkt}$, and $\Delta_{T}(M)$ has no \mkt\dash minor,
then $T$ is an allowable triangle of $M$.
In this appendix we consider the sporadic matroids listed
in Appendix~\ref{chp9}, and we give an outline of how a computer
check can determine all their allowable triangles.

Suppose that a sporadic matroid $M$ is represented in
Appendix~\ref{chp9} by the matrix $[I|A]$.
We adopt the convention that the columns of $I$ are labeled
by the positive integers $1,\ldots, r(M)$, while the columns
of $A$ are labeled by the negative integers
$-1,\ldots, -r(M^{*})$.

The next results can be checked by computer.\cross

\begin{prop}
\label{prop47}
There are~$13$ triangles in $M_{4,11}$.
Of these,~$3$ are allowable:
$\{2,\, -3,\, -7\}$, $\{3,\, -2,\, -7\}$, and
$\{4,\, -5,\, -7\}$.
\end{prop}

\begin{prop}
\label{prop38}
There are~$8$ triangles in $M_{5,12}^{a}$.
Of these~$4$ are allowable:
$\{1,\, 2,\, -1\}$, $\{3,\, -5,\, -6\}$, $\{4,\, 5,\, -4\}$
and $\{-2,\, -3,\, -7\}$.
Any pair of these triangles contains a cocircuit of
size four. 
\end{prop}

\begin{prop}
\label{prop43}
There are~$4$ triangles in $M_{6,13}$:
$\{1,\, 2,\, -1\}$, $\{3,\, 4,\, -2\}$,
$\{5,\, 6,\, -3\}$, and $\{-4,\, -5,\, -6\}$.
Each of these is allowable.
Any pair of these triangles contains a cocircuit of size
four.
\end{prop}

\begin{prop}
\label{prop44}
There are~$5$ triangles in $M_{7,15}$:
$\{1,\, 2,\, -1\}$, $\{3,\, -5,\, -6\}$,
$\{4,\, 5,\, -4\}$, $\{6,\, 7,\, -8\}$, and
$\{-2,\, -3,\, -7\}$.
Each of these is allowable.
Any pair of these triangles contains a cocircuit of
size four.
\end{prop}

\begin{prop}
\label{prop8}
There are~$6$ triangles in $M_{9,18}$:
$\{1,\, 2,\, -1\}$, $\{3,\, -5,\, -6\}$,
$\{4,\, 5,\, -4\}$, $\{6,\, 7,\, -8\}$,
$\{8,\, 9,\, -9\}$, and $\{-2,\, -3,\, -7\}$.
Each of these is allowable.
Any pair of these triangles contains a cocircuit of
size four.
\end{prop}

\begin{prop}
\label{prop32}
There are~$7$ triangles in $M_{11,21}$:
$\{1,\, 2,\, -1\}$, $\{3,\, -5,\, -6\}$,
$\{4,\, 5,\, -4\}$, $\{6,\, 7,\, -8\}$,
$\{8,\, 9,\, -9\}$, $\{10,\, 11,\, -10\}$,
and $\{-2,\, -3,\, -7\}$.
Each of these is allowable.
Any pair of these triangles contains a cocircuit of
size four.
\end{prop}

Any sporadic matroid that is not listed in
Propositions~\ref{prop47} to~\ref{prop32} contains no
allowable triangles.
We now sketch a proof of this fact.
We will make repeated use of the fact that if the triangle
$T$ is not allowable in the matroid $N$, and $N'$ has
$N$ as a minor, then $T$ is not allowable in $N'$.
This follows immediately from Proposition~\ref{prop30}.

There are~$12$ triangles in $C_{11}$, and we can check
by computer than none of them is allowable.\cross\
Any matroid produced from $C_{12}$ by deleting an
element is isomorphic to $C_{11}$.\cross\
Therefore any triangle of $C_{12}$ is not allowable in
a minor of $C_{12}$, and hence not allowable in
$C_{12}$ itself.
Thus $C_{12}$ has no allowable triangles.

Performing a \dy\ operation on \mkf\ produces a non-planar
graphic matroid, since the \dy\ and \yd\ operations
preserves planarity.
Thus performing a \dy\ operation on any triangle of
\mkf\ produces a matroid with an \mkt\dash minor.
Hence \mkf\ has no allowable triangles.

There are~$17$ triangles in $D_{12}$, and
we can check by computer that none is allowable.\cross\
Deleting any element other than~$4$ from $M_{4,13}$
produces a matroid isomorphic to $D_{12}$.\cross\
It follows that $M_{4,13}$ has no allowable triangles.
Deleting any element at all from $M_{4,14}$ produces a
minor isomorphic to $M_{4,13}$.\cross\
Therefore there are no allowable triangles in $M_{4,13}$.
Similarly, it is obviously true that deleting any element
from \pg{3,\, 2} produces a minor isomorphic to
$M_{4,14}$.
Thus \pg{3,\, 2} has no allowable triangles.

There are four triangles in $M_{5,11}$, and five in
$T_{12} / e$.
None of these is allowable.\cross\
Next we note that $M_{5,12}^{b}$ is an extension
of $T_{12} / e$ by the element $-7$.
Therefore any triangle that does not contain $-7$ is
not allowable.
The two triangles that contain~$-7$ are
$\{1,\, -6,\, -7\}$ and $\{-2,\, -4,\, -7\}$.
Neither of these is allowable.\cross

The matroid produced from $M_{5,13}$ by deleting any
of the elements in $\{1,\, 5,\, -7,\, -8\}$ is isomorphic to
$M_{5,12}^{b}$.\cross\
Any triangle avoids at least one of these elements,
so $M_{5,13}$ has no allowable triangles.

Finally, $T_{12}$ has no triangles.\cross

The next result can be checked by computer.\cross

\begin{prop}
\label{prop35}
The only matroids listed in Appendix~{\rm\ref{chp9}} that contain a
pair of intersecting triangles are: \mkf, $C_{11}$, $M_{4,11}$,
$C_{12}$, $D_{12}$, $M_{4,13}$, $M_{4,14}$, \pg{3,\, 2},
$M_{5,11}$, $T_{12} / e$, $M_{5,12}^{a}$, $M_{5,12}^{b}$, and $M_{5,13}$.
\end{prop}

The only sporadic matroids not listed in this previous
result are: $T_{12}$, $M_{6,13}$, $M_{7,15}$, $M_{9,18}$, $M_{11,21}$.

\backmatter

\providecommand{\bysame}{\leavevmode\hbox to3em{\hrulefill}\thinspace}

\begin{theindex}

  \item $2$\dash sum, 10

  \indexspace

  \item allowable triangle, 25

  \indexspace

  \item basepoint (of a $2$\dash sum), 10
  \item block, 20
  \item block cut-vertex graph, 20
  \item blocking sequence, 7

  \indexspace

  \item cofan, 6
  \item connectivity
    \subitem $(n,\, k)$\dash, 5
    \subitem $n$\dash, 5
    \subitem almost vertical $4$\dash, 5
    \subitem graph, 20
    \subitem internal $4$\dash, 5
    \subitem vertical $n$\dash, 5
  \item connectivity function, 5
  \item cut-vertex, 20
  \item cycle minor, 20
  \item cyclomatic number, 20

  \indexspace

  \item $\Delta_{4}^{+}$, 43
  \item $\Delta\textrm{-}Y$ operation, 16
  \item $\Delta(M;\, \mcal{T})$, $\nabla(M^{*};\, \mcal{T})$, 24

  \indexspace

  \item edge cut-set, 20

  \indexspace

  \item fan, 6
    \subitem length, 6
  \item fundamental graph, 7

  \indexspace

  \item good element (of a small separator), 49

  \indexspace

  \item legitimate set, 37

  \indexspace

  \item M\"{o}bius ladder
    \subitem cubic, 27
    \subitem quartic, 27
  \item M\"{o}bius matroid
    \subitem triadic, 28
    \subitem triangular, 27
  \item modular flat, 10

  \indexspace

  \item parallel connection, 10
  \item pivoting on an edge, 7

  \indexspace

  \item rim edges, 27
  \item rim elements, 28, 29

  \indexspace

  \item separation
    \subitem $k$\dash, 5
    \subitem exact $k$\dash, 5
    \subitem induced $k$\dash, 8
    \subitem vertical $k$\dash, 5
  \item separator
    \subitem $k$\dash, 5
    \subitem maximal small $3$\dash, 48
    \subitem maximal small vertical $3$\dash, 48
    \subitem small $3$\dash, 48
    \subitem small vertical $3$\dash, 48
    \subitem vertical $k$\dash, 5
  \item spanning triad, 19
  \item splitter, 9
  \item spoke edges, 27
  \item spoke elements, 28, 29

  \indexspace

  \item $T_{12}$, 9
  \item tip, 28, 29
  \item triad, 5
  \item triangle, 5

\end{theindex}

\end{document}